\documentclass[12pt]{amsart}

\usepackage[usenames]{color}
\usepackage{fullpage,url,mathrsfs,stmaryrd}
\usepackage{amsmath,amsthm,amssymb,mathrsfs,stmaryrd,color}
\usepackage[utf8]{inputenc}
\usepackage[T1]{fontenc}

\usepackage{relsize}
\usepackage[bbgreekl]{mathbbol}
\usepackage{amsfonts}

\DeclareSymbolFontAlphabet{\mathbb}{AMSb}
\DeclareSymbolFontAlphabet{\mathbbl}{bbold}
\newcommand{\prism}{{\mathlarger{\mathbbl{\Delta}}}}
\newcommand{\prismp}{{\prism'}}
\newcommand{\prismpp}{{\prism''}}
\newcommand{\nodal}{{ \mathlarger{\mathlarger{\propto}}}}

\newcommand{\BA}{{\mathbb{A}}}
\newcommand{\BP}{{\mathbb{P}}}

\newcommand{\SSigma}{{^!\Sigma}}

\newcommand{\BG}{{\mathbb{G}}}

\usepackage{amssymb}
\usepackage[all]{xy}
\usepackage{mathrsfs}
\usepackage{enumerate}
\usepackage{amscd}
\usepackage{eucal}

\usepackage{bbm}

\DeclareMathOperator{\colim}{{colim}}
\DeclareMathOperator{\llax}{{llax}}
\DeclareMathOperator{\Func}{{Func}}

\DeclareMathOperator{\Pol}{{Pol}}
\DeclareMathOperator{\Grpds}{{Grpds}}

\DeclareMathOperator{\Adm}{{Adm}}

\DeclareMathOperator{\Coeq}{{Coeq}}
\DeclareMathOperator{\Eq}{{Eq}}

\DeclareMathOperator{\pre}{{pre}}
\DeclareMathOperator{\Pre-div}{{Pre-div}}
\DeclareMathOperator{\Div}{{Div}}
\DeclareMathOperator{\Divp}{{\Div_+}}

\DeclareMathOperator{\lines}{{Lines}}

\DeclareMathOperator{\Ob}{{Ob}}
\DeclareMathOperator{\Sch}{{Sch}}

\DeclareMathOperator{\RStacks}{{RingStacks}}

\newcommand{\HHom}{\underline{\on{Hom}}}
\newcommand{\Ex}{\underline{\on{Ex}}}
\newcommand{\MMor}{\underline{\on{Mor}}}

\newcommand{\ppr}{\mathrm{pr}}

\newcommand{\cA}{{\mathcal A}}
\newcommand{\cB}{{\mathcal B}}
\newcommand{\cC}{{\mathcal C}}
\newcommand{\cD}{{\mathcal D}}
\newcommand{\cE}{{\mathcal E}}
\newcommand{\cF}{{\mathcal F}}
\newcommand{\cG}{{\mathcal G}}

\newcommand{\cI}{{\mathcal I}}

\newcommand{\cM}{{\mathcal M}}
\newcommand{\cN}{{\mathcal N}}
\newcommand{\cO}{{\mathcal O}}

\newcommand{\cQ}{{\mathcal Q}}

\newcommand{\cW}{{\mathcal W}}

\newcommand{\sA}{{\mathscr A}}

\newcommand{\sE}{{\mathscr E}}

\newcommand{\sG}{{\mathscr G}}
\newcommand{\sH}{{\mathscr H}}

\newcommand{\sK}{{\mathscr K}}
\newcommand{\sL}{{\mathscr L}}
\newcommand{\sM}{{\mathscr M}}
\newcommand{\sN}{{\mathscr N}}

\newcommand{\sP}{{\mathscr P}}

\newcommand{\sR}{{\mathscr R}}
\newcommand{\sS}{{\mathscr S}}
\newcommand{\sT}{{\mathscr T}}
\newcommand{\sU}{{\mathscr U}}

\newcommand{\sX}{{\mathscr X}}
\newcommand{\sY}{{\mathscr Y}}
\newcommand{\sZ}{{\mathscr Z}}

\newcommand{\fD}{{\mathfrak D}}

\newcommand{\fG}{{\mathfrak G}}

\newcommand{\fM}{{\mathfrak M}}

\newcommand{\fQ}{{\mathfrak Q}}

\newcommand{\fS}{{\mathfrak S}}

\newcommand{\fa}{{\mathfrak a}}

\newcommand{\fm}{{\mathfrak m}}

\newcommand{\fp}{{\mathfrak p}}

\newcommand{\nc}{\newcommand}

\nc\wh{\widehat}

\nc\on{\operatorname}

\nc\Gr{\on{Gr}}

\nc\Fl{\on{Fl}}

\newtheorem{cor}[subsubsection]{Corollary}
\newtheorem{lem}[subsubsection]{Lemma}
\newtheorem{prop}[subsubsection]{Proposition}

\newtheorem{thm}[subsubsection]{Theorem}

\newtheorem{quest}[subsubsection]{Question}

\theoremstyle{remark}
\newtheorem{rem}[subsubsection]{Remark}

\DeclareMathOperator{\BL}{{BL}}

\newcommand{\BF}{{\mathbb{F}}}
\newcommand{\BN}{{\mathbb{N}}}
\newcommand{\BQ}{{\mathbb{Q}}}
\newcommand{\BR}{{\mathbb{R}}}
\newcommand{\BV}{{\mathbb{V}}}
\newcommand{\BZ}{{\mathbb{Z}}}

\DeclareMathOperator{\Isom}{{Isom}}

\DeclareMathOperator{\Pic}{{Pic}}
\DeclareMathOperator{\red}{{red}} \DeclareMathOperator{\Spf}{{Spf}}  

\DeclareMathOperator{\Sym}{{Sym}}

\DeclareMathOperator{\Cone}{{Cone}}

\newcommand{\limto}{{\displaystyle\lim_{\longrightarrow}}}
\newcommand{\rightlim}{\mathop{\limto}}

\newcommand{\leftlim}{\mathop{\displaystyle\lim_{\longleftarrow}}}
\newcommand{\limfromn}{\leftlim\limits_{\raise3pt\hbox{$n$}}}
\newcommand{\limton}{\rightlim\limits_{\raise3pt\hbox{$n$}}}

\newcommand{\rightlimit}[1]{\mathop{\lim\limits_{\longrightarrow}}\limits%
                    _{\raise3pt\hbox{$\scriptstyle #1$}}}

\newcommand{\leftlimit}[1]{\mathop{\lim\limits_{\longleftarrow}}\limits%
                    _{\raise3pt\hbox{$\scriptstyle #1$}}}

\newcommand{\epi}{\twoheadrightarrow}
\newcommand{\iso}{\buildrel{\sim}\over{\longrightarrow}}
\newcommand{\mono}{\hookrightarrow}

\DeclareMathOperator{\prim}{{prim}}
\DeclareMathOperator{\Rig}{{Rig}}

\DeclareMathOperator{\RRig}{{{\sR}ig}}

\DeclareMathOperator{\Aut}{{Aut}}
\DeclareMathOperator{\Coker}{{Coker}}

\DeclareMathOperator{\dR}{{dR}}
\DeclareMathOperator{\HT}{{HT}}

\DeclareMathOperator{\bardR}{{\overline{dR}}}

\DeclareMathOperator{\Hdg}{{Hdg}}

\DeclareMathOperator{\End}{{End}}

\DeclareMathOperator{\Hom}{{Hom}}

\DeclareMathOperator{\Ker}{{Ker}} \DeclareMathOperator{\id}{{id}}
\DeclareMathOperator{\im}{{Im}} 
\DeclareMathOperator{\Inv}{{Inv}}

\DeclareMathOperator{\Mor}{{Mor}}
 \DeclareMathOperator{\op}{{op}}

\DeclareMathOperator{\Spec}{{Spec}}
\DeclareMathOperator{\Teich}{{Tei}}

\DeclareMathOperator{\WCart}{{WCart}}

\DeclareMathOperator{\Res}{{Res}}

\DeclareMathOperator{\Syn}{{Syn}}

\theoremstyle{definition}

\newtheorem{defin}[subsubsection]{Definition}

\newtheorem{ex}[subsubsection]{Example}

\numberwithin{equation}{section}

\newcommand{\Fr}{\operatorname{Fr}}

\begin{document}
\title[Prismatization]{Prismatization}
\author{Vladimir Drinfeld}
\address{University of Chicago, Department of Mathematics, Chicago, IL 60637}
\email{drinfeld@math.uchicago.edu}

\begin{abstract}
The eventual goal is to construct three related ``prismatization'' functors from the category of $p$-adic formal schemes to that of formal stacks. This should provide a good category of coefficients for prismatic cohomology in the spirit of $F$-gauges.

In this article we define and study the three versions of the prismatization of $\Spf\BZ_p$.
\end{abstract}


\maketitle

\tableofcontents

\section{Introduction}
Let $p$ be a prime.

In their remarkable work \cite{BS} B.~Bhatt and P.~Scholze  introduced \emph{prismatic cohomology}. This is a new and very powerful cohomology theory for $p$-adic formal schemes.  In the case of $\BF_p$-schemes (which is a particular type of $p$-adic formal schemes) prismatic cohomology turns out to be equal to crystalline cohomology, see \cite[\S 5]{BS}.

B.~Bhatt and J.~Lurie realized that the theory of \cite{BS} has a stacky reformulation\footnote{For a gentle introduction to this circle of ideas, see \cite{My IAS talk}.}; it is based on a certain functor $X\mapsto X^\prism$ defined on the category of bounded $p$-adic formal schemes, which is called \emph{prismatization}\footnote{In \cite{BL2} this functor is denoted by $X\mapsto\WCart_X$. }. 
This functor admits certain refinements\footnote{Presumably, the refinements produce a good theory of coefficients for prismatic cohomology (see \S\ref{ss:gauges and F-gauges}).}, which we denote by $X\mapsto X^\prismp$ and $X\mapsto X^\prismpp$.

This paper is devoted to a detailed study of the stacks
\[
\Sigma =(\Spf\BZ_p)^\prism , \quad \Sigma' =(\Spf\BZ_p)^\prismp ,\quad  \Sigma'' =(\Spf\BZ_p)^\prismpp .
\]
The construction of the functors
\[
X\mapsto X^\prism , \quad X\mapsto X^\prismp ,\quad X\mapsto X^\prismpp
\]
is sketched in \S\ref{ss:prismatization outline}-\ref{ss:X^prismp} below, but these functors are hardly mentioned outside the introduction (so the title of this article is somewhat misleading).

The functor $X\mapsto X^\prism$ is constructed in \cite{BL2}. The other two functors are constructed in the lecture notes \cite{Bh}.

The stack $\Sigma$ is also introduced and studied in \cite{BL} (but it is denoted there by $\WCart$). Let us note that $\Sigma$ is closely related to the earlier notion of the ``absolute prismatic site'' of $\Spf\BZ_p$ (in the sense of \cite[Rem.~4.7]{BS}).

The notation of \cite{Bh} is different from ours and from the one used in \cite{BL,BL2}. The relation between our notation and the more meaningful notation from the recent text \cite{Bh} is as follows:
\[
X^\prismp =X^\cN ,\quad X^\prismpp =X^{\Syn} , \quad \Sigma=\BZ_p^\prism ,\quad \Sigma'=\BZ_p^\cN ,\quad \Sigma''=\BZ_p^{\Syn}.
\]

\subsection{Format of the Bhatt-Lurie approach to prismatic theory}    \label{ss:format of stacky}
The theory has a derived variant (which is really good) and a non-derived one (which is good enough in many situations).

\subsubsection{Non-derived variant}   \label{sss:Non-derived variant}
Let $\widehat{\Sch}_{\BZ_p}$ be the category of $p$-adic formal schemes, so
\[
\widehat{\Sch}_{\BZ_p}=\underset{n}{\underset{\longleftarrow}{\lim}}\Sch_{\BZ/p^n\BZ}\, ,
\]
where $\Sch_{\BZ/p^n\BZ}$ is the category of schemes over $\BZ/p^n\BZ$. Let $\widehat{\Sch}_{\BZ_p}^b\subset\widehat{\Sch}_{\BZ_p}$ be the full subcategory of
bounded $p$-adic formal schemes (boundedness means that $H^0(U,\cO_X)$ has bounded $p^\infty$-torsion for any open affine $U\subset X$).

To every $X\in\widehat{\Sch}_{\BZ_p}^b$ Bhatt and Lurie associate a certain stack\footnote{By a stack we mean a stack of groupoids on the category of schemes equipped with the fpqc topology. In the case of a stack over $\Spf\BZ_p$ it suffices to consider $p$-nilpotent schemes.} $X^\prism$ over 
$\Spf\BZ_p$ equipped with an endomorphism $F:X^\prism\to X^\prism$; moreover, $F$ is a lift of Frobenius (i.e., the endomorphism of the stack $X^\prism\otimes\BF_p$ induced by $F$ is isomorphic to the Frobenius, and one fixes an isomorphism). The assignment $X\mapsto X^\prism$ is a functor; this functor is called \emph{prismatization}.

Define a \emph{crystal} on $X$ to be a complex of $\cO$-modules\footnote{If $X$ is quasisyntomic in the sense of \cite{BMS} this is the same as a crystal on the absolute prismatic site of~$X$, see \cite[Prop.~8.15]{BL2}.} on $X^\prism$. 
Define an \emph{$F$-crystal} on $X$ to be a crystal $M$ on $X$ equipped with a morphism\footnote{One could also try to impose a certain condition for the morphism $F^*M\to M$ (it should be an ``isogeny'' in some sense). However, this is not strictly necessary (e.g., because hopefully there is a better notion of \emph{effective $F$-gauge}, see \S\ref{sss:effectivity}).} $F^*M\to M$.

A morphism $f:X\to Y$ in $\widehat{\Sch}_{\BZ_p}^b$ induces a morphism $f^\prism:X^\prism\to Y^\prism$ and a functor $f^\prism_*$ from the category of $F$-crystals on $X$ to the category of $F$-crystals on $Y$. Under mild assumptions\footnote{According to Propositions 8.15 and 8.16(2) of \cite{BL2}, it suffices to assume that $f$ and $Y$ are quasisyntomic in the sense of \cite{BMS}.}
on $X,Y,f$, the  \emph{prismatic cohomology} of $X$ with respect to $f$ is just the $F$-crystal $f^\prism_* (\cO_{X^\prism})$.

\subsubsection{Derived variant}
Let $\widehat{\Sch}_{\BZ_p}^{der}$ be the $\infty$-category of derived $p$-adic formal schemes. Bhatt and Lurie define prismatization as a functor
from $\widehat{\Sch}_{\BZ_p}^{der}$ to the $\infty$-category of derived stacks equipped with a lift of Frobenius, from which the functor from \S\ref{sss:Non-derived variant} is obtained by pre-composing with the embedding $\widehat{\Sch}_{\BZ_p}^{b}\mono\widehat{\Sch}_{\BZ_p}^{der}$ and post-composing with the ``classical truncation'' functor
\[
\mbox{\{Derived stacks\}}\to\mbox{\{Classical stacks\}},
\]
where ``classical'' means ``non-derived''. Bhatt and Lurie proved that if $X\in\widehat{\Sch}_{\BZ_p}^{b}$ is quasisyntomic in the sense of \cite{BMS} then
the last step is unnecessary: the derived prismatization of $X$ is already classical in an appropriate sense. In particular, this is so if $X$ is noetherian and~l.c.i.

If $f:X\to Y$ is a morphism in $\widehat{\Sch}_{\BZ_p}^{der}$ satisfying very mild assumptions\footnote{According to Propositions 8.15 and 8.16(1) of \cite{BL2}, it is enough to assume, e.g., that $f$ has finite type and $Y$ is quasisyntomic.} then the  \emph{derived prismatic cohomology} of $X$ with respect to $f$ is just the direct image of the structure sheaf under the morphism of prismatizations corresponding to $f$.

\subsection{The stack $\Sigma =(\Spf\BZ_p)^\prism$}   \label{ss:about Sigma}
\subsubsection{Definition of $\Sigma =(\Spf\BZ_p)^\prism$} \label{sss:2Sigma as quotient} 
Let $W$ be the ring scheme (over $\BZ$) of $p$-typical Witt vectors. Let $x_0, x_1,\ldots$ be the usual coordinates on $W$.
Let $W_{\prim}$ be the formal completion of $W$ along the locally closed subscheme of $W$ defined by the equations $p=x_0=0$ and the inequality $x_1\ne 0$.

The group scheme $W^\times$ acts on $W_{\prim}$ by multiplication, so we can form the quotient stack\footnote{This stack is only slightly beyond the familiar world, see Proposition~\ref{p:Sigma as limit}(vii).}
\begin{equation}   \label{e:2Sigma as quotient}
\Sigma:=W_{\prim}/W^\times . 
\end{equation}

Since $W_{\prim}$ is a \emph{formal} scheme, $\Sigma$ is not an algebraic stack but rather a formal one.

Note that the Witt vector Frobenius $F:W\to W$ induces a canonical endomorphism $F:\Sigma\to\Sigma$ whose reduction modulo $p$ is the Frobenius endomorphism of $\Sigma\otimes\BF_p$.

Let us also note that the elements $p,V(1)\in W(\BZ_p)$ yield morphisms $p:\Spf\BZ_p\to\Sigma$ and $V(1):\Spf\BZ_p\to\Sigma$.

\subsubsection{$S$-points of $\Sigma$}    \label{sss:2Sigma(S)}
Given a scheme $S$, let $W_S:=W\times S$; this is a ring scheme over $S$.  By a $W_S$-module we mean a commutative affine group scheme over $S$ equipped with an action of $W_S$.  A $W_S$-module is said to be \emph{invertible} if it is locally isomorphic to $W_S$. An invertible $W_S$-module is essentially the same as a 
$W_S^\times$-torsor.

Now let us describe $\Sigma (S)$, i.e., the groupoid of $S$-points of $\Sigma$. We say that $S$ is \emph{$p$-nilpotent }if $p\in H^0(S,\cO_S)$ is locally nilpotent. The above definition of $\Sigma$ implies that if $S$ is not $p$-nilpotent then $\Sigma (S)=\emptyset$, and if $S$ is $p$-nilpotent then $\Sigma (S)$ identifies with the groupoid of pairs $(M,\xi)$, where $M$ is an invertible $W_S$-module and $\xi:M\to W_S$ is a $W_S$-morphism such that
for every $s\in S$ one has
\[
(M_s)_{\red}\subset\Ker\xi_1,\quad (M_s)_{\red}\not\subset\Ker\xi_2 ,
\]
where $\xi_n:M\to (W_n)_S$ is the composite map $M\overset{\xi}\longrightarrow W_S\epi (W_n)_S$.

\subsubsection{On the language of $W_S$-modules}    \label{sss:language of W_S-modules}
If $S$ is $p$-nilpotent and affine then an invertible $W_S$-module (in the sense of \S\ref{sss:2Sigma(S)}) is essentially the same as an invertible $W(R)$-module, where $R=H^0(S,\cO_S)$. So it is easy to avoid using $W_S$-modules while dealing with the stack $\Sigma$. 

However, $\Sigma'$ is defined using $W_S$-modules which are not locally free (see 
\S\ref{sss:idea of Sigma'}). So it is harder to avoid using $W_S$-modules while dealing with $\Sigma'$. (An attempt in this direction is made in Appendix~\ref{s:SSigma'}).

\subsection{Cones}   \label{ss:Cones}
The set-up that we are going to introduce now is used in the outline of the construction of the prismatization functor given in \S\ref{ss:prismatization outline} below.  It is also used in some parts of the main body of the article.

\subsubsection{Cone of a morphism of abelian groups}   \label{sss:Cones-groups}
If a group $G$ acts on a set $X$ then the \emph{quotient groupoid} is defined as follows: the set of objects is $X$, a morphism $x\to x'$ is an element $g\in G$ such that $gx=x'$, and the composition of morphisms is given by multiplication in $G$.

Let $d:A\to B$ be a homomorphism of abelian groups. Then $A$ acts on the set $B$: namely, $a\in A$ acts by $b\mapsto b+d(a)$. In this situation the quotient groupoid will be denoted by $\Cone (d)$. It has an additional structure
of Picard groupoid (i.e., a symmetric monoidal category in which all objects and morphisms are invertible); moreover, this Picard groupoid is strictly commutative. The details are explained in \cite[Expos\'e XVIII, \S 1.4]{SGA4}, where the Picard groupoid in question is denoted by ${\rm ch}(d)$.

There is no conflict between this understanding of $\Cone (d)$ and the usual one\footnote{The usual cone of $d:A\to B$ is an object of the DG category of complexes of abelian groups (namely, the complex $0\to A\to B\to 0$, where $B$ is placed in degree $0$). }: as explained in \emph{loc. cit.}, there is a canonical equivalence between the 2-category of strictly commutative Picard groupoids and the full subcategory of the DG category of complexes of abelian groups formed by complexes with cohomology concentrated in degrees $-1$ and $0$; moreover, this equivalence takes our $\Cone (d)$ to the usual cone of $d$.

\subsubsection{Cone of a morphism of commutative group schemes}   \label{sss:Cones-group schemes}
Now let $S$ be a scheme and $d:A\to B$ a morphism of commutative group schemes over $S$. One can consider $A$ and $B$ as fpqc sheaves of abelian groups on the category of $S$-schemes. So sheafifying 
\S\ref{sss:Cones-groups}, one gets a strictly commutative Picard stack $\Cone (d)$ over $S$. The details are explained in \cite[Expos\'e XVIII, \S 1.4]{SGA4}. 

As a stack, $\Cone (d)$ is just the quotient of $B$ by the action of $A$ by translations. So if
$A$ is flat over $S$ then the stack $\Cone (d)$ is algebraic in the sense of  Definition~\ref{d:algebraic stack}.
Note that $A$ is not required to have finite type over $S$ (e.g., one can take $A=W_S$).
This is because the class of algebraic stacks in our sense is larger than the one from \cite{LM}.

\subsubsection{Quasi-ideals}  \label{sss:2quasi-ideals}
Let $C$ be a (commutative) ring. By a \emph{quasi-ideal} in $C$ we mean a pair $(I,d)$, where $I$ is a $C$-module and $d:I\to C$ is a $C$-linear map such that 
$d(x)\cdot y=d(y)\cdot x$ for all $x,y\in I$. 

Now let $C$ be a ring scheme over some scheme $S$. Then by a \emph{quasi-ideal} in $C$ we mean a pair $(I,d)$, where $I$ is a commutative group $S$-scheme equipped with an action of $C$ and $d:I\to C$ is a $C$-linear morphism such that for every $S$-scheme $S'$ and every $x,y\in I(S')$ one has $d(x)\cdot y=d(y)\cdot x$.

A more detailed discussion of the notion of quasi-ideal can be found in \S\ref{ss:quasi-ideals}.

\subsubsection{Cone of a quasi-ideal}    \label{sss:Cones-rings}
By a \emph{ring groupoid} we mean a product-preserving functor $\Pol^{\op}\to\Grpds$, where $\Pol$ is the category of free unital commutative rings and
$\Grpds$ is the $(2,1)$-category of groupoids.\footnote{This definition has a high-tech reformulation: a ring groupoid is an animated ring (in the sense of \cite{CS}) whose $i$-th homotopy groups are zero for all $i>1$.} One has a similar notion of \emph{ring $S$-stack}, where $S$ is a scheme or a stack. This type of definition goes back to Lawvere \cite{La1,La2}.

If $C$ is a (commutative) ring and $(I,d)$ is a quasi-ideal in $C$ then the structure of strictly commutative Picard groupoid on $\Cone (d)$ can be upgraded to that of a ring groupoid. Si\-mi\-larly, if $C$ is a ring scheme over $S$ and $(I,d)$ is a quasi-ideal in $C$ then the structure of strictly commutative Picard $S$-stack on $\Cone (d)$ can be upgraded to that of a ring $S$-stack.

\subsubsection{Where to find details} \label{sss:Where to find details}
More details about the $(2,1)$-category of ring groupoids can be found in the expository note \cite{ring groupoid}. In particular, in \cite[\S 4]{ring groupoid} we explain a simple model for this $(2,1)$-category, which is due to E.~Aldrovandi and B.~Noohi. E.g., if $A,C$ are rings and $(I,d)$ is a a quasi-ideal in $C$ then the groupoid of $1$-morphsims from $A$ to $\Cone (d)$ is formed by commutative diagrams
\[
\xymatrix{
0\ar[r]&I\ar[r] \ar[rd]_d &\tilde A\ar[r]\ar[d]^f&A\ar[r]&0 \\
&& C& &
}
\]
in which the upper row is a ring extension, $f$ is a ring homomorphism, and $d(\tilde ax)=f(\tilde a)\cdot dx$ for all $\tilde a\in\tilde A$ and $x\in I$. 

\subsection{Outline of the construction of the prismatization functor}  \label{ss:prismatization outline}
The non-derived version of prismatization is a functor $X\mapsto X^\prism$ from $\widehat{\Sch}_{\BZ_p}^b$ (i.e., the category of bounded $p$-adic formal schemes) to the 2-category of pairs $(\sX ,F)$, where $\sX$ is a stack algebraic over $\Sigma$ and $F:\sX\to\sX$ is a lift of Frobenius.
For the precise meaning of the words ``algebraic over'', see \S\ref{ss:algebraic over} 
(which relies on \S\ref{sss:c and g} and \S\ref{ss:algebraic stacks}).

We will give a brief sketch of the construction of the functor $X\mapsto X^\prism$ and its derived version; the details are contained in \cite{BL2}.

\subsubsection{The ring stack $\sR_\Sigma=(\BA^1\hat\otimes\BZ_p)^\prism$}   \label{sss:ring stack R}
Let $\BA^1\hat\otimes\BZ_p$ denote the $p$-adic completion of $\BA^1=\BA^1_\BZ$. Then 
$(\BA^1\hat\otimes\BZ_p)^\prism :=\sR_\Sigma$, where $\sR_\Sigma$ is the ring stack over $\Sigma$ defined below.

Given a scheme $S$, we identify $\Sigma (S)$ with the groupoid of pairs $(M,\xi)$ as in \S\ref{sss:2Sigma(S)}. For $(M,\xi)\in\Sigma (S)$, we set
\begin{equation}
\sR_{M,\xi}:=\Cone (M\overset{\xi}\longrightarrow W_S).
\end{equation}
$\sR_{M,\xi}$ is a ring stack over $S$ because $(M,\xi)$ is a quasi-ideal in $W_S$. The formation of $\sR_{M,\xi}$ commutes with base change $S'\to S$, so the ring stacks $\sR_{M,\xi}$ define a ring stack $\sR_\Sigma$ over $\Sigma$.

\subsubsection{Prismatization of affine $p$-adic formal schemes}  \label{sss:prismatization of affines}
Let $X=\Spf A\in\widehat{\Sch}_{\BZ_p}^b$. Then the stack $X^\prism$ is defined as follows: for any scheme $S$, an object of the groupoid $X^\prism (S)$ is a triple 
$(M,\xi,\alpha )$, where $(M,\xi )\in\Sigma (S)$ and $\alpha :A\to \sR_{M,\xi} (S)$ is a morphism of ring groupoids.

One checks that $(\Spf\BZ_p)^\prism=\Sigma$ and $(\BA^1\hat\otimes\BZ_p)^\prism=\sR_\Sigma$, as promised.

Note that the procedure used above to define $X^\prism$ is quite general: it can be applied to \emph{any $B$-algebra stack over any base, where $B$ is any ring.} Bhatt's name for this procedure is \emph{transmutation}, see \cite[Rem.~2.3.8]{Bh}.
Here is a classical example of transmutation: if $L$ is a  finite extension of a field $k$ then $L$ determines an $L$-algebra scheme\footnote{For any $k$-algebra $C$, the $L$-algebra of $C$-points of this scheme is $C\otimes_kL$.} over $k$, and the corresponding transmutation functor is the Weil restriction $\Res_{L/k}\,$.

\subsubsection{Remarks}   \label{sss:points of R}
(i) If $S$ is affine then $\sR_{M,\xi} (S)=\Cone (M (S)\overset{\xi}\longrightarrow W (S))$ since $H^1(S,M)=0$ (here we use the assumption that $M$ is an invertible $W_S$-module).

(ii) Combining the previous remark with \S\ref{sss:Where to find details}, we see that the definition of  $X^\prism (S)$ from \S\ref{sss:prismatization of affines} is quite understandable.

\subsubsection{Prismatization of arbitrary $p$-adic formal schemes}    \label{sss:prismatization in general}
Bhatt and Lurie proved\footnote{In that proof it is important that the $p$-adic formal sheme $X$ is bounded.} that the construction of $X^\prism$ is Zariski-local (and even etale-local). This allows them to define $X^\prism$ for any $X\in\widehat{\Sch}_{\BZ_p}^b$.

\subsubsection{Prismatization of $\Spec\BF_p$}    \label{sss:prismatization of Spec F_p}
According to \cite[Rem.~3.13]{BL2}, $(\Spec\BF_p)^\prism=\Spf\BZ_p$, and the canonical morphism $(\Spec\BF_p)^\prism\to\Sigma$ is the morphism $p:\Spf\BZ_p\to\Sigma$ defined at the end of \S\ref{ss:about Sigma}.

\subsubsection{Derived version}   \label{sss:Derived version}
As explained in \cite{BL2}, the above prismatization functor has a derived version, which produces a derived stack over $\Sigma$ from a derived $p$-adic formal scheme (we are talking about the theory of derived schemes and stacks in which affine derived schemes are spectra of animated rings in the sense of \cite{CS}). The definition of $X^\prism (S)$ from \S\ref{sss:prismatization of affines} remains essentially valid, but now $X$ and $S$ are derived, so $A$ and $\sR_{M,\xi} (S)$ are animated rings whose higher homotopy groups are not required to be zero.

\subsubsection{(Non)-commutation with limits}  \label{sss:(Non)-commutation with limits}
Derived prismatization commutes with projective limits (for tautological reasons). 

According to \cite[Rem.~3.5]{BL2}, the non-derived prismatization functor considered in \S\ref{sss:prismatization of affines}-\ref{sss:prismatization in general} commutes with
projective limits of \emph{Tor-independent} diagrams in $\widehat{\Sch}_{\BZ_p}^b$ (i.e., those diagrams whose limit in $\widehat{\Sch}_{\BZ_p}^b$ agrees with the limit in the category $\widehat{\Sch}_{\BZ_p}^{der}$ of derived $p$-adic formal schemes). 
\emph{Without the Tor-independence assumption, the statement is false,} in general.
E.g., consider the following diagrams, which are Cartesian in $\widehat{\Sch}_{\BZ_p}^b$ but not in $\widehat{\Sch}_{\BZ_p}^{der}$:
\[
\xymatrix{
\BA^1\hat\otimes\BZ_p\ar[r]^{\id} \ar[d]_{\id} & \BA^1\hat\otimes\BZ_p\ar[d]^{\Delta}&\Spec\BF_p\ar[r]^{\id} \ar[d]_{\id}&\Spec\BF_p\ar@{^{(}->}[d]\\
\BA^1\hat\otimes\BZ_p\ar[r]^{\Delta} & \BA^2\hat\otimes\BZ_p&\Spec\BF_p\ar@{^{(}->}[r]&\Spf\BZ_p
}
\]
(here $\Delta$ is the diagonal map). Their images under the prismatization functor are the diagrams
\[
\xymatrix{
\sR_\Sigma\ar[r]^{\id} \ar[d]_{\id} &\sR_\Sigma\ar[d]^{\Delta}&\Spf\BZ_p\ar[r]^{\id} \ar[d]_{\id}&\Spf\BZ_p\ar[d]^p\\
\sR_\Sigma\ar[r]^-{\Delta} & \sR_\Sigma\times_\Sigma\sR_\Sigma&\Spf\BZ_p\ar[r]^p&\Sigma
}
\]
which are not Cartesian (for the second diagram, see Lemma~\ref{l:not mono} below).

\subsubsection{Motivation}  \label{sss:motivation of prismatization}
In my talk \cite{My IAS talk} I tried to explain that the stacky approach to prismatic cohomology of $\BF_p$-schemes (which is isomorphic to crystalline cohomology) is not so far from Grothendieck's ideas from \cite{Gr68}.
But in order to pass from $\BF_p$-schemes to arbitrary $p$-nilpotent schemes, one has to replace $\Cone (W\overset{p}\longrightarrow W)$ by $\Cone (W\overset{\xi}\longrightarrow W)$, where $\xi$ is a deformation of $p$.

\subsection{Specializations of $X^\prism$}
This subsection is sketchy (just as the previous one). One of its goals is to point out a drawback of the stack $\Sigma$ (see \S\ref{sss:drawback of Sigma}); fixing it is one of the motivations for introducing the more complicated stacks $\Sigma'$ and $\Sigma''$.

\subsubsection{$X^{\dR}$ and de Rham cohomology}   \label{sss:X^dR}
For $X\in\widehat{\Sch}_{\BZ_p}^b$ let $X^{\dR}$ be the stack over $\Spf\BZ_p$ obtained by base-changing the morphism $X^\prism\to\Sigma$ via the map
$p:\Spf\BZ_p\to\Sigma$ defined at the end of \S\ref{sss:2Sigma as quotient}. The notation $X^{\dR}$ is motivated by part (3) of \cite[Thm.~1.8]{BS}, which essentially says that if $X$ is smooth over $\Spf\BZ_p$ then $R\Gamma (X^{\dR},\cO_{X^{\dR}})$ canonically identifies with the de Rham cohomology of~$X$. A~slight\-ly different motivation is provided by Proposition~\ref{p:Cone (G_a^sharp to G_a)} of this article, which describes the stack 
$(\BA^1\hat\otimes\BZ_p)^{\dR}=\Cone (W\overset{p}\longrightarrow W)\hat\otimes\BZ_p$ 
in terms of a divided power version of the additive group (denoted by $\BG_a^\sharp$).

\subsubsection{Other specializations of $X^\prism$}   \label{sss:other specializations}
Given a morphism $S\to\Sigma$, we get a cohomology theory\footnote{Of course, to get a really good theory one should consider the derived version of $X^\prism$
and $X^\prism\times_\Sigma S$.} which to $X\in\widehat{\Sch}_{\BZ_p}^b$ associates the cohomology of $X^\prism\times_\Sigma S$ with coefficients in $\cO$. E.g., we already mentioned that the morphism  $p:\Spf\BZ_p\to\Sigma$ gives rise to de Rham cohomology. On the other hand, the morphism 
$V(1):\Spf\BZ_p\to\Sigma$ defined at the end of \S\ref{sss:2Sigma as quotient} gives rise to ``Hodge-Tate'' cohomology, see part (2) of \cite[Thm.~1.8]{BS}.

Thus one can think of $\Sigma$ as a \emph{stack parametrizing $p$-adic cohomology theories.}

\subsubsection{A drawback of $\Sigma$}   \label{sss:drawback of Sigma}
Hodge cohomology is a very simple cohomology theory:  if $X$ is smooth over $\Spf\BZ_p$ then the Hodge cohomology of $X$ is $\bigoplus\limits_iR\Gamma (X,\Omega_X^i[-i])$. Unfortunately, Hodge cohomology cannot be obtained using the procedure of \S\ref{sss:other specializations}. Fixing this is one of the motivations for introducing $\Sigma'$ and $\Sigma''$ (another one is  to  introduce  a good category of coefficients for prismatic cohomology, see \S\ref{ss:gauges and F-gauges}). The way we fix this is motivated by the following remark.

\subsubsection{Remark}   \label{sss:Hodge cohomology of A^1}
The Hodge cohomology of $\BA^1_{\BZ}$ can be written as $R\Gamma (\sR_{\rm Hodge},\cO)$, where $\sR_{\rm Hodge}$ is the ring stack over $\BZ$ defined by 
$\sR_{\rm Hodge}:=\Cone (\BG_a^\sharp\overset{0}\longrightarrow\BG_a)$; here $\BG_a^\sharp$ is the divided power version of the additive group. Moreover, $\sR_{\rm Hodge}=\Cone (M\overset{\xi}\longrightarrow W)$, where $M=\BG_a^\sharp\oplus\cI$, $\cI:=V(W)=\Ker (W\epi\BG_a)$, $\xi |_{\BG_a^\sharp}=0$, and $\xi |_\cI$ is the inclusion.

\subsection{The c-stacks $\Sigma'$ and $\Sigma''$}

\subsubsection{C-stacks and g-stacks} 
By a c-stack (resp. g-stack) we mean a stack of categories (resp. of groupoids) on the category of schemes equipped with the fpqc topology. We often say ``stack'' instead of ``g-stack''. All c-stacks considered so far were g-stacks.

By the  \emph{underlying g-stack} of a c-stack $\sX$ we mean the g-stack 
whose category of $S$-points is obtained from $\sX (S)$ by removing all non-invertible morphisms. 

\subsubsection{Idea of the definition of $\Sigma'$}  \label{sss:idea of Sigma'}
In \S\ref{sss:2Sigma(S)} the $W_S$-module $M$ was invertible. To define $\Sigma'$, we replace invertibility by a weaker \emph{admissibility} property from 
Definition~\ref{d:admissible}, see \S\ref{s:Sigma'} for more details. The definition of admissibility is motivated, in part, by the fact that the $W$-module $M$ from \S\ref{sss:Hodge cohomology of A^1} is admissible but not invertible.

Let us note that the spectral sequences relating Hodge cohomology with de Rham and Hodge-Tate cohomology are built into the definition\footnote{Presumably, the closed substack $\Delta'_0\subset\Sigma'$ (see \S\ref{sss:Delta'_0} and \S\ref{ss:Delta'_0 as c-stack}) provides the ``Hodge to Hodge-Tate'' spectral sequence, and the stack $\Sigma'_{\bardR}$ over $\Sigma'$ defined in \S\ref{ss:Hdg & bar dR} provides the ``Hodge to de Rham'' spectral sequence.}
of $\Sigma'$.

\subsubsection{Defining $\Sigma''$}
It turns out that $\Sigma'$ has disjoint open substacks $\Sigma_+$ and $\Sigma_-$ canonically isomorphic to $\Sigma$. We define a c-stack $\Sigma''$ by gluing $\Sigma_+$ with~$\Sigma_-$.  

The underlying g-stack of $\Sigma''$ contains $\Sigma$ as an open substack. The corresponding open c-substack of $\Sigma''$ itself is a kind of quotient of $\Sigma$ by the action of $F:\Sigma\to\Sigma$, denoted by $\Sigma_F$ (see \S\ref{sss:stacky lax quotient} for the precise definition of $\Sigma_F$).

\subsubsection{The ring stacks $\sR_{\Sigma'}$ and $\sR_{\Sigma''}$}  \label{sss:the ring stacks}
Similarly to \S\ref{sss:ring stack R}, we define in \S\ref{ss:Sigma'&ring stack} a ring stack $\sR_{\Sigma'}$ over $\Sigma'$. Warning: Remark~\ref{sss:points of R}(i) does not hold if $M$ is assumed to be admissible rather than invertible. 

In \S\ref{sss:R_Sigma''} we descend $\sR_{\Sigma'}$ to a ring stack $\sR_{\Sigma''}$ over $\Sigma''$.

\subsubsection{A hope} \label{sss:A hope}
If $S$ is a scheme then  associating to a morphism  $f:S\to\Sigma''$ the $f$-pullback of $\sR_{\Sigma''}$, we get a functor $\Sigma''(S)\to\RStacks_S$, where $\RStacks_S$ is the $(2,1)$-category of algebraic ring stacks over $S$.
I hope that this functor is fully faithful\footnote{Hopefully, this can be deduced from our Theorem~\ref{t:Sigma'' to Q}.}. If so  \emph{then one could think of $\Sigma''$ as a moduli stack of a certain type of ring stacks.} 
This would give a conceptual justification of the definition of $\Sigma''$.

\subsection{The functors $X\mapsto X^\prismp$ and $X\mapsto X^\prismpp$}   \label{ss:X^prismp}
Using the ring stacks $\sR_{\Sigma'}$ and $\sR_{\Sigma''}$ mentioned in \S\ref{sss:the ring stacks}, one defines $X^\prismp$ and $X^\prismpp$ similarly to \S\ref{sss:prismatization of affines} and \S\ref{sss:prismatization in general}. 
This construction also has a derived version in the spirit of \S\ref{sss:Derived version}. 
 The details can be found in \cite{Bh}, where $X^\prismp$ and $X^\prismpp$ (or rather their underlying g-stacks) are denoted by $X^\cN$ and $X^{\Syn}$, respectively.
 $X^{\Syn}$ is called there the \emph{syntomification} of $X$, and $X^\cN$ is called the \emph{filtered prismatization} of $X$ (``filtered'' refers to the Nygaard filtration, whence the notation $X^\cN$).
 
 Let us note that the morphisms $X^\prismp\to\Sigma'$ and $X^\prismpp\to\Sigma''$ are \emph{left fibrations} in the sense of \S\ref{ss:left&right fibrations}.
 
 The pragmatic justification of the definition of $\Sigma''$ is that the functor $X\mapsto X^\prismpp$ encodes the prismatic cohomology theory of Bhatt-Scholze \cite{BS}.  (In fact, the desire to ensure this led me to the definition of $\Sigma''$).

\subsubsection{The case $X=\Spec\BF_p$}   \label{sss:(Spec F_p)^prismp}
As explained in Bhatt's lectures \cite{Bh},  $(\Spec\BF_p)^\prismp$ has the following explicit description\footnote{See \cite[\S 3.3]{Bh} and  \cite[Thm.~5.4.1]{Bh}. The relation between our notation and that of \cite[\S 3.3]{Bh} is as follows: our $v_-$ and $v_+$ correspond to Bhatt's $t$ and $u$, respectively.}:
for every scheme $S$ such that $p\in H^0(S,\cO_S)$ is locally nilpotent, 
$(\Spec\BF_p)^\prismp (S)$ is the category of diagrams $$\cO_S\overset{v_+}\longrightarrow\sL\overset{v_-}\longrightarrow\cO_S,$$ 
where $\sL$ is a line bundle  on $S$ and $v_-v_+=p$; if $p\in H^0(S,\cO_S)$ is not locally nilpotent then 
$(\Spec\BF_p)^\prismp (S)=\emptyset$. Moreover, $(\Spec\BF_p)^\prismpp$ is obtained from 
$(\Spec\BF_p)^\prismp$ by gluing together the open substacks $(\Spec\BF_p)^\prism_\pm\subset (\Spec\BF_p)^\prismp$, where  $(\Spec\BF_p)^\prism_\pm$ is defined by the condition of invertibility of $v_\pm$ (both $(\Spec\BF_p)^\prism_+$ and $(\Spec\BF_p)^\prism_-$ are isomorphic to~$\Spf\BZ_p$).

\subsection{Gauges and $F$-gauges}  \label{ss:gauges and F-gauges}
Let $X\in\widehat{\Sch}_{\BZ_p}^{der}$.
\subsubsection{}
Define a \emph{gauge} (resp. an  \emph{$F$-gauge}) on $X$ to be a quasi-coherent complex of $\cO$-modules on $(X^\prismp)^{\rm g}$ (resp. on $(X^\prismpp)^{\rm g}$), where $(X^\prismp)^{\rm g}$ is the underlying g-stack of
$X^\prismp$. Probably $F$-gauges  form the natural category of coefficients for prismatic cohomology.
(The idea that the category of coefficients should be in the spirit of the Fontaine-Jannsen notion of $F$-gauge \cite{FJ} was suggested to me by P.~Scholze).

\subsubsection{} \label{sss:FJ}
If $X$ is the spectrum of a perfect field $k$ of characteristic $p$, the above definitions of gauge and $F$-gauge are equivalent  to those of Fontaine-Jannsen \cite{FJ}. This is explained in \cite[ \S 3.4]{Bh} and \cite[\S 4.2]{Bh}.

\subsubsection{}  \label{sss:effectivity}
Here is a somewhat controversial suggestion. I suggest to define an \emph{effective gauge} (resp. an  \emph{effective $F$-gauge}) on $X$ to be a quasi-coherent complex of contravariant $\cO$-modules\footnote{See \S\ref{sss:O-mododules on c-stacks} for the notion of a contravariant $\cO$-module on a c-stack.} on 
$X^\prismp$ (resp.~on $X^\prismpp$). In the situation of  \S\ref{sss:FJ}, the above definition agrees with the Fontaine-Jannsen definition of effectivity (see \cite{FJ},   \cite[Def.~3.4.8]{Bh}, \cite[Rem.~4.3.2]{Bh}).
For any $X$, one can define a canonical functor from the category of effective $F$-gauges on $X$ to the category of $F$-crystals\footnote{The crystal corresponding to an effective gauge $M$ is $f^*M$, where $f$ is the canonical morphism $X^\prism=X^\prismpp\times_{\Sigma''}\Sigma\to X^\prismpp$. The morphism $F^*f^*M\to f^*M$ is defined using the morphism $\id_{\Sigma'}\to F'_-$ from \S\ref{sss:id to F'_-} (we skip the details).} on $X$
 (see \S\ref{ss:format of stacky}). I hope that using \cite[Prop. 3.4.9]{G}, Theorem~\ref{t:naive Coeq is OK}, and Theorem~\ref{t:Sigma''}(ii), one can show that 

(i) the canonical functors
\[
\mbox{\{effective gauges\}}\to\mbox{\{gauges\}}, \quad \mbox{\{effective $F$-gauges\}}\to\mbox{\{$F$-gauges\}},
\]
are fully faithful, so \emph{effectivity is a property} (rather an additional structure);

(ii) an $F$-gauge is effective if and only if the corresponding gauge is.

\subsection{On the language of c-stacks}
\subsubsection{}
Recall that by a c-stack we mean a stack of categories on the category of schemes equipped with the fpqc topology; so the notion of c-stack is quite old.
In this article we introduce and use the notion of \emph{algebraic} c-stack and the related notion of formal c-stack (see \S\ref{ss:algebraic stacks} and \S\ref{sss:def formal scheme}).

The reader may prefer to replace all algebraic and formal c-stacks considered in this paper by the underlying g-stacks (i.e., stacks of groupoids). This approach (which is used in \cite{Bh}) would allow him to skip a considerable part of this article.

On the other hand, it is natural to consider $\Sigma'$ and $\Sigma''$ as c-stacks rather than g-stacks. E.g., $\Sigma''$ parametrizes certain ring stacks (see \S\ref{sss:A hope}), and there exist non-invertible morphisms between these ring stacks.

\subsubsection{}
The reader who prefers to play with the notion of algebraic c-stack could look at some simple examples, see \S\ref{ss:simple c-stacks}, \S\ref{sss:c-stacks via alg cats}, \S\ref{sss:fS'} -\S\ref{sss:drawing Gamma'},
and a part of \S\ref{ss:2toy model}.

\subsubsection{}
The c-stack $\Sigma'$ is \emph{left-fibered} over a very simple c-stack $(\BA^1/\BG_m)_-$ introduced in \S\ref{sss:BA^1/BG_m)_pm}. According to the definition of left fibration (see \S\ref{ss:left&right fibrations}), this implies that the fibers of the map $\Sigma'\to (\BA^1/\BG_m)_-$ are g-stacks.

In \S\ref{sss:C-stacks left-fibered over} we show that a c-stack left-fibered over $(\BA^1/\BG_m)_-$ is the same as a g-stack over $\BA^1$ equipped with an \emph{anti-action} of the multiplicative monoid. As mentioned at the end of \S\ref{sss:C-stacks left-fibered over}, anti-actions appear ``in nature'' (e.g., if $Y$ is any scheme, $Z\subset Y$ is a closed subscheme, and $X$ is the corresponding deformation to the normal cone then $X$ is a scheme over $\BA^1$ equipped with an anti-action of the multiplicative monoid).

\subsection{Organization of the article}   \label{ss:Organization}
\S\ref{s:c-stacks} is devoted to generalities on c-stacks (including g-stacks). 

In \S\ref{s:W-modules} we first discuss the ring scheme $W$.
In particular, in Lemma~\ref{l:G_a^sharp=W^(F)} and Proposition~\ref{p:Cone (G_a^sharp to G_a)} we construct canonical isomorphisms
\[
\Ker (W\overset{F}\longrightarrow W)\iso\BG_a^\sharp ,\quad \Cone (W\overset{p}\longrightarrow W)\iso\Cone (\BG_a^\sharp\to \BG_a),
\]
where $\BG_a^\sharp$ is the $p$-divided power version of the additive group (the second isomorphism was already mentioned in \S\ref{sss:X^dR}).
Then we discuss $W_S$-modules; in particular, we define and study the notion of \emph{admissible} $W_S$-module.

In \S\ref{s:Sigma} we study the stack $\Sigma$. Let us note that \S\ref{s:Sigma} relies only on a small part of \S\ref{s:c-stacks} and \S\ref{s:W-modules} (namely, \S\ref{ss:algebraic stacks}-\ref{ss:algebraic over}, \S\ref{sss:p-nilpotence}, and \S\ref{ss:G_a^sharp}-\ref{ss:W^times} are essentially enough to read \S\ref{s:Sigma}).

In \S\ref{s:Sigma'} we define and study the c-stack $\Sigma'$. A more detailed description of \S\ref{s:Sigma'} is given in~\S\ref{ss:intro to Sigma'}.

In \S\ref{s:2Sigma' as quotient} we give an ``explicit'' description of $\Sigma'$ as a quotient stack.

In \S\ref{s:Sigma' as quotient} we give a more economic (and also ``explicit'') description of $\Sigma'$ as a quotient stack.
 As an application, we prove Corollary~\ref{c:Sigma'_red}, which describes the reduced part of $\Sigma'$ in very understandable terms.

In \S\ref{s:Sigma''} we define and study the c-stack $\Sigma''$. A more detailed description of \S\ref{s:Sigma''} is given in~\S\ref{ss:intro to Sigma''}.

In Appendix~\ref{s:SSigma'} we give one more description of $\Sigma'$. Thus the total number of descriptions of $\Sigma'$ given in this article is $4$ (which agrees with the story about four blind men feeling an elephant).

In Appendix~\ref{s:q-de Rham} we discuss the $q$-de Rham prism from \cite[\S 16]{BS} and use it to prove a certain statement about $\Sigma$.

In Appendix~\ref{s:ring morphisms W to W_n} we prove a statement used in the proof of Proposition~\ref{p:Aut tilde W}.

\subsection{Acknowledgements}
I thank A.~Beilinson,  B.~Bhatt, J.~Lurie, A.~Mathew, N.~Rozenblyum, and P.~Scholze for valuable advice and references. I thank L.~Gurney for numerous remarks.

The author's work on this project was partially supported by NSF grants DMS-1303100,  DMS-2001425, and the School of Mathematics of the Institute for Advanced Study. I thank the Institute for its hospitality.

\section{Generalities on c-stacks}   \label{s:c-stacks}
\subsection{C-stacks and g-stacks}
\subsubsection{}   \label{sss:c and g}
The general notion of a stack on a site was introduced by J.~Giraud (who was a student of Grothendieck).

By a c-stack (resp. g-stack) we mean a stack of categories (resp. of groupoids) on the category of schemes equipped with the fpqc topology. Sometimes we say ``stack'' instead of ``g-stack''. 

\subsubsection{}
By the \emph{underlying groupoid} of a category $\cC$ we mean the groupoid obtained from $\cC$ by removing all non-invertible morphisms.
By the  \emph{underlying g-stack} of a c-stack $\sX$ we mean the g-stack whose category of $S$-points is the groupoid underlying $\sX (S)$.

We think of a c-stack as a g-stack with additional structure; we call it \emph{c-structure}.

\subsubsection{}
Given categories $\cA,\cB,\cC$ and functors $F:\cA\to\cC$ and $G:\cB\to\cC$,
we have the fiber product $\cA\times_{\cC}\cB$; this is the category of triples $(a,b,f)$, where $a\in\cA$, $b\in\cB$, and $f$ is an isomorphism $F(a)\iso G(b)$. Accordingly, we have the notion of fiber product for c-stacks.

\subsubsection{}
A \emph{substack} of a c-stack $\sX$ is the following data: for each scheme $S$ we are given a strictly full subcategory $\sY (S)\subset\sX (S)$ so that $\sY$ is a c-stack. The latter condition means that

(i) for every morphism $f:S'\to S$ one has $f^*(\sY (S))\subset\sY (S')$;

(ii) if $x\in\sX (S)$ and if for some fpqc covering family of morphisms $f_i:S_i\to S$ one has $f_i^*(x)\in\sY (S_i)$ then $x\in\sY (S)$.

If $\cC$ is a category then there is a canonical bijection between strictly full subcategories of $\cC$ and  strictly full subcategories of the underlying groupoid of $\cC$. Similarly, if $\sX$ is a c-stack then there is a canonical bijection between substacks of $\sX$ and substacks of the underlying g-stack of~$\sX$.

\subsubsection{}
Let $\sX$ be a c-stack. A substack $\sY\subset\sX$ is said to be \emph{open} (resp.~\emph{closed}) if for every scheme $S$ and every morphism $S\to\sX$, the fiber product $\sY\times_{\sX}S$ is an open (resp.~closed) subscheme of $S$. If $\sY\subset\sX$ is a closed substack then there is a unique open substack $\sU\subset\sX$ such that for every scheme $S$ and every morphism $S\to\sX$ one has
$\sU\times_{\sX}S=S\setminus (\sY\times_{\sX}S)$; the substack $\sU$ is denoted by $\sX\setminus\sY$.

\subsubsection{The topological space associated to a c-stack}   \label{sss:|X|}
Following \cite[\S 5]{LM}, to a c-stack $\sX$ one associates a topological space $|\sX |$ as follows. Given fields $k_1$ and $k_2$, we say that $x_1\in\sX (k_1)$ is equivalent to $x_2\in\sX (k_2)$ if there exists a diagram of fields $k_1\to k_{12}\leftarrow k_2$ such that the images of $x_1$ and $x_2$ in $\sX (k_{12})$ are isomorphic\footnote{To check transitivity, note that given a diagram of fields $k_1\to k_{12}\leftarrow k_2\to k_{23}\leftarrow k_3$, we have $k_{12}\otimes_{k_2}k_{23}\ne 0$, so $k_{12}\otimes_{k_2}k_{23}$ admits a homomorphism to a field.}.
Let $|\sX |$ be the set of equivalence classes.

For every scheme $S$ over $\sX$, we have a map $S=|S|\to |\sX |$. Equip $|\sX |$ with the weakest topology such that all such maps are continuous. Note that $|\sX |$ depends only on the g-stack underlying $\sX$. 

If $\sU\subset\sX$ is an open substack then $|\sU |$ is an open subset of $|\sX |$. Thus one gets a bijection between
open substacks of $\sX$ and open subsets of $|\sX |$.

\subsubsection{The reduced part of a c-stack}  \label{sss:X_red}
If $\sX$ is a c-stack let $\sX_{\red}$ denote the smallest closed substack of $\sX$ containing all field-valued points of $\sX$. We call 
$\sX_{\red}$ the \emph{reduced part} of $\sX$. We say that $\sX$ is \emph{reduced} if $\sX_{\red}=\sX$. 
The author does not know whether an open substack of any reduced c-stack is reduced.

\subsection{Some simple c-stacks}  \label{ss:simple c-stacks}
Here are some examples of c-stacks. The c-stack $(\BA^1/\BG_m)_-$ is very important for us.

\subsubsection{The c-stacks $(\BA^1/\BG_m)_\pm$}    \label{sss:BA^1/BG_m)_pm}
 Let $(\BA^1/\BG_m)_+$ (resp.~$(\BA^1/\BG_m)_-$) be the c-stack whose category of $S$-points is the category of pairs consisting  of an invertible $\cO_S$-module $\sL$ and a morphism $\cO_S\to\sL$ (resp.~a morphism $\sL\to\cO_S$). These c-stacks have the same underlying g-stack (namely, the usual quotient g-stack $\BA^1/\BG_m$). 

\subsubsection{A generalization}   \label{sss:general setting for +-}
In \S\ref{sss:quotients of A^1^dR} we will use the following generalization of the c-stacks $(\BA^1/\BG_m)_\pm$. 
Let $\sX$ be a c-stack equipped with an action of the multiplicative monoid $\BA^1_m$ (by which we mean 
$\BA^1$ equipped with the multiplication operation).  Then we define c-stacks $(\sX/\BG_m)_\pm$ as follows: if $S$ is a scheme then $(\sX/\BG_m)_+ (S)$ (resp. $(\sX/\BG_m)_- (S)$) is the category of pairs consisting of a line bundle $\sL$ on $S$ and an $\BA^1_m$-equivariant morphism $\sL^*\to\sX$ (resp.~$\sL\to\sX$); here $\sL$ and $\sL^*$ are viewed as schemes and ``morphism'' means ``morphism of c-stacks''. 
 
 \subsubsection{A c-stack related to the hyperbolic action of $\BG_m$ on $\BA^2$}  \label{sss:hyperbolic}
 Consider the c-stack whose category of $S$-points is the category of diagrams 
 \begin{equation}  \label{e:v_+,v_-}
 \cO_S\overset{v_+}\longrightarrow\sL\overset{v_-}\longrightarrow\cO_S,
 \end{equation}
  where $\sL$ is an invertible $\cO_S$-module. This c-stack maps to $\BA^1$: to a diagram \eqref{e:v_+,v_-} one associates $v_-v_+\in H^0(S,\cO_S)$. The g-stack underlying our c-stack is the quotient of $\BA^2$ by the hyperbolic action of $\BG_m$.

 \subsubsection{Gluing $(\BA^1/\BG_m)_+$ with $(\BA^1/\BG_m)_-$}   \label{sss:a quotient of P^1}
 Let $\sY$ be the c-stack whose category of $S$-points is the category of pairs $(\sL , \cE)$, where $\sL$ is a line bundle on $S$ and $\cE\subset\sL\oplus\cO_S$ is a line subbundle. The underlying g-stack of $\sY$ is just $\BP^1/\BG_m$. Let $\sY_+$ (resp.~$\sY_-$) be the open substack of $\sY$ defined by the condition that the projection $\cE\to\cO_S$ (resp.~$\cE\to\sL$) is an isomorphism. Then $\sY_\pm=(\BA^1/\BG_m)_\pm$ and $\sY_+\cap\sY_-=\BG_m/\BG_m=\Spec\BZ$. Let us note that a c-stack in the spirit of $\sY$ appears in \S\ref{sss:describing fS''(S)} below.

\subsection{Pre-algebraic and algebraic c-stacks}   \label{ss:algebraic stacks}
 Let $\sX$ be a c-stack. Given a scheme $S$ and objects $x_1,x_2\in\sX (S)$, we have the following contravariant functor on the category of pairs $(T,f)$, where $T$ is a scheme and $f:T\to S$ is a morphism:
 \[
 (T,f)\mapsto\Mor (f^*x_1,f^*x_2).
 \]
 This functor will be denoted by $\MMor (x_1,x_2)$.
 
 \begin{defin}  \label{d:pre-algebraic stack}
 A c-stack $\sX$ is said to be \emph{pre-algebraic} (resp.~\emph{$\MMor$-affine}) if for every scheme $S$ and every $x_1,x_2\in\sX (S)$, the functor $\MMor (x_1,x_2)$ is representable by an $S$-scheme (resp.~by a scheme affine over $S$).
 \end{defin}
 
 E.g., any ind-scheme is a pre-algebraic g-stack.

\begin{rem}    \label{r:not too serious}
 The author does not take the above definition of pre-algebraicity too seriously\footnote{It causes some problems (see Remark~\ref{r:what IS true} and the end of Remark~\ref{r:cat acting on scheme}). On the other hand, the notion of 
 $\MMor$-affineness seems to be reasonable.} (e.g., one could allow 
 $\MMor (x_1,x_2)$ to be some kind of an algebraic space rather than a scheme; on the other hand, one could require the morphism
 $\MMor (x_1,x_2)\to S$ to be quasi-compact and maybe quasi-separated).  In practice, almost all c-stacks and all g-stacks that appear in this article are $\MMor$-affine; however, the c-stack $\Sigma''$ from \S\ref{s:Sigma''} is not $\MMor$-affine.
 \end{rem}

 \begin{rem}   \label{r:Mor-affine}
A g-stack $\sX$ is  $\MMor$-affine if and only if the diagonal morphism $\sX\to\sX\times\sX$ is affine. 
 \end{rem}

\begin{defin}  \label{d:fpqc morphism}
We say that a morphism of schemes $Y_1\to Y_2$ is \emph{faithfully flat} if it is flat and every quasi-compact open $U\subset Y_2$ is the image of some quasi-compact open~$V\subset Y_1$.
\end{defin}

This is \emph{stronger than faithful flatness in the sense of EGA} (i.e., flatness+surjectivity). In practice, there will be no conflict between the two notions of faithful flatness since all schemes that are relevant for this article are quasi-compact.

 \begin{defin}  \label{d:algebraic stack}
A pre-algebraic c-stack $\sX$ is said to be \emph{algebraic} if there exists a 
faithfully flat morphism $f:X\to\sX$ with $X$ being a scheme.  The words ``faithfully flat'' mean here that for every scheme $S$ over $\sX$, the morphism of schemes $S\times_{\sX}X\to S$ is faithfully flat.
\end{defin}

 \begin{rem}   \label{r:finite type not assumed}
The morphism $f:X\to\sX$ from Definition~\ref{d:algebraic stack} is \emph{not assumed} to be locally of finite type. \end{rem}

 \begin{rem}  \label{r:algebraicity in terms of g-stacks}
 Let $\sX$ be a pre-algebraic c-stack. Then $\sX$ is algebraic if and only if its underlying g-stack is.
 \end{rem}

 \begin{ex}
 The c-stacks from \S\ref{ss:simple c-stacks} are algebraic (to construct $f:X\to\sX$, fix a trivialization of $\sL$).
 \end{ex}

 \subsection{Algebraic c-stacks and algebraic categories}
 \begin{defin}  \label{d:algebraic cat}
An \emph{algebraic groupoid} (resp.~\emph{algebraic category}) is a groupoid (resp.~category) in the category of schemes.
 \end{defin}

 \begin{rem}
 Let $\sX$ be a c-stack and $X$  a scheme equipped with a morphism $f:X\to~\sX$.  
 Given a scheme $S$, let $\cC_S$ be the category with $\Ob\cC_S=X(S)$ such that for every $x_1,x_2\in X(S)$ one has $\Mor_{\cC_S}(x_1,x_2)=\Mor_{\sX (S)}(f(x_1),f(x_2))$. The assignment $S\mapsto\cC_S$ is a contravariant functor from the category of schemes 
 to the 1-category of categories. If  $\sX$ 
is pre-algebraic  then the functor $S\mapsto\cC_S$ is representable by 
an algebraic category. 
(This is more familiar if $\sX$ is a g-stack, in which case we get an algebraic groupoid.)
 \end{rem}

 \begin{rem}   \label{r:cat acting on scheme}
 Let $\sX$ be an algebraic c-stack and let $f:X\to\sX$ be as in Definition~\ref{d:algebraic stack}. Let $\Gamma$ be the algebraic category 
 from the previous remark. Let $\Gamma^*$ be the algebraic groupoid 
 underlying $\Gamma$. The groupoid $\Gamma^*$ is flat (indeed, the scheme of morphisms of $\Gamma^*$ is $X\times_{\sX}X$, and the projections $X\times_{\sX}X\to X$ are flat).
  In fact, the category of triples $(\sX ,X,f)$ as above identifies with \emph{a full subcategory} of the category of pairs $(X,\Gamma )$, where $X$ is a scheme and $\Gamma$ is an algebraic category  with $\Ob\Gamma=X$ whose underlying groupoid $\Gamma^*$ is flat. Indeed, one can reconstruct $\sX$ from $(X,\Gamma )$ as follows: associating to a scheme $S$ the category of $S$-points of $\Gamma$, one gets a pre-stack, and $\sX$ is the associated stack, denoted by $X/\Gamma$ (this notation is standard if $\sX$ is a g-stack, i.e., if $\Gamma$ is a groupoid). The above-mentioned full subcategory consists of pairs $(X,\Gamma )$ for which the stack
$X/\Gamma$ is pre-algebraic in the sense of Definition~\ref{d:pre-algebraic stack} (and therefore algebraic). Pre-algebraicity of $X/\Gamma$ does not seem to be automatic because in Definition~\ref{d:pre-algebraic stack} we require the functor $\MMor (x_1,x_2)$ to be an $S$-scheme, but in our situation we only see that $\MMor (x_1,x_2)\times_S S'$ is a scheme for some quasi-compact faithfully flat morphism $S'\to S$. However, $X/\Gamma$ is pre-algebraic (and therefore algebraic) if the morphism $\Gamma\to X\times X$ is quasi-affine. Moreover, $X/\Gamma$ is algebraic if $\Gamma$ is a (possibly infinite) disjoint union of schemes affine over
$X\times X$: this follows from the next lemma.
  \end{rem}
 
 \begin{lem}   \label{l:one-point compactification}
Let $\pi :S'\to S$ be a faithfully flat quasi-compact morphism of schemes. Let $Y'$ be an $S'$-scheme equipped with a descent datum with respect to $\pi$. Suppose that $Y'$ is a (possibly infinite) disjoint union of schemes affine over $S'$. Then $Y'$ descends to a scheme $Y$ over $S$.
\end{lem}

The proof given below is based on the idea of ``one-point compactification''.

\begin{proof}
We can assume that $S, S'$ are affine; let $S'=\Spec A'$. Let $I'\subset H^0(Y',\cO_{Y'})$ be the ideal of functions with quasi-compact support. Then $I'$ is a (non-unital) $A'$-algebra. Applying the usual procedure of ``adding the unit'', one gets a unital $A'$-algebra $B'=A'\oplus I'$. Let $\overline{Y'}:=\Spec B'$. Then $Y'$ is the open subscheme of $\overline{Y'}$ obtained by removing the closed subscheme $Z':=\Spec (B'/I')\subset \overline{Y'}$. The affine scheme $\overline{Y'}$ descends to an $S$-scheme, and the closed subscheme $Z'\subset \overline{Y'}$ also descends.
\end{proof}

\subsubsection{Examples}  \label{sss:c-stacks via alg cats}
Let us describe the c-stacks from \S\ref{sss:BA^1/BG_m)_pm} and \S\ref{sss:hyperbolic} in terms of algebraic categories.

\medskip

(i) The c-stack $(\BA^1/\BG_m)_\pm$ from \S\ref{sss:BA^1/BG_m)_pm} identifies with $\BA^1/\Gamma_\pm$, where $\Gamma_\pm$ is the following algebraic category with 
$\Ob\Gamma_\pm =\BA^1$: for any ring $R$ and $v,\tilde v\in R=\BA^1 (R)$, one has
\[
\Mor_{\Gamma_+}(v,\tilde v)=\{\lambda\in R\,|\,\tilde v=\lambda v\}, \quad \Mor_{\Gamma_-}(v,\tilde v)=\{\lambda\in R\,|\, v=\lambda \tilde v\},
\]
and the composition of morphisms is given by multiplication of $\lambda$'s.

\medskip

(ii) The c-stack from \S\ref{sss:hyperbolic} identifies with $\BA^2/\Gamma$, where $\Gamma$ is the following algebraic category with 
$\Ob\Gamma =\BA^2$: for any ring $R$ and any pairs $(v_+, v_-)\in R^2$ and $(\tilde v_+, \tilde v_-)\in R^2$ one has
\[
\Mor_{\Gamma}((v_+, v_-),(\tilde v_+, \tilde v_-))=\{\lambda\in R\,|\,\tilde v_+=\lambda v_+,\, v_-=\lambda \tilde v_-\}, 
\]
and the composition of morphisms is given by multiplication of $\lambda$'s.

\subsection{Algebraic morphisms of c-stacks}   \label{ss:algebraic over}
\subsubsection{Algebraic and schematic morphisms}
Let $f :\sX\to\sY$ be a morphism of c-stacks. We say that  $f$ is \emph{algebraic} or that $\sX$ is \emph{algebraic over $\sY$} if for every morphism $S\to\sY$ with $S$ being a scheme, the fiber product
$\sX\times_{\sY}S$ is an algebraic c-stack. 

We say that  $f$ is \emph{schematic} if for every morphism $S\to\sY$ with $S$ being a scheme, the fiber product
$\sX\times_{\sY}S$ is a scheme.

\begin{lem}   \label{l:what IS true}
Let $f:\sX\to\sY$ be a morphism of c-stacks. If $\sY$ is algebraic and $\sX$ is pre-algebraic then $\sX$ is algebraic.
\end{lem}

\begin{proof}
There exists a faithfully flat morphism $Y\to\sY$ with $Y$ being a scheme. Since $\sX\times_{\sY}Y$ is algebraic, there exists a faithfully flat morphism $Z\to\sX\times_{\sY}Y$ with $Z$ being a scheme. Then the morphism $Z\to\sX\times_{\sY}Y$ is faithfully flat.
\end{proof}

\begin{cor}     \label{c:what IS true}
Let $f:\sX\to \sY$ and $g:\sY\to \sZ$ be morphsims between pre-algebraic c-stacks. If $f$ and $g$ are algebraic then so is $g\circ f$. \qed
\end{cor}

\begin{rem}    \label{r:what IS true}
The author doubts that in the situation of Lemma~\ref{l:what IS true} pre-algebraicity of  $\sX$ is  automatic without further assumptions (even if $\sX$ and $\sY$ are g-stacks). The reason for doubt is as follows. If $\sY$ is a g-stack and  $x_1,x_2\in\sX (S)$ then $\MMor (f(x_1),f(x_2))$ is a scheme and 
$\MMor (x_1,x_2)\times_{\MMor (f(x_1),f(x_2))}T$ is a scheme for some scheme $T$ faithfully flat over $\MMor (f(x_1),f(x_2))$. But this does not imply that
$\MMor (x_1,x_2)$ is a scheme, in general.
\end{rem}

\subsubsection{Faithfully flat morphisms}  \label{sss:Faithfully flat}
We say that a morphism of c-stacks $g :\sX_1\to\sX$ is \emph{very surjective} if
the corresponding map of fpqc-sheaves is surjective, i.e., if for every morphism $f:S\to\sX$ with $S$ being a scheme, there exists a 
faithfully flat morphism $h:S_1\to S$ such that  for some $f_1:S_1\to\sX_1$ there exists a $2$-isomorphism between $f\circ h$ and $g\circ f_1$.

An algebraic morphism of c-stacks is said to be \emph{faithfully flat} if it is flat and very surjective. (This agrees with the usage of ``faithfully flat'' in Definitions~\ref{d:fpqc morphism} and \ref{d:algebraic stack}.)

\subsection{Left and right fibrations}  \label{ss:left&right fibrations}

\subsubsection{Left and right fibrations of categories}  \label{sss:left fibrations of cats}

Let $\cC$ and $\cD$ be categories.
A functor $\Phi:\cC\to\cD$ is said to be a \emph{left fibration}\footnote{This terminology was introduced by A.~Joyal in the more general setting of $\infty$-categories and then used in \cite[\S 2.1]{Lu1}.}
 if for every $c\in\cC$ the functor $\cC_{c/}\to\cD_{\Phi(c)/}$ induced by $\Phi$ is an equivalence; here $\cC_{c/}$ is the category of objects of $\cC$ equipped with a morphism from $c$.

A left fibration over $\cD$ is the same as a category \emph{cofibered in groupoids} over $\cD$ (we always understand the word ``cofibered'' in Street's ``weak'' sense \cite{St} rather than in Grothendieck's ``strong'' sense\footnote{A discussion of both notions can be found at \url{https://ncatlab.org}.}).

A functor $\cC\to\cD$ is said to be a \emph{right fibration} if the corresponding functor between the opposite categories is a left fibration.

\subsubsection{A particular case} \label{sss:left-fibered subcategories}
If $\cC$ is a strictly full subcategory of $\cD$ then the embedding $\cC\to\cD$ is a left fibration if and only if for every $\cD$-morphism with source in $\cC$ the target is also in $\cC$.

\subsubsection{Left and right fibrations of c-stacks}  \label{sss:left fibrations of c-stacks}
A morphism of c-stacks $\sX\to\sY$ is said to be a \emph{left fibration} (resp.\emph{right fibration} if for every scheme $S$ the corresponding functor $\sX (S)\to\sY (S)$ is a left (resp.~right) fibration. In this situation we also say that $\sX$ is left-fibered (resp.~right-fibered) over $\sX$.

Suppose that $\sX$ is left-fibered over $\sY$. Then for every scheme $S$ equipped with a morphism $f:S\to\sY$ 
the fiber product $\sX_{S,f}:=\sX\times_{\sY ,f}S$ is a g-stack. Moreover, $\sX_{S,f}$ depends functorially on $f\in\sY (S)$, and the formation of $\sX_{S,f}$ commutes with base change $S'\to S$. In fact, a c-stack $\sX$ left-fibered over $\sY$ is \emph{the same} as such a collection of g-stacks $\sX_{S,f}$.

\subsubsection{The notion of group scheme (or ring stack) over a c-stack}  \label{sss:group scheme over c-stack}
Let $\sX$ be a c-stack left-fibered over a c-stack $\sY$. Let $\sX_{S,f}$ be as in \S\ref{sss:left fibrations of c-stacks}. 

Suppose that each $\sX_{S,f}$ is equipped with a structure of Picard stack (resp.~ring stack)  in a way compatible with base change $S'\to S$. Then we say that $\sX$ is a Picard stack (resp.~ring stack) over $\sY$.

Similarly, if the left fibration $\sX\to\sY$ is schematic and if each $\sX_{S,f}$ is equipped with a structure of group scheme (resp.~ring scheme) over $S$ in a way compatible with base change $S'\to S$ then we say that $\sX$ is a group scheme (resp.~ring scheme) over $\sY$.

Thus our group schemes or, say, ring stacks over a c-stack $\sY$ are left-fibered by definition. This restriction is convenient for the purposes of this article, but it is not really necessary (e.g., one could define a group scheme over $\sY$ to be any group object of the category of c-stacks schematic over $\sY$).

\subsubsection{C-stacks left-fibered over $(\BA^1/\BG_m)_\pm$}   \label{sss:C-stacks left-fibered over}
In \S\ref{sss:BA^1/BG_m)_pm} we introduced the c-stacks $(\BA^1/\BG_m)_\pm$.

A left fibration $\sX\to (\BA^1/\BG_m)_\pm$ defines a g-stack $X:=\sX\times_{(\BA^1/\BG_m)_\pm}\BA^1$ equipped with a morphism $\pi:X\to\BA^1$. In fact, a c-stack left-fibered over $(\BA^1/\BG_m)_\pm$ is \emph{the same as} a g-stack $X$ equipped with a morphism $\pi:X\to\BA^1$ and a certain additional piece of structure. In the case of 
$(\BA^1/\BG_m)_+$ the additional piece of structure is an action of the multiplicative monoid $\BA^1$ on $X$ such that the $\pi$ is equivariant.  In the case of $(\BA^1/\BG_m)_-$ the additional piece of structure is less familiar: it is an ``anti-action'' of the multiplicative monoid on $X$. Roughly, this means that given 
$\lambda,\mu\in\BA^1 (S)$ we are given a morphism $X_{\lambda\mu} \to X_\lambda$ so that certain diagrams commute (here $X_\lambda$ is the base change of $X$ with respect to $\lambda :S\to\BA^1$).

Anti-actions are discussed in \cite[ \S 2.1.14 ]{G_m-actions} (and also in Appendix C of~\cite{G_m-actions}) in a certain ``real life'' situation. Another  ``real life'' example: if $Y$ is any scheme, $Z\subset Y$ is a closed subscheme, and $X$ is the corresponding deformation to the normal cone then $X$ is a scheme over $\BA^1$ equipped with an anti-action of the multiplicative monoid.

In this article we avoid the language of anti-actions; instead, we treat c-stacks left-fibered over $(\BA^1/\BG_m)_-$ as explained in \S\ref{sss:left fibrations of c-stacks}.

\begin{lem}   \label{l:left-fibered substacks}
Let $\sX$ be a c-stack. The following properties of a substack $\sY\subset\sX$ are equivalent:

(i) $\sY$ is left-fibered over $\sX$;

(ii) if $S$ is a scheme and $\varphi$ is a morphism in $\sX (S)$ whose source is in $\sY (S)$ then the target of $\varphi$ is in $\sY (S)$;

(iii) for any scheme $S$, any morphisms $f,g: S\to\sX$, and any morphism $f\to g$ in $\sX (S)$ one has $f^{-1}(\sY )\subset g^{-1}(\sY )$.
\end{lem}

\begin{proof}
By \S\ref{sss:left-fibered subcategories}, we have (i)$\Leftrightarrow$(ii). In the situation of (iii) the inclusion $f^{-1}(\sY )\subset g^{-1}(\sY )$ means that for every morphism of schemes $\varphi :S'\to S$ if $\varphi^* (f)\in\sY (S)$ then $\varphi^* (g)\in\sY (S)$. So 
(iii)$\Leftrightarrow$(ii).
\end{proof}

\begin{rem}
Lemma~\ref{l:left-fibered substacks} remains valid if one replaces ``left'' by ``right'' in (i), interchanges ``source'' and ``target'' in (ii), and replaces $\subset$ by
$\supset$ in (iii).
\end{rem}

\begin{cor}   \label{c:left-fibered substacks}
Let $\sX$ be a c-stack. If a closed substack $\sY\subset\sX$ is left-fibered (resp.~right-fibered) over $\sX$ then the open substack $\sX\setminus\sY\subset\sX$ is right-fibered (resp.~left-fibered) over~$\sX$.
\end{cor}

\begin{proof}   
Use the equivalence (i)$\Leftrightarrow$(iii) from Lemma~\ref{l:left-fibered substacks}.
\end{proof}

\subsection{$\cO$-modules on c-stacks}  \label{ss:O-mododules on c-stacks} 
Let $\sX$ be a c-stack. 

\subsubsection{Contravariant and covariant $\cO_{\sX}$-modules}  \label{sss:O-mododules on c-stacks} 
By a \emph{contravariant $\cO_{\sX}$-module}  we mean a collection of contravariant functors 
\begin{equation}   \label{e:X(S) to modules}
\sX (S)\to\{\mbox{quasi-coherent }\cO_S\mbox{-modules}\}, \quad S\in\{\mbox{Schemes}\},
\end{equation}
which is compatible with base change $\tilde S\to S$. Contravariant $\cO_{\sX}$-modules form a symmetric monoidal category with respect to $\otimes$.

One also has a similar notion of covariant $\cO_{\sX}$-module.  Of course, a covariant $\cO_{\sX}$-module is the same a contravariant $\cO_{\sX^{\op}}$-module, where
$\sX^{\op}$ is the c-stack \emph{opposite to $\sX$} (this means that $\sX^{\op}(S)=\sX(S)^{\op}$ for any $S$).
Note that if $\sX$ is a g-stack then $\sX^{\op}=\sX$, so in this case a covariant $\cO_{\sX}$-module is the same as a contravariant $\cO_{\sX}$-module.

The category of contravariant or covariant $\cO_{\sX}$-modules is \emph{not abelian, in general} (e.g., kernels do not always exist). However, it is abelian if $\sX$ is algebraic.

\subsubsection{Duality}   \label{sss:Duality}
Suppose that a contravariant $\cO_{\sX}$-module $M$ is locally free and finitely generated (i.e., $f^*M$ is a vector bundle for any morphism $f:S\to\sX$ with $S$ being a scheme). Then for any morphism $f:S\to\sX$ with $S$ being a scheme one has the $\cO_S$-module $(f^*M)^*$. These modules give rise to a \emph{covariant} $\cO_{\sX}$-module $M^*$. Similarly, if $M$ is a covariant  $\cO_{\sX}$-module then $M^*$ is a contravariant one.

\subsubsection{Weakly invertible $\cO_{\sX}$-modules} \label{sss:Weakly invertible}
A covariant or contravariant $\cO_{\sX}$-module $M$ is said to be \emph{weakly invertible} if its pullback to the underlying g-stack of $\sX$ is invertible. In this case we often write $M^{-1}$ or $M^{\otimes -1}$ instead of $M^*$. The word ``weakly'' emphasizes the fact that passing to the inverse interchanges covariant and contravariant modules.

\subsubsection{Strongly invertible $\cO_{\sX}$-modules} \label{sss:Strongly invertible}
Invertible objects of the tensor category of covariant (resp.~contravariant) $\cO_{\sX}$-modules are called \emph{strongly invertible} covariant 
(resp.~contravariant) $\cO_{\sX}$-modules. A covariant (resp.~contravariant) $\cO_{\sX}$-module is strongly invertible if and only if it is weakly invertible and each of the corresponding functors \eqref{e:X(S) to modules} takes all morphisms to isomorphisms. Inverting these isomorphisms, one gets a \emph{canonical equivalence between the category of strongly invertible covariant $\cO_{\sX}$-modules and that of 
strongly invertible contravariant $\cO_{\sX}$-modules;} we will not distinguish the two categories.

\subsubsection{The relative $\Spec$}    \label{sss:Spec}
By a contravariant $\cO_{\sX}$-algebra we mean a commutative algebra in the symmetric monoidal category of contravariant $\cO_{\sX}$-modules. One has an anti-equivalence $\cA\mapsto\Spec\cA$ between the category of contravariant $\cO_{\sX}$-algebras and the category of affine left fibrations over $\sX$. If $S$ is a scheme then an $S$-point of $\Spec\cA$ is a pair $(f,\alpha)$, where $f\in\Mor (S,\sX)=\sX (S)$ and $\alpha :f^*\cA\to\cO_S$ is an $\cO_S$-algebra homomorphism. Following EGA, for any contravariant $\cO_{\sX}$-module $M$ we write
\[
\BV (M):=\Spec\Sym (M);
\]
$\BV (M)$ is an affine left fibration over $\sX$.

\subsubsection{The affine fibration corresponding to a weakly invertible $\cO_{\sX}$-module}    \label{sss:seno-soloma}
Let $\sX$ be a \mbox{c-stack} and $\sL$ a weakly invertible covariant $\cO_{\sX}$-module. Then $\BV (\sL ^*)$ is called the
affine fibration corresponding to $\sL$; this is a \emph{left} fibration over $\sX$ (because $\sL^*$ is a \emph{contra}variant $\cO_{\sX}$-module).

\subsubsection{Quotients of $\cO_{\sX}$ and closed substacks of $\sX$}    \label{sss:quotients of O_X}
A quotient of $\cO_{\sX}$ in the category of contravariant $\cO_{\sX}$-modules has a unique algebra structure. The functor $\Spec$ gives a bijection between such quotients and closed substacks of $\sX$ left-fibered over $\sX$. The quotient of $\cO_{\sX}$ corresponding to a closed substack $\sY\overset{i}\mono\sX$ will usually be denoted by $\cO_{\sY}$ rather than~$i_*\cO_{\sY}$.

\subsubsection{Variant}  \label{sss:Variant}
In \S\ref{sss:Spec}-\S\ref{sss:quotients of O_X} one can replace ``contravariant''  and ``left'' by ``covariant''  and ``right''.

\subsubsection{A corollary}    \label{sss:SpecCoker}
As a consequence of \S\ref{sss:quotients of O_X}-\ref{sss:Variant}, we see that if $\sM$ is a covariant (resp.~contravariant) $\cO_{\sX}$-module equipped with a morphism $f:\sM\to\cO_{\sX}$ then $\Coker f=\cO_{\sY}$, where $\sY\subset\sX$ is a closed substack right-fibered (resp.~left-fibered) over $\sX$.

\subsection{Infinitesimal neighborhoods and conormal sheaves}    \label{ss:infinitesimal neighb}
Let $\sX$ be a c-stack and $\sY\subset\sX$ a closed substack. By the \emph{$n$-th  infinitesimal neighborhood of $\sY$ in $\sX$} we mean the closed substack 
$\sY_n\subset\sX$  such that for any scheme $S$ over $\sX$, the subscheme $\sY_n\times_{\sX}S\subset S$ is equal to the $n$-th infinitesimal neighborhood of 
$\sY\times_{\sX}S$ in $S$. We have $\sY_0=\sY$.

\begin{lem}    \label{l:infinitesimal neighb}
If $\sY$ is left-fibered (resp.~right-fibered) over $\sX$ then so is $\sY_n$.
\end{lem}

\begin{proof}
Use the equivalence (i)$\Leftrightarrow$(iii) from Lemma~\ref{l:left-fibered substacks} and its right-fibered analog.
\end{proof}

\subsubsection{Conormal sheaf}   \label{ss:conormal definition}
As before, let $\sY_1$ be the first infinitesimal neighborhood of $\sY$ in $\sX$. Suppose that the c-stack $\sY_1$ is algebraic; then the category of contravariant (or covariant) $\cO_{\sY_1}$-modules is abelian.

If $\sY$ is left-fibered over $\sX$ we have the contravariant $\cO_{\sY_1}$-module $\cO_\sY$ (see  \S\ref{sss:quotients of O_X}), which is a quotient of $\cO_{\sY_1}$.
Then $\Ker (\cO_{\sY_1}\epi\cO_{\sY})$ is a contravariant $\cO_{\sY}$-module. It is called the \emph{conormal sheaf} of $\sY$ in $\sX$.

Similarly, if $\sY$ is right-fibered over $\sX$ one defines the conormal sheaf of $\sY$ in the category of \emph{covariant} $\cO_\sY$-modules. If $\sY$ is neither left-fibered nor right-fibered over $\sX$ we can only define the conormal sheaf as an $\cO$-module on the underlying g-stack of $\sY$.

\subsection{Formal schemes and formal c-stacks}   \label{ss:formal schemes}
\subsubsection{Definitions}   \label{sss:def formal scheme}
By a \emph{formal c-stack} we mean a c-stack $\sX$ 
which can be represented as
\begin{equation}  \label{e:formal scheme as colim}
\sX= \underset{\longrightarrow}{\lim}(\sX_1\mono \sX_2\mono\ldots ),
\end{equation}
where each $\sX_i$ is an algebraic c-stack, the morphisms $\sX_i\to \sX_{i+1}$ are closed embeddings, and for each $i$ and each quasi-compact scheme $S$ over $\sX$ the ideal of $\sX_i\times_{\sX}S$ in $\sX_{i+1}\times_{\sX}S$ is nilpotent.
We say that $\sX$ is a \emph{formal scheme} if $\sX$ admits a presentation \eqref{e:formal scheme as colim} such that each $\sX_i$ is a scheme\footnote{One can show 
that then for \emph{every} presentation \eqref{e:formal scheme as colim} each $\sX_i$ is a scheme.}.

Any formal c-stack is clearly pre-algebraic. By Remark~\ref{r:algebraicity in terms of g-stacks}, a pre-algebraic c-stack is formal if and only if its underlying g-stack is.

\subsubsection{Relation to Grothendieck's notion of formal scheme}  \label{sss:relation to EGA}
Let $\sX$ be a formal scheme in the sense of \S\ref{sss:def formal scheme}. Given a presentation 
\eqref{e:formal scheme as colim}, the underlying topological space of $\sX_n$ does not depend on $n$; denote it by $\sX$. Let $\cO_\sX$ be the following sheaf of topological rings on~$\sX$: if $U\subset \sX$ is a quasi-compact open subset then $H^0(U, \cO_\sX)$ is the projective limit of the discrete rings $H^0(U,\cO_{\sX_n})$. Then $(\sX,\cO_\sX)$ is a formal scheme in the sense of EGA I. It does not depend on the choice of a presentation \eqref{e:formal scheme as colim}. Thus the category of formal schemes in the sense of \S\ref{sss:def formal scheme} identifies with a full subcategory of the category of formal schemes in the sense of EGA I.

\subsubsection{Affine formal schemes}
From now on, the words ``formal scheme'' are understood in the sense of \S\ref{sss:def formal scheme}. A formal scheme is said to be affine if all its closed subschemes are. In terms of a presentation \eqref{e:formal scheme as colim} this just means that $\sX_n$ is affine for all (or some) $n$.

Associating to an affine formal scheme $\sX$ the ring of regular functions on $\sX$ equipped with its natural topology (see \S\ref{sss:relation to EGA}), one gets a fully faithful contravariant functor from the category of affine formal schemes to that of topological rings. A topological ring $A$ is in its essential image if and only if $A$ is separated, complete, and has a countable base of neighborhoods of zero formed by topologically nilpotent ideals. The formal scheme corresponding to such $A$ is denoted by $\Spf A$. If $S$ is a quasi-compact scheme then an $S$-point of $\Spf A$ is the same as a homomorphism $A\to H^0(S,\cO_S)$ with open kernel. If open ideals $I_1\supset I_2\supset\ldots $ in $A$ form a base of neighborhoods of zero then $\Spf A$ is the inductive limit of the schemes $\Spec (A/I_n)$.

\subsubsection{$S$-points of $\Spf\BZ_p$}   \label{sss:p-nilpotence}
Let $p$ be a prime. A scheme $S$ is said to be \emph{$p$-nilpotent} if the element $p\in H^0(S,\cO_S)$ is locally nilpotent. A ring $R$ is said to be $p$-nilpotent if $p$ is nilpotent in $R$.

It is clear that if $S$ is $p$-nilpotent then the formal scheme 
$\Spf\BZ_p ={\underset{\longrightarrow}{\lim}}\Spec (\BZ/p^n\BZ )$ has exactly one $S$-point; otherwise $\Spf\BZ_p$ has no $S$-points.

So a c-stack over $\Spf\BZ_p$ is the same as a c-stack $\sX$ with the following property: $\sX (S)\ne\emptyset$ only if $S$ is $p$-nilpotent. (We are mostly interested in such c-stacks.)

\subsubsection{$p$-adic formal schemes}  \label{sss:p-adic formal}
We say that a formal scheme $\sX$ is \emph{$p$-adic} if it has a representation \eqref{e:formal scheme as colim} with the following properties:
each $\sX_n$ is a scheme over $\BZ /p^n\BZ$ and for each $n$ the morphism $\sX_n\mono \sX_{n+1}$ induces an isomorphism $\sX_n\iso \sX_{n+1}\otimes (\BZ/p^n\BZ)$. Such a representation is unique if it exists: namely, $\sX_n=\sX\otimes (\BZ/p^n\BZ)$. So 
\[
\widehat{\Sch}_{\BZ_p}=\underset{n}{\underset{\longleftarrow}{\lim}}\Sch_{\BZ/p^n\BZ}\, ,
\]
where $\widehat{\Sch}_{\BZ_p}$ is the category of $p$-adic formal schemes and $\Sch_{\BZ/p^n\BZ}$ is the category of schemes over $\BZ/p^n\BZ$.

\subsubsection{$p$-adic completion}   \label{sss:p-adic completion}
The \emph{$p$-adic completion} of a c-stack $\sX$ is defined to be the c-stack
\[
\sX\hat\otimes\BZ_p:=\sX\times\Spf\BZ_p=\underset{n}{\underset{\longrightarrow}{\lim}} (\sX\times\Spec (\BZ/p^n\BZ )).
\]
In other words, for any $p$-nilpotent scheme $S$ one has $( \sX\hat\otimes\BZ_p)(S)=\sX (S)$, and if $S$ is not $p$-nilpotent then  $( \sX\hat\otimes\BZ_p)(S)=\emptyset$.

If $\sX$ is an algebraic  c-stack then $\sX\hat\otimes\BZ_p$ is algebraic over $\Spf\BZ_p$. If $\sX$ is a scheme then $\sX\hat\otimes\BZ_p$ is a $p$-adic formal scheme.

\begin{lem}   \label{l:algebraic over formal}
Let $\pi :\sY\to\sX$ be a morphism of c-stacks. Suppose that $\sY$ is pre-algebraic and $\sX$ is formal. Then $\pi$ is algebraic if and only if the $\pi$-preimage of each algebraic closed substack of $\sY$ is algebraic. If this is the case then $\sY$ is formal.
\end{lem}

\begin{proof}
To prove the ``only if'' statement, use Lemma~\ref{l:what IS true}.
\end{proof}

\subsubsection{Strongly adic c-stacks}   \label{sss:strongly adic}
Let $\sX$ be a c-stack. Let $\sX_n$ be the $n$-th infinitesimal neighborhood of $\sX_{\red}$ (see \S\ref{ss:infinitesimal neighb} and  \S\ref{sss:X_red}).
We say that $\sX$ is \emph{strongly adic} if

(i) each $\sX_n$ is algebraic,

(ii) $\sX= \underset{\longrightarrow}{\lim}\sX_n$,

(iii) the ideal of $\sX_{\red}$ in $\sX_n$ is finitely generated for each $n$ (or equivalently, for $n=1$).

\noindent It is clear that strongly adic c-stacks are formal.

\begin{rem}   \label{r:adicity in terms of g-stacks}
A strongly adic c-stack is clearly pre-algebraic. By Remark~\ref{r:algebraicity in terms of g-stacks}, a pre-algebraic c-stack is strongly adic if and only if its underlying g-stack is.
\end{rem}

\subsection{Effective Cartier divisors on pre-algebraic stacks}   \label{ss:effective Cartier}

\subsubsection{Pre-divisors on schemes}  \label{sss:Pre-div on schemes}
Let $S$ be a scheme. A quasi-coherent ideal $\cI\subset\cO_S$ is said to be locally principal if it satisfies any
of the following equivalent conditions:

(i) there exists a Zariski covering of $S$ by schemes $S_i$ such that the pullback of $\cI$ to each $S_i$ is a principal ideal;

(ii) there exists an fpqc covering with this property;

(iii) $\cI$ is finitely generated and the fiber of the $\cO_S$-module $\cI$ at each point of $S$ has dimen\-sion~$\le 1$.

We say that a closed subscheme $D\subset S$ is a \emph{pre-divisor} on $S$ if the corresponding ideal $\cI_D$ is locally principal. 

If $D_1,D_2\subset S$ are pre-divisors then the ideal $\cI_{D_1}\cdot\cI_{D_2}$ is locally principal; the corresponding pre-divisor is denoted by $D_1+D_2$. Thus pre-divisors on $S$ form a commutative monoid; we denote it by~$\Pre-div (S)$.

If $f:S'\to S$ is a morphism and $D\subset S$ is a pre-divisor then $f^{-1}(D)\subset S'$ is a pre-divisor. The map 
$\Pre-div (S)\to\Pre-div (S')$ defined by $D\mapsto f^{-1}(D)$ is a homomorphism of monoids.

\subsubsection{Pre-divisors on stacks}  \label{sss:Pre-div on stacks}
Let $\sX$ be a c-stack.  Let $\Sch_{\sX}$ be the category of schemes over~$\sX^{\rm g}$, where $\sX^{\rm g}$ is the underlying 
g-stack of $\sX$. By a pre-divisor on $\sX$ we mean a closed substack $D\subset\sX$ such that for every scheme $S\in \Sch_{\sX}$ the closed subscheme $D\times_{\sX}S\subset S$ is a pre-divisor on $S$. Let $\Pre-div (\sX )$ be the set of all pre-divisors on $\sX$. If $D_1,D_2\in\Pre-div (\sX )$ there is a unique $D\in\Pre-div (\sX )$ such that for every 
$S\in\Sch_{\sX}$ one has $D\times_{\sX}S=(D_1\times_{\sX}S)+(D_2\times_{\sX}S)$; we write $D=D_1+D_2$. Thus 
$\Pre-div (\sX )$ is a commutative monoid.

\subsubsection{What will be done}
We will define a submonoid of $\Pre-div (\sX )$, whose elements are called \emph{effective Cartier divisors}.\footnote{The definition is motivated by Examples~\S\ref{ex:Spf Z_p}-\ref{ex:formal schemes} below.} This definition is reasonable in the case of stacks which are pre-algebraic in the sense of Definition~\ref{d:pre-algebraic stack} (e.g., in the case of formal stacks). More precisely,
in the case of pre-algebraic g-stacks one has the usual relation between effective Cartier divisors and line bundles 
(see \S\ref{ss:divisors&line bundles}).

\subsubsection{Definition of effective Cartier divisor}  \label{sss:effective Cartier}
Let $D\in\Pre-div (\sX )$. For $S\in\Sch_{\sX}$, let $\cI_{D,S}\subset\cO_S$ be the ideal of $D\times_{\sX}S\subset S$. The $\cO_S$-module $\cI_{D,S}$ is the pushforward of an invertible $\cO$-module on some closed subscheme $Y_{D,S}\subset S$. 
If $f:S\to S'$ is a morphism in $\Sch_{\sX}$ then $f(Y_{D,S})\subset Y_{D,S'}$, and if $f$ is flat then 
$f^{-1}(Y_{D,S'})=~Y_{D,S}$. 

We say that $f$ is \emph{$D$-good} if $f(S)\subset Y_{D,S'}$. Note that if $f$ is $D$-good then so is $g\circ f$, where
 $g$ is any $\Sch_{\sX}$-morphism with source $S'$.

We say that $S$ is $D$-good if there exists a $D$-good $\sX$-morphism $S\to S'$ for some~$S'\in\Sch_{\sX}$. 

We say that $D$ is an \emph{effective Cartier divisor} if for every $S\in\Sch_{\sX}$ there exists a 
faithfully flat morphism $\tilde S\to S$ with $\tilde S$ being $D$-good. The set of all effective Cartier divisors on $\sX$ is denoted by $\Divp (\sX )$. By definition, $\Divp (\sX )$ depends only on the underlying g-stack of $\sX$.

According to Lemma~\ref{l:usual notion for algebraic stacks} below, if $\sX$ is algebraic then the above notion of effective Cartier divisor is equivalent to the usual one; this is quite obvious if $\sX$ is a scheme.

\begin{lem}  \label{l:summand of divisor}
Let $D\in\Divp (\sX )$, $D'\in \Pre-div (\sX )$, and $D'\subset D$. Then $D'\in\Divp (\sX )$.
\end{lem} 

\begin{proof}
Follows from the inclusion $Y_{D,S}\subset Y_{D',S}$, $S\in\Sch_{\sX}$. 
\end{proof}

\begin{ex}  \label{ex:Spf Z_p}
Let $\sX=\Spf\BZ_p$ and $D=\Spec (\BZ/p^m\BZ)$. Then the pre-divisor $D\subset\sX$ is an effective Cartier divisor. 
Indeed, if $S$ is a scheme over $\Spec (\BZ/p^n\BZ )$ then the morphism $S\to\Spec (\BZ/p^{n+m}\BZ )$ is $D$-good.
\end{ex}

\begin{ex}   \label{ex:formal schemes}
Let $\sX$ be the inductive limit of a diagram of schemes $$X_1\mono X_2\mono\ldots$$ 
(all arrows in the diagram are monomorphisms, and the inductive limit is taken in the category of fpqc sheaves on the category of schemes);
for example, $\sX$ could be a formal scheme. 
Then $\sX$ is a pre-algebraic stack. If $D\subset\sX$ is a pre-divisor then the ideal of $D\times_{\sX}X_n$ in $X_n$ is an invertible sheaf on a closed subscheme $Y_n\subset X_n$. It is easy to see that $D$ is an effective Cartier divisor if and only if for every $n$ and every quasi-compact open $U\subset X_n$ there exists $N\ge n$ such that the morphism $X_n\to X_N$ maps $U$ to $Y_N$. 
\end{ex}

\begin{ex}
Let $\sX =\Spf A$ be an affine formal scheme. Then an effective Cartier divisor on $\sX$ is the same as an invertible $A$-submodule $I\subset A$ satisfying the following condition: the induced topology on $I$ is the same as the one that comes from $I$ being a finitely generated projective $A$-module. By the Artin-Rees lemma, the condition holds automatically if $A$ is noetherian and the topology on $A$ is adic.
\end{ex}

\begin{lem}   \label{l:usual notion for algebraic stacks}
If $\sX$ is algebraic then the notion of effective Cartier divisor  from \S\ref{sss:effective Cartier} is equivalent to the usual one.
\end{lem} 

\begin{proof}
Choose a 
faithfully flat morphism $X\to\sX$.

If $D\subset\sX$ is an effective Cartier divisor in the usual sense then $D\in\Divp (\sX )$: indeed, for every $S\in\Sch_{\sX}$ the $\sX$-scheme $S\times_{\sX}X$ is $D$-good because the morphism $S\times_{\sX}X\to X$ is $D$-good.

Let us show that if $D\in\Divp (\sX )$ then $D$ is an effective Cartier divisor in the usual sense. After changing $X$, we can assume that $X$ is $D$-good, so there exists a $D$-good $\Sch_{\sX}$-morphism $f:X\to X'$. The problem is to show that
$\id_X :X\to X$ is $D$-good. Let $X'':= X'\times_{\sX}X$. The morphisms $f$ and $\id_X$ define a morphism $h:X\to X''$. Since $X''$ is flat over $X'$, we have $Y_{D,X''}=Y_{D,X'}\times_{X'}X''$. Since $f$ is $D$-good, we see that $h$ is $D$-good. The composite morphism $X\overset{h}\longrightarrow X''\to X$ equals $\id_X$, so $\id_X$ is $D$-good.
\end{proof}

\begin{lem}  \label{l:Divp is fpqc-local}
Let $g:\sX_1\to\sX$ be a morphism of stacks. Let $D\subset\sX$ be a closed substack and $D_1:=g^{-1}(D)$.

(i) If $g$ is a flat algebraic morphism and $D\in\Divp (\sX )$ then $g^{-1}(D)\in\Divp (\sX_1 )$.

(ii)  If $g$ is very surjective in the sense of \S\ref{sss:Faithfully flat} and $D_1\in\Divp (\sX_1 )$ then $D\in\Divp (\sX )$.
\end{lem} 

In particular, statement (ii) applies if $g$ is algebraic and faithfully flat in the sense of \S\ref{sss:Faithfully flat}.

\begin{proof}
(i) Let $S\in\Sch_{\sX_1}$ and suppose that an $\Sch_{\sX}$-morphism $f:S\to S'$ is $D$-good. Then the 
stack $\sY:=S'\times_{\sX}\sX_1$ is algebraic, so there exists a 
faithfully flat morphism  $S'_1\to\sY$.
Let $\tilde S:=S\times_{\sY}S'_1$ (this is a scheme because $\sY$ is algebraic). We have an $\sX_1$-morphism $\tilde S\to S_1'$. 
It is $D_1$-good because $Y_{D_1,S'_1}=Y_{D,S'}\times_{S'}S'_1$ by flatness of $S'_1$ over~$S'$. 

(ii) Let $S$ be a scheme and $f:S\to\sX$ a morphism. Let $h:S_1\to S$ and $f_1:S_1\to\sX_1$ be 
such that  for some $f_1:S_1\to\sX_1$ the diagram 
\[
\xymatrix{
S_1\ar[r]^h \ar[d]_{f_1} & S\ar[d]^f\\
\sX_1\ar[r]^g & \sX
}
\]
commutes.
Then $h^{-1}(D\times_{\sX}S)=D\times_{\sX}S_1=D_1\times_{\sX_1}S_1$ is a pre-divisor in $S_1$.
Since $h$ is 
faithfully flat, this implies that $D\times_{\sX}S$ is a pre-divisor in~$S$. 
A $D_1$-good $\Sch_{\sX_1}$-morphism $S_1\to S'$ is also a $D$-good $\Sch_{\sX}$-morphism.
\end{proof}

\begin{lem}
$\Divp (\sX)$ is a submonoid of $\Pre-div (\sX)$.
\end{lem}

\begin{proof}
For every $S\in\Sch_{\sX}$ there is a diagram 
$S\overset{\pi}\leftarrow\tilde S\overset{f}\rightarrow S'\overset{\pi'}\leftarrow\tilde S'\overset{f'}\rightarrow S''$
such that $\pi,\pi'$ are 
faithfully flat, $f$ is $D_1$-good, and $f'$ is $D_2$-good. Then the morphism 
$\tilde f:\tilde S\times_{S'}\tilde S'\to \tilde S'$ is $D_1$-good. Using this and $D_2$-goodness of $f'$, one checks that 
$f'\circ\tilde f:\tilde S\times_{S'}\tilde S'\to S''$ is $(D_1+D_2)$-good.  
\end{proof}

\subsection{Divisors and line bundles}   \label{ss:divisors&line bundles}
In this subsection we assume that $\sX$ is a \emph{pre-algebraic g-stack.}

\begin{prop}   \label{p:pre-algebraicity enough}

(i) Let $D\in\Divp (\sX )$. Then the morphism $\cO_{\sX}\epi\cO_D$ admits a kernel $\cO_{\sX}(-D)$ in the (possibly non-abelian) category of quasi-coherent $\cO_{\sX}$-modules. Moreover, the $\cO_{\sX}$-module $\cO_{\sX}(-D)$ is invertible.

(ii) Let $D\in\Divp (\sX )$. Let $\sL$ be an invertible $\cO_{\sX}$-module  and  $\varphi:\sL\to \cO_{\sX}$ a morphism with 
$\Coker \varphi=\cO_D$. Then $\varphi$ induces an isomorphism $\sL\iso \cO_{\sX}(-D)$.

(iii) If $D_1,D_2\in\Divp (\sX )$ then the morphism $\cO_{\sX}(-D_1)\otimes \cO_{\sX}(-D_2)\to\cO_{\sX}$ that comes from the canonical morphisms $\cO_{\sX}(-D_i)\to\cO_{\sX}$ induces an isomorphism 
$$\cO_{\sX}(-D_1)\otimes \cO_{\sX}(-D_2)\iso\cO_{\sX}(-(D_1+D_2)).$$
\end{prop}

\begin{proof}   
(i) Let $S\in\Sch_{\sX}$ be $D$-good. Let $\cC$ be the category of pairs $(S',f)$, where $S'\in\Sch_{\sX}$ and $f:~S\to~S'$ is a $D$-good $\Sch_{\sX}$-morphism. The category $\cC$ is non-empty (because $S$ is $D$-good), and each two objects of $\cC$ have a Cartesian product in $\cC$ (because $\sX$ is pre-algebraic). These two facts are known to imply that the nerve $N(\cC )$ is contractible\footnote{Here is a proof, which I learned from Omar Antolin Camarena's webpage. Choose $c_0\in\cC$ and define  $F, G:\cC\to\cC$ by $F(c):=c\times c_0$, $G(c)=c_0$. By \cite[\S 1, Prop.~2]{Q}, the natural transformations $G\leftarrow F\to\id_{\cC}$ induce homotopies between the corresponding maps $N(\cC )\to N(\cC )$.}.

Given $(S',f)\in\cC$, define an invertible $\cO_S$-module $\sL_{S',f}$ by $\sL_{S',f}=f^*\cI_{D,S'}$, where $\cI_{D,S'}$ is the ideal of $D\times_{\sX}S'\subset S'$ (viewed as an $\cO_{S'}$-module). The assignment 
$(S',f)\mapsto \sL_{S',f}$ is a functor from $\cC$ to the groupoid of invertible $\cO_S$-modules.
Since $N(\cC )$ is contractible, such a functor is the same as an invertible $\cO_S$-module $\sL_S$ (if you wish, $\sL_S$ is the projective limit of the functor). The line bundles $\sL_S$ define an invertible $\cO_{\sX}$-module $\sL$; it is equipped with a morphism 
$\varphi :\sL\to\cO_\sX$ such that $\Coker\varphi =\cO_D$.

The pair $(\sL ,\varphi )$ is the kernel of the morphism $\cO_{\sX}\epi\cO_D$: indeed,
if $M$ is a quasi-coherent $\cO_{\sX}$-module and $g:M\to\cO_{\sX}$ is  such that $\im (g_S:M_S\to\cO_S)\subset\cI_{D,S}$ for $S\in\Sch_{\sX}$ then $g$ uniquely factors as $M\to\sL\overset{\varphi}\longrightarrow\cO_{\sX}$.

(ii) It is clear that $\varphi$ induces an epimorphism $\sL\epi \cO_{\sX}(-D)$. Since $\sL$ and $\cO_{\sX}(-D)$ are invertible it is automatically an isomorphism.

(iii) Apply statement (ii) for $\sL=\cO_{\sX}(-D_1)\otimes \cO_{\sX}(-D_2)$ and $D=D_1+D_2$.
\end{proof}

\begin{rem}
Let $\sL$ be an invertible $\cO_\sX$-module and $\varphi:\sL\to\cO_{\sX}$ a morphism. Let $D\subset\sX$ be the closed substack such that
$\Coker\varphi =\cO_D$. It is clear that $D$ is a predivisor. It is also clear that $D$ is an effective Cartier divisor if and only if $\varphi$ satisfies the following condition: for every $S\in\Sch_{\sX}$ there exists a diagram $S\overset{\pi}\leftarrow\tilde S\overset{f}\rightarrow S'$ in 
$\Sch_{\sX}$ such that $\pi$ is 
faithfully flat and the morphism $f^*\sL_{S'}\to\sL_{\tilde S}$ kills $f^*\sK_{S'}$, where
$\sK_{S'}:=\Ker (\varphi_{S'}:\sL_{S'}\to\cO_{S'})$.
\end{rem}

\begin{rem}
By Proposition~\ref{p:pre-algebraicity enough}(i-ii), an effective Cartier divisor on $\sX$ is \emph{the same as} a pair
$(\sL ,\varphi )$ satisfying the condition from the previous remark.
\end{rem}

\section{$W_S$-modules}  \label{s:W-modules}
In most of this section we work with arbitrary schemes (rather than schemes over $\BZ_{(p)}$ or~$\BZ_p$). 

\subsection{Quasi-ideals} \label{ss:quasi-ideals} 
\subsubsection{Definition} \label{sss:quasi-ideals} 
Let $C$ be a (commutative) ring. By a \emph{quasi-ideal} in $C$ we mean a pair $(I,d)$, where $I$ is a $C$-module and $d:I\to C$ is a $C$-linear map such that 
\begin{equation}  \label{e:quasi-ideal}
d(x)\cdot y=d(y)\cdot x
\end{equation}
for all $x,y\in I$. 
Note that a quasi-ideal $(I,d)$ with $\Ker d=0$ is the same as an ideal in $C$.

\subsubsection{Remarks}  \label{sss:quasi-ideal rems}
(i) Condition~\eqref{e:quasi-ideal} holds automatically if the $C$-module $I$ is invertible.

(ii) If  $(I,d)$ is a quasi-ideal in $C$ then $I$ is a (non-unital) ring with respect to the multiplication operation 
$(x,y)\mapsto d(x)\cdot y$.

(iii) If $(I,d)$ is a quasi-ideal in a ring $C$ then one can define a DG ring $A$ as follows: $A^0=C$, $A^{-1}=I$, $A^i=0$ for $i\ne 0,1$, the differential in $A$ is given by $d:I\to C$, and the multiplication maps 
$$A^0\times A^0\to A^0, \quad A^0\times A^{-1}\to A^{-1}$$ 
come from the ring structure on $C$ and the $C$-module structure on $I$; note that the Leibnitz rule in $A$ is equivalent to \eqref{e:quasi-ideal}. Thus one gets an equivalence between the category of triples $(C,I,d)$ and the category of DG rings $A$ such that $A^i=0$ for $i\ne 0,-1$.

\subsubsection{Quasi-ideals in ring schemes} \label{sss:quasi-ideal schemes} 
Now let $C$ be a ring scheme over some scheme $S$. Then by a \emph{quasi-ideal} in $C$ we mean a pair $(I,d)$, where $I$ is a commutative group $S$-scheme equipped with an action of $C$ and $d:I\to C$ is a $C$-linear morphism such that for every $S$-scheme $S'$ and every $x,y\in C(S')$ condition \eqref{e:quasi-ideal} holds.

\subsection{The group scheme $\BG_a^\sharp$}   \label{ss:G_a^sharp}

\subsubsection{Definition of $\BG_a^\sharp$}   \label{sss:G_a^sharp}
Let $\BG_a^\sharp:=\Spec A$, where $A\subset\BQ [x]$ is the subring generated by the elements 
\[
u_n:=x^{p^n}/p^{\frac{p^n-1}{p-1}}, \quad n\ge 0. 
\]
It is easy to see that the ideal of relations between the $u_n$'s is generated by the relations $u_n^p=pu_{n+1}$.

Since $p^{\frac{p^n-1}{p-1}}\in (p^n)!\cdot\BZ_p^\times$, there is a unique homomorphism $\Delta :A\to A\otimes A$ such that $\Delta (x)=x\otimes 1+1\otimes x$. The pair 
$(A,\Delta)$ is a Hopf algebra over $\BZ$. So $\BG_a^\sharp$ is a group scheme over $\BZ$.

\subsubsection{Remarks}   \label{r:G_a^sharp}
(i) $\BG_a^\sharp\otimes\BZ_{(p)}$ is just the PD-hull of zero in $\BG_a\otimes\BZ_{(p)}$.

(ii) The embedding $\BZ [x]\mono A$ induces a morphism of group schemes 
\begin{equation}   \label{e:sharp to usual}
\BG_a^\sharp\to\BG_a .
\end{equation}

(iii) The morphism \eqref{e:sharp to usual} induces an isomorphism $\BG_a^\sharp\otimes\BQ\to\BG_a\otimes\BQ$.

\begin{lem}   \label{l:not additive}
Let $u_n\in A$ be as in \S\ref{sss:G_a^sharp}. If $n>0$ then $\Delta (u_n)-u_n\otimes 1-1\otimes u_n$ is not divisible by any prime.
\end{lem}

\begin{proof}
As a $\BZ$-module, $A$ has a basis formed by elements of the form $\prod\limits_i u_i^{a_i}$, where $0\le a_i<p$ and almost all numbers $a_i$ are zero.
The coefficient of $u_0\otimes\prod\limits_{i=0}^{n-1} u_i^{p-1}$ in $\Delta (u_n)$ equals $1$.
\end{proof}

\subsubsection{$\BG_a^\sharp$ as a quasi-ideal in $\BG_a$} \label{sss:as quasi-ideal}
There is a unique action of the ring scheme $\BG_a$ on $\BG_a^\sharp$ inducing the usual action of $\BG_a\otimes\BQ$ on $\BG_a^\sharp\otimes\BQ=\BG_a\otimes\BQ$. Thus 
$\BG_a^\sharp$ is a $\BG_a$-module. Moreover, the morphism \eqref{e:sharp to usual} makes $\BG_a^\sharp$ into a quasi-ideal in $\BG_a$ (in the sense of \S\ref{sss:quasi-ideal schemes}).

\subsubsection{$W^{(F)}$ as a quasi-ideal in $\BG_a$}   \label{sss:2as quasi-ideal}
Let $W$ be the ring scheme over $\BZ$ formed by $p$-typical Witt vectors. Let $W^{(F)}:=\Ker (F:W\to W)$. The action of $W$ on $W^{(F)}$ factors through $W/VW=\BG_a$. The composite map 
\[
W^{(F)}\mono W\epi W/VW=\BG_a
\]
 is a morphism of $\BG_a$-modules, which makes $W^{(F)}$ into a quasi-ideal in $\BG_a$.
 
 \begin{lem}   \label{l:G_a^sharp=W^(F)}
 $\BG_a^\sharp$ and $W^{(F)}$ are isomorphic as quasi-ideals in $\BG_a$. Such an isomorphism is unique.
 \end{lem}
 
 \begin{proof}
Uniqueness is clear. To construct the isomorphism, $\BG_a^\sharp\iso W^{(F)}$, 
we will use the approach to $W$ developed by Joyal \cite{JoyalDelta} (an exposition of this approach can be found in \cite{BorgerCourse} and \cite[\S 1]{BorgerGurney}).

Let $B$ be the coordinate ring of $W$. 
Let $F^*:B\to B$ be the homomorphism corresponding to $F:W\to W$. 
The map $W\otimes\BF_p\to W\otimes\BF_p$ induced by $F$ is the usual Frobenius, so there is a map $\delta:B\to B$ such that $F^*(b)=b^p+p\delta (b)$ for all $b\in B$ (of course, the map $\delta$ is neither additive nor multiplicative). 

The pair $(B,\delta)$ is a $\delta$-ring in the sense of \cite{JoyalDelta}. The main theorem of \cite{JoyalDelta} says that $B$ is the \emph{free $\delta$-ring} on $y_0$, where
$y_0\in B$ corresponds to the canonical homomorphism $W\to W/VW=\BG_a$. This means that as a ring, $B$ is freely generated by the elements $y_n:=\delta^n(y_0)$, $n\ge 0$. We have $F^*(y_n)=y_n^p+py_{n+1}$. The closed subscheme $$\{0\}\subset W=\Spec B$$ identifies with $\Spec B/(y_0,y_1,\ldots )$. This implies that the closed subscheme $W^{(F)}\subset W$ identifies with $\Spec (B/I)$, where the ideal $I\subset B$ is generated by $y_n^p+py_{n+1}$, $n\ge 0$. On the other hand, $B/I$ identifies with the ring $A$ from \S\ref{sss:G_a^sharp} via the epimorphism $B\epi A$ that takes $y_n$ to $(-1)^nu_n$.

Thus we have constructed an isomorphism of schemes $W^{(F)}\iso \BG_a^\sharp$. To show that it is an isomorphism of quasi-ideals, it suffices to check this for the corresponding isomorphism $W^{(F)}\otimes\BQ\iso \BG_a^\sharp\otimes\BQ$.
\end{proof}

\begin{rem}
By Remark~\ref{sss:quasi-ideal rems}(ii), $W^{(F)}$ and $\BG_a^\sharp$ are (non-unital) ring schemes, and the isomorphism from Lemma~\ref{l:G_a^sharp=W^(F)} is an isomorphism of ring schemes.
\end{rem}

\subsection{The group schemes $W^\times$ and $(W^\times)^{(F)}$}  \label{ss:W^times}
\begin{lem}   \label{l:invertible in W}
Let $R$ be a $p$-nilpotent ring. Then

(i) the ideal $V(W(R))$ is contained in the Jacobson radical of $W(R)$;

(ii) a Witt vector $\alpha\in W(R)$ is invertible if and only if its $0$th component is;

(iii) $\alpha\in W(R)$ is invertible if and only if $F(\alpha )$ is.
\end{lem}

\begin{proof}
The ideal $\Ker (W(R)\to W(R/pR))$ is nilpotent. So we can assume that $R$ is an $\BF_p$-algebra.

To prove (i) and (ii), it suffices to show that $Vx$ is topologically nilpotent for all $x\in W(R)$.
Indeed, since $VF=FV=p$ we have $(Vx)^n=p^{n-1}V(x^n)=V^n(F^{n-1}x)$. 

Statement (iii) follows from (ii) because $F:W\otimes\BF_p\to W\otimes\BF_p$ is the usual Frobenius.
\end{proof}

\begin{rem}     \label{r:invertible in W}
For \emph{any} ring $R$ one can show by induction on $n$ that an element of $W_n(R)$ is invertible if and only if all of its ghost components are.
\end{rem}

\subsubsection{The group scheme $(W^\times)^{(F)}$}
Let 
\[
(W^\times)^{(F)}:=\Ker (F:W^\times\to W^\times ),
\]
where $W^\times$ is the multiplicative group of the ring scheme~$W$. Then $(W^\times)^{(F)}$ identifies with the multiplicative group of the non-unital ring scheme\footnote{By definition, the multiplicative group of a non-unital ring $A$ is $\Ker (\tilde A^\times\to\BZ^\times )$, where $\tilde A:=\BZ\oplus A$ is the ring obtained by formally adding the unit to $A$.} $W^{(F)}$. 

On the other hand, let $\BG_m^\sharp$ be the multiplicative group of the non-unital ring scheme $\BG_a^\sharp$;  here $\BG_a^\sharp$ is equipped with the multiplication that comes as explained in \S\ref{sss:quasi-ideal rems}(ii) from the quasi-ideal structure described in \S\ref{sss:as quasi-ideal}. Note that $\BG_m^\sharp\otimes\BZ_{(p)}$ is the PD-hull of $1$ in $\BG_m\otimes\BZ_{(p)}$.

Lemma~\ref{l:G_a^sharp=W^(F)} provides an isomorphism $\BG_a^\sharp\iso W^{(F)}$. It is an isomorphism between quasi-ideals in $\BG_a$ and therefore a ring homomorphism. So it induces an isomorphism of group schemes
\begin{equation}  \label{e:(W^times)^(F)=G_m^sharp}
(W^\times)^{(F)}\iso\BG_m^\sharp .
\end{equation}

\begin{lem}      \label{l:G_m^sharp modulo p}
There is a unique isomorphism of group schemes over $\BF_p$
\begin{equation}     \label{e:G_m^sharp modulo p}
(W^{(F)}\times\mu_p)\otimes\BF_p \iso (W^\times)^{(F)}\otimes\BF_p 
\end{equation}
whose restriction to $\mu_p\otimes\BF_p$ is the Teichm\"uller embedding and whose restriction to $W^{(F)}\otimes\BF_p$ is given by $x\mapsto 1+Vx$.
\end{lem}

\begin{proof}
Let us only explain why the map $x\mapsto 1+Vx$ is a group homomorphism $$W^{(F)}\otimes\BF_p\to (W^\times)^{(F)}\otimes\BF_p\, .$$
This follows from the identities 
\[
F(1+Vx)=1+VFx, \quad (1+Vx)(1+Vy)=1+V(x+y+VF(xy)),
\]
which hold because in characteristic $p$ one has $VF=FV$.
\end{proof}

\begin{rem}
There is a canonical isomorphism
$( (W^\times)^{(F)}/\mu_p)\hat\otimes\BZ_p\iso W^{(F)}\hat\otimes\BZ_p$
of group schemes over $\Spf\BZ_p$ (see  \cite[Lemma.~ 3.5.18]{BL} or \cite[\S 4.6]{Formal group}). 
However, we do not need it in this article.
\end{rem}

\subsection{Faithful flatness of $F:W\to W$ and $F:W^\times\to W^\times$}  \label{ss:faithful flatness of F}
Joyal's description of $W$ (see the proof of Lem\-ma~\ref{l:G_a^sharp=W^(F)}) shows that the morphism $F:W\to W$ is faithfully flat.

Here is another proof. It suffices to check faithful flatness of $F:W_{n+1}\to W_n$ for each $n$. This reduces to proving faithful flatness of the two maps
\[
F:W_{n+1}\otimes\BZ[1/p]\to W_n\otimes\BZ[1/p], \quad F:W_{n+1}\otimes\BF_p\to W_n\otimes\BF_p.
\]
 The first map can be treated using ghost components. The second map is just the composite of the projection $W_{n+1}\otimes\BF_p\to W_n\otimes\BF_p$ and the usual Frobenius.
 
 The same argument proves faithful flatness of $F:W^\times\to W^\times$.
  
\subsection{The Picard stack $\Cone (\BG_a^\sharp\to \BG_a)$ in terms of $W$}  \label{ss:Cone (G_a^sharp to G_a)}
We will be using the notation $\Cone$ from \S\ref{sss:Cones-group schemes}.

\begin{prop}   \label{p:Cone (G_a^sharp to G_a)}
One has a canonical isomorphism of Picard stacks over $\BZ$
\[
\Cone (\BG_a^\sharp\to \BG_a)\iso\Cone (W\overset{p}\longrightarrow W)
\]
\end{prop}

\begin{proof}
By Lemma~\ref{l:G_a^sharp=W^(F)}, $\Cone (\BG_a^\sharp\to \BG_a)=\Cone (W^{(F)}\to W/VW)$. We have
\[
\Cone (W^{(F)}\to W/VW)=\Cone (VW\to W/W^{(F)})=\Cone (VW\overset{F}\longrightarrow W),
\]
where the second equality follows from \S\ref{ss:faithful flatness of F}. But
$\Cone (VW\overset{F}\longrightarrow W)=\Cone (W\overset{FV}\longrightarrow W)$
and $FV=p$.
\end{proof}

 \subsection{Generalities on $W_S$-modules}   \label{ss:W_S-module generalities}
 Let $W_S:=W\times S$; this is a ring scheme over $S$.  By a $W_S$-module we mean a commutative affine group scheme over $S$ equipped with an action of the ring scheme $W_S$. 
 
 \subsubsection{$\Hom_W$ and\, $\HHom_W$}
 If $M$ and $N$ are  $W_S$-modules we write $\Hom_W(M,N)$ for the group of all $W_S$-morphisms $M\to N$. 
 
 Let $\cA$ be the category of fpqc-sheaves of abelian groups on the category of $S$-schemes.
 Sometimes it is convenient to embed the category of $W_S$-modules into the bigger category of objects of $\cA$ equipped with a $W_S$-action. Given $W_S$-modules $M$ and $N$, one defines an object $\HHom_W(M,N)$ in the bigger category; namely, $\HHom_W(M,N)$ is the
 contravariant functor 
 \[
 S'\mapsto\Hom_W(M\times_SS',N\times_SS').
 \]
In some important cases this functor turns out to be representable; then $\HHom_W(M,N)$ is a $W_S$-module.

\subsubsection{The functor $M\mapsto M^{(n)}$} \label{sss:M^(n)}
Let $n\in\BZ$, $n\ge 0$. Let $M$ be a $W_S$-module. Precomposing the action of $W_S$ on $M$ with $F^n:W_S\to W_S$, we get a new $W_S$-module structure on the group scheme underlying $M$; the new $W_S$-module will be denoted by $M^{(n)}$.

\subsubsection{The functor $N\mapsto N^{(-1)}$} \label{sss:N^(-1)}
By \S\ref{ss:faithful flatness of F}, $F$ is faithfully flat. So the functor $M\mapsto M^{(1)}$ from \S\ref{sss:M^(n)} induces an equivalence between the category of $W_S$-modules and the full subcategory of $W_S$-modules killed by $W_S^{(F)}:=\Ker (W_S\overset{F}\longrightarrow W_S)$. The inverse functor is denoted 
by $N\mapsto N^{(-1)}$.

 \subsection{Examples of $W_S$-modules}
 Define $W_S$-modules $(\BG_a)_S$ and $(\BG^\sharp_a)_S$ as follows:
 \[
 (\BG_a)_S:=\BG_a\times S, \quad (\BG^\sharp_a)_S:=\BG^\sharp_a\times S, 
 \]
 where the ring scheme $W$ acts on $\BG_a$ via the canonical ring epimorphism $W\epi W/VW=\BG_a$.
 Applying \S\ref{sss:M^(n)} to the $W_S$-modules  $W_S$ and $(\BG_a)_S$, we get $W_S$-modules $W_S^{(n)}$ and $(\BG_a)_S^{(n)}$ for each integer $n\ge 0$.
 
 We have a $W_S$-module homomorphism $F:W_S\to W_S^{(1)}$, which is a faithfully flat map by~\S\ref{ss:faithful flatness of F}. Its kernel is denoted by $W_S^{(F)}$. By Lemma~\ref{l:G_a^sharp=W^(F)}, we have a canonical isomorphism $W_S^{(F)}\iso  (\BG^\sharp_a)_S$.
 
In addition to the exact sequence
\begin{equation}   \label{e:first exact sequence}
0\to W_S^{(F)}\to W_S\overset{F}\longrightarrow W_S^{(1)}\to 0.
\end{equation}
we have the exact sequence
\begin{equation}   \label{e:second exact sequence}
0\to W_S^{(1)}\overset{V}\longrightarrow W_S\to (\BG_a)_S\to 0.
\end{equation}

\subsection{Duality between exact sequences \eqref{e:first exact sequence} and \eqref{e:second exact sequence}}
The goal of this subsection is to prove Proposition~\ref{p:duality of exact sequences}.
\begin{lem}   \label{l:W^(F) to G_a^(n)}
(i) If $n>0$ then $\Hom_W(W_S^{(F)},(\BG_a)_S^{(n)})=0$.

(ii) The $W_S$-module morphisms $W_S^{(F)}\mono W_S\epi (\BG_a)_S$ induce an isomorphism
\[
H^0(S,\cO_S)=\Hom_W ((\BG_a)_S, (\BG_a)_S)\iso\Hom_W (W_S^{(F)}, (\BG_a)_S).
\]
\end{lem}

\begin{proof}
By Lemma~\ref{l:G_a^sharp=W^(F)}, we can replace $W_S^{(F)}$ by $(\BG^\sharp_a)_S$. We can assume that $S$ is affine, $S=\Spec R$.
Let $A$ and $u_n$ be as in \S\ref{sss:G_a^sharp}.   Recall that $(\BG^\sharp_a)_S=\Spec (A\otimes R)$.

Let $f\in\Hom_W(W_S^{(F)},(\BG_a)_S^{(n)})$. Since $f$ commutes with the action of Teichm\"uller elements of the Witt ring, we see that the function
$f\in A\otimes R$ is homogeneous of degree $p^n$. So $f=cu_n$ for some $c\in R$. If $n>0$ then $c=0$ by Lemma~\ref{l:not additive}.
\end{proof}

\begin{lem}   \label{l:each other's annihilators}
(i) The multiplication pairing
\begin{equation}   \label{e:multiplication pairing}
W_S\times W_S\to W_S
\end{equation}
kills $W_S^{(F)}\times V(W_S^{(1)})\subset W_S\times W_S$.

(ii) The kernel of the morphism $W_S\to\HHom_W(V(W_S^{(1)}),W_S)$ induced by \eqref{e:multiplication pairing} equals $W_S^{(F)}$.

(iii) The kernel of the morphism $W_S\to\HHom_W(W_S^{(F)},W_S)$ induced by \eqref{e:multiplication pairing} equals $V(W_S^{(1)})$.
\end{lem}

\begin{proof}
Statement (i) is clear. To prove (ii), use the section $V(1)$ of the $S$-scheme $V(W_S^{(1)})$. Statement (iii) follows from (i) and the equality
\[
\Ker ((\BG_a)_S\to\HHom_W(W_S^{(F)},W_S^{(F)}))=0;
\]
this equality is clear because $W_S^{(F)}=(\BG^\sharp_a)_S$. 
\end{proof}

By Lemma~\ref{l:each other's annihilators}(i), the pairing \eqref{e:multiplication pairing} and the exact sequences \eqref{e:first exact sequence}-\eqref{e:second exact sequence} yield $W_S$-bilinear pairings
\begin{equation}   \label{e:pairing on W^(1)}
W_S^{(1)}\times W_S^{(1)}\to W_S,
\end{equation}
\begin{equation}   \label{e:pairing between W^(1) and G_a}
(\BG_a)_S\times W_S^{(F)}\to W_S.
\end{equation}
The pairing \eqref{e:pairing on W^(1)} is symmetric; in fact, this is just the multiplication $W_S^{(1)}\times W_S^{(1)}\to W_S^{(1)}$ followed by $V:W_S^{(1)}\mono W_S$.
The pairing \eqref{e:pairing between W^(1) and G_a} is the composite 
$$(\BG_a)_S\times W_S^{(F)}\to W_S^{(F)}\mono W_S,$$
where the first map is the action of $(\BG_a)_S$ on $W_S^{(F)}$.

\begin{prop}    \label{p:duality of exact sequences}
The pairings  \eqref{e:pairing on W^(1)} and \eqref{e:pairing between W^(1) and G_a} induce isomorphisms
\begin{equation}   \label{e:Hom(W^(1),W}
W_S^{(1)}\iso\HHom_W(W_S^{(1)},W_S),
\end{equation}
\begin{equation}   \label{e:Hom(W^(F),W}
W_S^{(F)}\iso\HHom_W((\BG_a)_S,W_S),
\end{equation}
\begin{equation}     \label{e:Hom(G_a,W}
(\BG_a)_S\iso\HHom_W(W_S^{(F)},W_S).
\end{equation}
\end{prop}

\begin{proof}
The statements about $\HHom_W(W_S^{(1)},W_S)$ and $\HHom_W((\BG_a)_S,W_S)$ are easy because $W_S^{(1)}$ and $(\BG_a)_S$ appear as \emph{quotients} of $W_S$. 
More precisely, they are equivalent to Lem\-mas~\ref{l:each other's annihilators}(iii) and \ref{l:each other's annihilators}(ii), respectively.

To prove the statement about $\HHom_W(W_S^{(F)},W_S)$, use Lemma~\ref{l:W^(F) to G_a^(n)} and the filtration
\begin{equation}   \label{e:V^nW}
W_S\supset V(W_S^{(1)})\supset V^2(W_S^{(2)})\supset\ldots , 
\end{equation}
whose successive quotients are the $W_S$-modules $(\BG_a)_S^{(n)}$, $n\ge 0$.
\end{proof}

\subsection{More computations of $\Hom_W$}
\begin{prop}    \label{p:more HHoms}
(i) The action of $(\BG_a)_S$ on $W_S^{(F)}$ induces an isomorphism
\[
(\BG_a)_S\iso\HHom_W(W_S^{(F)},W_S^{(F)}).
\]

(ii) $\HHom_W(W_S^{(F)},W_S^{(1)})=0$.

(iii) The morphism $F:W_S\to W_S^{(1)}$ induces isomorphisms
\begin{equation}      \label{e:Hom(W^(1),W^(1)}
\HHom_W(W_S^{(1)},W_S^{(1)})\iso\HHom_W(W_S,W_S^{(1)})=W_S^{(1)},
\end{equation}
\begin{equation}      \label{e:Hom(W^(1),W^(F)}
\HHom_W(W_S^{(1)},W_S^{(F)})\iso\Ker (W_S^{(F)}\to (\BG_a)_S)=\Ker ((\BG_a^\sharp)_S\to (\BG_a)_S).
\end{equation}
\end{prop}

\begin{proof}
Statement (i) follows from \eqref{e:Hom(G_a,W}. Statement (ii) is deduced from Lemma~\ref{l:W^(F) to G_a^(n)}(i) using the filtration~\eqref{e:V^nW}. 
 The isomorphism~\eqref{e:Hom(W^(1),W^(1)} follows from the fact that  the map $F:W_S\to W_S^{(1)}$ induces an isomorphism $W_S/W_S^{(F)}\iso W_S^{(1)}$.
  The isomorphism~\eqref{e:Hom(W^(1),W^(F)} follows from \eqref{e:Hom(G_a,W}.
\end{proof}

\subsubsection{Remarks}
(i) Although the map $\BG_a^\sharp\otimes\BZ [1/p]\to\BG_a\otimes\BZ [1/p]$ is an isomorphism,
it is easy to see from \S\ref{sss:G_a^sharp} that $\Ker (\BG_a^\sharp\to\BG_a)\ne 0$.

(ii) By Proposition~\ref{p:Cone (G_a^sharp to G_a)}, the r.h.s. of \eqref{e:Hom(W^(1),W^(F)} can be rewritten as $\Ker (W_S\overset{p}\longrightarrow W_S)$.

\subsection{Extensions of $W_S^{(1)}$ by $W_S^{(F)}$}  \label{ss:Ex (W^(1),W^(F)}
Given $W_S$-modules $M$ and $N$, let $\Ex_W(M,N)$ denote the Picard stack over $S$ whose $S'$-points are extensions of $N\times_SS'$ by $M\times_SS'$. The following statement strengthens formula \eqref{e:Hom(W^(1),W^(F)}.

\begin{prop}   \label{p:Ex (W^(1),W^(F)}
One has a canonical isomorphism
\begin{equation}   \label{e:Ex (W^(1),W^(F)}
\Ex_W(W_S^{(1)},W_S^{(F)})\iso \Cone (W_S^{(F)}\to (\BG_a)_S)=\Cone ((\BG_a^\sharp)_S\to (\BG_a)_S).
\end{equation}
In particular, the stack $\Ex_W(W_S^{(1)},W_S^{(F)})$ is algebraic and 
$\MMor$-affine\footnote{For the notion of $\MMor$-affineness, see Definition~\ref{d:pre-algebraic stack} and Remark~\ref{r:Mor-affine}.}.
\end{prop}

\begin{proof}
Let $S'$ be an $S$-scheme. Pushing forward the canonical extension 
\[
0\to W_{S'}^{(F)}\to W_{S'}\overset{F}\longrightarrow W_{S'}^{(1)}\to 0
\]
via a morphism $W_{S'}^{(F)}\to W_{S'}^{(F)}$, one gets a new extension 
\begin{equation} \label{e:new extension}
0\to W_{S'}^{(F)}\to M\overset{\pi}\longrightarrow W_{S'}^{(1)}\to 0.
\end{equation}
Thus one gets an isomorphism of Picard groupoids
$$\Cone (\Hom_W(W_{S'},W_{S'}^{(F)})\to\Hom_W(W_{S'}^{(F)},W_{S'}^{(F)}))\iso\on{Ex}_W^{good}(W_{S'}^{(1)},W_{S'}^{(F)}) ,$$ 
where $\on{Ex}_W^{good}(W_{S'}^{(1)},W_{S'}^{(F)})$ is the groupoid of
those extensions \eqref{e:new extension} for which there exists a section 
$S'\to~\pi^{-1}(1)\subset M$.  For \emph{any} extension \eqref{e:new extension}, such a
section exists fpqc-locally on~$S'$ (because $\pi$ is faithfully flat and therefore $\pi^{-1}(1)$ is faithfully flat over $S'$). So we get an isomorphism of Picard stacks
$$\Cone (W_S^{(F)}\to\HHom_W(W_S^{(F)},W_S^{(F)}))\iso\Ex_W(W_S^{(1)},W_S^{(F)}) .$$ 
To get \eqref{e:Ex (W^(1),W^(F)}, it remains to use the canonical isomorphism $(\BG_a)_S\iso\HHom_W(W_S^{(F)},W_S^{(F)})$, see
Proposition~\ref{p:more HHoms}(i).
\end{proof}

Combining Propositions~\ref{p:Ex (W^(1),W^(F)} and \ref{p:Cone (G_a^sharp to G_a)}, we get a canonical isomorphism
\begin{equation} \label{e:Ex=Cone(p)}
\Ex_W(W_S^{(1)},W_S^{(F)})\iso\Cone (W_S\overset{p}\longrightarrow W_S).
\end{equation}
We will also give its direct construction (see \S\ref{sss:Ex=Cone(p)} below). It is based on the following

\begin{lem}   \label{l:rigidification providing zeta}
For every scheme $S$, every exact sequence $0\to W_{S}^{(F)}\overset{i}\longrightarrow M\overset{\pi}\longrightarrow W_{S}^{(1)}\to 0$ Zariski-locally on $S$ admits a rigidification of the following type: a $W_S$-morphism $$r:M\to W_S$$ such that $r|_{W_{S}^{(F)}}=\id$. All such rigidifications form a torsor over $\HHom_W(W_{S}^{(1)}, W_S)\simeq W_{S}^{(1)}$.
\end{lem}

\begin{proof}
The lemma is a consequence of the following fact, which can be easily deduced from Proposition~\ref{p:more HHoms}(i):
every extension of $W_{S}^{(1)}$ by $W_{S}$ splits Zariski-locally on $S$.

Here is a slightly more direct proof. We already know that $\HHom_W(W_{S}^{(1)}, W_S)\simeq W_{S}^{(1)}$, see~\eqref{e:Hom(W^(1),W}. Since every $W_{S}^{(1)}$-torsor is Zariski-locally trivial, it suffices to prove that $r$ exists fpqc-locally. So we can assume that there exists a $W_S$-morphism $\sigma :W_S\to M$ such that $\pi\circ\sigma=F$. A choice of $\sigma$ realizes our exact sequence as a pushforward of the canonical exact sequence
\begin{equation}   \label{e:can exact seq}
0\to W_S^{(F)}\to W_S\overset{F}\longrightarrow W_S^{(1)}\to 0
\end{equation}
via some $h:W_{S}^{(F)}\to W_{S}^{(F)}$. Constructing $r$ is equivalent to extending $h$ to a mor\-phism $W_S\to W_S$. This is possible by Proposition~\ref{p:more HHoms}(i).
\end{proof}

\begin{cor}   \label{c:zeta}
Every extension of $W_{S}^{(1)}$ by $W_{S}^{(F)}$ can be Zariski-locally on $S$ obtained as a pullback of \eqref{e:can exact seq}
via some $\zeta\in\End_W(W_{S}^{(1)})$. \qed
\end{cor}

\subsubsection{Direct construction of \eqref{e:Ex=Cone(p)}}  \label{sss:Ex=Cone(p)}
By Lemma~\ref{l:rigidification providing zeta} and Corollary~\ref{c:zeta}, 
$$\Ex_W(W_S^{(1)},W_S^{(F)})=\Cone (\HHom_W(W_S^{(1)},W_S)\overset{g}\longrightarrow\HHom_W(W_S^{(1)},W_S^{(1)})),$$ 
where the map~$g$ comes from $F:W_S\to W_S^{(1)}$. Using \eqref{e:Hom(W^(1),W}, \eqref{e:Hom(W^(1),W^(1)}, and the formula $FV=p$, one identifies 
$g$ with the map $W_S^{(1)}\overset{p}\longrightarrow W_S^{(1)}$.

\begin{lem} \label{l:Ex for perfect S}
Let $S$ be a perfect $\BF_p$-scheme. Then the Picard groupoid of $W_S$-module extensions of $W_S^{(1)}$ by 
$(\BG_a^\sharp)_S=W_S^{(F)}$ identifies with the group $H^0(S,\cO_S)$ as follows: the extension corresponding to
$u\in H^0(S,\cO_S)$ is the pushforward of the canonical extension 
\[
0\to W_S^{(F)}\to W_S\overset{F}\longrightarrow W_S^{(1)}\to 0
\]
via $u:W_S^{(F)}\to W_S^{(F)}$.
\end{lem}
\begin{proof}
By Proposition~\ref{p:Ex (W^(1),W^(F)}, a $W_S$-module extension of $W_S^{(1)}$ by 
$W_S^{(F)}$ is the same as an $S$-point of the $S$-stack $\Cone (W_S^{(F)}\to (\BG_a)_S)$, i.e., a
$W_S^{(F)}$-torsor $T$ over $S$ equipped with a $W_S^{(F)}$-equivariant $S$-morphism $T\to (\BG_a)_S$. It remains to show that any $W_S^{(F)}$-torsor $T$ has one and only one section $\sigma :S\to T$. Uniqueness is clear because $S$ is reduced. To construct~$\sigma$, first note that the absolute Frobenius $\Fr_T:T\to T$ factors as
$T\epi T/W_S^{(F)}=S\overset{f}\longrightarrow T$; then set $\sigma :=f\circ\Fr_S^{-1}$ (by perfectness, $\Fr_S$ is invertible).
\end{proof}

\subsection{The Teichm\"uller functor $\sL\mapsto [\sL ]$}   \label{ss:2Teichmuller}
This section can be skipped at first reading. Its goal is to prove a slight generalization of the description of $\Ex_W(W_S^{(1)}, W_S^{(F)})$ obtained in \S\ref{ss:Ex (W^(1),W^(F)}. To formulate the result, we will introduce the \emph{Teichm\"uller functor}.

Let $S$ be a scheme. A line bundle $\sL$ on $S$ is the same as a $\BG_m$-torsor on $S$. Applying to this $\BG_m$-torsor the Teichm\"uller homomorphism
\begin{equation}  \label{e:2Teichmuller map}
\BG_m\to W^\times ,\quad \lambda\mapsto [\lambda ],
\end{equation}
one gets a $W^\times$-torsor on $S$, which is the same as an invertible $W_S$-module. We denote this $W_S$-module by $[\sL ]$.

For any $S$-scheme $T$, a $T$-point of $[\sL ]$ is the same as a $\BG_m$-equivariant\footnote{Equivariance with respect to the multiplicative group implies equivariance with respect to the multiplicative \emph{monoid}.} (non-additive) morphism $\sL^*\times_S T\to W$, where $\sL^*$ is the dual line bundle and $\BG_m$ acts on $W$ 
via~\eqref{e:2Teichmuller map}. So the assignment $\sL\mapsto [\sL ]$ is a functor from the category of line bundles on $S$ to that of invertible $W_S$-modules. We call it the
\emph{Teichm\"uller functor}. Of course, this functor is not additive on morphisms.

Note that 
\begin{equation}   \label{e:2L' & L^p}
[\sL]'=[\sL ]\otimes_{W_S}W_S^{(1)}=[\sL^{\otimes p}]^{(1)}. 
\end{equation}

Tensoring the morphisms $W_S\overset{F}\longrightarrow W_S^{(1)} \overset{V}\longrightarrow W_S$ by $[\sL ]$, we get morphsims
\[
[\sL]\overset{F}\longrightarrow [\sL]' \overset{V}\longrightarrow [\sL].
\]
Tensoring the exact sequences \eqref{e:first exact sequence}-\eqref{e:second exact sequence} by $[\sL ]$, we get exact sequences
\begin{equation}  \label{e:3resolving L^sharp}
0\to\sL^\sharp\longrightarrow  [\sL ]\overset{F}\longrightarrow [\sL]'\to 0.
\end{equation}
\begin{equation}   \label{e:2L second exact sequence}
0\to [\sL]'\overset{V}\longrightarrow [\sL]\longrightarrow \sL\to 0,
\end{equation}
(Recall that $\sL^\sharp :=\sL\otimes_{(\BG_a)_S} W_S^{(F)}=\sL\otimes_{(\BG_a)_S}(\BG_a^\sharp)_S$.)

Here is a variant of formulas \eqref{e:Ex (W^(1),W^(F)} and \eqref{e:Ex=Cone(p)}.

\begin{prop}   \label{p:3describing Ex_W}
Let $\sL$ a line bundle on $S$. Then there are canonical isomorphisms of Picard stacks
\begin{equation}    \label{e:Ext(W^1,Lsharp)}
\Ex_W(W_S^{(1)}, \sL^\sharp )\iso\Cone (\sL^\sharp\to \sL), 
\end{equation}
\begin{equation}   \label{e:2Ext(W^1,Lsharp)}
\Ex_W(W_S^{(1)}, \sL^\sharp )\iso\Cone ([\sL]'\overset{p}\longrightarrow [\sL]'), 
\end{equation}
which are functorial in $\sL$.
\end{prop}

\begin{proof}
(i) The first isomorphism comes from the exact sequence $$0\to W_S^{(F)}\to W_S\overset{F}\longrightarrow W_S^{(1)}\to 0.$$

(ii) The exact sequence \eqref{e:3resolving L^sharp} induces a morphism 
$$\Cone ([\sL]'\overset{p}\longrightarrow [\sL]')\longrightarrow\Ex_W(W_S^{(1)}, \sL^\sharp ).$$
It is an isomorphism by Lemma~\ref{l:rigidification providing zeta} or Corollary~\ref{c:zeta}. 
\end{proof}

\subsection{Admissible $W_S$-modules} \label{ss:admissible W-modules}
\begin{defin}  \label{d:invertible}
A $W_S$-module is said to be \emph{invertible} if it is fpqc-locally isomorphic to~$W_S$.
\end{defin}

\begin{lem}   \label{l:topology-independence}
If $S$ is $p$-nilpotent then every invertible $W_S$-module $M$ is Zariski-locally isomorphic to $W_S$.
\end{lem}

\begin{proof}
We can assume that $S$ is affine and the line bundle $\sL:=M/V(M^{(1)})$ is trivial. Choose an isomorphism $\cO_S\iso\sL$ and lift it to a morphism $f:W_S\to M$. Then $f$ is an isomorphism by Lemma~\ref{l:invertible in W}(i).
\end{proof}

\subsubsection{Notation}   \label{sss:Lsharp}
Let $\sL$ be a line bundle on $S$.  Then $\sL$ is a module over the ring scheme $(\BG_a)_S$, and we set 
$$\sL^\sharp :=\sL\otimes_{(\BG_a)_S} W_S^{(F)}=\sL\otimes_{(\BG_a)_S}(\BG_a^\sharp)_S.$$
If $S$ is a $\BZ_{(p)}$-scheme then $\sL^\sharp$ is the PD-hull of the group scheme $\sL$ along its zero section.

\begin{lem}  \label{l:2 classes of W-modules}
(i) The functor $M\mapsto M^{(1)}$ from \S\ref{sss:M^(n)} induces an equivalence between the category of invertible $W_S$-modules and the category of $W_S$-modules fpqc-locally isomorphic to~$W_S^{(1)}$. The inverse functor is $N\mapsto N^{(-1)}$.

(ii) The functor $\sL\mapsto\sL^\sharp$ induces an equivalence between the category of line bundles $\sL$ on $S$ and the category of $W_S$-modules fpqc-locally isomorphic to $W_S^{(F)}$. The inverse functor is $M\mapsto\HHom_W(W_S^{(F)},M)$.
\end{lem}

\begin{proof}
Statement (i) follows from 
\S\ref{sss:N^(-1)}.
Statement (ii) follows from Proposition~\ref{p:more HHoms}(i).
\end{proof}

\subsubsection{Remark}  \label{sss:2 classes of W-modules}
It is clear that in statement (ii) of the above lemma one can replace ``fpqc" by ``Zariski''. If $S$ is $p$-nilpotent this is also true for statement (i) (this follows from Lem\-ma~\ref{l:topology-independence}).

 \begin{defin}    \label{d:admissible}
A $W_S$-module $M$ is said to be \emph{admissible} if there exists an exact sequence of $W_S$-modules
\begin{equation}  \label{e:M_0&M'}
0\to M_0\to M\to M'\to 0,
\end{equation}
where $M_0$ is locally isomorphic to $W_S^{(F)}$ and $M'$ is locally isomorphic to~$W_S^{(1)}$.
\end{defin}

\begin{lem}   \label{l:uniqueness of M_0}
The exact sequence \eqref{e:M_0&M'} is essentially unique if it exists. Moreover, it is functorial in $M$.
\end{lem}

\begin{proof}
Follows from Proposition \ref{p:more HHoms}(ii).
\end{proof}

\subsubsection{Remarks}   \label{sss:admissibility rems}
(i) By the previous lemma, admissibility of a $W_S$-module is an fpqc-local property.

(ii) By Lemma~\ref{l:2 classes of W-modules}(ii), 
the exact sequence  \eqref{e:M_0&M'} can be rewritten as
\begin{equation}  \label{e:L^sharp &M'}
0\to\sL^\sharp\to M\to M'\to 0,
\end{equation}
where $\sL$ is a line bundle on $S$. Here $\sL=\sL_M:=\HHom_W(W_S^{(F)},M_0)=\HHom_W(W_S^{(F)},M)$.

(iii) The exact sequence \eqref{e:first exact sequence} shows that any invertible $W_S$-module $M$ is admissible; in this case 
\begin{equation} \label{e:M' for invertibleM}
M'=M\otimes_{W_S} W_S^{(1)}
\end{equation}
and  $\sL=M\otimes_{W_S} (\BG_a)_S$. Formula \eqref{e:M' for invertibleM} can be rewritten as
\begin{equation} \label{e:2M' for invertibleM}
(M')^{(-1)}=M\otimes _{W_S, F}W_S.
\end{equation}
(as before, we are assuming that $M$ is invertible). Let us also note that if $M$ is invertible we have a canonical isomorphism
\begin{equation}  \label{e:L_M for invertible M}
M\otimes_{W_S}(\BG_a)_S\iso \HHom_W(W_S^{(F)},M):=\sL_M
\end{equation}
induced by the isomorphism \eqref{e:Hom(G_a,W}; here the homomorphism $W_S\to (\BG_a)_S$ maps a Witt vector to its zeroth component.

(iv) If $S$ is a $\BZ [p^{-1}]$-scheme then all admissible $W_S$-modules are invertible because for every open $S'\subset S$ one has
$\Ex_W(W_{S'}^{(1)},W_{S'}^{(F)})=0$ by \eqref{e:Ex (W^(1),W^(F)} or \eqref{e:Ex=Cone(p)}.

\begin{lem} \label{l:Adm of a field}
Let $S=\Spec k$, where $k$ is a field of characteristic $p$. Then 

(i) the admissible $W_S$-module $W_S^{(F)}\oplus W_S^{(1)}$ is not invertible;

(ii) if $k$ is perfect then any admissible $W_S$-module is isomorphic to $W_S$ or to $W_S^{(F)}\oplus W_S^{(1)}$.
\end{lem}

\begin{proof}
Statement (i) is clear because $W_S^{(F)}$ is not reduced as a scheme. Statement (ii) can be deduced from \eqref{e:Ex=Cone(p)} and the fact that
$W(k)/pW(k)=k$ (which follows from perfectness).
\end{proof}

\begin{lem} \label{l:dual of admissible}
Let $M$ be an admissible $W_S$-module. Then 

(i) $\HHom_W(M, W_S)$ exists as a  scheme affine over $S$;

(ii) the exact sequence \eqref{e:M_0&M'} induces an exact sequence
\[
0\to \HHom_W(M', W_S) \to \HHom_W(M, W_S)\to \HHom_W(M_0, W_S)\to 0.
\]
\end{lem} 

\begin{proof}
This is a reformulation of Lemma~\ref{l:rigidification providing zeta}.
\end{proof}

Here is a generalization of Lemma~\ref{l:dual of admissible}(i).

\begin{lem} \label{l:HHom of admissible}
Let $M$ and $\tilde M$ be admissible $W_S$-modules. Then $\HHom_W(M, \tilde M)$ exists as a  scheme affine over $S$.
\end{lem} 

\begin{proof}
By Corollary~\ref{c:zeta}, we can assume the existence of a Cartesian square
\[
\xymatrix{
\tilde M\ar[r] \ar[d] & W_S^{(1)}\ar[d]^\zeta\\
W_S\ar[r]^F & W_S^{(1)}
}
\]
Then it suffices to show that $\HHom_W(M,W_S)$ and $\HHom_W(M,W_S^{(1)})$ exist as schemes affine over $S$. 
For $\HHom_W(M,W_S)$ this is Lem\-ma~\ref{l:dual of admissible}(i). By Proposition~ \ref{p:more HHoms}(ii), we have
$$\HHom_W(M,W_S^{(1)})=\HHom_W(M',W_S^{(1)}),$$ 
so $\HHom_W(M,W_S^{(1)})$ is an invertible  $W_S^{(1)}$-module.
\end{proof}

\begin{lem} \label{l:quasi-ideal automatically}
Let $M$ be an admissible $W_S$-module and $\xi :M\to W_S$ a $W_S$-morphism. Then $(M,\xi )$ is a quasi-ideal in $W_S$.
\end{lem} 

\begin{proof}
We have to show that for every $S$-scheme $S'$ 
one has 
\begin{equation}  \label{e:skew form}
\xi (\alpha)\beta-\xi (\beta )\alpha=0 \quad \mbox{for all } \alpha ,\beta\in M(S').
\end{equation}
We can assume that $M$ is an extension of $W_S^{(1)}$ by $W_S^{(F)}$.

By \eqref{e:Hom(G_a,W}, the identity \eqref{e:skew form} holds if $\alpha$ and $\beta$ are sections of $W_S^{(F)}$.
So considering the l.h.s. of \eqref{e:skew form} when $\alpha$ is a section of $W_S^{(F)}$ and $\beta$ is arbitrary, we get a $W_S$-bilinear pairing
$W_S^{(F)}\times W_S^{(1)}\to W_S$. But all such pairings are zero by \eqref{e:Hom(W^(1),W^(1)} and Proposition~\ref{p:more HHoms}(ii).

Thus the l.h.s. of \eqref{e:skew form} defines a $W_S$-bilinear pairing
$B: W_S^{(1)}\times W_S^{(1)}\to W_S$. It is strongly skew-symmetric (i.e.,the restriction of $B$ to the diagonal is zero). So using the epimorphism $F:W_S\epi W_S^{(1)}$, we see that $B=0$.
\end{proof}

\subsection{Admissible $W_S$-modules in the case that $S$ is an $\BF_p$-scheme}  \label{ss:geometric Frobenius}
\subsubsection{The geometric Frobenius}   \label{sss:geometric Frobenius}
Let $S$ be an $\BF_p$-scheme. Then for any scheme $X$ over $S$ we have the geometric Frobenius $F:X\to\Fr_S^*X$. If $X=W_S$ then $\Fr_S^*X=W_S$, and the geometric Frobenius $W_S\to W_S$ is equal to the Witt vector Frobenius $W_S\to W_S$ (so there is no risk of confusion between the two $F$'s). If $M$ is a $W_S$-module then $F:M\to\Fr_S^*M$ is $F$-linear with respect to the $W_S$-action; in other words, we have a $W_S$-linear morphism 
\begin{equation}  \label{e:3geomFrobenius}
F:M\to (\Fr_S^*M)^{(1)}.
\end{equation}

From now on, let us assume that the $W_S$-module $M$ is \emph{admissible}, so we have the exact sequence
\begin{equation}   \label{e:2M_0&M'} 
0\to\sL^\sharp\to M\to M'\to 0.
\end{equation}
from \S\ref{sss:admissibility rems}(ii). 
Note that  the mor\-phism~\eqref{e:3geomFrobenius} kills $\sL^\sharp$ because the geometric Frobenius endomorphism of $W_S^{(F)}$ is trivial. Therefore
\eqref{e:3geomFrobenius} induces a morphism $M'\to (\Fr_S^*M)^{(1)}$ or equivalently, a morphism
\begin{equation}  \label{e:4geomFrobenius}
f_M:(M')^{(-1)}\to \Fr_S^*M.
\end{equation}
This morphism is functorial in $M$.

Set $P:=(M')^{(-1)}$. Applying \eqref{e:M' for invertibleM} or \eqref{e:2M' for invertibleM} to $P$ and using the fact that the Witt vector Frobenius is equal to the usual one, we see that 
$P'=\Fr_S^*P^{(1)}=\Fr_S^*M'.$
So we get an exact sequence
\begin{equation}  \label{e:sN appears}
0\to\cN^\sharp\to (M')^{(-1)}\to \Fr_S^*M'\to 0,
\end{equation}
where $\cN$ is a line bundle (namely, $\cN=(M')^{(-1)}\otimes_{W_S}(\BG_a)_S$). Thus \eqref{e:4geomFrobenius} induces a morphism $\cN^\sharp\to\Fr_S^*\sL^\sharp=(\sL^{\otimes p})^\sharp$ or equivalently, a morphism of line bundles
\begin{equation}  \label{e:5geomFrobenius}
\varphi_M:\cN\to\sL^{\otimes p}.
\end{equation}
Moreover, we get a commutative diagram with exact rows  
\begin{equation}   \label{e:Frobenius diagram}
\xymatrix{
0\ar[r]&\cN^\sharp\ar[r] \ar[d]_{\varphi_M^\sharp} & (M')^{(-1)}\ar[r]\ar[d]_{f_M}&\Fr_S^*M'\ar[r]\ar@{=}[d] &0 \\
0\ar[r]&(\sL^{\otimes p})^\sharp\ar[r] & \Fr_S^*M\ar[r] & \Fr_S^*M'\ar[r] &0
}
\end{equation}
\begin{lem}  \label{l:invertibility criterion}
For any $\BF_p$-scheme $S$, the following properties of an admissible $W_S$-module~$M$ are equivalent:

(a) $M$ is invertible;

(b) the morphism \eqref{e:4geomFrobenius} is an isomorphism;

(c) the morphism \eqref{e:5geomFrobenius} is an isomorphism.
\end{lem}

\begin{proof}
Diagram \eqref{e:Frobenius diagram} shows that (b)$\Leftrightarrow$(c). 
To prove that (a)$\Rightarrow$(b), it suffices to treat the case $M=W_S$, which is straightforward. 

Let us show that (b)$\Rightarrow$(a). By Corollary~\ref{c:zeta}, we can assume that the extension \eqref{e:2M_0&M'}
is the pullback of the canonical exact sequence $0\to W_S^{(F)}\to W_S\overset{F}\longrightarrow W_S^{(1)}\to 0$  
via an endomorphism of $W_{S}^{(1)}$. Such an endomorphism has the form
$u^{(1)}$, where $u\in\End_W(W_{S})=W(S)$. Then $\sL=\cN=\cO_S$, and one checks that $\varphi_M:\cO_S\to\cO_S$ is multiplication by $u_0^p$, where
$u_0$ is the zero component of the Witt vector $u$. Invertibility of $u_0$ is equivalent to that of $u$ by Lemma~\ref{l:invertible in W}(ii).
\end{proof}

\subsection{The c-stack $\Adm$}   \label{ss:Adm}
\subsubsection{Definition of $\Adm$}     \label{sss:Adm}
For a scheme $S$, let $\Adm (S)$ be the category whose objects are admissible $W_S$-modules and whose morphisms are those $W_S$-linear maps $M_1\to M_2$ that induce an \emph{iso}morphism $M'_1\iso M'_2$. We have a functor
\[
\Adm (S)\to\mbox{\{line bundles on }S\}, \quad M\mapsto\sL_M,
\]
where $\sL_M$ is as in Remark~\ref{sss:admissibility rems}(ii). This functor is a left fibration in the sense of
\S\ref{sss:left fibrations of cats}; equivalently, it makes $\Adm (S)$ into a category cofibered in groupoids over the category of line bundles on $S$.

By Lemma~\ref{l:uniqueness of M_0} or Remark~\ref{sss:admissibility rems}(i), the assignment $S\mapsto\Adm (S)$ is a c-stack for the fpqc topology.

\begin{prop}     \label{p:algebraically of Adm}
The c-stack $\Adm$ is algebraic and $\MMor$-affine.
\end{prop}

\begin{proof}
$\MMor$-affineness holds by Lemma~\ref{l:HHom of admissible}. To prove algebraicity, note that one 
has a faithfully flat morphism 
$\Ex_W(W_\BZ^{(1)},W_\BZ^{(F)})\to\Adm\,$, 
and $\Ex_W(W_\BZ^{(1)},W_\BZ^{(F)})$ is algebraic by Proposition~\ref{p:Ex (W^(1),W^(F)}. 
\end{proof}

In \ref{ss:Adm_Rig} we describe a concrete presentation of $\Adm$ (in particular, see Corollary~\ref{c:cG-torsor} and Lemma~\ref{l:cG-action explicitly}).
In \cite[Appendix~A]{version 1} we also describe a concrete presentation of the algebraic c-stacks $\Adm_+$ and $\widetilde{\Adm}$ defined in the next subsection.

\subsection{The diagram $\Adm_+\to\widetilde{\Adm}\to\Adm$}   \label{ss:AdmAdmAdm}

\subsubsection{The stack $\Inv$}   \label{sss:Inv}
Let $\Inv (S)$ be the groupoid of invertible $W_S$-modules. The g-stack $\Inv$ is algebraic: this is just the classifying stack of $W^\times$. By \S\ref{sss:admissibility rems}(iii), all invertible $W_S$-module are admissible, so we get a functor $\Inv (S)\to\Adm (S)$.

Recall that a scheme $S$ is said to be $p$-nilpotent if the element $p\in H^0(S,\cO_S)$ is locally nilpotent.

\begin{lem}    \label{l:Inv for p-nilpotent}
(i) If $S$ is $p$-nilpotent then the functor $\Inv (S)\to\Adm (S)$ is fully faithful.

(ii) For every $n\in\BN$, the morphism $\Inv\otimes\BZ/p^n\BZ\to\Adm\otimes\BZ/p^n\BZ$ is an affine open immersion. It is also a right fibration in the sense of \S\ref{ss:left&right fibrations}.
\end{lem}

\begin{proof}
If $R$ is a ring in which $p$ is nilpotent and $w\in W(R)$ is such that $F(w)\in W(R)^\times$ then $w\in W(R)^\times$ by Lemma~\ref{l:invertible in W}(iii). Statement (i) follows.

To prove (ii), it suffices to consider the case $n=1$. In this case it follows from the equivalence (a)$\Leftrightarrow$(c) in Lemma~\ref{l:invertibility criterion}.  For a more straightforward proof of statement~(ii) see \cite[Lemma 3.13.2]{version 1}.
\end{proof}

\subsubsection{The c-stacks $\Adm_+$ and $\widetilde{\Adm}$}  \label{sss:Adm_+ & tildeAdm}
Let $\Adm_+ (S)$ be the category whose objects are triples $(P,M,f)$, where $P~\in~\Inv (S)$, $M\in\Adm (S)$, and $f$ is an $\Adm (S)$-morphism $P\to M$; by a morphism 
$(P_1,M_1,f_1)\to (P_2,M_2,f_2)$ we mean a pair $(g,h)$, where $g:P_1\to P_2$ is an isomorphism, $h:M_1\to M_2$ is an  $\Adm (S)$-morphism, and $hf_1=f_2g$.

Let $\widetilde{\Adm} (S)$ be the category whose objects are triples $(P,M,\phi )$, where $P~\in~\Inv (S)$, $M\in\Adm (S)$, and $\phi :P'\iso M'$ is an isomorphism; by a morphism $(P_1,M_1,\phi_1)\to (P_2,M_2,\phi_2)$ we mean a pair $(g,h)$, where $g:P_1\to P_2$ is an isomorphism, $h:M_1\to M_2$ is an  $\Adm (S)$-morphism, and $h'\phi_1=\phi_2g'$.

Thus we have c-stacks $\Adm_+$ and $\widetilde{\Adm}$. The forgetful morphisms
\begin{equation}   \label{e:AdmAdmAdm}
\Adm_+\to\widetilde{\Adm}\to\Adm .
\end{equation}
are left fibrations in the sense of \S\ref{sss:left fibrations of c-stacks}.

\begin{lem}
The c-stacks $\Adm_+$ and $\widetilde{\Adm}$ are algebraic and $\MMor$-affine.
\end{lem}

\begin{proof}
$\MMor$-affineness holds by Lemma~\ref{l:HHom of admissible}. The morphism
\[
\Adm_+\to\Adm\times\Inv , \quad (P,M,f)\mapsto (P,M)
\]
is affine, and the c-stack $\Adm\times\Inv$ is algebraic, so $\Adm_+$ is algebraic. Algebraicity of $\widetilde{\Adm}$ is proved similarly.
\end{proof}

\subsubsection{The morphism $\Adm_+\to\Inv$}  \label{sss:Adm_+ to Inv}
Let $(P,M,f)\in\Adm_+ (S)$. Then we have line bundles $\sL=\sL_M:=\HHom_W (W^{(F)},M)$ and $\sL_P:=P/V(P')\simeq\HHom_W (W^{(F)},P)$. Moreover, $f$ induces a morphism $\varphi :\sL_P\to\sL$. Note that the exact sequence $0\to\sL^\sharp\to M\to M'\to 0$ is just the pushforward of the exact sequence
$0\to\sL_P^\sharp\to P\to P'\to 0$ with respect to $\varphi ^\sharp:\sL_P^\sharp\to\sL^\sharp$. So we can think of $\Adm_+ (S)$ as follows: an object of $\Adm_+ (S)$ is a triple 
$(P,\sL, \varphi :\sL_P\to\sL )$, where $P\in\Inv (S)$ and $\sL$ is a line bundle on $S$; a~morphism $(P_1,\sL_1,\varphi_1)\to (P_2,\sL_2,\varphi_2)$ is
a pair $(g,h)$, where $g:P_1\to P_2$ is an isomorphism, $h:\sL_1\to \sL_2$ is a morphism, and the corresponding diagram
\[
\xymatrix{
\sL_{P_1}\ar[r]^{\varphi_1} \ar[d] & \sL_1\ar[d]^h\\
\sL_{P_2}\ar[r]^{\varphi_2} &\sL_2
}
\]
commutes. 

Thus the morphism $\Adm_+\to\Inv$ is very simple. Indeed, one can think of $\varphi :\sL_P\to\sL$ as a section of the line bundle $\sN=\sL_P^{\otimes -1}\otimes\sL$, so we get an isomorphism
\begin{equation}   \label{e:Adm_+ in terms of Inv}
\Adm_+\iso\Inv\times (\BA^1/\BG_m)_+,
\end{equation}
where $(\BA^1/\BG_m)_+$ is the c-stack whose $S$-points are line bundles on $S$ equipped with a section.

\subsubsection{The open immersion $j:\Inv\mono\Adm_+$}  \label{sss:Inv to Adm_+}
For every scheme $S$, the groupoid $\Inv (S)$ identifies with the full subcategory of triples $(P,M,f)\in\Adm_+(S)$ such that $f$ is invertible. Invertibility of $f$ is equivalent to invertibility of the morphism $\varphi :\sL_P\to\sL$ from \S\ref{sss:Adm_+ to Inv}. So we get an open immersion $j:\Inv\mono\Adm_+$. One has 
\begin{equation}   \label{e:Adm_+^0}
j (\Inv )=\Adm_+\setminus \Adm_+^0,
\end{equation}
where $\Adm_+^0\subset\Adm_+$ is the closed substack defined by the equation $\varphi =0$.

The isomor\-phism~\eqref{e:Adm_+ in terms of Inv} identifies 
$j(\Inv )\subset\Adm_+$ with
$\Inv\times (\BG_m/\BG_m)\subset\Inv\times (\BA^1/\BG_m)_+$.

\begin{prop}   \label{p:torsor-gerbe}
(i) The morphisms \eqref{e:AdmAdmAdm} are faithfully flat.

(ii) The morphism $\widetilde{\Adm}\to\Adm$ is a $\BG_m^\sharp$-gerbe.

(iii) The morphism $\Adm_+\to\widetilde{\Adm}$ is an $H$-torsor, where $H$ is the following group scheme over $\widetilde{\Adm}$: an $S$-point of $H$ is a quadruple 
$(P,M,\phi ,\sigma)$, where $(P,M, \phi)\in\widetilde{\Adm}(S)$ and $\sigma$ is a section of the group scheme $(\sL_P^{\otimes -1}\otimes\sL_M)^\sharp$.
\end{prop}

\begin{proof}
Statement (ii) follows from the isomorphism \eqref{e:(W^times)^(F)=G_m^sharp}. Statement (iii) is clear. Statement~(i) follows from (ii) and (iii). 
\end{proof}

\section{The stack $\Sigma$}    \label{s:Sigma}
From now on, we fix a prime $p$. Recall that a scheme $S$ is said to be \emph{$p$-nilpotent} if the element $p\in H^0(S,\cO_S)$ is locally nilpotent; this is a natural class of ``test schemes", as explained in \S\ref{sss:p-nilpotence}.

In this section we usually write ``stack'' instead of ``g-stack''.

This section overlaps with \cite[\S 3]{BL}. 
The relation between our notation and that of \cite{BL} is as follows: our stacks $\Sigma$,  $\Delta_0$,  $W_{\prim}$ are denoted in 
\cite{BL} by $\WCart$, $\WCart^{\HT}$, $\WCart_0$, respectively. In \cite{Bh} the stacks $\Sigma$,  $\Delta_0$ are denoted by 
$\BZ_p^\prism$ and $\BZ_p^{\HT}$, respectively.

\subsection{The formal scheme $W_{\prim}$}   \label{ss:W_prim}
\subsubsection{A locally closed subscheme of $W$}   \label{sss:A locally closed subscheme}
Let $A\subset W\otimes\BF_p$ be the locally closed subscheme obtained by removing $\Ker (W\epi W_2)\otimes\BF_p$ from $\Ker (W\epi W_1)\otimes\BF_p$.
In terms of the usual coordinates $x_0, x_1,\ldots$ on the scheme $W$, the subscheme $A\subset W$ is defined by the equations $p=x_0=0$ and the inequality $x_1\ne 0$.

\subsubsection{Definition of $W_{\prim}$}    \label{sss:W_prim}
Define $W_{\prim}$ to be the formal completion of $W$ along the locally closed subscheme $A$ from \S\ref{sss:A locally closed subscheme}.
(In \cite{BL} and \cite{Bh} this formal scheme is denoted by $\WCart_0$ and $W_{dist}$ respectively, where ``dist'' stands for ``distinguished''.)

In other words, for any scheme $S$, an $S$-point of $W_{\prim}$ is a morphism $S\to W$ which maps $S_{\red}$ to $A$. 
If $S$ is $p$-nilpotent and if we think of a morphism $S\to W$ as a sequence of functions $x_n\in H^0(S,\cO_S )$ then the condition is that $x_0$ is locally nilpotent and $x_1$ is invertible. If $S$ is not $p$-nilpotent then $W_{\prim} (S)=\emptyset$. 

It is clear that $W_{\prim}$ is a formal affine scheme. In terms of the usual coordinates $x_0, x_1,\ldots$ on $W$, the  coordinate ring of $W_{\prim}$ is the completion of $\BZ_p[x_0,x_1,\ldots ][x_1^{-1}]$ with respect to the ideal $(p,x_0)$ or equivalently,  the $p$-adic completion of $\BZ_p [x_1,x_1^{-1},x_2,x_3,\ldots][[x_0]]$.

Probably $W_{\prim}$ should be called the formal scheme of \emph{primitive\footnote{The word ``primitive'' is used in a similar context in \cite[\S 2.2.2.1]{FF} and 
\cite[Definition~6.2.9]{SW}. On the other hand, \cite{Bh} uses the word ``distinguished'' instead of ``primitive of degree 1''.}  Witt vectors of degree $1$}. However, we always skip the words ``of degree $1$'' (primitive Witt vectors of degrees different from $1$ never appear in this work).

\begin{lem}      \label{l:F on W_prim}
(i) The morphism $F:W\to W$ induces a morphism $W_{\prim}\to W_{\prim}$ (also denoted by $F$).

(ii) The morphism $F:W_{\prim}\to W_{\prim}$ is schematic and faithfully flat. 

(iii) The following diagram is Cartesian:
\[
\xymatrix{
W_{\prim}\ar[r]^F \ar[d] &W_{\prim}\ar[d]\\
W\ar[r]^F & W
}
\]
\end{lem}

\begin{proof}
The morphism $F:W\otimes\BF_p\to W\otimes\BF_p$ is the usual Frobenius. Combining this with the description of $W_{\prim} (S)$ from \S\ref{sss:W_prim}, we see that
for any scheme $S$, the preimage of $W_{\prim} (S)\subset W(S)$ with respect to $F:W(S)\to W(S)$ equals $W_{\prim} (S)$. This proves (i) and~(iii). Statement (ii) follows from (iii) and \S\ref{ss:faithful flatness of F}.
\end{proof}

\subsection{Definition of $\Sigma$}   \label{ss:def of Sigma}
\subsubsection{Action of $W^*$ on $W_{\prim}$}   \label{sss:Action of W^* on W_prim}
The morphism
\begin{equation}   \label{e:action by division}
W^\times\times W_{\prim}\to W_{\prim}, \quad (\lambda ,x)\mapsto \lambda^{-1}x
\end{equation}
defines an action of  $W^*$ on $W_{\prim}$ (``action by division''). The reason why we prefer it to the action by multiplication is explained in \S\ref{sss:Why division} below. The difference between the two actions is irrelevant for most purposes.

\subsubsection{$\Sigma$ as a quotient stack}   \label{sss:Sigma as quotient}
The stack $\Sigma$ is defined as follows:
\begin{equation}   \label{e:Sigma as quotient}
\Sigma:=W_{\prim}/W^\times . 
\end{equation}
In other words, $\Sigma$ is the fpqc-sheafification of the presheaf of groupoids $$R\mapsto~W_{\prim} (R)/W(R)^\times.$$
It is also the Zariski sheafification of this presheaf: indeed, any $W^\times$-torsor on a $p$-nilpotent scheme is Zariski-locally trivial (this is a reformulation of Lemma~\ref{l:topology-independence}).

It is clear that $\Sigma$ is a $\MMor$-affine\footnote{For the notion of $\MMor$-affineness, see Definition~\ref{d:pre-algebraic stack} and Remark~\ref{r:Mor-affine}.} formal stack in the sense of \S\ref{sss:def formal scheme} and moreover, a strongly adic stack in the sense of \S\ref{sss:strongly adic}. 

Let us note that in \cite{BL} the stack $\Sigma$ is called the \emph{Cartier-Witt stack} and denoted by $\WCart$. In \cite{Bh} it is denoted by $\BZ_p^\prism$.

To describe $S$-points of $\Sigma$, we need the following definition.

\begin{defin}   \label{d:primitivity of xi}
Let $S$ be a $p$-nilpotent scheme and $M$ a $W_S$-module in the sense of \S\ref{ss:W_S-module generalities}. A $W_S$-morphism $\xi :M\to W_S$ is said to be \emph{primitive}\footnote{The true name should be ``primitive of degree $1$'', see the end of \S\ref{sss:W_prim}.} if for every $s\in S$ one has
\[
(M_s)_{\red}\subset\Ker\xi_1,\quad (M_s)_{\red}\not\subset\Ker\xi_2 ,
\]
where $\xi_n:M\to (W_n)_S$ is the composite map $M\overset{\xi}\longrightarrow W_S\epi (W_n)_S$.
\end{defin}

\subsubsection{$S$-points of $\Sigma$}   \label{sss:Sigma(S)}
Recall that a $W_S$-module is said to be invertible if it is locally isomorphic to $W_S$; here the word ``locally'' can be understood in either Zariski or the fpqc sense (see Definition~\ref{d:invertible} and Lemma~\ref{l:topology-independence}).  An invertible $W_S$-module is essentially the same as a $W_S^\times$-torsor. Note that if a $W_S$-module $M$ is invertible then $(M_s)_{\red}=M_s$ for all $s\in S$.\footnote{Moreover, if $S$ is $p$-nilpotent and affine then an invertible $W_S$-module (in the sense of \S\ref{sss:2Sigma(S)}) is essentially the same as an invertible $W(R)$-module, where $R=H^0(S,\cO_S)$. So it is easy to avoid using $W_S$-modules in the current section. (However, it is harder to avoid them in \S\ref{s:Sigma'}.)}

However, $\Sigma'$ is defined using $W_S$-modules which are not locally free (see 
\S\ref{sss:idea of Sigma'}). So it is harder to avoid using $W_S$-modules while dealing with $\Sigma'$. (An attempt in this direction is made in Appendix~\ref{s:SSigma'}).

It is clear from \eqref{e:Sigma as quotient} that if a scheme $S$ is not  $p$-nilpotent then $\Sigma (S)=\emptyset$, and if $S$ is $p$-nilpotent then $\Sigma (S)$ identifies with the underlying groupoid of the category of pairs $(M,\xi)$, where $M$ is an invertible $W_S$-module and $\xi:M\to W_S$ is a primitive $W_S$-morphism.\footnote{The words ``category of'' mean that a morphism $(M_1,\xi_1)\to (M_2,\xi_2)$ is a $W_S$-module homomorphism $f:M_1\to M_2$ such that $\xi_2\circ f=\xi_1$.} Moreover, the next lemma shows that 
the words ``underlying groupoid of'' are unnecessary.

\begin{lem}   \label{l:g-stack}
Let $M_1,M_2$ be invertible $W_S$-modules. Let $\xi_1,\xi_2: M\to W_S$ primitive $W_S$-morphisms. Then every morphism $(M_1,\xi_1)\to (M_2,\xi_2)$ is invertible.
\end{lem}

\begin{proof}
We have to show that if $R$ is a ring and $\alpha,\beta\in W (R)$ are such that $\beta\in W_{\prim} (R)$ and 
$\alpha\beta\in W_{\prim} (R)$ then $\alpha\in W (R)^\times$. We can assume that $R$ is a reduced $\BF_p$-algebra.
Then $\beta=Vx$ for some $x\in W(R)^\times$. So $\alpha\beta =V(F(\alpha)\cdot x)$. 
Since $\alpha\beta\in W_{\prim} (R)$, we see that $F(\alpha)$ is invertible by Lemma~\ref{l:invertible in W}(ii). So $\alpha$ is invertible by Lemma~\ref{l:invertible in W}(iii).
\end{proof}

\subsubsection{Remark} \label{sss:Why division} 
The reason for writing $\lambda^{-1}x$ in formula \eqref{e:action by division} is that $\xi :M\to W_{\prim}$ is a ``linear functional'' on $M$ rather than a ``vector'' in $M$.
See also \S\ref{sss:sL_Sigma in terms of W^times} below.

\medskip

The following lemma is not used in the rest of the article. It is a variant of the last statement of \cite[Lemma~3.6]{BS}.

\begin{lem}
If $S$ is affine and $(M, \xi)$ is as in \S\ref{sss:Sigma(S)} then $M^{\otimes p}$ is (noncanonically) trivial. 
\end{lem}

\begin{proof}
Using the Teichm\"uller embedding $\BG_m\subset W^\times$, we can factor the morphism $W_{\prim}\to \Sigma$ as follows:
\[
W_{\prim}\to W_{\prim}/\mu_p\to W_{\prim}/\BG_m\to W_{\prim}/W^\times=\Sigma .
\]
It suffices to show that the functor $\Mor (S,W_{\prim}/\mu_p)\to \Mor (S,\Sigma )$ is essentially surjective.
The functor $\Mor (S,W_{\prim}/\BG_m )\to \Mor (S,\Sigma )$ is essentially surjective because
$S$ is affine and the group $W^\times/\BG_m$ is pro-unipotent. Finally, let us show that the morphism
$$W_{\prim}/\mu_p\to W_{\prim}/\BG_m$$ 
admits a section. Indeed, let $x_0,x_1,\ldots$ be the usual coordinates on $W$ and let $W_{\prim}^1\subset W_{\prim}$ be the closed substack defined by the equation $x_1=1$; then $\mu_p\cdot W_{\prim}^1=W_{\prim}^1\,$, and the map $W_{\prim}^1/\mu_p\iso W_{\prim}/\BG_m$ is an isomorphism.
\end{proof}

\subsubsection{The morphism $F:\Sigma\to\Sigma$}   \label{sss:F:Sigma to Sigma}
By \eqref{e:Sigma as quotient}, the morphisms 
$$F:W_{\prim}\to W_{\prim}\, ,\quad F:W^\times\to W^\times$$ 
induce a morphism $F:\Sigma\to\Sigma$. By Lemma~\ref{l:F on W_prim}, the morphism $F:\Sigma\to\Sigma$ is algebraic and faithfully flat.  

For any $p$-nilpotent scheme $S$, the functor $F:\Sigma (S)\to\Sigma (S)$ has the following description in terms of \S\ref{sss:Sigma(S)}:
\[
F(M,\xi )=(M,\xi )\otimes_{W_S,F}W_S .
\]
By \eqref{e:2M' for invertibleM}, one can rewrite this formula as
\begin{equation}   \label{e:F(M,xi)}
F(M,\xi )=((M')^{(-1)},(\xi')^{(-1)}).
\end{equation}

\subsubsection{The morphism $\Sigma\to\hat\BA^1/\BG_m$}
Let $\hat\BA^1$ be the formal completion of $\BA^1$ at the point $0\in\BA^1\otimes\BF_p\,$; in other words, $\hat\BA^1=\Spf\BZ_p[[x]]$. The morphism $W\epi W_1=\BG_a=\BA^1$ induces morphisms 
\begin{equation}   \label{e:Wprim to hat A^1}
W_{\prim}\to\hat\BA^1,
\end{equation}
\begin{equation}   \label{e:Sigma to A^1/G_m}
\Sigma=W_{\prim}/W^\times\to\hat\BA^1/\BG_m ,
\end{equation}
where $\BG_m$ acts on $\hat\BA^1$ by division, see \S\ref{sss:Action of W^* on W_prim}.

\begin{lem}    \label{l:Sigma to A^1/G_m}
$\Sigma$ is algebraic and flat over $\hat\BA^1/\BG_m$.
\end{lem}

\begin{proof}
Follows from a similar property of the morphism \eqref{e:Wprim to hat A^1}, which is clear from \S\ref{sss:W_prim}.
\end{proof}

\subsubsection{Another description of the morphism $\Sigma\to\hat\BA^1/\BG_m$}   \label{sss:Sigma to A^1/G_m}
Let $(M,\xi )\in\Sigma (S)$. Applying the functor $\otimes_{W_S}(\BG_a)_S$ to $(M,\xi )$, we get a pair $(\sL , v:\sL\to\cO_S)$, where $\sL$ is a line bundle on~$S$.
Moreover, $v$ vanishes on $S_{\red}$, and $S$ is $p$-nilpotent, so $(\sL , v)$ defines an $S$-point of $\hat\BA^1/\BG_m$. This is the image of $(M,\xi )$ under the morphism \eqref{e:Sigma to A^1/G_m}.

\subsection{Some properties of $\Sigma$}

\begin{lem}   \label{l:Sigma(perfect)}
Let $S$ be a perfect $\BF_p$-scheme. Then $\Sigma (S)$ is a point.
\end{lem}

\begin{proof}
By \S\ref{sss:Sigma as quotient}, 
$\Sigma$ is the Zariski sheafification of the presheaf of groupoids 
$$R\mapsto~W_{\prim} (R)/W(R)^\times.$$
So it suffices to show that if $R$ is a perfect $\BF_p$-algebra then the groupoid $W_{\prim} (R)/W(R)^\times$ is a point. Since $R$ is reduced, all elements of $W_{\prim} (R)$ have the form $Vy$, where $y\in W(R)^\times$. If $y\in W(R)^\times$ then there is a unique $u\in W(R)^\times$ such that $Vy=u\cdot V(1)$: namely, $u=F^{-1}(y)$ (note that $F^{-1}$ exists by perfectness of $R$).
\end{proof}

\begin{cor} \label{c:|Sigma|}
The set $|\Sigma |$ (see \S\ref{sss:|X|}) has exactly one element. \qed
\end{cor}

\subsubsection{The stacks $\Sigma_n$}  \label{sss:Sigma_n}
Similarly to \S\ref{sss:W_prim}, define for each $n\in\BN$ a formal scheme $(W_n)_{\prim}$ to be the formal completion of $W_n$ along the following locally closed subscheme $A_n\subset W_n$:   the subscheme $A_1\subset W_1$ is the point $0$ in $W_1\otimes\BF_p=\BA^1_{\BF_p}$, and if $n>1$ then $A_n$ is obtained by removing the subscheme $\Ker (W_n\epi W_2)\otimes\BF_p$ from $\Ker (W_n\epi W_1)\otimes\BF_p$. Set $$\Sigma_n:=(W_n)_{\prim}/W_n^\times,$$
where $W_n^\times$ acts on $(W_n)_{\prim}$ 
by division (just as in \S\ref{sss:Action of W^* on W_prim}). 
Then $\Sigma_n$ is a $\MMor$-affine formal stack.

\begin{prop}   \label{p:Sigma as limit}
(i) $\Sigma$ is the projective limit of the stacks $\Sigma_n$.

(ii) $\Sigma_1$ is the formal stack $\hat\BA^1/\BG_m$, and $(\Sigma_1)_{\red}=(\Spf\BF_p)/\BG_m$.

(iii) Each morphism $\Sigma_{n+1}\to\Sigma_n$ is algebraic, of finite presentation, smooth, and of pure relative dimension $0$. 

(iv) Let $\sY_n\subset\Sigma_n$ be the preimage of the closed substack 
$$(\Spf\BZ_p)/\BG_m=\{0\}/\BG_m\subset\hat\BA^1/\BG_m=\Sigma_1.$$
Then $\sY_n$ is canonically isomorphic to the classifying stack $(\Spf\BZ_p)/H_n$, where
$$H_n=\Ker (W_n^\times\overset{F}\longrightarrow W_{n-1}^\times).$$

(v) Let $\sX$ be any stack which is algebraic over $\Sigma$ and flat over $\Sigma_1=\hat\BA^1/\BG_m$, then $\sX$ is flat over $\Sigma$. If, in addition, 
$\sX\ne\emptyset$ then $\sX$ is faithfully flat over $\Sigma$

(vi) The canonical morphism $\hat\BA^1\to\Sigma_1$ can be lifted to a morphism $\hat\BA^1\to\Sigma$.

(vii) Any morphism $\hat\BA^1\to\Sigma$ as in (vi) is a flat universal homeomorphism\footnote{A schematic morphism of stacks $\sX\to\sY$ is said to be a universal homeomorphism if it becomes such after any base change $S\to\sY$ with $S$ being a scheme.} of infinite type. (In particular, it yields a presentation of $\Sigma$ as a quotient of $\hat\BA^1$ by a flat groupoid of infinite type.)
\end{prop}

According to (vii), the formal stack $\Sigma$ is not far from the familar world. 

\begin{proof}
Statements (i)-(iii) are clear. 

Let us prove (iv). We have a locally closed embedding $V:W_{n-1}^\times\mono W$, and $\sY_n$ is the $p$-adic completion of the stack
$V(W_{n-1}^\times )/W_n^\times$. The element $V(1)\in W(\BZ)$ gives a $\BZ$-point of $V(W_{n-1}^\times)$.
Using this point and the identity $u\cdot Vw=V((Fu)\cdot w)$, we identify 
$V(W_{n-1}^\times )/W_n^\times$ with 
$\Cone (W_n^\times\overset{F}\longrightarrow W_{n-1}^\times )$ 
(here $\Cone$ is understood in the sense of \S\ref{sss:Cones-group schemes}). 
By \S\ref{ss:faithful flatness of F}, the morphism $F:W_n^\times\to W_{n-1}^\times$ is 
faithfully flat, so $\Cone (W_n^\times\overset{F}\longrightarrow W_{n-1}^\times )=(\Spec\BZ )/H_n$.

Let us prove the flatness part of (v). Since $\sX$ is flat over $\Sigma_1$, it suffices to show that $\sX\times_{\Sigma_1}(\Sigma_1)_{\red}$ is flat over 
$\Sigma\times_{\Sigma_1}(\Sigma_1)_{\red}$. This is automatic because $\Sigma\times_{\Sigma_1}(\Sigma_1)_{\red}$ is the classifying stack of a group scheme over $\BF_p$: indeed, statement (iv) implies that $$\Sigma_n\times_{\Sigma_1}(\Sigma_1)_{\red}=(\Spec\BF_p)/(H_n\otimes\BF_p).$$

Statement (vi) is a formal consequence of the smoothness part of (iii), but it is clear anyway.

The flatness part of (vii) follows from (v). To prove that the morphism
$\hat\BA^1\to\Sigma$ is a universal homeomorphism of infinite type, it suffices to note that this is true for the morphism $\Spec\BF_p =(\hat\BA^1)_{\red}\to\Sigma_{\red}\,$: indeed, the scheme $\underset{\longleftarrow}{\lim}\, (H_n\otimes\BF_p)=(W^\times)^{(F)}\otimes\BF_p$ has infinite type, and its reduced part equals $\Spec\BF_p$.
\end{proof}

\begin{rem}
By Proposition~\ref{p:Sigma as limit}(iv), for every algebraically closed field $k$ of characteristic~$p$ and each $n$ the stack $\Sigma_n$ has one and only one isomorphism class of $k$-points. However, \emph{let us warn the reader that the geometric fiber of the morphism $(\Sigma_{n+1})_{\red}\to (\Sigma_n)_{\red}$ is quite big}.
Indeed, Proposition~\ref{p:Sigma as limit}(iv) implies that this fiber is the stack $\Cone (H_{n+1}\to H_n)\otimes k$, which is a gerbe over a curve.
\end{rem}

\subsection{The morphisms $p:\Spf\BZ_p\to\Sigma$ and $V(1):\Spf\BZ_p\to\Sigma$}   \label{sss:p and V(1)}
The elements $p,V(1)\in W(\BZ_p)$ define morphisms $\Spf\BZ_p\to W_{\prim}$; composing them with the morphism $W_{\prim}\to\Sigma$, we get objects $p,V(1)$ of the groupoid 
$\Mor (\Spf\BZ_p, \Sigma )$. The automorphism group of the object $p\in\Mor (\Spf\BZ_p, \Sigma )$ is trivial\footnote{On the other hand, Lemma~\ref{l:Delta_0 as classifying stack} implies that the automorphism group of $V(1)\in\Mor (\Spf\BZ_p, \Sigma )$ is nontrivial: it equals $\BG_m^\sharp (\BZ_p)=1+p\BZ_p$. Let us also note that the group \emph{scheme} of automorphisms of $p\in\Mor (\Spf\BZ_p, \Sigma )$ is nontrivial (see the proof of Lemma~\ref{l:not mono}).} (because all ghost components of $p\in W(\BZ_p)$ are nonzero).

The morphism $F:\Sigma\to\Sigma$ induces a map $F:\Mor (\Spf\BZ_p, \Sigma )\to\Mor (\Spf\BZ_p, \Sigma )$ such that $F(V(1))=F(p)=p$.

Let us note that Bhatt and Lurie 
\cite{BL} call $p\in\Mor (\Spf\BZ_p, \Sigma )$ the \emph{de Rham point} (the reason is clear from our \S\ref{sss:X^dR}).

\subsection{The effective Cartier divisors $\Delta_n\subset\Sigma$}  \label{ss:Delta_n}
Since $\Sigma$ is not an algebraic stack but a \emph{formal} one,  the words ``effective Cartier divisor'' will be understood in the sense of \S\ref{sss:effective Cartier}.

\subsubsection{The Hodge-Tate locus $\Delta_0\subset\Sigma$}   \label{sss:Hodge-Tate locus}
We have a canonical morphism $\Sigma\to\hat\BA^1/\BG_m$, see~\eqref{e:Sigma to A^1/G_m}. 
Let $\Delta_0\subset\Sigma$ be the preimage of $\{ 0\}/\BG_m\subset\hat\BA^1/\BG_m$. 
By Lemma~\ref{l:Sigma to A^1/G_m}, $\Delta_0$ is an effective Cartier divisor on $\Sigma$; moreover,  $\Delta_0\otimes\BF_p$ is an effective Cartier divisor on~$\Sigma\otimes\BF_p$. 
Note that $\Sigma_{\red}=\Delta_0\otimes\BF_p$.

The divisor $\Delta_0$ is called the \emph{Hodge-Tate locus} in $\Sigma$. Let us note that in \cite{BL} this divisor is denoted by $\WCart^{\HT}$.
In \cite{Bh} it is denoted by $\BZ_p^{\HT}$.

The line bundle $\cO_{\Sigma}(-\Delta_0)$ can be described as follows. If $(M,\xi)\in\Sigma (S)$ then $M/V(M')$ is a line bundle on $S$ equipped with a morphism to $\cO_S$ with cokernel $\cO_{S\times_{\Sigma}\Delta_0}$. Letting $S$ and $(M,\xi)$ vary, one gets a line bundle on $\Sigma$, which is nothing but $\cO_{\Sigma}(-\Delta_0)$.

Note that the object $V(1)\in\Mor (\Spf\BZ_p, \Sigma )$ is in $\Mor (\Spf\BZ_p, \Delta_0)$. On the other hand, the
object $p\in\Mor (\Spf\BZ_p, \Sigma )$ is not in $\Mor (\Spf\BZ_p, \Delta_0)$.

The following lemma describes $\Delta_0$ as a classifying stack; it also appears in \cite{BL} and plays a central role there.

\begin{lem} \label{l:Delta_0 as classifying stack}
The morphism $V(1):\Spf\BZ_p\to\Delta_0$ induces an isomorphism 
\begin{equation}    \label{e:Delta_0 as classifying stack}
(\Spf\BZ_p)/\BG_m^\sharp\iso\Delta_0.
\end{equation}
\end{lem}

\begin{proof}
By Proposition~\ref{p:Sigma as limit}(i,iv), $\Delta_0=(\Spf\BZ_p)/(W^\times)^{(F)}$. By \eqref{e:(W^times)^(F)=G_m^sharp}, $(W^\times)^{(F)}=\BG_m^\sharp$.
\end{proof}

\begin{lem}    \label{l:F restricted to Delta_0}
The following diagram commutes:
\begin{equation}  \label{e:F on Delta_0}
\xymatrix{
\Delta_0\ar[d] \ar@{^{(}->}[r] &\Sigma\ar[d]^{F}\\
\Spf\BZ_p\ar[r]^{p}&\Sigma
}
\end{equation}
\end{lem}

Here the commutation is up to unique isomorphism because the automorphism group of $p:\Spf\BZ_p\to\Sigma$ is trivial (see \S\ref{sss:p and V(1)}) and the morphism $\Delta_0\to\Spf\BZ_p$ has a section.

\begin{proof}
Follows from the identity $FV=p$.
\end{proof}

Let us note that diagram \eqref{e:F on Delta_0} plays a central role in \cite{BL} (e.g., see \cite[Thm.~3.6.7]{BL}).

\subsubsection{Remarks related to Lemma~\ref{l:Delta_0 as classifying stack}}
(i) It is easy to show that for any $\BG_m^\sharp$-module $M$ over $\BF_p$ one has $H^i(\BG_m^\sharp,M)=0$ for $i>1$. This fact easily implies that \emph{$\Sigma$ has cohomological dimension $1$.} According to B~Bhatt, this is not clear from \cite{BS} and was one of the reasons why he and Lurie suggested the stacky approach to prismatic cohomology. 

(ii) If $n>2$ then the stack $\Sigma_n$ from \S\ref{sss:Sigma_n} has \emph{infinite} cohomological dimension: this follows from a similar statement about the group scheme $H_n\otimes\BF_p$, where $H_n$ is as in Proposition~\ref{p:Sigma as limit}(iv). Thus $\Sigma$ is cohomologically ``better" than the stacks $\Sigma_n$ for $n>2$. This is somewhat similar to the following: if $O$ is a discrete valuation ring with maximal ideal $\fm$ than $O$ has homological dimension $1$ even though $O/\fm^n$ has infinite homological dimension if $n>1$.

(iii) In \S\ref{s:Sigma'} we will introduce a substack $\Delta'_0\subset\Sigma'$ and prove an analog of Lemma~\ref{l:Delta_0 as classifying stack} for~$\Delta'_0$ (see 
Lemmas~\ref{l:Delta'_0 as g-stack} and \ref{l:Delta'_0 as c-stack}). An important application is briefly mentioned in \S\ref{sss:goal of Delta'_0 as c-stack}.

\subsubsection{The divisors $\Delta_n\subset\Sigma$}
For $n\ge 0$ let $\Delta_n:=(F^n)^{-1}(\Delta_0)\subset\Sigma$. By \S\ref{sss:F:Sigma to Sigma}, the morphism $F:\Sigma\to\Sigma$ is flat, so each $\Delta_n$ is an effective Cartier divisor on $\Sigma$. We also have the effective Cartier divisor $\Sigma\otimes\BF_p\subset\Sigma$.

\begin{lem}  \label{l:Delta_m cap Delta_n}
(i) If $m<n$ then $\Delta_m\cap\Delta_n=\Delta_m\otimes\BF_p$.

(ii) The effective Cartier divisor $\Delta_m\otimes\BF_p\subset\Sigma\otimes\BF_p$ equals $p^m\cdot(\Delta_0\otimes\BF_p)$. In particular, $\Delta_m\otimes\BF_p\subset \Delta_{m+1}\otimes\BF_p$. 

(iii) $\Sigma\otimes\BF_p=\bigcup\limits_{m=0}^\infty (\Delta_m\otimes\BF_p )$.

(iv) Let $f:S\to\Sigma$ be a morphism with $S$ being a quasi-compact scheme. Then $f^{-1}(\Delta_n)=S\otimes\BF_p$ for $n$ big enough.
\end{lem}

Part (iv) of the lemma says that in a certain sense $\Delta_n$ ``tends" to $\Sigma\otimes\BF_p$ as $n$ goes to $\infty$. Toy model: instead of $\Sigma$ and $\Delta_n$, consider $\Spf\BR [[x,y]]$ and $D_n:=\Spf\BR [[x,y]]/(y-x^{p^n})$; then $D_n$ ``tends" to $\Spf\BR [[x,y]]/(y)$ as $n$ goes to $\infty$. 

\begin{proof}
To prove (i), we can assume that $m=0$. The divisor $\Delta_n\times_\Sigma W_{\prim}\subset W_{\prim}$ is given by the equation 
\begin{equation}  \label{e:w_n=0}
x_0^{p^n}+px_1^{p^{n-1}}+\ldots +p^nx_n=0,
\end{equation}
 where the $x_i$'s are the usual coordinates on $W$. So
$(\Delta_0\cap\Delta_n)\times_\Sigma W_{\prim}$ is given by the system of equations $x_0=x_0^{p^n}+px_1^{p^{n-1}}+\ldots +p^nx_n=0$. It is equivalent to the system $x_0=p=0$ (because $x_1$ is invertible on $W_{\prim}$).

Statements (ii)-(iii) are clear because $F:\Sigma\otimes\BF_p\to\Sigma\otimes\BF_p$ is the usual Frobenius.

Let us prove (iv). There exists $l$ such that $x_0^l=0$ on $S\times_\Sigma W_{\prim}$.  If $p^n\ge l$ then $f^{-1}(\Delta_n)=S\otimes\BF_p$ because on the formal subscheme of $W_{\prim}$ defined by $x_0^l=0$ equation \eqref{e:w_n=0} becomes equivalent to $p=0$.
\end{proof}

\subsubsection{Remark}    \label{sss:Informal remarks on Sigma}
Using Proposition~\ref{p:Sigma as limit}(vii), one can show that the stacks $\Sigma$, $\Delta_n$, and $\Sigma\otimes\BF_p$ are ``regular" in the following sense:
each of them can be represented as  $(\Spf A)/\Gamma$, where $A$ is a complete regular local ring and $\Gamma$ is a flat groupoid acting on $\Spf A$ (the morphisms $\Gamma\to\Spf A$ have infinite type).

\begin{prop}   \label{p:divisors on Sigma}
The monoid of effective Cartier divisors on $\Sigma$ is freely generated by $\Sigma\otimes\BF_p$ and $\Delta_n$, $n\ge 0$.
\end{prop}

A proof will be given in \S\ref{ss:divisors on Sigma} of Appendix~\ref{s:q-de Rham}.

\subsection{The contracting property of $F:\Sigma\to\Sigma$}  \label{ss:contracting}
\subsubsection{Generalities on contracting functors}
Let $\cC$ be a category and $F:\cC\to\cC$ a functor. In this situation one defines $\cC^F$ to be the category of pairs $(c,\alpha)$, where $c\in\cC$ and $\alpha :c\iso F(c)$ is an isomorphism; $\cC^F$ is called
the \emph{category of fixed points} of $F:\cC\to\cC$. The canonical functor $\cC^F\to\cC$ is faithful but not necessarily fully faithful. 

On the other hand, one sets
\[
\cC [F^{-1}]:=\underset{\longrightarrow}{\lim} 
(C\overset{F}\longrightarrow\cC\overset{F}\longrightarrow\cC\overset{F}\longrightarrow\ldots );
\]
the category $\cC [F^{-1}]$ is called the \emph{localization of $\cC$ with respect  to $F$}. 

We say that $F:\cC\to\cC$ is \emph{contracting} if $\cC [F^{-1}]$ is a point (i.e., a category with one object and one morphism). 

\begin{lem}  \label{l:contracting nonsense}
If $F:\cC\to\cC$ is contracting then $\cC^F$ is a point.
\end{lem}

\begin{proof}
(i) Any $c_1, c_2\in\cC$ have isomorphic images in $\cC [F^{-1}]$, so $F^n(c_1)\simeq F^n(c_2)$ for some $n$.

(ii) If $c_1, c_2\in\cC$, $F(c_1)\simeq c_1$, and $F(c_2)\simeq c_2$ then $c_1\simeq c_2$: this follows from (i).

(iii) Let us show that $\cC^F\ne\emptyset$. Since $\cC [F^{-1}]\ne\emptyset$ we have $\cC\ne\emptyset$. Choose $c\in\cC$ and apply (i) for $c_1=c$, $c_2=F(c)$. We see that  there exists $n$ such that $F^n(c)\simeq F^n(F(c))=F(F^n(c))$, so $F^n(c)$ belongs to the essential image of the functor $\cC^F\to\cC$.

(iv) The lemma holds if $\cC$ is a set: this follows from (ii) and (iii).

(v) To prove the lemma in general, it remains to show that if $c,c'\in\cC^F$ then $\Mor_{\cC^F}(c,c')$ is a point.
Let $S:=\Mor_{\cC}(c,c')$. Since $c,c'\in\cC^F$, the functor $F$ induces a map 
$$f:S\to \Mor_{\cC}(F(c),F(c'))=\Mor_{\cC}(c,c')=S,$$
and $\Mor_{\cC^F}(c,c')$  is just the set of fixed points $S^f$.  The map $f$ is contracting because $F$ is. So $S^f$ is a point by (iv).
\end{proof}

\begin{rem}
Lemma~\ref{l:contracting nonsense} is a consequence of the following general statement: if $\cC$ is a category and
$F,G:\cC\to\cC$ are functors equipped with an isomorphism $GF\iso FG$ then the canonical functor 
$(\cC^G)[F^{-1}]\to (\cC[F^{-1}])^G$ is an equivalence. To deduce the lemma from this statement, take $G=F$.
\end{rem}

Now let us return to the stack $\Sigma$. 

\begin{prop}     \label{p:contracting}
For every quasi-compact quasi-separated  $p$-nilpotent scheme $S$, the functor $\Sigma (S)\to\Sigma (S)$ induced by the morphism $F:\Sigma\to\Sigma$ is contracting.
\end{prop}

The proof will be given in \S\ref{sss:contracting proof}.

\begin{cor}  \label{c:Sigma^F}
Let $\Sigma^F$ be the stack such that $\Sigma^F(S)=\Sigma (S)^F$ for every scheme $S$. Then

(i) $\Sigma^F=\Spf\BZ_p$;

(ii) the morphism $\Spf\BZ_p =\Sigma^F\to\Sigma$ is the morphism $p:\Spf\BZ_p\to\Sigma$ from \S\ref{sss:p and V(1)}.
\end{cor}

\begin{proof}
Statement (i) follows from Lemma~\ref{l:contracting nonsense} and Proposition~\ref{p:contracting}. Statement (ii) follows from (i) and the fact that $p\in W(\BZ_p)$ satisfies $F(p)=p$.
\end{proof}

\subsubsection{Warnings}
(i) According to Proposition~\ref{p:density} below, the morphism $\Sigma^F\to\Sigma$ does not factor through any closed substack of $\Sigma$ different from $\Sigma$.

(ii) The following lemma shows that $\Sigma^F$ is not a \emph{sub}stack of $\Sigma$.

\begin{lem}   \label{l:not mono}
The morphism $p:\Spf\BZ_p\to\Sigma$ is not a monomorphism.
\end{lem}

\begin{proof}
Recall that $\Sigma :=W_{\prim}/W^\times$, so it suffices to show that the stabilizer of the point $p\in W_{\prim}(\BF_p)$ in $W^\times\hat\otimes\BF_p$ is nontrivial. In $W_{\prim}(\BF_p)$ one has $p=V(1)$, so the stabilzer in question equals $(W^\times)^{(F)}\otimes\BF_p$, which is nontrivial.
\end{proof}

\subsubsection{Proof of Proposition~\ref{p:contracting}}   \label{sss:contracting proof}
By definition, $\Sigma$ is the stack associated to the pre-stack $S\mapsto W_{\prim}(S)/W(S)^\times$. So it suffices to prove the following two lemmas.

\begin{lem}   \label{l:contracting1}
Let $R$ be a $p$-nilpotent ring and $x\in W_{\prim}(R)$. Then there exists $n\in\BN$ such that $F^n(x)=pu$ for some $u\in W(R)^\times$.
\end{lem}

\begin{proof}
Write $x=[x_0]+Vy$, where $x_0\in R$, $y\in W(R)$. The assumption $x\in W_{\prim}(R)$ means that $x_0$ is nilpotent and the $0$-th component of the Witt vector $y$ is invertible. Then $y\in W(R)^\times$ by Lemma~\ref{l:invertible in W}(ii). If $n\in\BN$ is such that $x_0^{p^n}=0$ then $F^n(x)=pF^{n-1}(y)$ and $F^{n-1}(y)\in W(R)^\times$.
\end{proof}

\begin{lem}   \label{l:contracting2}
Let $R$ be a $p$-nilpotent ring. Let $u\in W(R)^\times$ and $pu=p$. Then $F^n(u)=1$ for some $n\in\BN$.
\end{lem}

\begin{proof}
Since $p\in R$ is nilpotent, it suffices to show that for every $n\in\BN$ one has
\begin{equation}   \label{e:to prove}
F^n(u-1)\in W(p^nR):=\Ker (W(R)\epi W(R/p^nR)).
\end{equation}

We will proceed by induction. Let $n=1$. Let $\bar u\in W(R/pR)$ be the image of $u$.
We have $FV(\bar u-1)=p(\bar u-1)=0$. In $W(R/pR)$ we have $FV=VF$, so $VF(\bar u-1)=0$ and $F(\bar u-1)=0$
(by injectivity of $V$).

Now assume that \eqref{e:to prove} holds for some $n$. Then
$$VF^{n+1}(u-1)=V(1)\cdot F^n(u-1)=(V(1)-p)\cdot F^n(u-1)$$
(the second equality holds because $p(u-1)=0$). Since $V(1)-p\in W(pR)$, we get
$$VF^{n+1}(u-1)\in W(pR)\cdot W(p^nR)\subset W(p^{n+1}R).$$
But $V:W(R/p^{n+1}R)\to W(R/p^{n+1}R)$ is injective, so $F^{n+1}(u-1)\in W(p^{n+1}R)$.
\end{proof}

\subsection{Density of the image of $p:\Spf\BZ_p\to\Sigma$}
\begin{prop}   \label{p:density}
Let $\sY\subset\Sigma$ be a closed substack such that the morphism $p:\Spf\BZ_p\to\Sigma$ factors through $\sY$. Then $\sY=\Sigma$.
\end{prop}

Before proving the proposition, let us formulate the following

\begin{cor}    \label{c:density}
The canonical homomorphism $\BZ_p\to H^0(\Sigma , \cO_{\Sigma})$ is an isomorphism.
\end{cor}

\begin{proof}
Let $\varphi :H^0(\Sigma , \cO_{\Sigma})\to\BZ_p$ be the pullback with respect to $p:\Spf\BZ_p\to\Sigma$. The composite map $\BZ_p\to H^0(\Sigma , \cO_{\Sigma})\overset{\varphi}\longrightarrow\BZ_p$ equals $\id_{\BZ_p}$.
On the other hand, $\Ker\varphi=0$ by Proposition~\ref{p:density}.
\end{proof}

Another proof of the equality $H^0(\Sigma , \cO_{\Sigma})=\BZ_p$ is given in \S\ref{sss:functions on Sigma}.

To prove Proposition~\ref{p:density}, we need two lemmas.

\begin{lem}   \label{l:[p^2]/p}
There exists $a\in W(\BZ_p)$ such that $pa=[p^2]$.
\end{lem}

\begin{proof}
The element $[p^2]$ has ghost components $p^{2p^n}$, $n=0,1,2,\ldots$. So we have to prove the existence of
an element of $W(\BZ_p)$ with ghost components $p^{2p^n-1}$. By By Dwork's lemma\footnote{E.g., see \S 4.6 on p.~213 of \cite[Ch.~VII]{Laz}.},
it suffices to show that
$p^{2p^n-1}-p^{2p^{n-1}-1}$ is divisible by $p^n$ for each $n\ge 1$. Indeed, it is easy to check that $2p^{n-1}-1\ge n$ for $n\ge 1$.
\end{proof}

\begin{lem}  \label{l:orbit of p}
Let $R$ be a $p$-nilpotent ring. Suppose that the image of $z\in W(R)$ in $W(R/p^2R)$ equals $p$. Then $z=pu$ for some $u\in W(R)^\times$.
\end{lem}

\begin{proof}
We have $z=p+\sum\limits_{i=0}^\infty V^i[p^2 b_i]$ for some $b_i\in R$. So $z=pu$, where $u=1+\sum\limits_{i=0}^\infty V^i(a[b_i])$ and $a\in W(\BZ_p)$ is as in Lemma~\ref{l:[p^2]/p}. Invertibility of $u$ follows from Lemma~\ref{l:invertible in W}(ii) because $a-p\in V (W(\BZ_p))$.
\end{proof}

\subsubsection{Proof of Proposition~\ref{p:density}}
Since $\Sigma=W_{\prim}/W^\times$, it suffices to prove that if $Y\subset W_{\prim}$ is a $W^\times$-stable closed formal subscheme such that $p:\Spf\BZ_p\to W_{\prim}$ factors through $Y$ then $Y=W_{\prim}$. 
Let us show that if $f$ is a regular function on $W_{\prim}$ which vanishes on $Y$ then $f=0$.
Lemma~\ref{l:orbit of p} tells us that  for any $p$-nilpotent ring $R$ and any element $z\in W(R)$ whose image in $W(R/p^2R)$ equals $p$ the pullback of $f$ via $z:\Spec R\to W$ is zero. This easily implies that the Taylor series of $f$ at 
the point $p\in\Mor (\Spf\BZ_p ,W)$ is zero. So $f=0$. \qed

\subsection{The Picard group of $\Sigma$}   \label{ss:Pic Sigma}
Let $\Pic'\Sigma$ be the Picard groupoid of pairs consisting of a line bundle on $\Sigma$ and a trivialization of its pullback via $p:\Spf\BZ_p\to\Sigma$.

\begin{lem}   \label{l:Pic=Pic'}
The objects of $\Pic'\Sigma$ have no nontrivial automorphisms (so the Picard groupoid $\Pic'\Sigma$ is a group).
Moreover, the canonical map $\Pic'\Sigma\to\Pic\Sigma$ is an isomorphism of groups.
\end{lem}

\begin{proof}
Follows from Corollary~\ref{c:density}.
\end{proof}

The morphism $F:\Sigma\to\Sigma$ induces a homomorphism $F^*:\Pic\Sigma\to\Pic\Sigma$.

\begin{prop} \label{p:invertibility of 1-F^*}
The endomorphism $1-F^*\in\End (\Pic\Sigma )$ is invertible.
\end{prop}

\begin{proof}
By Lemma~\ref{l:Pic=Pic'}, we can work with $\Pic'\Sigma$ instead of $\Pic\Sigma$. Let us construct the inverse endomorphism $(1-F^*)^{-1}\in\End (\Pic'\Sigma )$.
Let $\sM$ be a line bundle on $\Sigma$ trivialized at $p:\Spf\BZ_p\to~\Sigma$. Then 
\begin{equation} \label{e:infinite otimes}
(1-F^*)^{-1}(\sM )=\bigotimes\limits_{n=0}^\infty (F^n)^*\sM ,
\end{equation}
where the r.h.s is understood as follows: for every quasi-compact quasi-separated  $p$-nilpotent scheme $S$ equipped with a morphism $f:S\to\Sigma$, the $f$-pullback of the r.h.s. of \eqref{e:infinite otimes} is the line bundle
\begin{equation} \label{e:2infinite otimes}
\bigotimes\limits_{n=0}^\infty f^*(F^n)^*\sM=\bigotimes\limits_{n=0}^\infty (F^n\circ f)^*\sM .
\end{equation}

It remains to explain the meaning of the infinite tensor product \eqref{e:2infinite otimes}. 
By Proposition~\ref{p:contracting}, the functor $F:\Sigma (S)\to\Sigma (S)$ is contracting; more precisely, it contracts $\Sigma (S)$ to $p\in\Sigma (S)$ (see Corollary~\ref{c:Sigma^F}). Since $\sM$ is trivialized at $p\in\Sigma (S)$, we see that the sequence of line bundles $(F^n\circ f)^*\sM$ is equipped with an 
``asymptotic trivialization", i.e., an element of the inductive limit of the sets $\Isom ((F^n\circ f)^*\sM ,\cO_S)$. Because of this, the infinite tensor product \eqref{e:2infinite otimes} makes sense (just as in the theory of automorphic representations).
\end{proof}

\begin{rem}   \label{r:completeness of Pic Sigma}
It is not hard to show that $\Pic\Sigma$ viewed as a module over the polynomial ring $\BZ [F^*]$ is $m$-adically complete and separated, where $m\subset \BZ [F^*]$ is the ideal generated by $F^*$ and $p$. So $\Pic\Sigma$ is a module over $\BZ_p[[F^*]]$. 
\end{rem}

\begin{quest}    \label{q: Pic Sigma}
Is it true that the $\BZ_p[[F^*]]$-module $\Pic\Sigma$ is freely generated by the class of the divisor $\Delta_0$ from \S\ref{sss:Hodge-Tate locus}?
\end{quest}

\subsection{The Breuil-Kisin-Tate module $\cO_{\Sigma}\{1\}$}   \label{ss:BK}
This subsection is based on what I learned from private communications with Bhatt and Scholze.
Following them, we introduce an invertible $\cO_\Sigma$-module denoted by $\cO_{\Sigma}\{1\}$. According to Bhatt and Scholze, it plays the role of the Tate module in the theory of prismatic cohomology. A justification of this claim can be found (at least between the lines) in \cite[\S 16]{BS}  and to some extent, in \S 6.2 of the earlier article \cite{BMS}. For a more detailed justification, see \cite{BL} (especially, \cite[\S 2.7.4]{BL}, \cite[\S 7]{BL}, and  \cite[Thm.~8.3.1]{BL}).

\subsubsection{The line bundle $\sL_{\Sigma}$}     \label{sss:sL_Sigma}

Let $S$ be a scheme equipped with a morphism $S\to\Sigma$. This morphism is given by a pair $(M,\xi )$, where $M$ is an invertible $W_S$-module and $\xi :M\to W_S$ is a primitive morphism. Let $\sL_S=M\otimes_{W_S}(\BG_a)_S$ (the homomorphism $W\epi\BG_a$ takes a Witt vector to its zeroth component). Then $\sL_S$ is a line bundle on $S$ equipped with a morphism $v:\sL_S\to\cO_S$ (both $\sL_S$ and $v$ already appeared in \S\ref{sss:Sigma to A^1/G_m}). As $S$ varies, we get an invertible $\cO_\Sigma$-module $\sL_{\Sigma}$ equipped with a morphism $\sL_\Sigma\to\cO_\Sigma$. One has $\Coker (\sL_\Sigma\to\cO_\Sigma )=\cO_{\Delta_0}$, where $\Delta_0$ is the effective divisor from \S\ref{sss:Hodge-Tate locus}. So
\[
\sL_\Sigma =\cO_\Sigma (-\Delta_0 ).
\]

\subsubsection{Remark}     \label{sss:sL_Sigma in terms of W^times}  
In terms of the presentation $\Sigma =W_{\prim}/W^\times$, the line bundle $\sL_{\Sigma}$ corresponds to the homomorphism $W^{\times}\to W_1^{\times}=\BG_m$
(here we use that $W^\times$ acts on $W_{\prim}$ by division, see \S\ref{sss:Action of W^* on W_prim}).

\subsubsection{Trivialization of $\sL_{\Sigma}$ at $p:\Spf\BZ_p\to\Sigma$} \label{sss:trivializing sL_Sigma}
The morphism $p:\Spf\BZ_p\to\Sigma$ corresponds to the following pair $(M,\xi )$: the module $M$ is $W_S$ (where $S=\Spf\BZ_p$) and $\xi :M\to W_S$ is multiplication by $p$. Since $M=W_S$, we get a trivialization of the pullback of $\sL_\Sigma$ via $p:\Spf\BZ_p\to\Sigma$. 

The same trivialization can be described as follows. The canonical morphism $\sL_\Sigma\to\cO_\Sigma$ identifies the pullback of $\sL_\Sigma$ via $p:\Spf\BZ_p\to\Sigma$ with the submodule $p\BZ_p\subset\BZ_p$, and we choose $p$ as a generator of $p\BZ_p$.

Thus $\sL_{\Sigma}\in\Pic'\Sigma$, where $\Pic'\Sigma$ is as in \S\ref{ss:Pic Sigma}.

\begin{defin}    \label{d:BK}
$\cO_{\Sigma}\{1\}\in\Pic'\Sigma$ is the image of $\sL_{\Sigma}$ under $(1-F^*)^{-1}:\Pic'\Sigma\iso \Pic'\Sigma$.
The line bundle $\cO_{\Sigma}\{1\}$ is called the \emph{Breuil-Kisin-Tate} module (or \emph{BK-Tate module}).
\end{defin}

Let us note that the name for $\cO_{\Sigma}\{1\}$ used in \cite{BL} is ``Breuil-Kisin twist'' or ``Breuil-Kisin line bundle''; in \cite{BL} it is denoted by $\cO_{\WCart}\{1\}$.

\medskip

Since $\sL_{\Sigma}=\cO_\Sigma (-\Delta_0)$, one has a canonical isomorphism
\begin{equation}   \label{e:shtuka structure}
F^*\cO_{\Sigma}\{1\}\iso \cO_{\Sigma}\{1\}(\Delta_0).
\end{equation}

By \eqref{e:infinite otimes}, one has 
\begin{equation}  \label{e:3infinite otimes}
\cO_{\Sigma}\{1\}=\bigotimes\limits_{n=0}^\infty (F^n)^*\sL_\Sigma\, .
\end{equation}

\subsubsection{The $F$-crystals $\cO_{\Sigma}\{n\}$, $n\le 0$}    \label{sss:powers of BK}
For $n\in\BZ$ set
\[
\cO_{\Sigma}\{n\}:=(\cO_{\Sigma}\{1\})^{\otimes n}.
\]
By \eqref{e:shtuka structure}, we have a canonical isomorphism
\[
F^*\cO_{\Sigma}\{n\}\iso \cO_{\Sigma}\{n\}(n\Delta_0).
\]
So for $n\le 0$ we get a canonical morphism $F^*\cO_{\Sigma}\{n\}\to\cO_{\Sigma}\{n\}$. If you wish, this is a structure of $F$-crystal on 
$\cO_{\Sigma}\{n\}$ for $n\le 0$. The restriction $n\le 0$ is not surprising: it appears already in the theory of crystalline cohomology (in this theory the role of
$\Sigma$ is played by $\Spf\BZ_p$ and the role of the $F$-crystal $\cO_{\Sigma}\{n\}$ is played by $\BZ_p$ equipped with the endomorphism of multiplication by $p^{-n}$).

Let us note that the pullback of $\cO_{\Sigma}\{1\}$ to the $q$-de Rham prism has a very simple description, see \S\ref{ss:BK-Tate on Q} of Appendix~\ref{s:q-de Rham}.

\subsubsection{The Bhatt-Lurie isomorphism}    \label{sss:BL}
Let $\cO_{\Sigma\otimes\BF_p}\{n\}$ be the restriction of $\cO_{\Sigma}\{n\}$ to $\Sigma\otimes\BF_p$. Restricting the isomorphism \eqref{e:shtuka structure} to $\Sigma\otimes\BF_p$, we get an isomorphism
\[
\cO_{\Sigma\otimes\BF_p}\{p\}\iso \cO_{\Sigma\otimes\BF_p}\{1\}(\Delta_0)
\]
because the restriction of $F:\Sigma\to\Sigma$ to $\Sigma\otimes\BF_p$ is the usual Frobenius. So we get a canonical isomorphism
\begin{equation}   \label{e:BL}
\BL :\cO_{\Sigma\otimes\BF_p}\{1-p\}\iso\cO_{\Sigma\otimes\BF_p}(-\Delta_0)\iso\sL_{\Sigma\otimes\BF_p},
\end{equation}
where $\sL_{\Sigma\otimes\BF_p}$ is the restriction of $\sL_{\Sigma}$ to $\Sigma\otimes\BF_p$. We call \eqref{e:BL} the \emph{Bhatt-Lurie isomorphism.}

\begin{lem} \label{l:restricting F^*sL_Sigma to Delta_0}
(i) If $n>0$ then the restriction of $(F^n)^*\sL_{\Sigma}$ to $\Delta_0$ is canonically trivial.

(ii) The restrictions of $\cO_{\Sigma}\{1\}$ and $\sL_{\Sigma}$ to $\Delta_0$ are canonically isomorphic.
\end{lem}

\begin{proof}
By Lemma~\ref{l:F restricted to Delta_0}, the trivialization of the pullback of $\sL_{\Sigma}$ via $p:\Spf\BZ_p\to\Sigma$ induces a trivialization of the restriction of
$(F^n)^*\sL_{\Sigma}$ to $\Delta_0$. This proves statement (i). It implies (ii) by formula~\eqref{e:3infinite otimes}.
\end{proof}

\begin{lem} \label{l:restricting BK to Sigma_red}
Let $\sL_{\Sigma_{\red}}$ and $\cO_{\Sigma_{\red}}\{1\}$ be the restrictions of $\sL_{\Sigma}$ and $\cO_{\Sigma}\{1\}$ to the substack $\Sigma_{\red}=\Delta_0\otimes\BF_p$. Then

(i) $\cO_{\Sigma_{\red}}\{1\}=\sL_{\Sigma_{\red}}$;

(ii) $\sL_{\Sigma_{\red}}^{\otimes p}$ is canonically isomorphic to $\cO_{\Sigma_{\red}}$;

(iii) In terms of the description of $\Sigma_{\red}=\Delta_0\otimes\BF_p$ as the classifying stack of $(W^\times)^{(F)}\otimes\BF_p$ (see Lemma~\ref{l:Delta_0 as classifying stack}), $\sL_{\Sigma_{\red}}$ corresponds to the following  module over $(W^\times)^{(F)}\otimes\BF_p$: as a vector space, it equals $\BF_p$, and the group action is
given by the  homomorphism $(W^\times)^{(F)}\otimes\BF_p\to\mu_p$ that takes a Witt vector to its $0$th component.
\end{lem}

\begin{proof}
$\Sigma_{\red}\subset\Sigma\otimes\BF_p$, so $\sL_{\Sigma_{\red}}^{\otimes p}=\Fr^*\sL_{\Sigma_{\red}}$  and $\Fr^*\sL_{\Sigma_{\red}}$  is the restriction of
$F^*\sL_{\Sigma}$ to $\Sigma_{\red}$. 
So statements (i)-(ii) follow from Lemma~\ref{l:restricting F^*sL_Sigma to Delta_0}. 
Statement (iii) follows from \S\ref{sss:sL_Sigma in terms of W^times}.  
\end{proof}

\section{The c-stack $\Sigma'$}   \label{s:Sigma'}
The main goal of this section is to define the c-stack $\Sigma'$ and the basic pieces of structure on it. 
Explicit presentations of $\Sigma'$ as a quotient stack will be introduced in  \S\ref{s:2Sigma' as quotient} and \S\ref{s:Sigma' as quotient}.

\subsection{Introduction}  \label{ss:intro to Sigma'}
\subsubsection{A possible order of reading}
The reader who likes pictures or formulas  could start reading \S\ref{ss:toy model} and/or \S\S\ref{s:2Sigma' as quotient}-\ref{s:Sigma' as quotient} shortly after \S\ref{ss:Sigma'&ring stack}.

\subsubsection{Organization of  \S\ref{s:Sigma'}}
In \S\ref{ss:Sigma'&ring stack} we define the c-stack $\Sigma'$ and the ring stacks 
$\sR_{\Sigma'}$ and $\sR_{\Sigma}$ mentioned in the introduction to the article.

In \S\ref{ss:Sigma_+}-\ref{ss:piece of structure} we introduce the basic pieces of structure\footnote{Some of these pieces of structure will be later described in terms of a concrete realization of $\Sigma'$ as a quotient stack, see \S\ref{ss:Who is who in terms of sZ}.} on $\Sigma'$ (substacks, morphisms, line bundles, etc.). 
In \S\ref{ss:Delta'_0 as c-stack} we explicitly describe the Hodge-Tate locus $\Delta'_0\subset\Sigma'$ as a c-stack.

In \S\ref{ss:Hdg & bar dR} we discuss certain c-stacks $\Sigma'_{\Hdg}$ and $\Sigma'_{\bardR}$ over $\Sigma'$. Presumably, they encode
the ``Hodge to de Rham'' spectral sequence.

In \S\ref{ss:2Sigma'_+}-\ref{ss:Sigma'_red} we define a c-stack $\Sigma'_+$ faithfully flat over $\Sigma'$ and use it to prove some statements about $\Sigma'$.
The reader may prefer to replace arguments based on $\Sigma'_+$ by those based on an explicit presentation of $\Sigma'$ as a quotient stack (two such presentations are given in \S\ref{s:2Sigma' as quotient} and \S\ref{s:Sigma' as quotient}).

In \S\ref{ss:toy model} (which is independent from the most part of \S\ref{s:Sigma'}) we describe a ``toy model'' for~$\Sigma'$, which is essentially the quotient of the coordinate cross in $\BA^2_{\BF_p}$ by the hyperbolic action of~$\BG_m$.

\subsubsection{Comparison with the notation of \cite{Bh}}
In \cite[\S 5.3]{Bh} Bhatt considers the g-stack underlying $\Sigma'$ and denotes it by $\BZ_p^\cN$ (where $\cN$ stands for ``Nygaard''). 
His notation for the morphism $j_+:\Sigma\mono\Sigma'$ from our \S\ref{ss:Sigma_+} is $j_{\HT}:\BZ_p^\prism\mono\BZ_p^\cN$.
His notation for the morphisms $j_-:\Sigma\mono\Sigma'$ and $F':\Sigma'\to\Sigma$ from our Lemma~\ref{l:F' & j_-} is $j_{\dR}:\BZ_p^\prism\mono\BZ_p^\cN$ and
$\pi:\BZ_p^\cN\to\BZ_p^\prism$. The underlying g-stack of the substack $\Delta'_0\subset\Sigma'$ defined in our \S\ref{sss:Delta'_0} is denoted in \cite{Bh} by $(\BZ_p^\cN )_{t=0}\,$.
The map $i_{\dR}:\BA^1/\BG_m\to\BZ_p^\cN$ from \cite[\S 5.3]{Bh} corresponds to the map $\fp: (\BA^1/\BG_m)_-\hat\otimes\BZ_p\to \Sigma'$ from our Lemma~\ref{l:Sigma'_bardR}.

\subsection{Definition of $\Sigma'$ and $\sR_{\Sigma'}$}  \label{ss:Sigma'&ring stack}
\begin{defin}  \label{d:Sigma'}
If $S$ is a $p$-nilpotent scheme, let $\Sigma'(S)$ be the following category: its objects are pairs $(M,\xi)$, where $M$ is an admissible $W_S$-module (see Definition~\ref{d:admissible}) and $\xi :M\to W_S$ is a primitive  $W_S$-morphism (see Definition~\ref{d:primitivity of xi}); a morphism from $(M_1,\xi_1)$ to $(M_2,\xi_2)$ is defined to be a $W_S$-module homomorphism $f:M_1\to M_2$ such that $\xi_2\circ f=\xi_1$. If $S$ is not $p$-nilpotent then $\Sigma'(S):=\emptyset$.
\end{defin}

By Remark~\ref{sss:admissibility rems}(i), $\Sigma'$ is a c-stack on the category of schemes equipped with the fpqc topology. Later we will see that $\Sigma'$ is a very nice formal c-stack (see Theorem~\ref{t:Sigma' algebraic}(iii), Corollary~\ref{c:strongly adic} and especially \S\ref{ss:Sigma' as lim}).

\begin{lem}    \label{l:the 3 points of Sigma'}
If $S$ is the spectrum of a perfect field of characteristic $p$ then $\Sigma'(S)$ has exactly $3$ isomorphism classes. Here are their representatives:
\[
M=W_S, \xi=p; \quad M=W_S^{(F)}\oplus W_S^{(1)}, \xi =(0,V); \quad M=W_S^{(F)}\oplus W_S^{(1)}, \xi =(1,V).
\]
\end{lem}

\begin{proof}
Use Lemma~\ref{l:Adm of a field}.
\end{proof}

\begin{cor}
The set $|\Sigma'|$ (see \S\ref{sss:|X|}) has exactly 3 elements.  \qed
\end{cor}

For a description of $|\Sigma'|$ as a topological space, see Corollary~\ref{c:topology on |Sigma'|}.

\subsubsection{The ring stacks $\sR_{\Sigma'}$ and $\sR_{\Sigma}$}   \label{sss:R_Sigma'}
By Lemma~\ref{l:quasi-ideal automatically}, if $M$ is an admissible $W_S$-module then any morphism $\xi: M\to W_S$ makes $M$ into a quasi-ideal in $W_S$ in the sense of \S\ref{sss:2quasi-ideals} (if $M$ is invertible this is clear even without Lemma~\ref{l:quasi-ideal automatically}). So we can form  the ring stack $\Cone (\xi )$ over $S$, see
\S\ref{sss:Cones-group schemes}-\ref{sss:Cones-rings}. Letting $(M,\xi )$ vary, we get a ring stack $\sR_{\Sigma'}$ over $\Sigma'$ (in the sense of \S\ref{sss:group scheme over c-stack}) and a ring stack $\sR_{\Sigma}$ over $\Sigma$.

\subsection{The open embedding $j_+:\Sigma\iso\Sigma_+\subset\Sigma'$}  \label{ss:Sigma_+}
Recall that invertible $W_S$-modules are admissible. So by \S\ref{sss:Sigma(S)}, we have a functor $j_+:\Sigma (S)\to\Sigma'(S)$. By Lemma~\ref{l:g-stack}, this functor is fully faithful. Its essential image is denoted by $\Sigma_+(S)$. (Strictly speaking, there is no difference between $\Sigma$ and $\Sigma_+$; we distinguish them for psychological reasons.)

\begin{lem}   \label{l:Sigma_+ open affine}
$\Sigma_+\subset\Sigma'$ is an open substack, which is affine over $\Sigma'$.
\end{lem}

\begin{proof}
Follows from Lemma~\ref{l:Inv for p-nilpotent}(ii). 
\end{proof}

Later we will show that $\Sigma_+$ is \emph{right-fibered over $\Sigma'$}, see Corollary~\ref{c:Sigma_+ right-admissible}.

\subsection{The morphism $F':\Sigma'\to\Sigma$}   \label{ss:F'}
Let $S$ be a $p$-nilpotent scheme, $M$ an admissible $W_S$-module, and $\xi :M\to W_S$ a $W_S$-linear morphism. Let $M'$ be the invertible $W_S^{(1)}$-module from Definition~\ref{d:admissible}. The second part of Lemma~\ref{l:uniqueness of M_0} implies that $\xi$ induces a $W_S^{(1)}$-module morphism 
$\xi':M'\to W_S^{(1)}$. We say that $\xi'$ is \emph{primitive} if the corresponding morphism $(\xi')^{(-1)}:(M')^{(-1)}\to W_S$ is. 

\begin{lem}  \label{l:primitivity of xi'}
$\xi$ is primitive if and only if  $\xi'$ is.   
\end{lem}

\begin{proof}
Since $S$ is $p$-nilpotent, the morphism $F:W_S\to W_S^{(1)}=W_S$ is bijective; moreover, for every $s\in S$ and $n\in\BN$ it induces a bijective morphism $(W_n)_s\to (W_n)_s$. So the lemma follows directly from Definition~\ref{d:primitivity of xi}.
\end{proof}

By Lemma~\ref{l:primitivity of xi'}, 
we get a morphism
\begin{equation}  \label{e:F' defined}
F':\Sigma'\to\Sigma ,\quad  F'(M,\xi):=((M')^{(-1)},(\xi')^{(-1)}).
\end{equation}
By \eqref{e:F(M,xi)}, we have
\begin{equation}   \label{e:F'j_+}
F'\circ j_+=F.
\end{equation}
\begin{lem}  \label{l:f' invertible}
For any morphism $f:(M_1,\xi_1)\to (M_2,\xi_2)$ in $\Sigma'(S)$, the corresponding morphism $f':M_1'\to M_2'$ is an isomorphism.
\end{lem}

\begin{proof}
Apply Lemma~\ref{l:g-stack} to 
$(f')^{(-1)}:((M'_1)^{(-1)},(\xi'_1)^{(-1)})\to ((M'_2)^{(-1)},(\xi'_2)^{(-1)})$.
 \end{proof}

\begin{cor}  \label{c:Sigma' to Adm}
The assignment $(M,\xi)\mapsto M$ defines a morphism 
\begin{equation}    \label{e:Sigma' to Adm}
\Sigma'\to\Adm ,
\end{equation}
where $\Adm$ is the c-stack from \S\ref{sss:Adm}. 
\end{cor}

\begin{proof}
Follows from Lemma~\ref{l:f' invertible} (which is necessary because of the definition of morphisms in $\Adm (S)$ given in \S\ref{sss:Adm}). 
\end{proof}

\begin{cor}  \label{c:Sigma_+ right-admissible}
$\Sigma_+$ is right-fibered over $\Sigma'$ in the sense of \S\ref{ss:left&right fibrations}.  
\end{cor}

\begin{proof}
Follows from Corollary~\ref{c:Sigma' to Adm} and Lemma~\ref{l:Inv for p-nilpotent}(ii).
\end{proof}

\begin{thm}   \label{t:Sigma' algebraic}
(i) The morphism $F':\Sigma'\to\Sigma$ is algebraic.

(ii) $\Sigma'$ is $\MMor$-affine in the sense of Definition~\ref{d:pre-algebraic stack}.

(iii) $\Sigma'$ is a formal c-stack in the sense of~\S\ref{sss:def formal scheme}.

\end{thm}

Later we will show that the formal c-stack $\Sigma'$ is very nice (see Corollary~\ref{c:strongly adic} and especially \S\ref{ss:Sigma' as lim}).

\begin{proof}
Statement (ii) follows from  Lemma~\ref{l:HHom of admissible}.

To prove statement (i), it suffices to construct a Cartesian square
\begin{equation}   \label{e:Sigma and fS}
\xymatrix{
\Sigma'\ar[r]^{F'} \ar[d] &\Sigma\ar[d]\\
\sX'\ar[r] & \sX
}
\end{equation}
where $\sX$ and $\sX'$ are algebraic c-stacks.

Let $\sX:=W/W^\times$; this is an algebraic g-stack. The morphism $\Sigma=W_{\prim}/W^\times\to\sX$ is clear; it identifies $\Sigma$ with
the formal completion of $\sX$ along a locally closed substack.

Let $\Adm$ be the c-stack from \S\ref{sss:Adm}. For any scheme $S$, let $\sX'(S)$ be the category of pairs $(M,\xi)$, where $M\in\Adm (S)$, $\xi\in\Hom_W(M,W_S)$. The morphism $\sX'\to\Adm$ is affine by Lemma~\ref{l:dual of admissible}(i). The c-stack $\Adm$ is algebraic by Proposition~\ref{p:algebraically of Adm}. Moreover, the c-stack $\sX'$ is $\MMor$-affine by Lemma~\ref{l:HHom of admissible}.
So the c-stack $\sX'$ is algebraic by Lemma~\ref{l:what IS true}.

The morphism $\sX'\to \sX$ takes $(M,\xi)\in\sX' (S)$ to $((M')^{(-1)},(\xi')^{(-1)})\in\sX(S)$.   

The canonical morphism $\Sigma'\to\sX'$ is clear from Corollary~\ref{c:Sigma' to Adm}. 

Diagram \eqref{e:Sigma and fS} clearly commutes. 
Moreover, it is Cartesian by Lemma~\ref{l:primitivity of xi'}. 

Thus we have proved (i).

The Cartesian square \eqref{e:Sigma and fS} shows that 
$\Sigma'$ is the formal completion of the algebraic c-stack $\sX'$ along a locally closed substack. This proves (iii).
\end{proof}

Composing $F':\Sigma'\to\Sigma$ with the morphism \eqref{e:Sigma to A^1/G_m}, we get a morphism
\begin{equation}   \label{e:Sigma' to A^1/G_m}
\Sigma'\to\hat\BA^1/\BG_m ,
\end{equation}
where $\hat\BA^1$ is the formal completion of $\BA^1$ along $0\in \BA^1\otimes\BF_p$.
Combining Theorem~\ref{t:Sigma' algebraic}(i,ii) with Lemma \ref{l:Sigma to A^1/G_m}, we get

\begin{cor}     \label{c:Sigma' algebraic}
The morphism \eqref{e:Sigma' to A^1/G_m} is algebraic. 
\end{cor}

\subsection{The left fibration $\Sigma'\to (\BA^1/\BG_m)_-\hat\otimes\BZ_p$}  \label{ss:theleftfibration}
\subsubsection{The weakly invertible covariant $\cO_{\Sigma'}$-module $\sL_{\Sigma'}$}   \label{sss:L_Sigma'}

Let $(M,\xi )\in\Sigma'(S)$. Then by \S\ref{sss:admissibility rems}(ii), we get a canonical commutative diagram

\begin{equation}    \label{e:diagram corresponding to xi}
\xymatrix{
0\ar[r]&\sL^\sharp\ar[r] \ar[d]^{v_-} & M\ar[r]\ar[d]^\xi&M'\ar[r]\ar[d]^{\xi'} &0 \\
0\ar[r]&(\BG_a^\sharp)_S\ar[r] & W_S\ar[r]^F & W_S^{(1)}\ar[r] &0
}
\end{equation}
where $\sL$ is a line bundle on $S$. By Lemma~\ref{l:2 classes of W-modules}(ii), the morphism 
$v_-:\sL^\sharp\to (\BG_a^\sharp)_S$ comes from a unique morphism of line bundles 
$v_-:\sL\to\cO_S$. Letting $S,M,\xi$ vary, we get a weakly invertible covariant $\cO_{\Sigma'}$-module $\sL_{\Sigma'}$ in the sense of \S\ref{sss:Weakly invertible}.
It is equipped with a morphism $v_-:\sL_{\Sigma'}\to\cO_{\Sigma'}$. On the other hand, we have the invertible $\cO_{\Sigma}$-module $\sL_{\Sigma}$ and the morphism $v:\sL_{\Sigma}\to\cO_{\Sigma}$ (see \S\ref{sss:Sigma to A^1/G_m} and \S\ref{sss:sL_Sigma}).

\begin{lem}    \label{l:j_+^*(sL,v_-)}
The pullback of $(\sL_{\Sigma'},v_-)$ via $j_+:\Sigma\mono\Sigma'$ is the pair $(\sL_{\Sigma},v)$ from \S\ref{sss:sL_Sigma}.
\end{lem}

\begin{proof}
Follows from the isomorphism \eqref{e:L_M for invertible M}.
\end{proof}

\subsubsection{The left fibration $\Sigma'\to (\BA^1/\BG_m)_-\hat\otimes\BZ_p$}   \label{sss:theleftfibration}
In \S\ref{sss:BA^1/BG_m)_pm} we defined a c-stack $(\BA^1/\BG_m)_-$.
The pair $(\sL_{\Sigma'} ,v_-:\sL_{\Sigma'}\to\cO_{\Sigma'})$ defines a morphism $\Sigma'\to (\BA^1/\BG_m)_-$.
Moreover, since $\Sigma'$ is over $\Spf\BZ_p$, we get a morphism
\begin{equation}    \label{e:2left fibration}
\Sigma'\to (\BA^1/\BG_m)_-\hat\otimes\BZ_p ,
\end{equation}
where $(\BA^1/\BG_m)_-\hat\otimes\BZ_p:=(\BA^1/\BG_m)_-\times\Spf\BZ_p$ is the $p$-adic completion of $(\BA^1/\BG_m)_-$ in the sense of \S\ref{sss:p-adic completion}.

The morphism \eqref{e:2left fibration} is a left fibration in the sense of \S\ref{sss:left fibrations of c-stacks} (this easily follows from Lemma~\ref{l:f' invertible}).

By Lemma~\ref{l:j_+^*(sL,v_-)} one has a commutative diagram 
\begin{equation}   \label{e:restricting left fibration to Sigma}
\xymatrix{
\Sigma\ar@{^{(}->}[r]^{j_+} \ar[d] & \Sigma'\ar[d]\\
\hat\BA^1/\BG_m\ar@{^{(}->}[r] & (\BA^1/\BG_m)_-\hat\otimes\BZ_p
}
\end{equation}
whose vertical arrows are the morphisms \eqref{e:Sigma to A^1/G_m} and \eqref{e:2left fibration}.

\subsection{The open substack $\Sigma_-\subset\Sigma'$}
\subsubsection{Definition of $\Sigma_-$}  \label{sss:Sigma_-}
Let $\Sigma_-(S)$ be the category of pairs $(M,\xi)\in\Sigma'(S)$ such that the corresponding map $v_-:\sL\to\cO_S$ is an isomorphism. In other words, $\Sigma_-$ is the preimage of the open substack $(\BG_m/\BG_m)\hat\otimes\BZ_p\subset(\BA^1/\BG_m)_-\hat\otimes\BZ_p$ with respect to the left fibration \eqref{e:2left fibration}. The substack $\Sigma_-\subset\Sigma'$ is clearly open and affine over $\Sigma'$. It is also clear that the morphisms 
$\Sigma_-\to\Sigma'$ and $\Sigma_-\to(\BA^1/\BG_m)_-\hat\otimes\BZ_p$ are left fibrations.

Let us note that the morphisms $\Sigma_+\to\Sigma'$ and $\Sigma_+\to(\BA^1/\BG_m)_-\hat\otimes\BZ_p$ are \emph{not} left 
fibrations.\footnote{Moreover, Corollary~\ref{c:recieves morphism from Sigma_+} below says that every object of $\Sigma'(S)$ fpqc-locally on $S$ receives a morphism from some object of $\Sigma_+(S)$.}

\begin{lem}   \label{l:F' & j_-}
The morphism $F':\Sigma'\to\Sigma$ defined by formula  \eqref{e:F' defined} induces an isomorphism $\Sigma_-\iso\Sigma$. The inverse isomorphism $j_-:\Sigma\iso\Sigma_-$ takes $(M,\xi)\in\Sigma (S)$ to $(\tilde M,\tilde\xi)$, where $\tilde M=M^{(1)}\times_{W_S^{(1)}}W_S$ and $\tilde\xi:\tilde M\to W_S$ is the projection (here we are using the morphism $F:W_S\to W_S^{(1)}$)
\end{lem}

\begin{proof}
If the left vertical arrow of a diagram \eqref{e:diagram corresponding to xi} is an isomorphism then the upper row of \eqref{e:diagram corresponding to xi} is the pullback of the lower one via $\xi':M'\to W_S^{(1)}$.
\end{proof}

Thus we have open immersions $j_\pm :\Sigma\iso\Sigma_\pm\mono\Sigma'$ such that 
\[
F'\circ j_+=F, \quad F'\circ j_-=\id_\Sigma .
\]

\begin{lem}  \label{l: Sigma_+ cap Sigma_-}
$\Sigma_+\cap\Sigma_-=\emptyset$.
\end{lem}

\begin{proof}
Follows from \eqref{e:restricting left fibration to Sigma} because $\hat\BA^1/\BG_m$ does not meet $(\BG_m/\BG_m )\hat\otimes\BZ_p$.
\end{proof}

\subsection{The morphism $F'_-:\Sigma'\to\Sigma_-$}
Define $F'_\pm :\Sigma'\to\Sigma_\pm$ by $F'_\pm:=j_\pm\circ F'$. 

\subsubsection{Explicit description of $F'_-$} \label{sss:F'_ explicitly}
Combining the definitions of $j_-$ and $F'$ (see Lemma~\ref{l:F' & j_-} and formula \eqref{e:F' defined}), we see that
$F'_-$ takes $(M,\xi )\in\Sigma' (S)$ to $(\tilde M,\tilde \xi )\in\Sigma_- (S)$, where $\tilde M$ is the extension of $M'$ by $(\BG_a^\sharp)_S$ obtained as the pushforward of the upper row of \eqref{e:diagram corresponding to xi} via $v_-:\sL^\sharp\to (\BG_a^\sharp)_S=W_S^{(F)}$ and $\tilde\xi:\tilde M\to W_S$ is induced by $\xi :M\to W_S$. In \S\ref{sss:what F'_- is} we will reformulate this more abstractly in terms of the left fibration $\Sigma'\to (\BA^1/\BG_m)_-\hat\otimes\BZ_p$.
In \S\ref{sss:id to F'_-}  we will interpret the canonical morphism
\begin{equation}   \label{e:M to tilde M}
(M,\xi )\to (\tilde M,\tilde \xi )
\end{equation}
in terms of an adjunction between $j_-$ and $F'$.

\subsubsection{An abstract categorical remark}  \label{sss:abstract remark}
Let $\cD$ be a category with a final object $d_0$. Let $\Phi:\cC\to\cD$ be a left fibration. Let $\cC_{d_0}\subset\cC$ be the fiber over $d_0$ (i.e., the category of objects $c\in\cC$ equipped with an isomorphism $\Phi (c)\iso d_0$). Then for every $c\in\cC$ there is an essentially unique morphism $c\to c'$ such that $c'\in\cC_{d_0}\subset\cC$. The assignment $c\mapsto c'$ is a functor; denote it by $\Psi :\cC\to\cC_{d_0}$. This functor is left adjoint to the inclusion $\cC_{d_0}\mono\cC$.

\subsubsection{What $F'_-$ is}  \label{sss:what F'_- is}
Now let $S$ be a scheme. Let $\cC=\Sigma'(S)$ and $\cD=((\BA^1/\BG_m)_-\hat\otimes\BZ_p)(S)$. Let $\Phi:~\cC\to~\cD$ be the left fibration induced by \eqref{e:2left fibration}. The category $\cD$  has a final object~$d_0$ (namely, the unique $S$-point of
$(\BG_m/\BG_m)\hat\otimes\BZ_p\subset (\BA^1/\BG_m)_-\hat\otimes\BZ_p$). Then $\cC_{d_0}=\Sigma_-(S)\subset\Sigma'(S)$, and  \S\ref{sss:F'_ explicitly} can be reformulated as follows:
$F'_-:\Sigma'(S)\to\Sigma_-(S)$ is the functor $\Psi :\cC\to\cC_{d_0}$ from~\S\ref{sss:abstract remark}.

\subsubsection{The canonical morphism $\id_{\Sigma'}\to F'_-$}  \label{sss:id to F'_-}
Recall that c-stacks form a 2-category (so endomorphisms of a c-stack form a monoidal category).

By \S\ref{sss:what F'_- is}, the 1-morphism $F':\Sigma'\to\Sigma$ is left adjoint to $j_-:\Sigma\mono\Sigma'$. The unit of the adjunction is a morphism
$\id_{\Sigma'}\to j_-\circ F'=:F'_-$. This is the morphism \eqref{e:M to tilde M}.

The counit of the adjuncton is an isomorphism, so $F'_-$ is an idempotent monad.

\subsection{The isomorphism $j_+^*\sR_{\Sigma'}\iso j_-^*\sR_{\Sigma'}$}   \label{sss:sR' descends}
In \S\ref{sss:R_Sigma'} we defined the ring stack $\sR_{\Sigma'}$ over $\Sigma'$ and the ring stack $\sR_{\Sigma}$ over $\Sigma$. Tautologically, 
$j_+^*\sR_{\Sigma'}=\sR_{\Sigma}$. 

Let us construct an isomorphism $j_-^*\sR_{\Sigma'}\iso\sR_{\Sigma}$. This is equivalent to constructing an 
isomorphism 
$$F'^*\sR_{\Sigma}\iso\sR_{\Sigma_-},$$
where $\sR_{\Sigma_-}$ is the restriction of $\sR_{\Sigma'}$ to $\Sigma_-$. Let $f:S\to\Sigma_-\subset\Sigma'$ correspond to a pair 
$(M,\xi )\in\Sigma_-(S)$. Then $f^*\sR_{\Sigma'}=\Cone (M\overset{\xi}\longrightarrow W_S)$, and 
$$f^*F'^*\sR_{\Sigma}=\Cone ((M')^{(-1)}\overset{\xi'^{(-1)}}\longrightarrow W_S)=\Cone (M'\overset{\xi'}\longrightarrow W_S^{(1)}).$$
Diagram \eqref{e:diagram corresponding to xi} yields a morphism of Picard stacks 
$$g:\Cone (M\overset{\xi}\longrightarrow W_S)\to \Cone (M'\overset{\xi'}\longrightarrow W_S^{(1)}).$$ 
By the definition of $\Sigma_-(S)$, the left vertical arrow of \eqref{e:diagram corresponding to xi} is an isomorphism. So $g$ is an isomorphism.

Thus we have constructed a canonical isomorphism 
\begin{equation}   \label{e:sR' descends}
j_+^*\sR_{\Sigma'}\iso j_-^*\sR_{\Sigma'}.
\end{equation}

\subsection{The Breuil-Kisin-Tate module $\cO_{\Sigma'}\{1\}$}        \label{ss:BK-Tate on Sigma'} 

\subsubsection{Recollections on $\sL_{\Sigma'}$ and $\sL_{\Sigma}$}  \label{sss:recollections on L}
(i) In \S\ref{sss:L_Sigma'} we defined a weakly invertible covariant $\cO_{\Sigma'}$-module $\sL_{\Sigma'}$ equipped with a morphism 
$v_-:\sL_{\Sigma'}\to\cO_{\Sigma'}$. Recall that $\Sigma_-\subset\Sigma'$ is the locus where $v_-$ is an isomorphism. So $v_-$ induces an isomorphism
$j_-^*\sL_{\Sigma'}\iso\cO_{\Sigma'}$. By Lemma~\ref{l:j_+^*(sL,v_-)}, $j_+^*\sL_{\Sigma'} =\sL_{\Sigma}$.

(ii) In \S\ref{sss:trivializing sL_Sigma} we chose a trivialization of the pullback of $\sL_\Sigma$ via $p:\Spf\BZ_p\to\Sigma$. It induces a trivialization of the pullback of 
$\sL_{\Sigma'}$ via $\Spf\BZ_p\overset{p}\longrightarrow\Sigma\overset{j_\pm}\mono\Sigma'$.

\subsubsection{The BK-Tate module $\cO_{\Sigma'}\{1\}$}        \label{sss:BK-Tate on Sigma'} 
We have the invertible $\cO_{\Sigma}$-module $\cO_{\Sigma}\{1\}$, see Definition~\ref{d:BK} and formula~\eqref{e:3infinite otimes}. We also have the morphism $F':\Sigma'\to\Sigma$.
Set 
\begin{equation}   \label{e:BK-Tate' definition}
\cO_{\Sigma'}\{1\}:=\sL_{\Sigma'}\otimes F'^*\cO_\Sigma\{1\}.
\end{equation}
This is a weakly invertible covariant $\cO_{\Sigma'}$-module. We call it the \emph{Breuil-Kisin-Tate} module on $\Sigma'$ (sometimes shortened to \emph{BK-Tate module}).

\begin{lem}    \label{l:O_Sigma'{1} descends}
One has canonical isomorphisms 
\begin{equation}   \label{e:O_Sigma'{1} descends}
j_+^*\cO_{\Sigma'}\{1\}\iso \cO_\Sigma\{1\}\iso j_-^*\cO_{\Sigma'}\{1\}.
\end{equation}
\end{lem}

\begin{proof}
Recall that $F'\circ j_+=F$ and $F'\circ j_-=\id_\Sigma$. Combining this with \S\ref{sss:recollections on L}(i), we see that
\[
j_+^*\cO_{\Sigma'}\{1\}=\sL_{\Sigma}\otimes F^*\cO_{\Sigma}\{1\} ,\quad j_-^*\cO_{\Sigma'}\{1\}=\cO_{\Sigma}\{1\}.
\]
It remains to use \eqref{e:shtuka structure}.
\end{proof}

\begin{rem}
By \eqref{e:3infinite otimes}, 
$\cO_{\Sigma}\{1\}=\bigotimes\limits_{n=0}^\infty (F^n)^*\sL_\Sigma\, $. By Lemma~\ref{l:j_+^*(sL,v_-)}, $j_+^*\sL_{\Sigma'} =\sL_{\Sigma}\,$. 
Using these facts and the equalities $F'_+:=j_+\circ F'$ and $F'\circ j_+=F$, one can rewrite ~\eqref{e:BK-Tate' definition} as
\begin{equation}   \label{e:BK-Tate' formula}
\cO_{\Sigma'}\{1\}=\bigotimes\limits_{n=0}^\infty ((F'_+)^n)^*\sL_{\Sigma'}\, ,
\end{equation}
where the infinite tensor product has to be understood appropriately.
\end{rem}

\begin{rem}   \label{r:trivializing O_Sigma'{1}}
By Definition~\ref{d:BK}, $\cO_{\Sigma}\{1\}$ is trivialized at $p:\Spf\BZ_p\to\Sigma$.  By~\eqref{e:O_Sigma'{1} descends}, this trivialization induces a trivialization of the pullback of 
$\cO_{\Sigma'}\{1\}$ via $\Spf\BZ_p\overset{p}\longrightarrow\Sigma\overset{j_\pm}\mono\Sigma'$. One gets the same trivialization by combining \S\ref{sss:recollections on L}(ii) with either \eqref{e:BK-Tate' definition} or \eqref{e:BK-Tate' formula}.
\end{rem}

\subsubsection{The $\cO_{\Sigma'}$-modules $\cO_{\Sigma'}\{n\}$}
For $n\in\BZ$ set
\[
\cO_{\Sigma'}\{n\}:=(\cO_{\Sigma'}\{1\})^{\otimes n}.
\]
If $n\ge 0$ this is a weakly invertible covariant $\cO_{\Sigma'}$-module. If $n\le 0$ it is a weakly invertible \emph{contravariant} $\cO_{\Sigma'}$-module (see \S\ref{sss:Weakly invertible}).

\subsubsection{The Bhatt-Lurie isomorphism}    \label{sss:2BL}
Let $\cO_{\Sigma'\otimes\BF_p}\{n\}$ be the restriction of $\cO_{\Sigma'}\{n\}$ to $\Sigma'\otimes\BF_p$. Combining \eqref{e:BL} and \eqref{e:BK-Tate' definition}, we get an isomorphism
\begin{equation}   \label{e:2BL}
\BL':\cO_{\Sigma'\otimes\BF_p}\{1-p\}\iso\sL_{\Sigma'\otimes\BF_p}^{\otimes (1-p)}\otimes \cN_{\Sigma'\otimes\BF_p},
\end{equation}
where $\cN_{\Sigma'\otimes\BF_p}$ is the restriction of $F'^*\sL_\Sigma=F'^*\cO_{\Sigma}(-\Delta_0)$ to $\Sigma'\otimes\BF_p$ and $\sL_{\Sigma'\otimes\BF_p}$ is the
restriction of $\sL_{\Sigma'}$. We call \eqref{e:2BL} the \emph{Bhatt-Lurie isomorphism.} Later we will identify the r.h.s. of \eqref{e:2BL} with the ideal of the closed substack $\Sigma'_{\red}\subset\Sigma'\otimes\BF_p$, see \S\ref{sss:3BL}.

\subsubsection{Restricting \eqref{e:2BL} to $\Sigma_\pm$}  \label{sss:Restricting to Sigma_pm}
We have a commutative diagram  
\begin{equation}   \label{e:BL diagram}
\xymatrix{
j_\pm^*\cO_{\Sigma'\otimes\BF_p}\{1-p\}\ar[r]^-\sim \ar[d]^\wr &j_\pm^*\sL_{\Sigma'\otimes\BF_p}^{\otimes (1-p)}\otimes j_\pm^*\cN_{\Sigma'\otimes\BF_p}\ar[d]^\wr\\
\cO_{\Sigma\otimes\BF_p}\{1-p\}\ar[r]^-\sim  & \sL_{\Sigma\otimes\BF_p}
}
\end{equation}
in which the horizontal arrows come from the  Bhatt-Lurie isomorphisms \eqref{e:2BL} and  \eqref{e:BL}, the left vertical arrow comes from 
\eqref{e:O_Sigma'{1} descends}, and the right vertical arrow comes from the isomorphisms
\[
j_+^*\sL_{\Sigma'}\iso\sL_{\Sigma}, \quad j_-^*\sL_{\Sigma'}\iso\cO_{\Sigma},
\]
described in \S\ref{sss:recollections on L}(i) and the canonical isomorphisms
\[
 j_+^*\cN_{\Sigma'\otimes\BF_p}=j_+^*F'^*\sL_{\Sigma\otimes\BF_p}\iso\Fr^*\sL_{\Sigma\otimes\BF_p}\iso\sL_{\Sigma\otimes\BF_p}^{\otimes p},
\]
\[
 j_-^*\cN_{\Sigma'\otimes\BF_p}=j_-^*F'^*\sL_{\Sigma\otimes\BF_p}\iso\sL_{\Sigma\otimes\BF_p}.
\]

\begin{lem}   \label{l:conservativity of cO_Sigma'{1}}
Let $S$ be a scheme and $ \beta_1, \beta_2\in\Sigma'(S)$. Then a morphism $ \beta_1\to \beta_2$ is an isomorphism if and only if it induces an isomorphism
$ \beta_1^*\cO_{\Sigma'}\{ 1\}\to \beta_2^*\cO_{\Sigma'}\{ 1\}$.
\end{lem}

\begin{proof}
Since $\Sigma$ is a g-stack, the $\cO_{\Sigma'}$-module $F'^*\cO_\Sigma\{1\}$ from formula \eqref{e:BK-Tate' definition} is strongly invertible in the sense of \S\ref{sss:Strongly invertible}. So the lemma is equivalent to a similar statement for $\sL_{\Sigma'}$. The latter is obvious.
\end{proof}

\subsection{Some closed substacks of $\Sigma'$}
\subsubsection{The Hodge-Tate locus $\Delta'_0\subset\Sigma'$}  \label{sss:Delta'_0}
In \S\ref{sss:L_Sigma'} we defined a weakly invertible $\cO_{\Sigma'}$-module $\sL_{\Sigma'}$ and a morphism $v_-:\sL_{\Sigma'}\to\cO_{\Sigma'}$.
Let $\Delta'_0\subset\Sigma'$ be the closed substack defined by the equation $v_-=0$. In terms of the morphism
$\Sigma'\to(\BA^1/\BG_m)_-\hat\otimes\BZ_p$ from \S\ref{sss:theleftfibration}, $\Delta'_0$ is the preimage of the closed substack 
$(\{ 0\}/\BG_m)_-\hat\otimes\BZ_p\subset(\BA^1/\BG_m)_-\hat\otimes\BZ_p$.
We call $\Delta'_0$ the \emph{Hodge-Tate locus} in $\Sigma'$.

By \S\ref{sss:SpecCoker} (or directly by \S\ref{sss:left-fibered subcategories}), the substack $\Delta'_0\subset\Sigma'$ is right-fibered over $\Sigma'$ in the sense of \S\ref{ss:left&right fibrations}. 

We have $\Sigma_-=\Sigma'\setminus\Delta'_0$, so $\Delta'_0\cap\Sigma_-=\emptyset$.
On the other hand, by Lemma~\ref{l:j_+^*(sL,v_-)} or by \eqref{e:restricting left fibration to Sigma}, we have
\begin{equation}  \label{e:Delta'_0 & Delta_0}
\Delta'_0\cap\Sigma_+=j_+(\Delta_0),
\end{equation}
 where $\Delta_0\subset\Sigma$ is as in 
\S\ref{sss:Hodge-Tate locus}. 

The equality  $\Sigma_+\setminus\Delta'_0=\Sigma_+\cap\Sigma_-=\emptyset$ means that $\Sigma_+$ is contained in the formal neighborhood of $\Delta'_0$. 
But $\Sigma_+\not\subset\Delta'_0$ because $j_+^{-1}(\Delta'_0)=\Delta_0\ne\Sigma$. 

By definition, the closed substack $\Delta'_0\subset\Sigma'$ is a pre-divisor in the sense of \S\ref{sss:Pre-div on stacks}. In Corollary~\ref{c:Delta'_0 Cartier divisor} we will show that $\Delta'_0$ is an effective Cartier divisor in $\Sigma'$ in the sense of \S\ref{sss:effective Cartier}.
Combined with Proposition~\ref{p:pre-algebraicity enough}(i-ii), this implies that the pullback of $\sL_{\Sigma'}$ to the underlying g-stack $(\Sigma')^{\rm g}$ of $\Sigma'$ equals $\cO_{(\Sigma')^{\rm g}} (-\Delta'_0 )$. By a slight abuse of notation, we will sometimes write $\cO_{\Sigma'}(-\Delta'_0)$ instead of $\sL_{\Sigma'}$ and $\cO_{\Sigma'}(\Delta'_0)$ instead of~$\sL_{\Sigma'}^{-1}$.

In \S\ref{ss:Delta'_0 as c-stack} we will describe the c-stack $\Delta'_0$ very explicitly.

\subsubsection{Some $\cO_{\Sigma'\otimes\BF_p}$-modules}
In \S\ref{sss:Y_+} we will  define a closed substack $\sY_+\subset\Sigma'\otimes\BF_p$ such that $\Sigma_+=\Sigma'\setminus\sY_+$. To do this, we need the following set-up.

Just as in \S\ref{sss:2BL}, let $\sL_{\Sigma'\otimes\BF_p}$ be the restriction of $\sL_{\Sigma'}$ to $\Sigma'\otimes\BF_p$ and let $\cN_{\Sigma'\otimes\BF_p}$ be the restriction of $F'^*\cO_{\Sigma}(-\Delta_0)$ to $\Sigma'\otimes\BF_p$. Note that $\cN_{\Sigma'\otimes\BF_p}$ is a \emph{strongly} invertible $\cO_{\Sigma'\otimes\BF_p}$-module (because it is the pullback of a line bundle on a g-stack), while $\sL_{\Sigma'\otimes\BF_p}$ is a weakly invertible covariant $\cO_{\Sigma'\otimes\BF_p}$-module (we are using the terminology of \S\ref{ss:O-mododules on c-stacks}). We need the diagram
\begin{equation}    \label{e:2alpha^p beta}
\xymatrix{
\cN_{\Sigma'\otimes\BF_p}\ar[r]^\varphi \ar[rd]_\gamma & \sL_{\Sigma'\otimes\BF_p}^{\otimes p}\ar[d]^{v_-^{\otimes p}}\\
 & \cO_{\Sigma'\otimes\BF_p}
}
\end{equation}
in which $\gamma$ is induced by the canonical morphism $v:\sL_\Sigma=\cO_\Sigma (-\Delta_0)\mono\cO_\Sigma$ and $\varphi$ comes from the morphism \eqref{e:5geomFrobenius}.  This diagram commutes: to check this, note that the morphism $\varphi_M$ from formula \eqref{e:5geomFrobenius} is functorial in $M$ and then apply this functoriality to $\xi:M\to W_S$.

\subsubsection{The closed substacks $\sY_\pm\subset\Sigma'\otimes\BF_p$}  \label{sss:Y_+}
Set $\sY_-:=\Delta'_0\otimes\BF_p\subset\Sigma'\otimes\BF_p$; then $\Sigma_-=\Sigma'\setminus\sY_-$. The closed substack $\sY_-$ is right-fibered over 
$\Sigma'\otimes\BF_p$ because $\Delta'_0$ is right-fibered over $\Sigma'$ (see \S\ref{sss:Delta'_0}).

Now let $\sY_+\subset\Sigma'\otimes\BF_p$ be the closed substack defined by the equation $\varphi =0$, where $\varphi$ is as in \eqref{e:2alpha^p beta}. Then
$\Sigma_+=\Sigma'\setminus\sY_+$: this follows from the equivalence (a)$\Leftrightarrow$(c) in Lemma~\ref{l:invertibility criterion}.

The closed substack $\sY_+$ is left-fibered over $\Sigma'\otimes\BF_p$: to see this, write
\[
\cO_{\sY_+}=\Coker (\cN_{\Sigma'\otimes\BF_p}\otimes\sL_{\Sigma'\otimes\BF_p}^{\otimes (-p)}\overset{\varphi}\longrightarrow \cO_{\Sigma'\otimes\BF_p})
\]
and apply \S\ref{sss:SpecCoker} to the contravariant $\cO_{\Sigma'\otimes\BF_p}$-module 
$\cN_{\Sigma'\otimes\BF_p}\otimes\sL_{\Sigma'\otimes\BF_p}^{\otimes (-p)}$.

By definition, the closed substacks $\sY_\pm\subset\Sigma'\otimes\BF_p$ are pre-divisors in the sense of \S\ref{sss:Pre-div on stacks}. Later we will show that they are effective Cartier divisors in $\Sigma'\otimes\BF_p$, see Lemma~\ref{l:Y_+ and Y_- are divisors} below. 

As already noted,
\begin{equation}    \label{e:Sigma_pm complement of Y_pm}
\Sigma_\pm=\Sigma'\setminus\sY_\pm .
\end{equation}
So $\sY_\pm\cap\Sigma_\pm =\emptyset$. By \eqref{e:Delta'_0 & Delta_0} and the equality $\sY_-:=\Delta'_0\otimes\BF_p$ we have
\begin{equation}  \label{e:Y_- cap Sigma_+}
\sY_-\cap\Sigma_+=j_+(\Delta_0\otimes\BF_p ).
\end{equation}

We claim that
\begin{equation}     \label{e:Y_+ cap Sigma_-}
\sY_+\cap\Sigma_-=j_-(\Delta_0\otimes\BF_p ).
\end{equation}
Indeed, diagram \eqref{e:2alpha^p beta} shows that on $\Sigma_-$ the equation $\varphi =0$ is equivalent to the equation $\gamma =0$, where $\gamma$ is the $F'$-pullback of $v:\sL_{\Sigma\otimes\BF_p}\to\cO_{\Sigma\otimes\BF_p}$. So 
$\sY_+\cap\Sigma_-=\Sigma_-\cap (F')^{-1}(\Delta_0\otimes\BF_p )$. This proves \eqref{e:Y_+ cap Sigma_-} because $F'\circ j_-=\id_\Sigma\,$.

\begin{lem}   \label{l:Y_+ and Y_- are divisors}
(i) The pre-divisors $\sY_\pm\subset\Sigma'\otimes\BF_p$ are effective Cartier divisors.

(ii) The closed substack $(F')^{-1}(\Delta_0\otimes\BF_p)\subset\Sigma'\otimes\BF_p$ is an effecive Cartier divisor. Moreover,
\begin{equation}   \label{e:(F')^{-1}(Delta_0 otimes F_p)}
(F')^{-1}(\Delta_0\otimes\BF_p)=\sY_++p\sY_-.
\end{equation}
\end{lem}

\begin{proof}
The first part of (ii) follows from flatness of $F':\Sigma'\to\Sigma$. Since $(F')^{-1}(\Delta_0\otimes\BF_p)$ is the locus of zeros of the morphism $\gamma$ from diagram \eqref{e:2alpha^p beta}, we see that \eqref{e:(F')^{-1}(Delta_0 otimes F_p)} holds (at least, as an equality between pre-divisors). Statement (i) follows from (ii) and
Lemma~\ref{l:summand of divisor}.
\end{proof}

\subsubsection{Remarks}
(i) A slightly different proof of Lemma~\ref{l:Y_+ and Y_- are divisors}(i) will be given in Remark~\ref{r:Y_+ and Y_- are divisors}.

(ii) In \S\ref{s:Sigma' as quotient} we will describe the c-stack $(F')^{-1}(\Delta_0\otimes\BF_p)$ and its closed substacks $\sY_\pm$ very explicitly, see
Corollaries~\ref{c:sR} and \ref{c:Y_pm}.

(iii) By \eqref{e:(F')^{-1}(Delta_0 otimes F_p)}, the reduction modulo $p$ of the stack $(F')^{-1}(\Delta_0)$ is reducible. Let us note that the stack $(F')^{-1}(\Delta_0)$ itself is irreducible (it is easy to check this using \S\ref{sss:Some morphisms}(i) below).

\subsubsection{The reduced part of $\Sigma'$}    \label{sss:reduced part of Sigma'}
Let $\Sigma'_{\red}$ be the reduced part of $\Sigma'$ (i.e., the smallest closed substack of $\Sigma'$ containing all field-valued points of $\Sigma'$).
By Theorem~\ref{t:Sigma' algebraic}(iii), the c-stack $\Sigma'_{\red}$ is algebraic.

It is clear that $\Sigma'_{\red}\subset\Sigma'\otimes\BF_p$. Moreover,
since $\Sigma_\pm=\Sigma'\setminus\sY_\pm$ and $\Sigma_+\cap\Sigma_-=\emptyset$, we have $\Sigma'_{\red}\subset\sY_++\sY_-$.
In Proposition~\ref{p:sX'} we will show that 
\begin{equation}   \label{e:Sigma'_red=Y_+ +Y_-}
\Sigma'_{\red}=\sY_++\sY_-.
\end{equation}
This is not surprising because $(\sY_++\sY_-)\cap\Sigma_\pm =j_\pm (\Sigma_{\red})$ by \eqref{e:Sigma_pm complement of Y_pm}-\eqref{e:Y_+ cap Sigma_-}.
In \S\ref{s:Sigma' as quotient} we will describe the c-stack $\Sigma'_{\red}$ very explicitly, see Corollary~\ref{c:Sigma'_red}.

Combining \eqref{e:Sigma'_red=Y_+ +Y_-} with the next lemma, we see that $\Sigma'_{\red}$ is left-fibered over $\Sigma'$.

\begin{prop}  \label{p:left/right-fibered}
Let $m,n\in\BZ$, $m,n\ge 0$. 

(i) If $n\le pm$ then the effective Cartier divisor $m\sY_++n\sY_-\subset\Sigma'\otimes\BF_p$ is left-fibered over~$\Sigma'\otimes\BF_p$

(ii) If $n\ge pm$ then $m\sY_++n\sY_-$ is right-fibered over~$\Sigma'\otimes\BF_p$.
\end{prop}

\begin{proof}
By \S\ref{sss:SpecCoker}, it suffices to write $\cO_{m\sY_++n\sY_-}$ as $\Coker (\sM_{m,n}\overset{f_{m,n}}\longrightarrow\cO_{\Sigma'\otimes\BF_p})$, where the pair $(\sM_{m,n},f_{m,n})$ lives in the category of contravariant (resp.~covariant) $\cO_{\Sigma'\otimes\BF_p}$-modules if $n\le pm$ (resp.~if $n\ge pm$). 

Set $\sM_{m,n}:=\cN_{\Sigma'\otimes\BF_p}^{\otimes m}\otimes\sL_{\Sigma'\otimes\BF_p}^{\otimes (n-pm)}$.
If $n\le pm$ then let $f_{m,n}$ be the composite morphism
\[
\sM_{m,n}=\cN_{\Sigma'\otimes\BF_p}^{\otimes m}\otimes\sL_{\Sigma'\otimes\BF_p}^{\otimes (n-pm)}\to \cN_{\Sigma'\otimes\BF_p}^{\otimes m}\otimes\sL_{\Sigma'\otimes\BF_p}^{\otimes (-pm)}\overset{\varphi^{\otimes m}}\longrightarrow\cO_{\Sigma'\otimes\BF_p}  ,
\]
where the first arrow is $\id_{\cN_{\Sigma'\otimes\BF_p}^{\otimes m}}\otimes v_-^{\otimes n}$.  If $n\ge pm$ then let $f_{m,n}$ be the composite morphism
\[
\sM_{m,n}=\cN_{\Sigma'\otimes\BF_p}^{\otimes m}\otimes\sL_{\Sigma'\otimes\BF_p}^{\otimes (n-pm)}  \to \sL_{\Sigma'\otimes\BF_p}^{\otimes n} 
\overset{v_-^{\otimes n}}
\longrightarrow\cO_{\Sigma'\otimes\BF_p} 
\]
where the first arrow is $\varphi^{\otimes m}\otimes\id_{\sL_{\Sigma'\otimes\BF_p}^{\otimes (n-pm)}}\, $.
\end{proof}

\subsubsection{The Bhatt-Lurie exact sequence}       \label{sss:3BL} 
The proof of Proposition~\ref{p:left/right-fibered} provides a morphism $f_{1,1}:\sM_{1,1}\to\cO_{\Sigma'\otimes\BF_p}$ such that 
$\Coker f_{1,1}=\cO_{\sY_++\sY_-}$; namely, 
$$\sM_{1,1}:=\sL_{\Sigma'\otimes\BF_p}^{\otimes (1-p)}\otimes\cN_{\Sigma'\otimes\BF_p},$$
and $f_{1,1}$ is induced by he composite morphism $\cN_{\Sigma'\otimes\BF_p}\overset{\varphi}\longrightarrow \sL_{\Sigma'\otimes\BF_p}^{\otimes p}
\overset{v_-}\longrightarrow \sL_{\Sigma'\otimes\BF_p}^{\otimes (p-1)}$, where $\varphi$ comes from the morphism \eqref{e:5geomFrobenius} (just as in diagram \eqref{e:2alpha^p beta}). We also have the Bhatt-Lurie isomorphism $\BL':\cO_{\Sigma'\otimes\BF_p}\{1-p\}\iso\sM_{1,1}\,$, see formula \eqref{e:2BL}. Combining this with the fact that $\sY_++\sY_-$ is an effective Cartier divisor on $\Sigma'\otimes\BF_p$ (see Lemma~\ref{l:Y_+ and Y_- are divisors}(i)), we get an exact sequence
\begin{equation}    \label{e:BL exact}
0\to\cO_{\Sigma'\otimes\BF_p}\{1-p\}\overset{f}\longrightarrow \cO_{\Sigma'\otimes\BF_p}\to\cO_{\sY_++\sY_-}\to 0, \quad f:=f_{1,1}\circ\BL'.
\end{equation}
Later we will prove the equality $\sY_++\sY_-=\Sigma'_{\red}$, see
Proposition \ref{p:sX'}. Assuming this, we can rewrite \eqref{e:BL exact} as an exact sequence
\begin{equation}    \label{e:2BL exact}
0\to \cO_{\Sigma'\otimes\BF_p}\{1-p\}\overset{f}\longrightarrow \cO_{\Sigma'\otimes\BF_p}\to\cO_{\Sigma'_{\red}}\to 0.
\end{equation}

By Lemma~\ref{l:O_Sigma'{1} descends}, $j_\pm^*\cO_{\Sigma'}\{1\}= \cO_\Sigma\{1\}$, so $j_\pm^*(f)$ is a morphism
$\cO_{\Sigma\otimes\BF_p}\{1-p\}\to\cO_{\Sigma\otimes\BF_p}$. 
The following lemma will allows us to descend the exact sequence \eqref{e:2BL exact} to $\Sigma''$, see Proposition~\ref{p:Sigma''_red as divisor}(ii).

\begin{lem}   \label{l:f_11 descends}
Both $j_+^*(f)$ and $j_-^*(f)$ are equal to the composite morphism 
\[
\cO_{\Sigma\otimes\BF_p}\{1-p\}\iso\sL_{\Sigma\otimes\BF_p}\overset{v}\longrightarrow\cO_{\Sigma\otimes\BF_p},
\]
where the first arrow is the Bhatt-Lurie isomorphism \eqref{e:BL}.
\end{lem}

\begin{proof}
By \S\ref{sss:Restricting to Sigma_pm}, it suffices to prove commutativity of the diagram
\begin{equation}    \label{e:triangular diagram}
\xymatrix{
j_\pm^*(\sL_{\Sigma'\otimes\BF_p}^{\otimes (1-p)}\otimes\cN_{\Sigma'\otimes\BF_p})\ar[r]^-{j_\pm^*(f_{1,1})} \ar[d]^\wr & \cO_{\Sigma\otimes\BF_p}\\
\sL_{\Sigma\otimes\BF_p}\ar[ur]_v 
}
\end{equation}
in which the vertical arrow is as in \eqref{e:BL diagram}.

We have $\sL_{\Sigma'\otimes\BF_p}^{\otimes (1-p)}\otimes\cN_{\Sigma'\otimes\BF_p}=
\cO_{\Sigma'\otimes\BF_p}((p-1)\sY_-)\otimes\cO_{\Sigma'\otimes\BF_p}(-(F')^{-1}(\Delta_0\otimes\BF_p))$,
and the commutative diagram \eqref{e:2alpha^p beta} shows that the morphism
$f_{1,1}$ from \S\ref{sss:3BL} is equal to the composite map
\[
\cO_{\Sigma'\otimes\BF_p}((p-1)\sY_-)\otimes\cO_{\Sigma'\otimes\BF_p}(-(F')^{-1}(\Delta_0\otimes\BF_p))\overset{\alpha}\simeq
\cO_{\Sigma'\otimes\BF_p}(-\sY_+-\sY_-)\mono\cO_{\Sigma'\otimes\BF_p},
\]
where $\alpha$ is the obvious isomorphism coming from the equality $(F')^{-1}(\Delta_0\otimes\BF_p)=\sY_++p\sY_-$ proved in Lemma~\ref{l:Y_+ and Y_- are divisors}. This implies commutativity of \eqref{e:triangular diagram}.
\end{proof}
 
 \subsubsection{A non-substack of $\Sigma'$}      \label{sss:nonsubstack}
 In \S\ref{ss:Hdg & bar dR} we will define a c-stack $\Sigma'_{\bardR}$ over $\Sigma'$ such that the composite map 
 $\Sigma'_{\bardR}\to\Sigma'\to (\BA^1/\BG_m)_-\hat\otimes\BZ_p$ is an isomorphism. Equivalently, we will define in \S\ref{ss:Hdg & bar dR} a section 
 $\fp :(\BA^1/\BG_m)_-\hat\otimes\BZ_p\to\Sigma'$. One can think of $\Sigma'_{\bardR}$ as a ``kind of complement" to $\Sigma_+=\Sigma'\setminus\sY_+$, except that $\Sigma'_{\bardR}$ is not a \emph{sub}stack of $\Sigma'$; see \S\ref{sss:kind of complement} for more details.

\subsection{A piece of structure on $\Sigma'\otimes\BF_p$}  \label{ss:piece of structure}
In this subsection we work with  $\Sigma\otimes\BF_p$ and $\Sigma'\otimes\BF_p$. So $F'$ will denote a morphism $\Sigma'\otimes\BF_p\to\Sigma\otimes\BF_p$. Same convention for $j_\pm$ and $F'_\pm$.

Recall that c-stacks form a 2-category (so endomorphisms of a c-stack form a monoidal category). By \S\ref{sss:id to F'_-}, $F'$ and $j_-$ form an adjoint pair;
the unit of the adjunction is a canonical morphism 
\begin{equation}   \label{e:unit of adjunction}
\id_{\Sigma'}\to j_-\circ F'=F'_-.
\end{equation}
We are going to describe a somewhat similar relation between $F'$ and $j_+$. More precisely, the 2-categorical properties of the morphism \eqref{e:queer adjunction} constructed below are somewhat similar to those of the morphism \eqref{e:unit of adjunction}.

\subsubsection{The morphism $f:F'_+\to\Fr_{\Sigma'\otimes\BF_p}$}  \label{sss:piece of structure}
Let $S$ be an $\BF_p$-scheme and $(M,\xi )\in\Sigma'(S)$. Recall that $F'_+$ takes $(M,\xi )$ to $((M')^{(-1)},(\xi')^{(-1)})\in\Sigma'(S)$.
Using the geometric Frobenius of $M$, we defined in \S\ref{sss:geometric Frobenius} a morphism $f_M:(M')^{(-1)}\to \Fr_S^*M$, see formula \eqref{e:4geomFrobenius}.  By functoriality of the morphism \eqref{e:4geomFrobenius}, our $f_M$ is a morphism $((M')^{(-1)},(\xi')^{(-1)})\to (\Fr_S^*M, \Fr_S^*\xi )$ in the category $\Sigma'(S)$. Letting $(S,M,\xi )$ vary, we get a morphism
\begin{equation}   \label{e:queer adjunction}
f:F'_+\to\Fr_{\Sigma'\otimes\BF_p}
\end{equation}
in the monoidal category $\End (\Sigma'\otimes\BF_p)$.

\subsubsection{Properties of $f:F'_+\to\Fr_{\Sigma'\otimes\BF_p}$}   \label{sss:2-categorical properties}
One can check the following properties of the morphism $f:F'_+\to\Fr_{\Sigma'\otimes\BF_p}$, of which only the first one will be used:

(i) the morphism $F'_+\circ j_+\to\Fr_{\Sigma'\otimes\BF_p}\circ j_+$ induced by $f$ is equal to the isomorphism
\[
F'_+\circ j_+ = j_+\circ F'\circ j_+\simeq j_+\circ\Fr_{\Sigma\otimes\BF_p}\simeq \Fr_{\Sigma'\otimes\BF_p}\circ j_+;
\]

(ii) the morphism $F'\circ F'_+\to F'\circ\Fr_{\Sigma'\otimes\BF_p}$ induced by $f$ is equal to the isomorphism
\[
F'\circ F'_+ = F'\circ j_+\circ F' \simeq \Fr_{\Sigma\otimes\BF_p}\circ F'  \simeq  F'\circ\Fr_{\Sigma'\otimes\BF_p};
\]

(iii) the morphism $F'_+\circ j_-\to\Fr_{\Sigma'\otimes\BF_p}\circ j_-$ induced by $f$ is equal to the morphism
\[
F'_+\circ j_- = j_+\circ F'\circ j_-\simeq j_+\to j_-\circ F'\circ j_+\simeq  j_-\circ\Fr_{\Sigma\otimes\BF_p}\simeq \Fr_{\Sigma'\otimes\BF_p}\circ j_-,
\]
where the morphism $j_+\to j_-\circ F'\circ j_+$ comes from \eqref{e:unit of adjunction}.

\begin{rem}  \label{r:uniquely characterizes}
One can show that property (i) uniquely characterizes the morphism \eqref{e:queer adjunction}. This easily follows from Corollary~\ref{c:recieves morphism from Sigma_+} below.
\end{rem}

\subsection{$\Delta'_0$ as a c-stack}  \label{ss:Delta'_0 as c-stack}
\subsubsection{The goal}   \label{sss:goal of Delta'_0 as c-stack}
The Hodge-Tate locus $\Delta'_0\subset\Sigma'$ was defined in \S\ref{sss:Delta'_0}.
Recall that the g-stack $\Delta'_0\cap\Sigma_+=j_+(\Delta_0)\simeq\Delta_0$ canonically identifies with the classifying stack
$(\Spf\BZ_p)/\BG_m^\sharp$ (see Lemma~\ref{l:Delta_0 as classifying stack}). In this subsection we will get a 
description of the whole c-stack $\Delta'_0$ in the same spirit, see Lemma~\ref{l:Delta'_0 as c-stack} and \S\ref{sss:reality check}.
Let us note that using this description one can translate into the prismatic language the  Deligne-Illusie argument \cite{D-I} for the degeneration of the ``Hodge to de Rham'' spectral sequence in characteristic $0$; this translation is beyond the scope of this article\footnote{For a somewhat different but essentially equivalent explanation of the  prismatic point of view on the Deligne-Illusie argument, see \cite[Rem.~4.7.18-4.7.23]{BL} and \cite[Rem.~5.16]{BL2}.}.

\begin{lem}   \label{l:S-points of Delta'_0}
(i) For any $p$-nilpotent scheme $S$, the category $\Delta'_0 (S)$ identifies with the category of pairs
consisting of a line bundle $\sL$ on $S$ and a $W_S$-module extension of $W^{(1)}_S$ by $\sL^\sharp=\sL\otimes W^{(F)}_S$. (If $S$ is not $p$-nilpotent then $\Delta'_0 (S)$ is, of course, empty.)

(ii) The following diagram commutes:  
\begin{equation}  \label{e:F' on Delta'_0}
\xymatrix{
\Delta'_0\ar[d] \ar@{^{(}->}[r] &\Sigma'\ar[d]^{F'}\\
\Spf\BZ_p\ar[r]^{p}&\Sigma
}
\end{equation}
\end{lem}

Note that diagram \eqref{e:F' on Delta'_0} extends diagram \eqref{e:F on Delta_0} (assuming that the upper right term of \eqref{e:F on Delta_0} is interpreted as $\Sigma_+$).

\begin{proof}
If $(M,\xi )\in\Delta'_0 (S)$ then $\xi:M\to W_S$ factors as $M\epi M'\to W_S$. By
\eqref{e:Hom(W^(1),W}, any morphism $M'\to W_S$ factors through $V(W^{(1)}_S)\subset W_S$. So $\xi:M\to W_S$ factors as
\[
M\epi M'\overset{g}\longrightarrow W^{(1)}_S\overset{V}\mono W_S.
\]
Then $\xi':M'\to W_S^{(1)}$ equals $FVg=pg$. Since $\xi'$ is primitive, the morphism $g:M'\to W^{(1)}_S$ is an isomorphism, so we can identify $M'$ with $W^{(1)}_S$ using $g$. 

Thus an object of $\Delta'_0 (S)$ is the same as a pair consisting of a line bundle $\sL$ on $S$ and an exact sequence $0\to\sL^\sharp\to M\to W^{(1)}_S\to 0$. This proves (i).

Recall that $\xi'=pg$. So after identifying $M'$ with $W^{(1)}_S$ using $g$, we have $\xi'=p$. This proves (ii).
\end{proof}

\begin{cor}  \label{c:S-points of Delta'_0}
The restrictions of $\cO_{\Sigma'}\{1\}$ and $\sL_{\Sigma'}$ to $\Delta'_0$ are canonically isomorphic.
\end{cor}

\begin{proof}
Follows from commutativity of \eqref{e:F' on Delta'_0} because $\cO_{\Sigma'}\{1\}\!:=\sL_{\Sigma'}\otimes F'^*\cO_\Sigma\{1\}$ and the pull\-back of
$\cO_\Sigma\{1\}$ via $p:\Spf\BZ_p\to\Sigma$ is canonically trivial
(see \S\ref{sss:BK-Tate on Sigma'} and Definition~\ref{d:BK}).
\end{proof}

In \S\ref{sss:X^dR} we more or less\footnote{We skipped the verification mentioned in \S\ref{sss:prismatization in general}, and we didn't explain why the stack $X^\prism$ (resp.~ $X^{\dR}$) is algebraic over $\Sigma$ (resp. over $\Spf\BZ_p$).} defined a functor $X\mapsto X^{\dR}$, where $X$ is a $p$-adic formal scheme. It is easy to check\footnote{Combine \S\ref{sss:prismatization of affines} with Proposition~\ref{p:Cone (G_a^sharp to G_a)}.} that
\begin{equation}   \label{e:A^1^dR}
(\BA^1\hat\otimes\BZ_p)^{\dR }=\Cone (\BG_a^\sharp\to\BG_a)\hat\otimes\BZ_p,
\end{equation}
where $\hat\otimes\BZ_p$ stands for the $p$-adic completion (see~\S\ref{sss:p-adic completion}) and $\Cone$ is understood as a Picard stack, see \S\ref{sss:Cones-group schemes}. On the other hand, one can use \eqref{e:A^1^dR} as an \emph{ad~hoc} definition of $(\BA^1\hat\otimes\BZ_p)^{\dR }$.

\begin{lem}   \label{l:Delta'_0 as g-stack}
The g-stack underlying $\Delta'_0$ is canonically isomorphic to $(\BA^1\hat\otimes\BZ_p)^{\dR }/\BG_m$, where the action of
$\BG_m$ on  $(\BA^1\hat\otimes\BZ_p)^{\dR }$ is induced by its action on $\BG_a$ and $\BG_a^\sharp$ by multiplication. 
\end{lem}

\begin{proof}
Follows from Lemma~\ref{l:S-points of Delta'_0} and the isomorphism 
\begin{equation}     \label{e:2Ex (W^(1),W^(F)}
\Ex_W(W_S^{(1)},W_S^{(F)})\iso\Cone ((\BG_a^\sharp)_S\to (\BG_a)_S)
\end{equation}
from Proposition~\ref{p:Ex (W^(1),W^(F)}.
\end{proof}

\subsubsection{The c-stack $((\BA^1\hat\otimes\BZ_p)^{\dR }/\BG_m)_+$}  \label{sss:quotients of A^1^dR}
The g-stack $(\BA^1\hat\otimes\BZ_p)^{\dR }$ carries an action (by multiplication) 
of the multipli\-ca\-tive monoid $\BA^1_m$ (not merely the multiplicative group $\BG_m$). So according to \S\ref{sss:general setting for +-},
one has the c-stack $((\BA^1\hat\otimes\BZ_p)^{\dR }/\BG_m)_+$. 
Here is an explicit description of $((\BA^1\hat\otimes\BZ_p)^{\dR }/\BG_m)_+$ (which could be used as a definition): if $S$ is a $p$-nilpotent scheme then 
$((\BA^1\hat\otimes\BZ_p)^{\dR }/\BG_m)_+(S)$ is the category of pairs consisting of a line bundle $\sL$ on $S$ and an $S$-morphism 
$$S\to\Cone (\sL^\sharp\to \sL);$$ 
if $S$ is not $p$-nilpotent then $((\BA^1\hat\otimes\BZ_p)^{\dR }/\BG_m)_+(S)=\emptyset$.

The $\BA^1_m$-equivariant morphism $\BA^1\hat\otimes\BZ_p\to (\BA^1\hat\otimes\BZ_p)^{\dR }$ induces a canonical morphism
\begin{equation}   \label{e:A^1 to (A^1)^dR}
(\BA^1/\BG_m)_+\hat\otimes\BZ_p\to ((\BA^1\hat\otimes\BZ_p)^{\dR }/\BG_m)_+,
\end{equation}
where $(\BA^1/\BG_m)_+$ is the c-stack from \S\ref{sss:BA^1/BG_m)_pm}. In terms of the above description of the category $((\BA^1\hat\otimes\BZ_p)^{\dR }/\BG_m)_+(S)$, the map \eqref{e:A^1 to (A^1)^dR} comes from the canonical morphism
\[
\sL\to\Cone (\sL^\sharp\to \sL).
\]

\begin{lem}   \label{l:Delta'_0 as c-stack}
The c-stack $\Delta'_0$ is canonically isomorphic to $((\BA^1\hat\otimes\BZ_p)^{\dR }/\BG_m)_+$.
\end{lem}

\begin{proof}
Similarly to \eqref{e:2Ex (W^(1),W^(F)}, for any line bundle $\sL$ on $S$ one has a canonical isomorphism
\begin{equation}    \label{e:3Ex(W^1,Lsharp)}
\Ex_W(W_S^{(1)}, \sL^\sharp )\iso\Cone (\sL^\sharp\to \sL), 
\end{equation}
see formula \eqref{e:Ext(W^1,Lsharp)}. It remains to combine \eqref{e:3Ex(W^1,Lsharp)} with  Lemma~\ref{l:S-points of Delta'_0}.
\end{proof}

\subsubsection{On the isomorphism $\Delta_0\iso (\Spf\BZ_p)/\BG_m^\sharp$}  \label{sss:reality check}
Recall that $\Delta_0$ is an open substack of $\Delta'_0$.
In Lemma~\ref{l:Delta_0 as classifying stack} we constructed a canonical isomorphism 
$\Delta_0\iso (\Spf\BZ_p)/\BG_m^\sharp$.  Combining it with Lemma~\ref{l:Delta'_0 as g-stack}, 
we get an open embedding 
\begin{equation}  \label{e:embedding to explain}
(\Spf\BZ_p)/\BG_m^\sharp\mono (\BA^1\hat\otimes\BZ_p)^{\dR }/\BG_m. 
\end{equation}

Here is a way to think of \eqref{e:embedding to explain}.
First, one can check that similarly to \eqref{e:A^1^dR}, one has
$$(\BG_m\hat\otimes\BZ_p)^{\dR }=\Cone (\BG_m^\sharp\to\BG_m)\hat\otimes\BZ_p\, ;$$
on the other hand, the reader can consider this formula as \emph{ad~hoc} definition of $(\BG_m\hat\otimes\BZ_p)^{\dR }$. 
 Then  $(\Spf\BZ_p)/\BG_m^\sharp=(\BG_m\hat\otimes\BZ_p)^{\dR }/\BG_m\,$, and \eqref{e:embedding to explain} is the morphism
$$(\BG_m\hat\otimes\BZ_p)^{\dR }/\BG_m\to (\BA^1\hat\otimes\BZ_p)^{\dR }/\BG_m$$ induced by the unique morphism of stacks
$(\BG_m\hat\otimes\BZ_p)^{\dR }=((\BA^1\setminus\{ 0\})\hat\otimes\BZ_p)^{\dR }\to (\BA^1\hat\otimes\BZ_p)^{\dR }$ such that the following diagram commutes:
\[
\xymatrix{
(\BA^1\setminus\{ 0\})\hat\otimes\BZ_p\ar[r] \ar[d] &  \BA^1\hat\otimes\BZ_p\ar[d]\\
(\BA^1\setminus\{ 0\})\hat\otimes\BZ_p)^{\dR }\ar[r] & (\BA^1\hat\otimes\BZ_p)^{\dR }
}
\]

\subsection{The morphisms $\Sigma'_{\Hdg}\to\Sigma'_{\bardR}\to\Sigma'$}   \label{ss:Hdg & bar dR}
The c-stacks  $\Sigma'_{\Hdg}$, $\Sigma'_{\bardR}$, and the morphisms 
$\Sigma'_{\Hdg}\to\Sigma'_{\bardR}\to\Sigma'$ will be defined abstractly. However, the notation is motivated by their
relation to the ``Hodge to de Rham'' spectral sequence (which is beyond the scope of this article).

\subsubsection{The stack $\Sigma'_{\bardR}$}    \label{sss:2Sigma'_bardR}
Recall that if $S$ is a $p$-nilpotent scheme then $\Sigma' (S)$ is the category of pairs $(M,\xi)$, where $M$ is an admissible $W_S$-module and $\xi :M\to W_S$ is a primitive $W_S$-morphism.
Now define a c-stack $\Sigma'_{\bardR}$ as follows: $\Sigma' _{\bardR}(S)$ is the category of pairs $(M,\xi)\in\Sigma' (S)$ equipped with a splitting $\sigma :M'\to M$, where $M'$ is as in Definition~\ref{d:admissible}. 
By definition, we have a canonical morphism 
\begin{equation}  \label{e:2bardR}
\Sigma'_{\bardR}\to\Sigma'.
\end{equation}

\begin{lem}
$\Sigma'_{\bardR}\times_{\Sigma'}\Sigma_+=\emptyset$.
\end{lem}

\begin{proof}
If $(M,\xi)\in\Sigma_+(S)$ then $M$ is invertible, so a splitting $M'\to M$ cannot exist.
\end{proof}

\subsubsection{Remarks}  \label{sss:kind of complement}
(i) Recall that $\Sigma_-=\Sigma'\setminus\Delta'_0$. Similarly, one can think of $\Sigma'_{\bardR}$ as a ``kind of complement" to $\Sigma_+$ except that $\Sigma'_{\bardR}$ is not a \emph{sub}stack of $\Sigma'$, see Lemma~\ref{l:Sigma'_bardR}(iv) below. 

(ii) Recall that $\Sigma_+=\Sigma'\setminus\sY_+$, see \S\ref{sss:Y_+}. Let us compare $\Sigma'_{\bardR}$ with $\sY_+$. Unlike $\sY_+$, the c-stack $\Sigma'_{\bardR}$ is \emph{not} over $\BF_p$, see Lemma~\ref{l:Sigma'_bardR}(i) below. Later we will show that the morphism $\Sigma'_{\bardR}\otimes\BF_p\to\Sigma$ factors through $\sY_+$ and the morphism 
$\Sigma'_{\bardR}\otimes\BF_p\to\sY_+$ is faithfully flat; moreover, $\sY_+$ identifies with the classifying stack of an explicit flat group scheme of infinite type over
$\Sigma'_{\bardR}\otimes\BF_p$ (informally, $\sY_+$ is ``much more stacky" than $\Sigma'_{\bardR}$). For precise statements, see Corollary~\ref{c:Y_pm}(i) and
Remark~\ref{r:2kind of complement}.

\begin{lem}  \label{l:Sigma'_bardR}
(i) The composite map $\Sigma'_{\bardR}\to\Sigma'\to (\BA^1/\BG_m)_-\hat\otimes\BZ_p$ is an isomorphism. 
Its inverse 
\begin{equation}
\fp: (\BA^1/\BG_m)_-\hat\otimes\BZ_p\iso \Sigma'_{\bardR}  
\end{equation} 
is as follows: if $S$ is a $p$-nilpotent scheme and $\sL$ is a line bundle on $S$ equipped with a morphism $v_-:\sL\to\cO_S$ then 
$\fp (\sL ,v_-)=(M,\xi,\sigma)$, where $M=\sL^\sharp\oplus W^{(1)}_S$, $\sigma :W^{(1)}_S\to M$ is the obvious embedding, and
$\xi :\sL^\sharp\oplus W^{(1)}_S\to W_S$ equals $(v_-,V)$.

(ii) $\Sigma'_{\bardR}\times_{\Sigma'}\Sigma_-=\Spf\BZ_p$.

(iii) The morphism $\Spf\BZ_p=\Sigma'_{\bardR}\times_{\Sigma'}\Sigma_-\to\Sigma_-=\Sigma$
equals $p\in\Sigma (\BZ_p)$.

(iii\,$'$) The following diagram commutes:
\[
\xymatrix{
\Spf\BZ_p \ar@{^{(}->}[r] \ar[d]_p & (\BA^1/\BG_m)_-\hat\otimes\BZ_p\ar[d]^\fp\\
\Sigma\ar@{^{(}->}[r]^{j_-} & \Sigma'
}
\]
Here the upper horizontal arrow is the  immersion $\Spf\BZ_p=(\BG_m/\BG_m)\hat\otimes\BZ_p\mono (\BA^1/\BG_m)_-\hat\otimes\BZ_p$.

(iv) The morphism \eqref{e:2bardR} is not a monomorphism.

(v) The following diagram commutes:  
\begin{equation}  
\xymatrix{
\Sigma'_{\bardR}\ar[d] \ar[r] &\Sigma'\ar[d]^{F'}\\
\Spf\BZ_p\ar[r]^{p}&\Sigma
}
\end{equation}
\end{lem}

\begin{proof}
Similarly to the proof of Lemma~\ref{l:S-points of Delta'_0}, one shows that if $(M,\xi,\sigma)\in \Sigma'_{\bardR}(S)$ then the composite map $M'\overset{\sigma}\longrightarrow M\overset{\xi}\longrightarrow W_S$ induces an isomorphism $M'\iso V(W_S^{(1)})$ and therefore an isomorphism $V^{-1}\xi\sigma: M'\iso W_S^{(1)}$. Statement (i) follows.

To prove (v), we have to compute the morphism $\xi':M'\to W_S^{(1)}$ induced by $\xi$. We have $\xi'=F\xi\sigma=FV\circ (V^{-1}\xi\sigma )=p\cdot (V^{-1}\xi\sigma )$. Statement (v) follows.

Statement (ii) follows from (i) because $\Sigma_-$ is the preimage in $\Sigma'$ of the open substack 
$\Spf\BZ_p=(\BG_m/\BG_m)\hat\otimes\BZ_p\subset (\BA^1/\BG_m)_-\hat\otimes\BZ_p$.

Statement (iii) follows from (v) because the restriction of $F':\Sigma'\to\Sigma$ to $\Sigma_-=\Sigma$ is the identity morphism.
Statement (iii$'$) is a reformulation of (iii).

Statement (iv) follows ffrom (iii) because the morphism $p: \Spf\BZ_p\to\Sigma$ is not a monomorphism (see Lemma~\ref{l:not mono}).
\end{proof}

\subsubsection{The stack $\Sigma'_{\Hdg}$}   \label{sss:2Hdg}
Let $\Sigma'_{\Hdg}:=\Sigma'_{\bardR}\times_{\Sigma'}\Delta'_0$. We will give an explicit description of 
$\Sigma'_{\Hdg}$ and the morphisms $\Sigma'_{\bardR}\leftarrow\Sigma'_{\Hdg}\to\Delta'_0$.

Let $\{0\}\subset\BA^1=\Spec\BZ [x]$ be the closed subscheme $x=0$; then $\{0\}\simeq\Spec\BZ$. The c-stacks 
$(\{0\}/\BG_m)_\pm$ from \S\ref{sss:general setting for +-} are canonically isomorphic to each other: for any scheme $S$, both $(\{0\}/\BG_m)_+(S)$ and $(\{0\}/\BG_m)_-(S)$ is just the category of line bundles on $S$.

By Lemma~\ref{l:Sigma'_bardR}(i), the composite map 
$\Sigma'_{\bardR}\to\Sigma'\to (\BA^1/\BG_m)_-\hat\otimes\BZ_p$ is an isomorphism. By the definition of $\Sigma'_{\Hdg}$, this isomorphism induces an isomorphism
\[
\Sigma'_{\Hdg}\iso (\{0\}/\BG_m)_-\hat\otimes\BZ_p\subset (\BA^1/\BG_m)_-\hat\otimes\BZ_p.
\]

By Lemma~\ref{l:Delta'_0 as c-stack}, $\Delta'_0=((\BA^1\hat\otimes\BZ_p)^{\dR }/\BG_m)_+$. The canonical morphism
$\Sigma'_{\Hdg}\to\Delta'_0$ is just the morphism
$(\{0\}/\BG_m)_-\hat\otimes\BZ_p=(\{0\}/\BG_m)_+\hat\otimes\BZ_p\to ((\BA^1\hat\otimes\BZ_p)^{\dR }/\BG_m)_+$ induced by the composite map $\{0\}\hat\otimes\BZ_p\mono\BA^1\hat\otimes\BZ_p\to (\BA^1\hat\otimes\BZ_p)^{\dR }$.

\subsection{The c-stack $\Sigma'_+$}   \label{ss:2Sigma'_+}
We will construct a c-stack $\Sigma'_+$, which is closely related to $\Sigma'$ but more understandable because of the
Cartesian square \eqref{e:Sigma'_+ in terms of Sigma}. We will use $\Sigma'_+$ to study $\Sigma'$. (However, the reader may prefer to find alternative proofs of statements about $\Sigma'$ using the presentation $\Sigma'=\sZ/\sG$ from \S\ref{s:Sigma' as quotient}.)

Informally, $\Sigma'_+$ is a ``convenient but slightly wrong" version of $\Sigma'$. (Let me add that $\Sigma'_+$ was my first guess of what $\Sigma'$ should be.)

\subsubsection{The diagram $\Sigma'_+\to\tilde\Sigma'\to\Sigma'$}    \label{sss:SigmaSigmaSigma}
In \S\ref{s:W-modules} we constructed algebraic c-stacks $\Adm$, $\Adm_+$, $\widetilde{\Adm}$, and a diagram of left fibrations $\Adm_+\to\widetilde{\Adm}\to\Adm$, see \S\ref{sss:Adm}, \S\ref{sss:Adm_+ & tildeAdm}, and formula~\eqref{e:AdmAdmAdm}. Base-changing this diagram via the morphism $\Sigma'\to\Adm$ from Corollary~\ref{c:Sigma' to Adm}, we get a diagram of formal c-stacks
\begin{equation} \label{e:SigmaSigmaSigma}
\Sigma'_+\to\tilde\Sigma'\to\Sigma',
\end{equation}
whose arrows are left fibrations. A description of this diagram in ``explicit'' terms can be found in \cite[Appendix~B]{version 1}.

\subsubsection{Diagram \eqref{e:SigmaSigmaSigma} in terms of $S$-points}   \label{sss:explicit SigmaSigmaSigma}
An object of $\tilde\Sigma'(S)$  consists of an object $(M,\xi_M)\in\Sigma'(S)$, an invertible $W_S$-module~$P$, and an isomorphism $P'\iso M'$; the functor $\tilde\Sigma'(S)\to\Sigma'(S)$ forgets $P$.

The category $\Sigma'_+(S)$ has three equivalent descriptions\footnote{Description (iii) is a tautological reformulation of the definition of $\Sigma'_+$ from \S\ref{sss:SigmaSigmaSigma}. To see that (i) and (ii) are equivalent to (iii), one can use Lemma~\ref{l:f' invertible}.}:

(i) $\Sigma'_+(S)$ is the category of triples consisting of an object $(M,\xi_M)\in\Sigma'(S)$, an object $(P,\xi_P)\in\Sigma_+(S)$, and a morphism $(P,\xi_P)\to (M,\xi_M)$.

(ii) An object of $\Sigma'_+(S)$ is an object $(P,\xi_P)\in\Sigma_+(S)$ with an additional piece of data. To define it, note that $\xi_P$ gives rise to a line bundle 
$\sL_P:=P/V(P')$ and a morphism $\bar\xi_P:\sL_P\to\cO_S$. The additional piece of data is a factorization of $\bar\xi_P$ as 
\begin{equation}   \label{e:2factorization as v_-v_+}
\sL_P\overset{v_+}\longrightarrow\sL\overset{v_-}\longrightarrow\cO_S
\end{equation}
 for some line bundle $\sL$.

(iii) An object of $\Sigma'_+(S)$  consists of an object $(M,\xi_M)\in\Sigma'(S)$, an invertible $W_S$-module~$P$, and a morphism $P\to M$ inducing an isomorphism $P'\iso M'$.

The functor $\Sigma'_+(S)\to\tilde\Sigma'(S)$ is clear if one thinks of $\Sigma'_+(S)$ in terms of (iii): the functor forgets
the morphism $P\to M$ but remembers the  isomorphism $P'\iso M'$.

\begin{prop}     \label{p:SigmaSigmaSigma}
(a) The morphisms \eqref{e:SigmaSigmaSigma} are algebraic and faithfully flat.

(b) The morphism $\tilde\Sigma'\to\Sigma'$ is a $\BG_m^\sharp$-gerbe.

(c) The morphism $\Sigma'_+\to\tilde\Sigma'$ is a torsor with respect to a flat group scheme over $\tilde\Sigma'$; the group scheme is fpqc-locally isomorphic to $\BG_a^\sharp\times\tilde\Sigma'$.

(d)  The morphism $\Sigma'_+\to\Sigma'$ is an isomorphism over the open substack $\Sigma_+\subset\Sigma'$.
\end{prop}

\begin{proof}
Statements (a)-(c) follow from Proposition~\ref{p:torsor-gerbe}.
If one thinks of $\Sigma'_+$ in terms of \S\ref{sss:explicit SigmaSigmaSigma}(i) then
statement (d) immediately follows from the fact that $\Sigma'_+$ is a g-stack 
\end{proof}

By \S\ref{sss:explicit SigmaSigmaSigma}(i), faithful flatness of the morphism $\Sigma'_+\to\Sigma'$ can be reformulated as follows.

\begin{cor}   \label{c:recieves morphism from Sigma_+}
Let $S$ be a quasi-compact scheme. Then for every object $(M,\xi)\in\Sigma'(S)$ there exists a quasi-compact faithfully flat $S$-scheme $\tilde S$ such that the  image of $(M,\xi)$ in $\Sigma'(\tilde S)$ receives a morphism from an object of $\Sigma_+(\tilde S)$.   \qed
\end{cor}

\subsubsection{$\Sigma'_+$ in terms of $\Sigma$}   \label{sss:Sigma'_+ in terms of Sigma}
Let $\sX$ be the following c-stack\footnote{$\sX$ is a variant of the c-stack from \S\ref{sss:hyperbolic}.}: for any scheme $S$, an object of $\sX (S)$ is a diagram
\begin{equation}    \label{e:L to M to O_S}
\sL\to\sM\to\cO_S
\end{equation}
of line bundles over $S$, and a morphism in $\sX (S)$ is a morphism of diagrams \eqref{e:L to M to O_S} inducing an \emph{iso}morphism between the line bundles $\sL$. Forgetting the middle term of \eqref{e:L to M to O_S}, we get a morphism from $\sX$ to the g-stack $\BA^1/\BG_m$.

The description of $\Sigma'_+$ from \S\ref{sss:explicit SigmaSigmaSigma}(ii) yields a Cartesian square
\begin{equation}   \label{e:Sigma'_+ in terms of Sigma}
\xymatrix{
\Sigma'_+\ar[r] \ar[d] & \sX\ar[d]\\
\Sigma\ar[r]& \BA^1/\BG_m
}
\end{equation}
whose lower horizontal arrow is obtained by composing the morphism \eqref{e:Sigma to A^1/G_m} with the morphism 
$\hat\BA^1/\BG_m\to\BA^1/\BG_m$.

It is easy to see that the diagram
\begin{equation}   \label{e:F'F}
\xymatrix{
\Sigma'_+\ar[r] \ar[d] & \Sigma'\ar[d]^{F'}\\
\Sigma\ar[r]^{F}& \Sigma
}
\end{equation}
commutes; here the upper horizontal arrow comes from \eqref{e:SigmaSigmaSigma}, and the lower vertical arrow comes from \eqref{e:Sigma'_+ in terms of Sigma}.

\begin{lem}  \label{l:all faithfully flat}
All morphisms in diagram \eqref{e:F'F} are algebraic and faithfully flat.
\end{lem}

\begin{proof}
The horizontal arrows of \eqref{e:F'F} are algebraic and faithfully flat by \S\ref{sss:F:Sigma to Sigma} and Proposition~\ref{p:SigmaSigmaSigma}(a). It remains to show that the left vertical arrow is algebraic and faithfully flat. This follows from the Cartesian square \eqref{e:Sigma'_+ in terms of Sigma}, whose right vertical arrow is clearly faithfully flat.
\end{proof}

\begin{cor}   \label{c:Sigma' flat}
The composite morphism
$\Sigma'\overset{F'}\longrightarrow\Sigma\to\hat\BA^1/\BG_m$ is algebraic and  flat. Here the second arrow is \eqref{e:Sigma to A^1/G_m}.
\end{cor}

\begin{proof}
Algebraicity follows from Corollary~\ref{c:Sigma' algebraic}.
The morphism $F':\Sigma'\to\Sigma$ is flat by Lemma~\ref{l:all faithfully flat}. 
The morphism $\Sigma\to\hat\BA^1/\BG_m$ is flat by Lemma \ref{l:Sigma to A^1/G_m}.
\end{proof}

\begin{lem}   \label{l:fD_pm}
Let $\fD_\pm\subset\Sigma'_+$ be the closed substack defined by the equation $v_\pm=0$, where $v_+$ and $v_-$ are as in \eqref{e:2factorization as v_-v_+}. 

(i) $\fD_\pm$ is an effective Cartier divisor on $\Sigma'_+$, and $\fD_\pm\otimes\BF_p$ is an effective Cartier divisor on $\Sigma'_+\otimes\BF_p$.

(ii) Let $m,n\in\BZ$, $m,n\ge 0$. The effective Cartier divisor $m\fD_+ +n\fD_-$ is right-fibered over $\Sigma'_+$ if $m\le n$ and left-fibered if $m\ge n$ (in particular, $\fD_+$ is left-fibered over $\Sigma'_+$ and $\fD_-$ is right-fibered).

(iii) The pullback to $\Sigma'_+$ of the effective Cartier divisor $\Delta_0\subset\Sigma$ equals $\fD_++\fD_-$.

(iv) $\Sigma_\pm\times_{\Sigma'}\Sigma'_+=\Sigma'_+\setminus\fD_\pm$.

(v) The morphism $\Sigma'_+\to\Sigma$ induces an isomorphism $\Sigma'_+\setminus\fD_\pm\iso\Sigma$.

(vi) The composite morphism $\Sigma\iso\Sigma'_+\setminus\fD_+=\Sigma_+\times_{\Sigma'}\Sigma'_+\to\Sigma_+\iso\Sigma$ equals $\id_\Sigma$.

(vii) The composite morphism $\Sigma\iso\Sigma'_+\setminus\fD_-=\Sigma_-\times_{\Sigma'}\Sigma'_+\to\Sigma_-\iso\Sigma$ equals $F$.
\end{lem}

\begin{proof}
$\fD_+$ and $\fD_-$ are pre-divisors in the sense of \S\ref{sss:Pre-div on schemes}-\ref{sss:Pre-div on stacks}, whose sum is equal to 
$\Delta_0\times_\Sigma~\Sigma'_+$; in particular, $\fD_\pm\subset\Delta_0\times_\Sigma~\Sigma'_+$.
 By Lemmas~\ref{l:all faithfully flat} and \ref{l:Divp is fpqc-local}(i), $\Delta_0\times_\Sigma\Sigma'_+$ is an effective Cartier divisor on~$\Sigma'_+$, and 
$\Delta_0\times_\Sigma(\Sigma'_+\otimes\BF_p)$ is an effective Cartier divisor on~$\Sigma'_+\otimes\BF_p$.
So statement~(i) follows from Lemma~\ref{l:summand of divisor}.

Let $\sX$ be the c-stack from \S\ref{sss:Sigma'_+ in terms of Sigma}. By definition, $\fD_\pm$ is the preimage of a certain effective Cartier divisor $\fD_\pm^{\sX}\subset\sX$. So statement (ii) follows from a similar statement for $m\fD_+^{\sX} +n\fD_-^{\sX}$,
which is checked straightforwardly using \S\ref{sss:left-fibered subcategories}.

The proof of the remaining statements of the lemma is straightforward.
\end{proof}

\begin{cor} \label{c:Delta'_0 Cartier divisor}
$\Delta'_0$ is an effective Cartier divisor on $\Sigma'$.
\end{cor}

\begin{proof}
By Lemmas~\ref{l:Divp is fpqc-local}(ii) and Lemma~\ref{l:all faithfully flat},  it suffices to check that $\Delta'_0\times_{\Sigma'}\Sigma'_+$ is an effective Cartier divisor on $\Sigma'_+$.
This follows from Lemma~\ref{l:fD_pm}(i) because we have $\Delta'_0\times_{\Sigma'}\Sigma'_+=~\fD_-$.
\end{proof}

\subsubsection{The canonical morphism $g:\Sigma'\otimes\BF_p\to\Sigma'_+\otimes\BF_p$}  \label{sss:g}
Define a morphism 
\begin{equation}   \label{e:g}
g:\Sigma'\otimes\BF_p\to\Sigma'_+\otimes\BF_p
\end{equation}
as follows:
for every $\BF_p$-scheme $S$ the functor $g:\Sigma'(S)\to\Sigma'_+  (S)$ takes $(M,\xi )\in\Sigma'(S)$ to the diagram
\begin{equation}   \label{e:g(M,xi)}
P\overset{f_M}\longrightarrow\Fr_S^*M\overset{\Fr_S^*\xi}\longrightarrow W_S, \quad P:=(M')^{(-1)},
\end{equation}
where $f_M:P\to \Fr_S^*M$ is the morphism \eqref{e:4geomFrobenius} (which comes from the geometric Frobenius~$M\to\Fr_S^*M^{(1)}$).

In terms of \S\ref{sss:piece of structure}, $g$ takes $(M,\xi )\in\Sigma'(S)$ to the morphism $F'_+(M,\xi )\to\Fr_{\Sigma'\otimes\BF_p}(M,\xi )$ induced by \eqref{e:queer adjunction}.

\begin{prop}  \label{p:g}
(i) Each of the composite morphisms
\[
\Sigma'_+\otimes\BF_p\to\Sigma'\otimes\BF_p\overset{g}\longrightarrow\Sigma'_+\otimes\BF_p \, ,
\]
\[
\Sigma'\otimes\BF_p\overset{g}\longrightarrow\Sigma'_+\otimes\BF_p\to\Sigma'\otimes\BF_p 
\]
is isomorphic to the Frobenius, and the diagram
\begin{equation}  \label{e:g&F'}
\xymatrix{
\Sigma'\otimes\BF_p\ar[r]^g \ar[rd]_{F'} &\Sigma'_+\otimes\BF_p\ar[d]\\
&\Sigma\otimes\BF_p
}
\end{equation}
commutes. Here the morphism
$\Sigma'_+\otimes\BF_p\to\Sigma'\otimes\BF_p$ comes from \eqref{e:SigmaSigmaSigma}, and the vertical arrow of \eqref{e:g&F'} comes from \eqref{e:Sigma'_+ in terms of Sigma}. 

(ii) If $S$ is a perfect $\BF_p$-scheme then the functor $\Sigma'_+  (S)\to\Sigma'  (S)$ induced by \eqref{e:SigmaSigmaSigma} is an equivalence.
\end{prop}

\begin{proof}
To show that the composite morphism $\Sigma'_+\otimes\BF_p\to\Sigma'\otimes\BF_p\overset{g}\longrightarrow\Sigma'_+\otimes\BF_p$ is isomorphic to the Frobenius, use \S\ref{sss:2-categorical properties}(i). The other parts of statement (i) are clear.
 
Statement (ii) follows from (i).
\end{proof}

\begin{lem}   \label{l:Y_+ in terms of g}
Let $\sY_\pm\subset\Sigma'\otimes\BF_p$ be as in \S\ref{sss:Y_+}. Let $\fD_\pm\subset\Sigma'_+$ be as in Lemma~\ref{l:fD_pm}. Then

(i) $\sY_+=g^{-1}(\fD_+\otimes\BF_p)$;

(ii) $\sY_+\times_{\Sigma'}\Sigma'_+$ is equal to the effective Cartier divisor $p\cdot (\fD_-\otimes\BF_p)$ in $\Sigma'_+\otimes\BF_p$;

(iii) $\sY_-\times_{\Sigma'}\Sigma'_+=\fD_-\otimes\BF_p$.
\end{lem}

\begin{proof}
Let $(M,\xi)\in\Sigma'(S)$. The morphism $f_M$ from diagram \eqref{e:g(M,xi)} induces a morphism $\varphi_M :\sL_P\to\Fr_S^*\sL_M$, where $\sL_P$ and $\sL_M$ are the line bundles on $S$ corresponding to $P$ and $M$ in the usual way. Then $g(M,\xi )\in\fD_+ (S)$ if and only if $\varphi_M =0$. By the definition of $\sY_+$,  one has $(M,\xi)\in\sY_+(S)\Leftrightarrow\varphi_M =0$. This proves (i).

By Proposition~\ref{p:g}(i), the composite morphism 
$\Sigma'_+\otimes\BF_p\to\Sigma'\otimes\BF_p\overset{g}\longrightarrow\Sigma'_+\otimes\BF_p $ 
is equal to the Frobenius. Combining this with statement (i), we get (ii).

Statement (iii) is clear because $\sY_-:=\Delta'_0\otimes\BF_p$.
\end{proof}

\begin{rem}  \label{r:Y_+ and Y_- are divisors}
By Lemma~\ref{l:Y_+ and Y_- are divisors}(i),  the closed substacks $\sY_\pm\subset\Sigma'\otimes\BF_p$ are effective Cartier divisors. This also follows from
Lemma~\ref{l:Y_+ in terms of g}(ii,iii) combined with Lemma~\ref{l:Divp is fpqc-local}(ii) and faithful flatness of $\Sigma'_+$ over $\Sigma'$.
\end{rem}

\subsection{On the reduced part of $\Sigma'$}  \label{ss:Sigma'_red}
Let $\sY_\pm\subset\Sigma'\otimes\BF_p$ be as in \S\ref{sss:Y_+}; 
these are effective Cartier divisors on $\Sigma'\otimes\BF_p$, see Remark~\ref{r:Y_+ and Y_- are divisors}.

\begin{prop}   \label{p:sX'}
(i) $\Sigma'_{\red}$ is equal to the effective Cartier divisor $\sY_+ +\sY_-\subset\Sigma'\otimes\BF_p$.

(ii) $\Sigma'_{\red}\cap\Sigma_\pm=(\Sigma_\pm )_{\red}=j_\pm (\Delta_0\otimes\BF_p )$.

(iii) $\Sigma'_{\red}$ is left-fibered over $\Sigma'$.
\end{prop}   

\begin{proof}
By the definition of $\sX_{\red}$ (see \S\ref{sss:X_red}), we have  $\Sigma'_{\red}\cap\Sigma_\pm\supset (\Sigma_\pm )_{\red}$.
We already know that $\Sigma'_{\red}\subset\sY_++\sY_-$ and $(\sY_++\sY_-)\cap\Sigma_\pm =(\Sigma_\pm )_{\red}$, see \S\ref{sss:reduced part of Sigma'}.
Statement (ii) follows. 

To prove (i), it remains to show that $(\sY_++\sY_-)\cap (\Sigma_+\cup\Sigma_-)$ is schematically dense in  the stack $\sY_++\sY_-\, $, i.e., that $\sY_++\sY_-$ is the smallest closed substack of $\Sigma'$ containing $(\sY_++\sY_-)\cap (\Sigma_+\cup\Sigma_-)$. Moreover, since $\Sigma'_+$ is faithfully flat over $\Sigma'$, it suffices to check a similar density statement after base change to $\Sigma'_+$ and then to $\Sigma'_+\times_\Sigma W_{\prim}$. On $\Sigma'_+\times_\Sigma W_{\prim}$ this is straightforward: by Lemma~\ref{l:Y_+ in terms of g}(ii-iii), we know the preimage of $\sY_++\sY_-$ in $\Sigma'_+$ and in $\Sigma'_+\times_\Sigma W_{\prim}\,$.

Statement (iii) follows from (i) and Proposition~\ref{p:left/right-fibered}(i).
\end{proof}

\begin{cor}  \label{c:strongly adic}
The c-stack $\Sigma'$ is strongly adic in the sense of \S\ref{sss:strongly adic}. 
\end{cor}

\begin{proof}
We will use the morphism $F':\Sigma'\to\Sigma$. By Theorem~\ref{t:Sigma' algebraic}(i,ii) and Lemma~\ref{l:algebraic over formal}, the $F'$-preimage of each algebraic closed substack of $\Sigma$ is algebraic. 
By \eqref{e:(F')^{-1}(Delta_0 otimes F_p)}, 
$$(F')^{-1}(\Sigma_{\red})=(F')^{-1}(\Delta_0\otimes\BF_p)=\sY_++p\sY_-\, .$$ 
Combining this with  Proposition~\ref{p:sX'}(i),
we see that the ideal of $\Sigma'_{\red}$ in $(F')^{-1}(\Sigma_{\red})$ is nilpotent (its $p$-th power is zero). 
Properties (i)-(ii) from \S\ref{sss:strongly adic} follow. Property (iii) from \S\ref{sss:strongly adic} holds because
$\Sigma'_{\red}$ is an effective Cartier divisor in~$\Sigma'\otimes\BF_p$ (see Proposition~\ref{p:sX'}).
\end{proof}

In \S\ref{s:Sigma' as quotient} we will describe the c-stacks $\Sigma'_{\red}$ and $(F')^{-1}(\Delta_0\otimes\BF_p)$ very explicitly, see
Corol\-laries~\ref{c:sR} and \ref{c:Sigma'_red}.

\subsection{A toy model for $\Sigma'$}  \label{ss:toy model}

\subsubsection{What will be done}
In the first part of this subsection  we define a very simple algebraic c-stack $\fS'$ over $\BF_p$ and a morphism $\fS'\to\Sigma'$
such that for any \emph{perfect} $\BF_p$-scheme $S$ the corresponding functor $\fS' (S)\to\Sigma'(S)$ is an equivalence. We think of $\fS'$ as a toy model for $\Sigma'$. 
Let us note that according to \S\ref{sss:(Spec F_p)^prismp}, $\fS'=(\Spec\BF_p)^\prismp\otimes\BF_p$. 

In the second part of this subsection (which begins in \S\ref{sss:sS}) we factor the morphism $\fS'\to\Sigma'$ as $\fS'\to\sS\to\Sigma'$ so that
for any perfect $\BF_p$-scheme $S$ the functors $$\fS' (S)\to\sS (S)\to\Sigma'(S)$$ are equivalences. The c-stack $\sS$ is very simple (just as $\fS'$). 
In \S\ref{s:Sigma' as quotient} we will show that the morphism $\sS\to\Sigma'$ identifies $\Sigma'_{\red}$ with the classifying stack of a certain flat group scheme of infinite type over $\sS$, see Corollary~\ref{c:Sigma'_red}.

\subsubsection{The c-stack $\fS'$}   \label{sss:fS'}

Let $\fS'$ be the c-stack over $\BF_p$ whose category of $S$-points is the category of diagrams 
\begin{equation} \label{e:2v_+,v_-}
 \cO_S\overset{v_+}\longrightarrow\sL\overset{v_-}\longrightarrow\cO_S, \quad v_-v_+=0,
 \end{equation}
where $\sL$ is an invertible $\cO_S$-module. (The morphisms in this category are required to induce the identity on the two copies of $\cO_S$.)

The c-stack $\fS'$ is clearly algebraic. Note that $\fS'$ is a closed substack of the c-stack from~\S\ref{sss:hyperbolic}. Also note that the closed substack of $\fS'$ defined by the equation $v_-=0$ (resp.~$v_+=0$) identifies with $(\BA^1/\BG_m)_+\otimes\BF_p$ (resp.~$(\BA^1/\BG_m)_-\otimes\BF_p$), where $(\BA^1/\BG_m)_\pm$ is as in \S\ref{sss:BA^1/BG_m)_pm}.

Let $C:=\Spec \BF_p[v_+,v_-]/(v_+v_-)$. Taking $\sL=\cO_S$ in diagram \eqref{e:2v_+,v_-}, one gets a canonical faithfully flat morphism $C\to\fS'$.
This morphism identifies the underlying g-stack of $\fS'$ with $C/\BG_m$, where $\BG_m$ acts on $C$ so that $\deg v_\pm=\pm 1$. By Remark~\ref{r:cat acting on scheme}, the c-stack $\fS'$ itself identifies with $C/\Gamma'$, where $\Gamma'$ is an algebraic category (in the sense of Definition~\ref{d:algebraic cat}), 
whose scheme of objects is $C$.  Explicitly, the scheme of morphisms of $\Gamma'$ is the closed subscheme of $\BA^1\times C\times C$ defined by the equations 
$$\tilde v_+=\lambda v_+, \quad v_-=\lambda\tilde  v_-,$$
where $\lambda\in\BA^1_{\BF_p}$, $(v_+,v_-)\in C$ is the source, and $(\tilde v_+,\tilde v_-)\in C$ is the target. A way of visualizing $\Gamma'$ will be suggested in \S\ref{sss:drawing Gamma'}.

Let $\fS_+\subset\fS'$ be the open substack $v_\pm\ne 0$. Let $\fS_-\subset\fS'$ be the open substack $v_-\ne 0$.
Then $\fS_\pm\simeq\Spec\BF_p$.

We think of $(\fS', \fS_+,\fS_-)$ 
as a toy model for $(\Sigma', \Sigma_+,\Sigma_-)$; this point of view will be justified by Proposition~\ref{p:fS' to Sigma'} below.

\subsubsection{Visualizing $\Gamma'$}  \label{sss:drawing Gamma'}
One could draw a picture of $\Gamma'$ as follows: draw the ``coordinate cross'' $C$ and on the axis $v_-=0$ (resp. $v_+=0$) draw arrow(s) pointing towards $(0,0)$ (resp.~ away from $(0,0)$). The arrows symbolize morphisms from the points of the line $v_-=0$ to $(0,0)$ and from $(0,0)$ to the points of the line $v_+=0$. 

\subsubsection{Defining a morphism $\fS'\to\Sigma'$}   \label{sss:fS' to Sigma'}
We give a direct definition. However, the reader may prefer the approach from Remark~\ref{r:fS' to Sigma'} below.

Let $S$ be an $\BF_p$-scheme\footnote{Essentially the same construction works for any $p$-nilpotent scheme if one replaces in \eqref{e:2v_+,v_-} the equation 
$v_-v_+=0$ 
 by the equation $v_-v_+=p$}. Let $\sL , v_+, v_-$ be as in diagram~\eqref{e:2v_+,v_-}. Pushing forward the canonical $W_S$-module extension 
$$0\to (\BG_a^\sharp)_S\to W_S\overset{F}\longrightarrow W_S^{(1)}\to 0$$ 
via $v_+:(\BG_a^\sharp)_S\to\sL^\sharp$, one gets an extension $0\to (\BG_a^\sharp)_S\to M\to W_S^{(1)}\to 0$ and a commutative diagram
\[
\xymatrix{
0\ar[r]&(\BG_a^\sharp)_S\ar[r] \ar[d]^{v_+} & W_S\ar[r]^F\ar[d]&W_S^{(1)}\ar[r]\ar[d]^{\id} &0 \\
0\ar[r]&\sL^\sharp\ar[r] & M\ar[r] & W_S^{(1)}\ar[r] &0
}
\]
Its left square is a push-out diagram, and the diagram 
\begin{equation}   \label{e:push-out diagram}
\xymatrix{
(\BG_a^\sharp)_S\ar@{^{(}->}[r]  \ar[d]_{v_+} & W_S\ar[d]^p\\
\sL^\sharp\ar[r]^{v_-} & W_S
}
\end{equation}
commutes (because $S$ is an $\BF_p$-scheme). So \eqref{e:push-out diagram}
yields a morphism $\xi :M\to W_S$. The pair $(M,\xi)$ defines an $S$-point of $\Sigma'$. So we have constructed a morphism
\begin{equation}   \label{e:fS' to Sigma'}
\fS'\to\Sigma' .
\end{equation}
It is easy to check that 
$$\fS'\times_{\Sigma'}\Sigma_\pm=\fS_\pm .$$

Recall that $\fS_\pm\simeq\Spec\BF_p$. Lemma~\ref{l:Sigma(perfect)} says that if $S$ is a perfect $\BF_p$-scheme then $\Sigma (S)$ is a point; equivalently, it says that  if $S$ is perfect the functor $\fS_\pm (S)\to\Sigma_\pm (S)$  is an equivalence.

\begin{prop}   \label{p:fS' to Sigma'}
If an $\BF_p$-scheme $S$ is perfect then the functor $\fS'(S)\to\Sigma'(S)$  constructed in \S\ref{sss:fS' to Sigma'} is an equivalence.
\end{prop}

\begin{proof}
Let $(M,\xi )\in\Sigma'(S)$. Then  $((M')^{(-1)},(\xi')^{(-1)})\in\Sigma (S)$ (this is the image of  $(M,\xi )$ under $F':\Sigma'\to\Sigma$).
By Lemma~\ref{l:Sigma(perfect)} and perfectness of $S$, the category $\Sigma (S)$ is a point, so there exists a unique isomorphism $${(M',\xi' )}\iso (W_S^{(1)},p).$$ 
Thus diagram \eqref{e:diagram corresponding to xi} can be rewritten as
\begin{equation} \label{e:2diagram corresponding to xi} 
\xymatrix{
0\ar[r]&\sL^\sharp\ar[r] \ar[d]^{v_-} & M\ar[r]^\pi\ar[d]^\xi&W_S^{(1)}\ar[r]\ar[d]^{p} &0 \\
0\ar[r]&(\BG_a^\sharp)_S\ar[r] & W_S\ar[r]^F & W_S^{(1)}\ar[r] &0
}
\end{equation}
Recall that $p:W_S^{(1)}\to W_S^{(1)}$ equals the composite morphism $W_S^{(1)}\overset{V}\longrightarrow W_S\overset{F}\longrightarrow W_S^{(1)}$, so
commutativity of the right square of \eqref{e:2diagram corresponding to xi} means that $F\circ (\xi -V\circ\pi)=0$ or equivalently, $(\xi -V\circ\pi)(M)\subset (\BG_a^\sharp)_S$. We can think of  $\xi-V\circ \pi$ as a splitting of the $v_-$-pushforward of the upper row of \eqref{e:2diagram corresponding to xi}.

Thus we have shown that the groupoid of objects of $\Sigma'(S)$ corresponding to a given pair 
$(\sL, v_-:\sL\to\cO_S)$ equals 
\begin{equation}   \label{e:first Ker}
\Ker (\on{Ex}_W(W_S^{(1)},\sL^\sharp )\overset{v_-}\longrightarrow\on{Ex}_W(W_S^{(1)},(\BG_a^\sharp)_S)),
\end{equation}
where $\on{Ex}_W$ stands for the Picard groupoid of $W_S$-module extensions and $\Ker$ denotes the categorical fiber over $0$. On the other hand, the groupoid of objects of $\fS'(S)$ corresponding to a given pair 
$(\sL, v_-)$ is the set 
\begin{equation}   \label{e:second Ker}
\Ker (H^0(S,\sL)\overset{v_-}\longrightarrow H^0(S,\cO_S)).
\end{equation}
By Lemma~\ref{l:Ex for perfect S}, \eqref{e:first Ker} is equal to \eqref{e:second Ker}.
\end{proof}

\begin{rem}   \label{r:fS' to Sigma'}
One has $\fS'=\Sigma'_+\times_{\Sigma}\Spec\BF_p$, where the morphism $\Spec\BF_p\to\Sigma$ comes from 
$p\in W(\BF_p)$ 
(to see this, think of $\Sigma'_+$ in terms of \S\ref{sss:explicit SigmaSigmaSigma}(ii)). Moreover, the morphism 
$\fS'\to\Sigma'$ from \S\ref{sss:fS' to Sigma'} is the composite map $\Sigma'_+\times_{\Sigma}\Spec\BF_p\to\Sigma'_+\to\Sigma'$, where the second map is as in \S\ref{sss:SigmaSigmaSigma}-\ref{sss:explicit SigmaSigmaSigma}. So 
Proposition~\ref{p:fS' to Sigma'} follows from Lemma~\ref{l:Sigma(perfect)}  and 
Proposition~\ref{p:g}(ii).  
\end{rem}

\begin{cor}  \label{c:topology on |Sigma'|}
(i) The morphism \eqref{e:fS' to Sigma'} induces a homeomorphism $|\fS'|\to |\Sigma'|$, where $|\fS'|$ and $|\Sigma'|$ are the topological spaces corresponding to $\fS'$ and $\Sigma'$ (see \S\ref{sss:|X|}).

(ii) The substacks $\Sigma_\pm\subset\Sigma'$ are the only open substacks of $\Sigma'$ different from $\emptyset$ and $\Sigma'$.
\end{cor}

\begin{proof}
Statement (i) immediately follows from Proposition~\ref{p:fS' to Sigma'}.

Recall that the underlying g-stack of $\fS'$ is $C/\BG_m$, where $C$ is the coordinate cross over~$\BF_p$.
The topological space $|\fS'|=|C/\BG_m|$ is a quotient of $|C|$; as a set, $|\fS'|$ is the set of $\BG_m$-orbits on $C$. So statement (ii) follows from (i) and the equality $\fS'\times_{\Sigma'}\Sigma_\pm=\fS_\pm$. 
\end{proof}

\subsubsection{The c-stack $\sS$}   \label{sss:sS}
For every $\BF_p$-scheme $S$, we define $\sS (S)$ to be the category of diagrams 
\begin{equation}   \label{e:2sS (S)}
\sL\overset{v_-}\longrightarrow\cO_S\overset{u}\longrightarrow\sL^{\otimes p}, \quad uv_-=0,
\end{equation}
where $\sL$ is a line bundle on $S$. Forgetting $u$, we get an affine left fibration 
\begin{equation}  \label{e:2sS to A^1/G_m}
\sS\to (\BA^1/\BG_m)_-\otimes\BF_p
\end{equation}

It is clear that $\sS$ is an algebraic c-stack whose underlying g-stack is a quotient of an $\BF_p$-scheme of finite type by an action of $\BG_m$.

\subsubsection{The morphisms $\fS'\rightleftarrows\sS$}  \label{sss:2fS' & sS}
In \S\ref{sss:fS'} we defined a c-stack $\fS'$ over $\BF_p$ as follows: for any $\BF_p$-scheme $S$, an $S$-point of
$\fS'$ is a diagram
\begin{equation} \label{e:4v_+,v_-}
 \cO_S\overset{v_+}\longrightarrow\sL\overset{v_-}\longrightarrow\cO_S, \quad v_-v_+=0,
 \end{equation}
where $\sL$ is an invertible $\cO_S$-module. Given such a diagram, we get a diagram \eqref{e:2sS (S)} by setting $u:=v_+^{\otimes p}$; thus we get a morphism $\fS'\to\sS$. Define a morphism $\sS\to\fS'$ by associating to a diagram \eqref{e:2sS (S)} the diagram 
$$\cO_S\overset{u}\longrightarrow\sL^{\otimes p}\overset{v_-^{\otimes p}}\longrightarrow\cO_S\, .$$
Each of the composite morphisms $\fS'\to\sS\to\fS'$ and $\sS\to\fS'\to\sS$ is isomorphic to the Frobenius.

Note that the above morphisms $\fS'\to\sS$ and $\sS\to\fS'$ are \emph{not flat.}

\subsubsection{The morphism $\sS\to\Sigma'$}  \label{sss:2sS to sX'}
To an $\BF_p$-scheme $S$ and a diagram \eqref{e:2sS (S)} we have to associate an object $(M,\xi )\in\Sigma'(S)$. To this end, 
 we will use the Teichm\"uller functor $\sL\mapsto [\sL ]$ from \S\ref{ss:2Teichmuller}. We have a canonical exact sequence 
\begin{equation}  \label{e:2Resolving L^sharp}
0\to\sL^\sharp\longrightarrow  [\sL ]\overset{F}\longrightarrow [\sL]'\to 0.
\end{equation}
and a canonical isomorphism $[\sL]'\iso [\sL^{\otimes p}]^{(1)}$, see  \eqref{e:3resolving L^sharp} and \eqref{e:2L' & L^p}.
Let $u$ be as in \eqref{e:2sS (S)}. Pulling back \eqref{e:2Resolving L^sharp} via the morphism $[u]^{(1)}:W_S^{(1)}\to[\sL^{\otimes p}]^{(1)}=[\sL]'$, we get an exact sequence
$0\to\sL^\sharp\longrightarrow  M\longrightarrow W_S^{(1)}\to 0$, which fits into a commutative diagram
\begin{equation} 
\xymatrix{
0\ar[r]&\sL^\sharp\ar[r] \ar@{=}[d] & M\ar[r]\ar[d]&W_S^{(1)}\ar[r]\ar[d]^{[u]^{(1)}} &0 \\
0\ar[r]&  \sL^\sharp\ar[r] & [\sL]\ar[r]^F & [\sL]'\ar[r] &0
}
\end{equation}
By definition, $M$ is a submodule of $[\sL ]\times_SW_S^{(1)}$. Now define $\xi :M\to W_S$ to be the restriction of the morphism $[\sL ]\times_SW_S^{(1)}\overset{([v_-],V)}\longrightarrow W_S$. 
One has a commutative diagram
\begin{equation}   \label{e:diagram for xi:M to W_S}
\xymatrix{
0\ar[r]&\sL^\sharp\ar[r] \ar[d]_{v_-} & M\ar[r]\ar[d]^\xi&W_S^{(1)}\ar[r]\ar[d]^p &0 \\
0\ar[r]&  W_S^{(F)}\ar[r] & W_S\ar[r]^F & W_S^{(1)}\ar[r] &0
}
\end{equation}

\begin{lem}    \label{l:2fS' to sS to Sigma'}
The morphism $\fS'\to\Sigma'$ constructed in \S\ref{sss:fS' to Sigma'} is equal to the composite morphism 
$\fS'\to\sS\to\Sigma'$, where the first arrow is as in \S\ref{sss:2fS' & sS} and the second one is as in \S\ref{sss:2sS to sX'}.
\end{lem}

\begin{proof}
Suppose we are given an $S$-point of $\fS'$, i.e., a diagram \eqref{e:4v_+,v_-}. Applying to it the construction of  \S\ref{sss:fS' to Sigma'}, we get an object $(M,\xi)\in\Sigma'(S)$. Applying to the same $S$-point of $\fS'$ the morphisms $\fS'\to\sS\to\Sigma'$, we get an object $(\underline M,\underline\xi)\in\Sigma'(S)$. Let us identify $M$ 
with~$\underline M$ (checking that after this identification $\xi=\underline\xi$ is left to the reader). To this end, consider the diagram
\[
\xymatrix{
0\ar[r]&(\BG_a^\sharp)_S\ar[r] \ar[d]_{v_+} & W_S\ar[r]^F\ar[d]_{[v_+]} &W_S^{(1)}\ar[r]\ar[d]^{[u]^{(1)}} &0 \\
0\ar[r]&\sL^\sharp\ar[r] & [\sL ]\ar[r]^F & [\sL^{\otimes p}]^{(1)}\ar@{=}[d]\ar[r]&0\\
&&&[\sL]'
}
\]
where $v_+$ is as in \eqref{e:4v_+,v_-} and $u=v_+^{\otimes p}$. The diagram commutes, so the pushforward of the upper row via $v_+:(\BG_a^\sharp)_S\to\sL^\sharp$ identifies with the pullback of the lower row via 
$[u]^{(1)}:W_S^{(1)}\to [\sL^{\otimes p}]^{(1)}$. Thus $M$ identifies with $\underline M$.
\end{proof}

\subsubsection{The pullback of $\cO_{\Sigma'}\{ 1\}$ to $\sS$ and $\fS'$}    \label{sss:cO_fS'}
Let $\sL_{\sS}$ be the weakly invertible covariant $\cO_{\sS}$-module whose pullback via the morphism $S\to\sS$ corresponding to a diagram \eqref{e:2sS (S)} is the line bundle $\sL$ from \eqref{e:2sS (S)}. Similarly, we have the tautological weakly invertible covariant $\cO_{\fS'}$-module  $\sL_{\fS'}$.

\begin{lem}     \label{l:cO_fS'}
The pullback of the Breuil-Kisin-Tate module $\cO_{\Sigma'}\{ 1\}$ to $\sS$ (resp.~$\fS'$) is canonically isomorphic to $\sL_{\sS}$ (resp.~$\sL_{\fS'}$).
\end{lem}

\begin{proof}
It suffices to prove the statement about $\sS$. By \eqref{e:BK-Tate' definition},
$\cO_{\Sigma'}\{1\}:=\sL_{\Sigma'}\otimes F'^*\cO_\Sigma\{1\}$. So it remains to show that the pullback of $F'^*\cO_\Sigma\{1\}$ to $\sS$ is canonically trivial. By \eqref{e:diagram for xi:M to W_S}, one has a commutative diagram
\[
\xymatrix{
\sS\ar[r] \ar[d] & \Sigma'\ar[d]^{F'}\\
\Spec\BF_p\ar[r] & \Sigma
}
\]
Finally, the pullback of $\cO_\Sigma\{1\}$ to $\Spec\BF_p$ is canonically trivial (see \S\ref{sss:trivializing sL_Sigma} and Definition~\ref{d:BK}).
\end{proof}

\section{A realization of $\Sigma'$ as a quotient}   \label{s:2Sigma' as quotient}
As already said in \S\ref{sss:group scheme over c-stack}, by a group scheme over a c-stack $\sY$ we mean a c-stack left-fibered in group schemes over $\sY$.

Recall that $\Sigma'$ is a c-stack left-fibered over $(\BA^1/\BG_m)_-\hat\otimes\BZ_p$. In this section (more precisely, in \S\ref{ss:Sigma'_Rig}) we describe an ``explicit'' realization of $\Sigma'$. More precisely, we construct an isomorphism
\[
\Sigma'\iso\Sigma'_{\Rig}/\fG,
\]
where $\Sigma'_{\Rig}$ is a certain c-stack left-fibered in formal schemes\footnote{By this we mean a left fibration over $(\BA^1/\BG_m)_-\hat\otimes\BZ_p$ whose base change via any morphism from a scheme to $(\BA^1/\BG_m)_-\hat\otimes\BZ_p$ is a formal scheme.} over $(\BA^1/\BG_m)_-\hat\otimes\BZ_p$ and $\fG$ is a certain group scheme over $(\BA^1/\BG_m)_-\hat\otimes\BZ_p$. The idea is to rigidify admissible $W_S$-modules in a certain way (which is described in \S\ref{sss:Adm_Rig}).

In \S\ref{s:Sigma' as quotient} we describe a more economic realization of $\Sigma'$ as a quotient. 
The reader may prefer to look at the more economic realization first.

\subsection{The c-stack $\Adm_{\Rig}$}  \label{ss:Adm_Rig}
\subsubsection{Recollections on the c-stack $\Adm$}   \label{sss:Adm recollections}
Recall that a $W_S$-module is admissible if and only if there exists an exact sequence
\begin{equation}  \label{e:2L^sharp &M'}
0\to\sL^\sharp\to M\to M'\to 0,
\end{equation}
where $\sL$ is a line bundle on $S$ and $M'$ is fpqc-locally isomorphic to $W_S^{(1)}$. By Lemmas~\ref{l:uniqueness of M_0} and \ref{l:2 classes of W-modules}(ii),  this exact sequence is unique and functorial in $M$. Recall that $\Adm$ is the following c-stack: for a scheme $S$, the objects of $\Adm (S)$ are admissible $W_S$-modules and the morphisms are those $W_S$-linear maps $M_1\to M_2$ that induce an \emph{iso}morphism $M'_1\iso M'_2$. One has a left fibration
\[
\Adm\to\lines , \quad M\mapsto\sL ,
\]
where $\lines$ is the c-stack whose $S$-points are line bundles on $S$.

\subsubsection{The c-stack $\Adm_{\Rig}$}    \label{sss:Adm_Rig}
We will use the Teichm\"uller functor $\sL\mapsto [\sL ]$ from \S\ref{ss:2Teichmuller}. 
Let $M\in\Adm (S)$, and let $\sL$ and $M'$ be as in \eqref{e:2L^sharp &M'}. By a \emph{rigidification} of $M$ we will mean a pair $(r,\pi )$, where $r:M\to [\sL ]$ and $\pi :M\to W_S^{(1)}$ are $W_S$-linear morphisms such that the restriction of $r$ to $\sL^\sharp$ is equal to the identity and $\pi$ induces an isomorphism $M'\iso W_S^{(1)}$. Sometimes we will consider $(r,\pi )$ as a single morphism $M\to [\sL ]\oplus W_S^{(1)}$. It is easy to see that this is a monomorphism.

Let $\Adm_{\Rig} (S)$ be the category of triples $(M,r,\pi)$, where $M\in\Adm (S)$ and $(r, \pi )$ is a rigidification of $M$. Then $\Adm_{\Rig}$ is a c-stack left-fibered over $\Adm$.

\subsubsection{Explicit description of $\Adm_{\Rig}$}   \label{sss:Adm_Rig explicitly}
We can think of $\Adm_{\Rig}(S)$ as the category of pairs consisting of a line bundle $\sL$ on $S$ and a commutative diagram with exact rows
\begin{equation}   \label{e:3diagram for M} 
\xymatrix{
0\ar[r]&\sL^\sharp\ar[r] \ar@{=}[d] & M\ar[r]^\pi\ar[d]^r&W_S^{(1)}\ar[r]\ar[d]^{r'} &0 \\
0\ar[r]&  \sL^\sharp\ar[r] & [\sL]\ar[r]^F & [\sL^{\otimes p}]^{(1)}\ar[r] &0
}
\end{equation}
where the lower row is obtained from \eqref{e:3resolving L^sharp} via \eqref{e:2L' & L^p}. We have $r'=\zeta^{(1)}$ for some element
$\zeta\in\Hom_W(W_S, [\sL^{\otimes p}])=[\sL^{\otimes p}](S)$. Since the right square of \eqref{e:3diagram for M} is Cartesian, the whole diagram 
\eqref{e:3diagram for M} can be reconstructed from $\zeta$; in particular,
\begin{equation}  \label{e:M by equations} 
M=\Ker ([\sL]\oplus W_S^{(1)}\overset{(F,-\zeta^{(1)})}\longrightarrow [\sL^{\otimes p}]^{(1)}).
\end{equation} 
Thus $\Adm_{\Rig}(S)$ identifies with the category of pairs $(\sL ,\zeta )$, where $\sL$ is a line bundle over $S$ and $\zeta :S\to [\sL^{\otimes p}]$ is a section.
In particular, we see that the left fibration $\Adm_{\Rig}\to\lines$ is schematic and, moreover, affine.

\subsubsection{The group scheme $\cG$}   \label{sss:cG}
We are going to define a group scheme $\cG$ over the c-stack $\lines$. Later we will show that the morphism $\Adm_{\Rig}\to\Adm$ is a $\cG$-torsor, see
Corollary~\ref{c:cG-torsor}.

For any scheme $S$, let $\cG (S)$ be the category of pairs consisting of a line bundle $\sL$ on $S$ and a matrix
\begin{equation}  \label{e:cG}
\left(\begin{matrix} 1 &\alpha \\0&w\end{matrix}\right), \quad w\in W^\times (S),\, \alpha\in [\sL^{\otimes p}](S).
\end{equation}
Forgetting the matrix, we get a flat affine left fibration $\cG\to\lines$. Moreover, multiplication of matrices makes $\cG$ into a group scheme over $\lines$.

\begin{lem}  \label{l:cG-torsor}
Let $M\in\Adm (S)$, and let $\sL$ be the corresponding line bundle on $S$. Then

(i) $M$ admits a rigidfiication fpqc-locally on $S$;

(ii) the set of all rigidifications $(r,\pi ):M\to [\sL ]\oplus W_S^{(1)}$ carries a free transtitive action of the group $\Aut_1 ([\sL ]\oplus W_S^{(1)})$ formed by those
$W_S$-linear automorphisms of $[\sL ]\oplus W_S^{(1)}$ which induce the identity on $[\sL ]\subset [\sL ]\oplus W_S^{(1)}$;

(iii) the group $\Aut_1 ([\sL ]\oplus W_S^{(1)})$ identifies with the group of matrices \eqref{e:cG} as follows: the automorphism of $[\sL ]\oplus W_S^{(1)}$ corresponding to a matrix $\left(\begin{matrix} 1 &\alpha \\0&w\end{matrix}\right)$ is given by the matrix
\begin{equation}  \label{e:matrix with V}
\left(\begin{matrix} 1 &V\circ \alpha^{(1)} \\0&w^{(1)} \end{matrix}\right),
\end{equation}
where $w^{(1)}:W_S^{(1)}\to W_S^{(1)}$ is induced by $w:W_S\to W_S$ and $V\circ \alpha^{(1)}$ is the composite morphism
$W_S^{(1)}\overset{\alpha^{(1)}}\longrightarrow [\sL^{\otimes p}]^{(1)}\overset{V}\longrightarrow [\sL ]$.
\end{lem}

\begin{proof}
Statement (i) follows from Lemma~\ref{l:rigidification providing zeta}. The verification of (ii) is straightforward. Statement (iii) follows from \eqref{e:Hom(W^(1),W} and \eqref{e:Hom(W^(1),W^(1)}.
\end{proof}

\begin{cor}  \label{c:cG-torsor}
$\Adm_{\Rig}$ is a $\cG$-torsor over $\Adm$, so we get an isomorphism $$\Adm_{\Rig}/\cG\iso\Adm.$$
\end{cor}

\begin{proof}
Follows from Lemma~\ref{l:cG-torsor}.
\end{proof}

\subsubsection{The $\cG$-action in explicit terms}   \label{sss:cG-action explicitly}
By \S\ref{sss:Adm_Rig explicitly}, the fiber of the functor $\Adm_{\Rig}(S)\to\lines (S)$ over $\sL\in\lines (S)$ identifies with the set of sections 
$\zeta:S\to[\sL^{\otimes p}]$. By Lemma~\ref{l:cG-torsor}(ii-iii), the group of matrices \eqref{e:cG} acts on this set. The following lemma describes this action.

\begin{lem}  \label{l:cG-action explicitly}
The action of $\left(\begin{matrix} 1 &\alpha \\0&w\end{matrix}\right)$ on $[\sL^{\otimes p}](S)$ is given by
$\tilde\zeta =w^{-1}(\zeta +p\alpha)$.
\end{lem}

In other words, we get the affine-linear transformation of $\zeta$ with matrix $w^{-1}\cdot\left(\begin{matrix} 1 &p\alpha \\0&w\end{matrix}\right)$. 

\begin{proof}
By Lemma~\ref{l:cG-torsor}(iii), we see that in terms of $r,\pi$ the action is as follows:
\begin{equation}  \label{e:action in terms of r,pi}
\tilde r=r+V\circ\alpha^{(1)}\circ\pi, \quad \tilde\pi =w^{(1)}\circ\pi. 
\end{equation}
By \S\ref{sss:Adm_Rig explicitly}, $(r,\pi )$ and $\zeta$ are related by $F\circ r=\zeta^{(1)}\circ\pi$. Then
\[
F\circ\tilde r=F(r+V\circ\alpha^{(1)}\circ\pi)=(\zeta +p\alpha)^{(1)}\circ\pi=\tilde\zeta^{(1)}\circ\tilde\pi ,
\]
where $\tilde\zeta =w^{-1}(\zeta +p\alpha)$.
\end{proof}

In \S\ref{ss:Sigma'_Rig} we will need the following lemma, whose verification is straightforward.

\begin{lem}  \label{l:dual of rigidifed W-module}
Let $(M,r,\pi )\in \Adm_{\Rig}(S)$. Let $\sL$ be the line bundle on $S$ corresponding to $M$. 
Let $\xi :M\to W_S$ be a $W_S$-linear morphism and let $v_-:\sL\to\cO_S$ be induced by $\xi$. Then 

(i) there exists a unique $\gamma\in\Hom (W_S,W_S)=W(S)$ such that $\xi=[v_-]\circ r+V\circ\gamma^{(1)}\circ\pi$;

(ii) If a rigidification $(\tilde r,\tilde\pi )$ is related to $(r,\pi )$ by \eqref{e:action in terms of r,pi} then $\xi=[v_-]\circ \tilde r+V\circ\tilde\gamma^{(1)}\circ\tilde\pi$, where
$\tilde\gamma =w^{-1}(\gamma-[v_-^{\otimes p}](\alpha ))$. \qed
\end{lem}

\subsection{The c-stack $\Sigma'_{\Rig}$ and its relation to $\Sigma'$}   \label{ss:Sigma'_Rig}
\begin{defin} \label{d:Sigma'_Rig}
$\Sigma'_{\Rig}:=\Sigma'\times_{\Adm}\Adm_{\Rig}$, where $\Adm_{\Rig}$ is as in \S\ref{sss:Adm_Rig}.
\end{defin}

\subsubsection{The group scheme $\fG$ and the isomorphism $\Sigma'_{\Rig}/\fG\iso\Sigma'$}
Recall that $\Sigma'$ is left-fibered over the c-stack $(\BA^1/\BG_m)_-\hat\otimes\BZ_p$. Recall that an $S$-point of $(\BA^1/\BG_m)_-$ is a pair $(\sL , v_-)$, where 
$\sL$ is a line bundle on $S$ and $v_-\in\Hom (\sL ,\cO_S)$. So forgetting $v_-$ we get a morphism $(\BA^1/\BG_m)_-\to\lines$.

In \S\ref{sss:cG} we defined a flat affine group scheme $\cG$ over $\lines$. Let $\fG$ be the pullback of $\cG$ to $(\BA^1/\BG_m)_-\hat\otimes\BZ_p$. Thus an
$S$-point of $\fG$ is an $S$-point $(\sL , v_-)$ of $(\BA^1/\BG_m)_-\hat\otimes\BZ_p$ plus a matrix
\begin{equation}  \label{e:fG}
\left(\begin{matrix} 1 &\alpha \\0&w\end{matrix}\right), \quad w\in W^\times (S),\, \alpha\in [\sL^{\otimes p}](S).
\end{equation}

By Corollary~\ref{c:cG-torsor}, $\Sigma'_{\Rig}$ is a $\fG$-torsor over $\Sigma'$, so we get an isomorphism $$\Sigma'_{\Rig}/\fG\iso\Sigma'.$$

\subsubsection{The problem}  \label{sss:The problem}
The above isomorphism yields a $\fG$-equivariant $W_{\Sigma'_{\Rig}}$-module $M_{\Sigma'_{\Rig}}$ equipped with a $\fG$-equivariant morphism 
$\xi:M_{\Sigma'_{\Rig}}\to W_{\Sigma'_{\Rig}}$. Our goal is to give an explicit description of $\Sigma'_{\Rig}$ and the pair $(M_{\Sigma'_{\Rig}},\xi)$ together with the
action of $\fG$ on $\Sigma'_{\Rig}$ and~$M_{\Sigma'_{\Rig}}$. Such a description immediately follows from \S\ref{ss:Adm_Rig}. Let us just formulate the answer with brief comments.

\subsubsection{The answer}     \label{sss:The answer}
We will use the ``coordinates'' $\zeta$ and $\gamma$ introduced in \S\ref{sss:Adm_Rig explicitly} and Lemma~\ref{l:dual of rigidifed W-module}, respectively.

(i) For any scheme $S$, $\Sigma'_{\Rig}(S)$ is the category of quadruples $(\sL ,v_-,\zeta ,\gamma )$, where 
$$(\sL ,v_-)\in (\BA^1/\BG_m)_-(S), \quad \zeta\in [\sL^{\otimes p}](S), \quad \gamma\in W(S),$$ and the following condition holds:
\begin{equation}  \label{e:2who is primitive}
[v_-^{\otimes p}](\zeta )+p\gamma\in W_{\prim}(S).
\end{equation}
The morphism $\Sigma'_{\Rig}\to (\BA^1/\BG_m)_-\hat\otimes\BZ_p$ forgets  $\zeta$ and $\gamma$.
The morphism $\Sigma'_{\Rig}\to\Adm_{\Rig}$ forgets  $v_-$ and $\gamma$. Condition \eqref{e:2who is primitive} will be explained in (iii) below.

(ii) Let us first describe $M_{\Sigma'_{\Rig}}$ as a c-stack affine over $\Sigma'_{\Rig}$ and then describe the $W_{\Sigma'_{\Rig}}$-module structure on it. The $S$-points of
$M_{\Sigma'_{\Rig}}$  are collections $(\sL ,v_-,\zeta ,\gamma , x,y)$, where $(\sL ,v_-,\zeta ,\gamma )\in \Sigma'_{\Rig}(S)$ 
and $x\in [\sL ](S)$, $y\in W(S)$ satisfy the equation 
\begin{equation}   \label{e:Fx=y zeta}
Fx=y\zeta.
\end{equation}
The $W_{\Sigma'_{\Rig}}$-module structure on $M_{\Sigma'_{\Rig}}$ is as follows: the group operation is addition of pairs $(x,y)$, and an element $a\in W(S)$ acts on 
$M_{\Sigma'_{\Rig}}(S)$ by
\begin{equation} \label{e:y is in W^(1)}
(x,y)\mapsto (ax,F(a)y).
\end{equation}
Formulas \eqref{e:Fx=y zeta}-\eqref{e:y is in W^(1)} are essentially equivalent to \eqref{e:M by equations}.

(iii) The morphism $\xi:M_{\Sigma'_{\Rig}}\to W_{\Sigma'_{\Rig}}$ is given by the formula
\[
\xi (\sL ,v_-,\zeta ,\gamma , x,y)= (\sL ,v_-,\zeta ,\gamma , [v_-]x+V(\gamma y)),
\]
which comes from Lemma~\ref{l:dual of rigidifed W-module}(i).

Note that condition \eqref{e:2who is primitive} is equivalent to primitivity of $\xi$. 
This follows from Lemma~\ref{l:primitivity of xi'} and the formula 
\begin{equation}   \label{e:F(vx+gamma Vy)}
F([v_-]x+\gamma\cdot Vy)=y([v_-^{\otimes p}](\zeta )+p\gamma).
\end{equation}

(iv) A matrix $\left(\begin{matrix} 1 &\alpha \\0&w\end{matrix}\right)\in \fG (S)$ acts as follows:
\begin{equation}  \label{e:tilde zeta}
\tilde\zeta =w^{-1}(\zeta +p\alpha),
\end{equation}
\begin{equation}   \label{e:tilde gamma}
\tilde\gamma =w^{-1}(\gamma-[v_-^{\otimes p}](\alpha )),
\end{equation}
\begin{equation}    \label{e:tilde x}
\tilde x=x+V(\alpha y), \quad \tilde y=wy.
\end{equation}
Formulas \eqref{e:tilde zeta}-\eqref{e:tilde gamma}
already appeared in Lemmas~\ref{l:cG-action explicitly} and \ref{l:dual of rigidifed W-module}(ii),  respectively.
The transformation \eqref{e:tilde x} corresponds to the matrix  \eqref{e:matrix with V}.

\section{A more economic realization of $\Sigma'$ as a quotient}   \label{s:Sigma' as quotient}
As already said in \S\ref{sss:group scheme over c-stack}, by a group scheme over a c-stack $\sY$ we mean a c-stack left-fibered in group schemes over $\sY$.

In \S\ref{ss:Sigma'_Rig} we constructed an isomorphism $\Sigma'\iso\Sigma'_{\Rig}/\fG$, where $\Sigma'_{\Rig}$ is an explicit c-stack left-fibered over 
$(\BA^1/\BG_m)_-\hat\otimes\BZ_p$ and $\fG$ is an explicit group scheme over $(\BA^1/\BG_m)_-\hat\otimes\BZ_p$. We are going to define a closed substack
$\sZ\subset\Sigma'_{\Rig}$ and a commutative group subscheme $\sG\subset\fG$ flat over $(\BA^1/\BG_m)_-\hat\otimes\BZ_p$ such that the morphism
$\sZ/\sG\to\Sigma'_{\Rig}/\fG$ is an isomorphism. Thus we get a more economic realization of $\Sigma'$ as a quotient (namely, $\Sigma'\iso\sZ/\sG$); see \S\ref{ss:def of sZ & sG}-\ref{ss:Comments on sZ$ and sG} for the details. 

In \S\ref{ss:Who is who in terms of sZ} we express the main pieces of structure on $\Sigma'$ in terms of $\sZ/\sG$.

Then we use the isomorphism $\Sigma'\iso\sZ/\sG$ to get an explicit description of the c-stacks 
$\Sigma'_{\red}$ and $(F')^{-1}(\Sigma_{\red})\otimes\BF_p$ (see Corollaries~\ref{c:sR} and \ref{c:Sigma'_red}). We also get a nice presentation of $\Sigma'$ as a projective limit, see~\S\ref{ss:Sigma' as lim}.

Recall that in \S\ref{sss:fS' to Sigma'} we defined a canonical morphism  $\fS'\to\Sigma'_{\red}$, where $\fS'$ is the ``toy model'' from \S\ref{sss:fS'}.
In \S\ref{ss:Sigma'_red to fS'} we define a morphism $\Sigma'_{\red}\to\fS'$ such that each of the composite maps
$\fS'\to\Sigma'_{\red}\to\fS'$ and $\Sigma'_{\red}\to\fS'\to\Sigma'_{\red}$ is equal to the Frobenius.

The BK-Tate module $\cO_{\Sigma'}\{ 1\}$ defines a $\BG_m$-torsor $\widetilde{\Sigma'}\to (\Sigma')^{\rm g}$, where  $(\Sigma')^{\rm g}$ is the g-stack underlying
$\Sigma'$. In \S\ref{ss:tilde Sigma'} we give an explicit description of $\widetilde{\Sigma'}_{\!\red}$, see Proposition~\ref{p:tilde Sigma'}.

\medskip

I find the descriptions of $\Sigma'_{\red}$ and $\widetilde{\Sigma'}_{\!\red}$ from Corollary~\ref{c:Sigma'_red} and Proposition~\ref{p:tilde Sigma'} to be amusing.\footnote{In particular, the isomorphisms $\Sigma'_{\red}\times_{\Sigma'}\Sigma_+\iso\Sigma'_{\red}\times_{\Sigma'}\Sigma_-$ and
 $\widetilde{\Sigma'}_{\!\red}\times_{\Sigma'}\Sigma_+\iso\widetilde{\Sigma'}_{\!\red}\times_{\Sigma'}\Sigma_-$ come from the isomorphism
$(W^\times)^{(F)}\otimes\BF_p\iso (W^{(F)}\times\mu_p)\otimes\BF_p$ from Lemma~\ref{l:G_m^sharp modulo p}.} I was unable to guess them
 before computing.

\subsection{Definition of $\sZ$ and $\sG$}   \label{ss:def of sZ & sG}
\subsubsection{Recollections on $\Sigma'_{\Rig}$ and $\fG$}   \label{sss:Recollections on Sigma'_Rig and fG}
Let $S$ be a scheme over $(\BA^1/\BG_m)_-\hat\otimes\BZ_p$; this means that we are given a line bundle $\sL$ on $S$ and a morphism $v_-:\sL\to\cO_S$.
Then an $S$-point of $\Sigma'_{\Rig}$ is a pair
\begin{equation}  \label{e:3who is primitive}
(\zeta ,\gamma),   \mbox{ where }\zeta\in [\sL^{\otimes p}](S), \quad \gamma\in W(S),\quad  [v_-^{\otimes p}](\zeta )+p\gamma\in W_{\prim}(S).
\end{equation}
On the other hand, $\fG (S)$ is the group of matrices
\begin{equation}  \label{e:2fG}
\left(\begin{matrix} 1 &\alpha \\0&w\end{matrix}\right), \quad w\in W^\times (S),\, \alpha\in [\sL^{\otimes p}](S).
\end{equation}
The action of $\fG$ on $\Sigma'_{\Rig}$ is given by 
\begin{equation}  \label{e:tilde zeta,tilde gamma}
\tilde\zeta =w^{-1}(\zeta +p\alpha), \quad \tilde\gamma =w^{-1}(\gamma-[v_-^{\otimes p}](\alpha )).
\end{equation}

\subsubsection{Definition of $\sZ$ and $\sG$}   \label{sss:def of sZ & sG}
Define $\sZ\subset\Sigma'_{\Rig}$ by the equation $\gamma =1$. Define $\sG\subset\fG$ by the equation $w=1-[v_-^{\otimes p}](\alpha )$. Then the action of $\sG$ given by
\eqref{e:tilde zeta,tilde gamma} preserves $\sZ$.  

\begin{lem}  \label{l:Sigma'=sZ/sG}
The morphism $\sZ/\sG\to\Sigma'_{\Rig}/\fG$ is an isomorphism. 
\end{lem}

\begin{proof}
By \eqref{e:tilde zeta,tilde gamma}, it suffices to show that $\gamma$ is invertible on the locus $v_-=0$. This is clear because
$[v_-^{\otimes p}](\zeta )+p\gamma\in W_{\prim}(S)$.
\end{proof}

In \S\ref{ss:Sigma'_Rig} we constructed an isomorphism $\Sigma'_{\Rig}/\fG\iso\Sigma'$.

\begin{cor}   \label{c:Sigma'=sZ/sG}
The composite morphism $\sZ/\sG\to\Sigma'_{\Rig}/\fG\iso\Sigma'$ is an isomorphism.  \qed
\end{cor}

The morphism $\Sigma'_{\Rig}/\fG\to\Sigma'$ is given by the pair $(M_{\Sigma'_{\Rig}},\xi)$  from \S\ref{sss:The answer}(ii-iv). Accordingly, the morphism
$\sZ/\sG\to\Sigma'$ is given by the pullback of this pair to $\sZ\subset\Sigma'_{\Rig}$ (i.e., by setting $\gamma =1$ in the formulas from \S\ref{ss:Sigma'_Rig}).

\subsection{Comments on $\sZ$ and $\sG$}   \label{ss:Comments on sZ$ and sG}
\subsubsection{Comments on $\sZ$}  \label{sss:def of sZ}
Let $S$ be a scheme over $(\BA^1/\BG_m)_-\hat\otimes\BZ_p$, i.e., a $p$-nilpotent scheme $S$ equipped with a line bundle $\sL$ and a morphism $v_-:\sL\to\cO_S$. 
By \S\ref{sss:Recollections on Sigma'_Rig and fG}-\ref{sss:def of sZ & sG}, the set $\sZ (S)$ consists of sections $\zeta: S\to [\sL^{\otimes p}]$ such that
\begin{equation}  \label{e:primitivity condition}
[v_-^{\otimes p}](\zeta)+p\in W_{\prim}(S).
\end{equation}
Here $W_{\prim}$ is as in \S\ref{sss:W_prim} and $[\sL^{\otimes p}]$ is the invertible $W_S$-module obtained from $\sL^{\otimes p}$ by applying the Teichm\"uller functor (see \S\ref{ss:2Teichmuller} )

Write $\zeta$ as $\sum\limits_{i=0}^\infty V^i[\zeta_i]$, where $\zeta_i\in H^0(S,\sL^{\otimes p^{i+1}})$. If $S$ is $p$-nilpotent then 
\eqref{e:primitivity condition} means that

(i) the function $v_-^{\otimes p}(\zeta_0)\in H^0(S,\cO_S)$ is locally nilpotent,

(ii) the function $v_-^{\otimes {p^2}}(\zeta_1)+1\in H^0(S,\cO_S)$ is invertible.

\noindent If $S$ is not $p$-nilpotent then $\sZ (S)=\emptyset$ because $W_{\prim}(S)=\emptyset$.

The morphism $\sZ\to (\BA^1/\BG_m)_-\hat\otimes\BZ_p$ (forgetting $\zeta$) is a left fibration. Moreover, for any scheme $S$ over
$(\BA^1/\BG_m)_-\hat\otimes\BZ_p$ the fiber product of $\sZ$ and $S$ over $(\BA^1/\BG_m)_-\hat\otimes\BZ_p$ is an affine formal scheme.

\subsubsection{The group scheme $\sG$ and its action on $\sZ$}  \label{sss:sG & its action}
Let $S$, $\sL$, $v_-$ be as before. By \S\ref{sss:Recollections on Sigma'_Rig and fG}-\ref{sss:def of sZ & sG}, $\sG (S)$ is the group of invertible matrices of the form
\begin{equation}  \label{e:fM_alpha}
\fM_\alpha:=\left(\begin{matrix} 1 &\alpha \\0&1-[v_-^{\otimes p}](\alpha )\end{matrix}\right), \quad \alpha\in [\sL^{\otimes p}](S).
\end{equation}
One has $\fM_{\alpha_1}\fM_{\alpha_2} =\fM_{\alpha_1*\alpha_2}$, where
\begin{equation}   \label{e:operation in sG}
\alpha_1*\alpha_2=\alpha_1+\alpha_2-[v_-^{\otimes p}](\alpha_1)\cdot\alpha_2.
\end{equation}
If you wish, $\sG (S)$ is the set of all $\alpha\in [\sL^{\otimes p}](S)$ such that the Witt vector $1-[v_-^{\otimes p}] (\alpha )\in W(S)$ is invertible, and the group operation is given by \eqref{e:operation in sG}.

The action of $\fM_\alpha$ on $\sZ (S)$ is given by the affine-linear transformation
\begin{equation} \label{e:affine-linear transformation}
\zeta\mapsto\frac{\zeta +p\alpha}{1-[v_-^{\otimes p}](\alpha )}.
\end{equation}

Formula~\eqref{e:operation in sG} shows that $\sG$ is commutative.
If $S$ is the spectrum of a field then the fiber product $\sG\times_{(\BA^1/\BG_m)_-}S$ is isomorphic to $W_S^\times$ if $v_-\ne 0$ and to $W_S$ if $v_-= 0$.
More precisely, we have the homomorphism  $\sG\to W^\times\times ((\BA^1/\BG_m)_-\hat\otimes\BZ_p)$ of group schemes over $(\BA^1/\BG_m)_-\hat\otimes\BZ_p$ given by
\begin{equation}    \label{e:sG to W^times}
\alpha\mapsto 1-[v_-^{\otimes p}](\alpha ),  
\end{equation}
whose restriction to the open locus $v_-\ne 0$ is an isomorphism.
Informally, one can think of $\sG$ as the $[-v_-^{\otimes p}]$-rescaled version\footnote{Prototype: for any $\BZ [\lambda ]$-algebra $R$, let $G(R)$ be the set $\{x\in R\,|\,1+\lambda x\in R^\times\}$ equipped with the operation $x_1*x_2=x_1+x_2+\lambda x_1x_2$; then $G$ is a group scheme over $\BZ [\lambda ]$ (this is the $\lambda$-rescaled version of $\BG_m$).}  of $W^\times$; if you wish, $\sG$ is a degeneration of $W^\times$ into $W$.

\subsubsection{Remark}   \label{sss:strange descent}
In \S\ref{sss:def of sZ}-\ref{sss:sG & its action} the pair $(\sL ,v_-)$ occurs only via its tensor power $(\sL^{\otimes p} ,v_-^{\otimes p})$. This means that the group scheme
$\sG$ is a pullback of a certain group scheme $\underline{\sG}$ over $(\BA^1/\BG_m)_-\hat\otimes\BZ_p$ with respect to the morphism
\begin{equation}   \label{e:kind of Frob}
(\BA^1/\BG_m)_-\hat\otimes\BZ_p\to (\BA^1/\BG_m)_-\hat\otimes\BZ_p ,  \quad (\sL ,v_-)\mapsto (\sL^{\otimes p} ,v_-^{\otimes p}),
\end{equation}
and similarly, $\sZ$ is a pullback of a certain $\underline{\sZ}$. So $\Sigma'=\sZ/\sG$ is also a pullback\footnote{However, the universal pair $(M,\xi)$ over $\Sigma'$ is not a pullback via \eqref{e:kind of Frob}.} via the mor\-phism~\eqref{e:kind of Frob}. The author does not understand this conceptually.

\subsection{Who is who}   \label{ss:Who is who in terms of sZ}
\subsubsection{Some preimages}   \label{sss:Some preimages}
Let us describe the preimages in $\sZ$ of some substacks of $\Sigma'$.

(i) The closed substack $\Delta'_0\times_{\Sigma'}\sZ\subset\sZ$ is given by the equation $v_-=0$.
The open substack $\Sigma_-\times_{\Sigma'}\sZ\subset\sZ$ is given by the inequality $v_-\ne 0$.

Recall that $\sY_-:=\Delta'_0\otimes\BF_p$. So $\sY_-\times_{\Sigma'}\sZ$ is the effective Cartier divisor in $\sZ\otimes\BF_p$ given by the equation $v_-=0$.

(ii) In \S\ref{sss:Y_+} we defined the effective Cartier divisor $\sY_+$ in $\Sigma'\otimes\BF_p$. One checks that the divisor $\sY_+\times_{\Sigma'}\sZ\subset\sZ\otimes\BF_p$ is given by the equation $\zeta_0=0$. Reality check: this locus is $\sG$-stable (because of $p$ in the numerator of \eqref{e:affine-linear transformation}).

The open substack $\Sigma_+\times_{\Sigma'}\sZ\subset\sZ$ is the complement of $\sY_+\times_{\Sigma'}\sZ$ in $\sZ$, so it is 
given by the inequality $\zeta_0\ne 0$.

(iii) $\sZ_{\red}$ is an effective Cartier divisor in $\sZ\otimes\BF_p$. It is the sum of the effective Cartier divisors
$\sY_-\times_{\Sigma'}\sZ$ and $\sY_+\times_{\Sigma'}\sZ$. One has $\Sigma'_{\red}\times_{\Sigma'}\sZ=\sZ_{\red}$.

(iv) In Lemma~\ref{l:Sigma'_bardR}(i) we defined a section
\[
\fp :(\BA^1/\BG_m)_-\hat\otimes\BZ_p\iso\Sigma'_{\bardR}\to\Sigma'
\]
 of the morphism $\Sigma'\to (\BA^1/\BG_m)_-\hat\otimes\BZ_p$. One checks that this section is equal to the composite morphism $(\BA^1/\BG_m)_-\hat\otimes\BZ_p\to\sZ\to\sZ/\sG=\Sigma'$, where the first arrow is given by setting $\zeta=0$.

\subsubsection{Some morphisms}   \label{sss:Some morphisms}
(i) 
Recall that $\Sigma =W_{\prim}/W^\times$, where $W^\times$ acts on $W_{\prim}$ by division, see \S\ref{sss:Action of W^* on W_prim}.
Formula~\eqref{e:F(vx+gamma Vy)} (in which we have to set $\gamma =1$) implies that the morphism 
\begin{equation}   \label{e:F'}
\sZ/\sG\iso\Sigma'\overset{F'}\longrightarrow\Sigma =W_{\prim}/W^\times
\end{equation}
is induced by the morphism 
\begin{equation}  \label{e:2xi'}
\sZ\to W_{\prim}, \quad (\sL, v_-,\zeta )\mapsto [v_-^{\otimes p}](\zeta )+p
\end{equation}
and the homomorphism $\sG\to W^\times$ given by \eqref{e:sG to W^times}.
Here $\sL, v_-,\zeta$ are as in \S\ref{sss:def of sZ}.

(ii) Recall that the restriction of $F':\Sigma'\to\Sigma$ to $\Sigma_-\subset\Sigma'$ is the isomorphism $j_-^{-1}$. The fact that this restriction is an isomorphism is clear from \eqref{e:2xi'}: indeed, 
if $v_-$ is invertible then $\zeta$ can be expressed in terms of $u:=[v_-^{\otimes p}](\zeta )+p$. 
The morphism 
$$j_-:W_{\prim}/W^\times=\Sigma\mono\Sigma'=\sZ/\sG$$ 
comes from the morphism $W_{\prim}\to\sZ$ that takes a Witt vector 
$u\in W_{\prim}(S)$ to the following triple $(\sL ,v_-,\zeta )\in\sZ (S)$:
\[
\sL=\cO_S , \quad v_-=\id_{\cO_S}, \quad\zeta=u-p.
\]

(iii) One checks that the isomorphism $j_+^{-1}:\Sigma_+\iso\Sigma$ is induced by the following morphism $\sZ\times_{\Sigma'}\Sigma_+\to\Sigma$:
\begin{equation}   \label{e:sZ_+ to Sigma_+}
(\sL, v_-,\zeta )\mapsto (M,\xi_\zeta :M\to W_S ), \quad M:=[\sL ],\; \xi_\zeta:=[v_-]+V(\zeta^{-1}).
\end{equation}
Note that $\zeta$ is invertible by \S\ref{sss:Some preimages}(ii),
$\zeta^{-1}$ is a section of $[\sL^{\otimes (-p)}]=[\sL^{-1}]'$, and $V$ is a morphism $[\sL^{-1}]'\to [\sL^{-1}]$ (see~\S\ref{ss:2Teichmuller}), so
$V(\zeta^{-1})$ is a section of $[\sL^{-1}]$, i.e., a morphism $[\sL ]\to W_S$. 

(iii$'$) The factorization of the morphism \eqref{e:sZ_+ to Sigma_+} as $\sZ\times_{\Sigma'}\Sigma_+\to (\sZ\times_{\Sigma'}\Sigma_+)/\sG\iso\Sigma$ can be seen from the following: one can check that 
\begin{equation}  \label{e:zeta & tilde zeta}
\mbox{ if } \tilde\zeta =\frac{\zeta +p\alpha}{1-[v_-^{\otimes p}](\alpha )} \mbox { then } \xi_{\tilde\zeta}=(1+V(\zeta^{-1}\alpha))^{-1}\xi_\zeta. 
\end{equation}

\begin{rem}   \label{r:price for economy}
It seems that the rather complicated formula \eqref{e:zeta & tilde zeta} 
is a price for using the ``economic" presentation $\Sigma'=\sZ/\sG$. This presentation of $\Sigma'$ was derived from the presentation $\Sigma' =\Sigma'_{\Rig}/\fG$ using the ``normalization'' $\gamma =1$ (see \S\ref{sss:def of sZ & sG});
on  the other hand, to describe the isomorphism $j_+^{-1}:\Sigma_+\iso\Sigma$ it is more convenient to use something like\footnote{The caveat is due to the fact that $\zeta$ is a section of $[\sL^{\otimes p}]$, so setting $\zeta =1$ is possible only after choosing an isomorphism $\sL^{\otimes p}\iso\cO_S$.} the ``normalization'' $\zeta =1$.
\end{rem}

\subsection{Description of the c-stack $(F')^{-1}(\Delta_0)\otimes\BF_p$}
\subsubsection{The preimage of $(F')^{-1}(\Delta_0\otimes\BF_p)$ in $\sZ$ }   \label{sss:sT}
Let $\sT\subset\sZ\otimes\BF_p$ be the divisor $v_-^{\otimes p}(\zeta_0)=0$, where $\zeta_0$ is the $0$th component of the ``Witt vector" $\zeta$.
By \S\ref{sss:Some morphisms}(i), the preimage in $\sZ$ of the substack $(F')^{-1}(\Delta_0)\otimes\BF_p\subset\Sigma'$ equals $\sT$, so the isomorphism 
$\sZ/\sG\iso\Sigma'$ from Corollary~\ref{c:Sigma'=sZ/sG} induces an isomorphism
\begin{equation}   \label{e:sR/sG}
\sT/\sG\iso (F')^{-1}(\Delta_0)\otimes\BF_p .
\end{equation}

\subsubsection{The substack $\sT^{\Teich}\subset\sT$}      \label{sss:sT Teich}
Let $\sZ^{\Teich}\subset\sZ$  be defined by the equation $\zeta=[\zeta_0]$ (in other words, by the condition that $\zeta$ is Teichm\"uller).
Let $\sT^{\Teich}:=\sT\cap\sZ^{\Teich}$.  

\subsubsection{The group scheme $\sG^{(F)}\subset\sG$}  \label{sss:sG^(F)}
Recall that for any scheme $S$ over $(\BA^1/\BG_m)_-\hat\otimes\BZ_p$, the set $\sG (S)$ consists of all $\alpha\in [\sL^{\otimes p}](S)$ such that $1-[v_-^{\otimes p}] (\alpha )\in W^\times (S)$, and the group operation on $\sG (S)$ is given by $\alpha_1*\alpha_2=\alpha_1+\alpha_2-[v_-^{\otimes p}](\alpha_1)\cdot\alpha_2$.

Let $\sG^{(F)}\subset\sG$ be the group subscheme defined by the condition $F(\alpha )=0$. If $S$ is the spectrum of a field then $\sG\times_{(\BA^1/\BG_m)_-}S$ is isomorphic to $(W_S^\times )^{(F)}=(\BG_m^\sharp)_S$ if $v_-\ne 0$ and to $W_S^{(F)}=(\BG_a^\sharp)_S$ if $v_-= 0$. Similarly to \S\ref{sss:sG & its action}, one can think of
$\sG^{(F)}$ as the $(-v_-^p)$-rescaled version of the multiplicative group of the ring $\BG_a^\sharp$; if you wish, $\sG^{(F)}$ is a degeneration of $\BG_m^\sharp$ into $\BG_a^\sharp$. 

Note that the map 
$\sG^{(F)}\to (\BA^1/\BG_m)_-\hat\otimes\BZ_p$ is a flat universal homeomorphism of infinite type (because if you disregard the group structure then $\sG^{(F)}$ is just the Cartesian product of $W^{(F)}$ and $(\BA^1/\BG_m)_-\hat\otimes\BZ_p$).

\begin{rem}  \label{r:2strange descent}
As noted in \S\ref{sss:strange descent}, the group scheme  $\sG$ is the pullback of a certain group scheme $\underline{\sG}$ with respect to the morphism
\eqref{e:kind of Frob}. Similarly, the subgroup $\sG^{(F)}\subset\sG$ is the pullback of a certain subgroup $\underline{\sG}^{(F)}\subset\underline{\sG}$ via the 
morphism \eqref{e:kind of Frob}. We will use this in \S\ref{sss:tilde Sigma'_red to C_1/H} and in the proof of Lemma~\ref{l:open parts of Sigma'_red}.
\end{rem}

\begin{prop}   \label{p:sR}
Let $\sG^{(F)}$ be the group scheme from \S\ref{sss:sG^(F)}. Let $\sT$ and $\sT^{\Teich}$ be as in \S\ref{sss:sT}-\ref{sss:sT Teich}.
Then

(i) the action of $\sG^{(F)}$ on $\sT$ is trivial;

(ii) the action morphism $(\sG/\sG^{(F)})\times_{(\BA^1/\BG_m)_-}\sT\to\sT$ induces an isomorphism
$$(\sG/\sG^{(F)})\times_{(\BA^1/\BG_m)_-}\sT^{\Teich}\iso\sT.$$
\end{prop}

\begin{proof}
Let $S$ be a scheme over $(\BA^1/\BG_m)_-\otimes\BF_p$. If $\zeta_0\in H^0 (S, \sL^{\otimes p})$ is such that 
$v_-^{\otimes p}(\zeta_0)=~0$ then for every section $\alpha :S\to [\sL^{\otimes p}]$ we have
\begin{equation}   \label{e:2affine-linear}
\frac{[\zeta_0] +p\alpha}{1-[v_-^{\otimes p}](\alpha )}=[\zeta_0]+V((h( F(\alpha)))), \quad \mbox{where } h(\beta ):=\frac{\beta}{1-[v_-^{\otimes p^2}](\beta )}\,.
\end{equation}
(We have used the identity $VF=p$, which holds because $S$ is over $\BF_p$.) 
Formula~\eqref{e:2affine-linear} implies that the action of $\sG^{(F)}$ on $\sT^{\Teich}$ is trivial. It also implies (ii). Statement (i) follows.
\end{proof}

Combining Proposition~\ref{p:sR} with \eqref{e:sR/sG}, we get

\begin{cor}  \label{c:sR}
The morphism $\sT^{\Teich}\to\sT/\sG=(F')^{-1}(\Delta_0)\otimes\BF_p$ identifies $(F')^{-1}(\Delta_0)\otimes\BF_p$ with the classifying stack of the pullback of 
$\sG^{(F)}$ to $\sT^{\Teich}$.  \qed
\end{cor}

\subsection{Description of $\Sigma'_{\red}$} 
In \S\ref{sss:sS} we defined a c-stack $\sS$, and in \S\ref{sss:2sS to sX'} we defined a morphism $\sS\to\Sigma'$. It factors through $\Sigma'_{\red}$ because
 $\sS$ is reduced.

\begin{cor}  \label{c:Sigma'_red}
The morphism $\sS\to\Sigma'_{\red}$ from \S\ref{sss:2sS to sX'} identifies $\Sigma'_{\red}$ with the classifying stack of the pullback of $\sG^{(F)}$ to $\sS$, where
$\sG^{(F)}$ is as in \S\ref{sss:sG^(F)}.
\end{cor}

\begin{proof}
We have $\Sigma'_{\red}\subset (F')^{-1}(\Delta_0)\otimes\BF_p$ and $\sZ_{\red}\subset\sT$. Moreover, 
the isomorphism $\sZ_{\red}/\sG\iso\Sigma'_{\red}$ is the restriction of \eqref{e:sR/sG}. So by Corollary~\ref{c:sR}, $\Sigma'_{\red}$ identifies with the classifying stack of the pullback of $\sG^{(F)}$ to $\sZ_{\red}\cap\sZ^{\Teich}$. The description of $\sZ_{\red}$ from \S\ref{sss:Some preimages}(iii) shows that 
$\sZ_{\red}\cap\sZ^{\Teich}$ identifies with the c-stack $\sS$ from \S\ref{sss:sS} so that the morphism $\sZ_{\red}\cap\sZ^{\Teich}\to\Sigma'$ identifies with the morphism $\sS\to\Sigma'$ from \S\ref{sss:2sS to sX'}.   
\end{proof}

Recall that by definition, $\sS$ is the c-stack over $\BF_p$ whose $S$-points are diagrams
\[
\sL\overset{v_-}\longrightarrow\cO_S\overset{u}\longrightarrow\sL^{\otimes p}, \quad uv_-=0,
\]
where $\sL$ is a line bundle on $S$. Let $\sS_{v_-=0}$ (resp.~$\sS_{u=0}$) be the closed substack of $\sS$ defined by the condition $v_-=0$ 
(resp.~$u=0$). Similarly, we have the open substacks $\sS_{v_-\ne 0}$ and~$\sS_{u\ne 0}$.

\begin{lem}  \label{l:open parts of Sigma'_red}
(i) $\sS\times_{\Sigma'}\Sigma_-=\sS_{v_-\ne 0}=\Spec\BF_p$.

(i\,$'$) The pullback of $\sG^{(F)}$ to  $\sS_{v_-\ne 0}$ equals $(W^\times )^{(F)}\times\sS_{v_-\ne 0}=(W^\times )^{(F)}\otimes\BF_p$.

(i\,$''$)   The morphism $\sS\times_{\Sigma'}\Sigma_-\to\Sigma'_{\red}\cap\Sigma_-$ identifies $\Sigma'_{\red}\cap\Sigma_-$ with the classifying stack
$(\Spec\BF_p)/(W^\times)^{(F)}$.

(ii) $\sS\times_{\Sigma'}\Sigma_+=\sS_{u\ne 0}=(\Spec\BF_p)/\mu_p\,$.

(ii\,$'$) The pullback of $\sG^{(F)}$ to  $\sS_{u\ne 0}$ is the constant group scheme $W^{(F)}\times\sS_{u\ne 0}\,$.

(ii\,$''$) The morphism $\sS\times_{\Sigma'}\Sigma_+\to\Sigma'_{\red}\cap\Sigma_+$ identifies $\Sigma'_{\red}\cap\Sigma_+$ with the classifying stack
$(\Spec\BF_p)/(\mu_p\times W^{(F)})$.
\end{lem}

\begin{proof}
Statements (i) and (ii) follow from the description of $\Sigma_\pm\times_{\Sigma'}\sZ$ in \S\ref{sss:Some preimages}(i-ii).

Statement (i\,$'$) is clear from \S\ref{sss:sG^(F)}.
Statement (ii\,$'$) follows from Remark~\ref{r:2strange descent}, the inclusion $\sS_{u\ne 0}\subset\sS_{v=0}\,$, and the fact that on $\sS_{u\ne 0}$ the line bundle $\sL^{\otimes p}$ is trivialized via $u$.

Statements (i$''$) and (ii$''$) follow from (i$'$) and (ii$'$) combined with Corollary~\ref{c:Sigma'_red}.
\end{proof}

\begin{rem}
The isomorphism $j_-j_+^{-1}:\Sigma_+\to\Sigma_-$ induces an isomorphism $$\Sigma'_{\red}\cap\Sigma_+\iso\Sigma'_{\red}\cap\Sigma_-.\,$$ The latter induces by 
Lemma~\ref{l:open parts of Sigma'_red}(i\,$'$, ii\,$'$) an isomorphism  $$(W^{(F)}\times\mu_p)\otimes\BF_p \iso (W^\times)^{(F)}\otimes\BF_p.$$
Using \S\ref{sss:Some morphisms}, one checks that this is the isomorphism \eqref{e:G_m^sharp modulo p}.
\end{rem}

\begin{lem}  \label{l:sS_v_-= & sS_u=}
One has
\begin{equation}  \label{e:2sS_v_-=}
\sS\times_{\Sigma'}\sY_-=\sS_{v_-=0}\, ,
\end{equation}
\begin{equation}  \label{e:2sS_u=}
\sS\times_{\Sigma'}\sY_+=\sS_{u=0}\, .
\end{equation}
\end{lem}

\begin{proof}
\eqref{e:2sS_v_-=} is clear because $\sY_-:=\Delta'_0\otimes\BF_p$.   One has $\sS=\sZ_{\red}\cap\sZ^{\Teich}$, so \eqref{e:2sS_u=} follows from 
\S\ref{sss:Some preimages}(ii).
\end{proof}

\begin{cor}    \label{c:Y_pm}
(i) $\sY_+$ identifies with the classifying stack of the pullback of $\sG^{(F)}$ to $\sS_{u=0}$.

(ii) $\sY_-$ identifies with the classifying stack of the pullback of $\sG^{(F)}$ to $\sS_{v_-=0}$. 
\end{cor}

Note that the pullback of $\sG^{(F)}$ to $\sS_{v_-=0}$ is an fpqc-locally constant group scheme with fiber $\BG_a^\sharp$. 

\begin{proof}
Combine Corollary~\ref{c:Sigma'_red} and Lemma~\ref{l:sS_v_-= & sS_u=}.
\end{proof}

\begin{rem}
One can easily deduce Corollary~\ref{c:Y_pm}(ii) from the description of $\Delta'_0$ given in~\S\ref{ss:Delta'_0 as c-stack} (recall that $\sY_-:=\Delta'_0\otimes\BF_p$).
\end{rem}

\begin{rem}  \label{r:2kind of complement}
One has $\sS_{u=0}=\Sigma'_{\bardR}\otimes\BF_p$, where $\Sigma'_{\bardR}$ is as in \S\ref{ss:Hdg & bar dR}. 
More precisely, both $\sS_{u=0}$ and $\Sigma'_{\bardR}\otimes\BF_p$ are c-stacks over $\Sigma'\otimes\BF_p$, and it is straightforward to check that they are uniquely isomorphic as such. 
Thus Corollary~\ref{c:Y_pm}(i) describes the relation between $\Sigma'_{\bardR}\otimes\BF_p$ and 
the closed substack $\sY_+\subset\Sigma'\otimes\BF_p$. 
\end{rem}

\subsubsection{$F':\Sigma'_{\red}\to\Sigma_{\red}$ as a morphism of classifying stacks} \label{sss:F'_red}
By Corollary~\ref{c:Sigma'_red} and Lem\-ma~\ref{l:Delta_0 as classifying stack}, 
$\Sigma'_{\red}$ is the classifying stack of the pullback of $\sG^{(F)}$ to $\sS$, and $\Sigma_{\red}=\Delta_0\otimes~\BF_p$ is  the classifying stack of the
group scheme $(W^\times )^{(F)}\otimes\BF_p$ over $\BF_p$. By \S\ref{sss:Some morphisms}(i), the morphism $F':~\Sigma'_{\red}\to\Sigma_{\red}$ comes from the homomorphism $\sG^{(F)}\to (W^\times)^{(F)}$ induced by the homomorphism $\sG\to W^\times$ given by \eqref{e:sG to W^times}.

\begin{lem}      \label{l:BK-Tate on Sigma'_red}
(i) Let $\cO_{\Sigma'_{\red}}\{ 1\}$ be the restriction of the BK-Tate module $\cO_{\Sigma'}\{ 1\}$ to $\Sigma'_{\red}$. Then 
$\cO_{\Sigma'_{\red}}\{ 1\}=\sL_{\Sigma'_{\red}}\otimes\sM$, where $\sL_{\Sigma'_{\red}}$ is the restriction of $\sL_{\Sigma'}$ to $\Sigma'_{\red}$ and
\[
\sM:=(F')^*\sL_{\Sigma_{\red}}.
\]

(ii) $\sM$ is a strongly invertible $\cO_{\Sigma'_{\red}}$-module.

(iii) One has a canonical isomorphism $\sM^{\otimes p}=\cO_{\Sigma'_{\red}}$.

(iv) By virtue of Corollary~\ref{c:Sigma'_red}, consider $\sM$ as a line bundle on $\sS$ equipped with an action of $\sG^{(F)}$. As such, $\sM$ is the trivial line bundle $\cO_{\sS}$ equipped with the action given by the composite homomorphism 
\[
\sG^{(F)}\to(W^\times_{\sS})^{(F)}\to (\mu_p)_{\sS}\mono(\BG_m)_{\sS} ,
\]
of group schemes over $\sS$, where the first map is given by \eqref{e:sG to W^times} and the second one takes a Witt vector to its $0$th component.
\end{lem}

\begin{proof}
To prove (i), combine \eqref{e:BK-Tate' definition} with Lemma~\ref{l:restricting BK to Sigma_red}(i).

Since $\Sigma_{\red}$ is a g-stack,  $\sL_{\Sigma_{\red}}$ is strongly invertible.  Statement (ii) follows because $\sM:=(F')^*\sL_{\Sigma_{\red}}$. 

Statements (iii) follows from Lemma~\ref{l:restricting BK to Sigma_red}(ii).

Statement (iv) follows from \S\ref{sss:F'_red} and Lemma~\ref{l:restricting BK to Sigma_red}(iii). 
\end{proof}

\subsection{$\Sigma'$ as a projective limit} \label{ss:Sigma' as lim}
Proposition~\ref{p:Sigma as limit} provides a nice presentation of $\Sigma$ as a projective limit. 
We will describe a similar presentation of $\Sigma'$.

\subsubsection{The stacks $\Sigma'_n$}   \label{sss:Sigma'_n}
For each $n\in\BN$ define $\sG_n$ similarly to the definition of $\sG$ in \S\ref{ss:def of sZ & sG} (or its description in \S\ref{sss:sG & its action})
but with $[\sL^{\otimes p}]$ replaced by 
$[\sL^{\otimes p}]_n:=[\sL^{\otimes p}]/V^n([\sL^{\otimes {p^{n+1}}}])$. Define $\sZ_n$ similarly to the definition of $\sZ$ in \S\ref{ss:def of sZ & sG} (or its description in \S\ref{sss:def of sZ}) but with 
$[\sL^{\otimes p}]$ replaced by $[\sL^{\otimes p}]_n$ and $W_{\prim}$ replaced by the formal scheme $(W_n)_{\prim}$ from \S\ref{sss:Sigma_n}. Set 
$$\Sigma'_n:=\sZ_n/\sG_n.$$
By Corollary~\ref{c:Sigma'=sZ/sG}, $\Sigma'=\sZ/\sG$, so
\begin{equation}   \label{e:Sigma' as lim}
\Sigma'= \underset{n}{\underset{\longleftarrow}{\lim}}\,\Sigma'_n .
\end{equation}

\subsubsection{Description of $\sZ_1$}  \label{sss:sZ_1}
The case $n=1$ is easy because $[\sL^{\otimes p}]_1=\sL^{\otimes p}$. In particular, $\sZ_1$ is as follows. Consider the c-stack $\cA$ of finite presentation over $\BZ$ whose $S$-points are diagrams $\sL\overset{v_-}\longrightarrow\cO_S\overset{u}\longrightarrow\sL^{\otimes p}$, where $\sL$ is a line bundle on $S$. Then $\sZ_1$ is the formal completion of $\cA$ along the closed substack defined by the equations $uv_-=0$, $p=0$. Note that 
$$(\sZ_1)_{\red}=\sS,$$ 
where $\sS$ is as in \S\ref{sss:sS}. 

\subsubsection{$\Sigma'_1\otimes\BF_p$ and $(\Sigma'_1)_{\red}$}   \label{sss:Sigma'_1_red}
The action of $\sG_1$ on $\sZ_1\otimes\BF_p$ is trivial (the term $p\alpha$ in the numerator of \eqref{e:affine-linear transformation} becomes zero). So 
$\Sigma'_1\otimes\BF_p$ is the classifying stack of the pullback of $\sG_1$ to $\sZ_1\otimes\BF_p$, and $(\Sigma'_1)_{\red}$ is the classifying stack of the pullback of $\sG_1$ to $(\sZ_1)_{\red}=\sS$.

\begin{lem}    \label{l:Sigma' as limit}
(i) $\Sigma'_1$ is the formal completion of a smooth algebraic c-stack over $\BZ$ of relative dimen\-sion~$0$ along an effective divisor in its reduction modulo $p$. 

(ii) Each morphism $\Sigma'_{n+1}\to\Sigma'_n$ is an algebraic left fibration. It is smooth, of finite presentation, and of pure relative dimension $0$. \qed
\end{lem}

\subsubsection{$\Sigma_\pm$ as a projective limit}
Let $(\Sigma_-)_1\subset\Sigma'_1$ be the open substack $v_-\ne 0$. Let $(\Sigma_+)_1\subset\Sigma'_1$ be the open substack whose preimage in $\sZ_1$ is the locus $u\ne 0$, where $u$ is as in \S\ref{sss:sZ_1}. 
By~\S\ref{sss:Some preimages}(i-ii), $\Sigma_\pm=\Sigma'\times_{\Sigma'_1}(\Sigma_\pm)_1$, so
\[
\Sigma_\pm= \underset{\longleftarrow}{\lim}\,(\Sigma_\pm)_n\, , \mbox{ where } (\Sigma_\pm)_n:=\Sigma'_n\times_{\Sigma'_1}(\Sigma_\pm)_1 .
\]
On the other hand, one has a canonical isomorphism $\Sigma_\pm\iso\Sigma$, and $\Sigma$ is the projective limit of the stacks $\Sigma_n$ from \S\ref{sss:Sigma_n}.
Thus we get three presentations of $\Sigma$ as a projective limit. It is easy to check that they are related as follows:
\[
(\Sigma_-)_n=\Sigma_n, \quad (\Sigma_+)_n=\Sigma_{n+1}.
\]

\subsubsection{A morphism $\sZ_1\to\Sigma'$}
Consider the closed substack $\sZ^{\Teich}\subset\sZ$  defined by the equation $\zeta=[\zeta_0]$ (i.e., by the condition that $\zeta$ is Teichm\"uller). The map $\sZ^{\Teich}\to\sZ_1$ is an isomorphism, so we get a closed embedding $\sZ_1\mono \sZ$ and therefore a morphism
\begin{equation}   \label{e:sZ_1 to Sigma'}
\sZ_1\to\Sigma' .
\end{equation}
The morphism \eqref{e:sZ_1 to Sigma'} is schematic because the morphism $\sZ_1\to\Sigma'_1$ is.

Let us prove the following analog of Proposition~\ref{p:Sigma as limit}(v,vii). 

\begin{prop}  \label{p:Sigma' as limit}
(i) Let $\sX$ be any c-stack which is algebraic over $\Sigma'$ and flat over $\Sigma'_1$. Then $\sX$ is flat over~$\Sigma'$.

(ii) The morphism \eqref{e:sZ_1 to Sigma'} is a flat universal homeomorphism\footnote{A schematic morphism of stacks $\sX\to\sY$ is said to be a universal homeomorphism if it becomes such after any base change $S\to\sY$ with $S$ being a scheme.} of infinite type. 
\end{prop}

Combining statement (ii) with \S\ref{sss:sZ_1}, we see that the formal c-stack $\Sigma'$ is not far from the familiar world.

To prove Proposition~\ref{p:Sigma' as limit}, we need the following

\begin{lem}   \label{l:gerbes&flatness}
Suppose we have a commutative diagram of algebraic stacks
\[
\xymatrix{
\sP\ar[r]^f&\sR_1\ar[r]^g\ar[d]_{\pi_1} &\sR_2\ar[ld]^{\pi_2}\\
&\sR_3&
}
\]
in which $\pi_1$ and $\pi_2$ are gerbes\footnote{By this we just mean here that the corresponding morphisms between the underlying g-stacks are gerbes.}. If $g\circ f$ is flat then so is $f$.
\end{lem}

\begin{proof}
We can assume that $\pi_2$ is an 
isomorphism, so $g$ is a gerbe. Let us factor $f$ as $\sP\overset{\alpha}\longrightarrow\sP\times_{\sR_2}\sR_1 \overset{\beta}\longrightarrow\sR_1.$ We claim that $\alpha$ and $\beta$ are flat. Indeed, $\beta$ is a base change of the flat morphism $g\circ f:\sP\to\sR_2$, and $\alpha$ is a base change of the diagonal morphism $\sR_1\to\sR_1\times_{\sR_2}\sR_1$, which is \emph{flat} because $g$ is a gerbe.
\end{proof}

\begin{proof}[Proof of Proposition~\ref{p:Sigma' as limit}]
To prove (i), it suffices to show that $\sX\times_{\Sigma'_1}(\Sigma'_1)_{\red}$ is flat over
$\Sigma'\times_{\Sigma'_1}(\Sigma'_1)_{\red}$. To this end, note that $\Sigma'\times_{\Sigma'_1}(\Sigma'_1)_{\red}=\Sigma'_{\red}$ (see the smoothness part of Lemma~\ref{l:Sigma' as limit}(ii)) and
apply Lemma~\ref{l:gerbes&flatness} in the following situation:
$$\sR_1=\Sigma'_{\red}, \quad \sR_2=(\Sigma'_1)_{\red}, \quad \sR_3=\sS , \quad \sP =\sX\times_{\Sigma'}\Sigma'_{\red}=\sX\times_{\Sigma'_1}(\Sigma'_1)_{\red}\,.$$ Lemma~\ref{l:gerbes&flatness} is applicable by Corollary~\ref{c:Sigma'_red} and \S\ref{sss:Sigma'_1_red}.

The flatness part of (ii) follows from~(i). To prove that  \eqref{e:sZ_1 to Sigma'} is a universal homeomorphism of infinite type, it suffices to prove these properties for the  morphism $\sS=(\sZ_1)_{\red}\to\Sigma'_{\red}$  induced by  \eqref{e:sZ_1 to Sigma'}. They follow from Corollary~\ref{c:Sigma'_red} because the map 
$\sG^{(F)}\to (\BA^1/\BG_m)_-\hat\otimes\BZ_p$ is a universal homeomorphism of infinite type (see \S\ref{sss:sG^(F)}).
\end{proof}

\subsection{The mophisms $\fS'\rightleftarrows\Sigma'_{\red}$}   \label{ss:Sigma'_red to fS'}

\subsubsection{The diagram $\fS'\rightleftarrows\sS\rightleftarrows\Sigma'_{\red}$}   \label{sss:Sigma'_red to fS'}
By Corollary~\ref{c:Sigma'_red}, $\Sigma'_{\red}$ is the classifying stack of a flat group scheme over $\sS$, so we have faithfully flat canonical morphisms $\sS\rightleftarrows\Sigma'_{\red}$. Combining them with the morphisms $\fS'\rightleftarrows\sS$ from \S\ref{sss:2fS' & sS}, we get a diagram
\[
\fS'\rightleftarrows\sS\rightleftarrows\Sigma'_{\red}\, .
\]
By Lemma~\ref{l:2fS' to sS to Sigma'}, the composite morphism  $\fS'\to\sS\to\Sigma'$ is equal to the one from \S\ref{sss:fS' to Sigma'}.

\begin{rem}   \label{r:Why should we study them}
We have to study the mophisms 
\begin{equation}   \label{e:fS' and Sigma'_red}
\fS'\rightleftarrows\Sigma'_{\red}
\end{equation}
because they have analogs for $\Sigma''_{\red}$ (see \S\ref{ss:2toy model}-\ref{ss:Sigma''_red to fS''} below), while the morphisms
$\sS\rightleftarrows\Sigma'_{\red}$ do not. Let us note that the mophisms \eqref{e:fS' and Sigma'_red} are  \emph{not flat} (this easily follows from non-flatness
of the morphisms $\fS'\rightleftarrows\sS$, see \S\ref{sss:2fS' & sS}).
\end{rem}

\begin{lem}    \label{l:fS' and Sigma'_red}
(i) Each of the composite morphisms $\fS'\to\Sigma'_{\red}\to\fS'$ and $\Sigma'_{\red}\to\fS'\to\Sigma'_{\red}$ is uniquely isomorphic to the Frobenius.

(ii) The pullback of the BK-Tate module $\cO_{\Sigma'}\{ 1\}$ to $\fS'$ canonically identifies with the covariant $\cO_{\fS'}$-module $\sL_{\fS'}$ from \S\ref{sss:cO_fS'}.

(iii) The pullback of $\sL_{\fS'}$ to $\Sigma'_{\red}$ is canonically  isomorphic to $\cO_{\Sigma'_{\red}}\{ p\}$ and also to~$\sL_{\Sigma'_{\red}}^{\otimes p}$.

(iv) If $p|n$ then the non-derived pushforward of $\cO_{\Sigma'_{\red}}\{ n\}$ via the morphism $\Sigma'_{\red}\to\fS'$ equals $\sL_{\fS'}^{\otimes (n/p)}$. 
\end{lem}

The case where $p$ does not divide $n$ is treated in Corollary~\ref{c:coarse moduli} below.

\begin{proof}
(i) By \S\ref{sss:2fS' & sS}, each of the composite morphisms $\fS'\to\sS\to\fS'$ and $\sS\to\fS'\to\sS$ is isomorphic to the Frobenius. So the composite morphisms $\fS'\to\Sigma'_{\red}\to\fS'$ is isomorphic to the Frobenius, and the composite morphism $\Sigma'_{\red}\to\fS'\to\Sigma'_{\red}$ can be rewritten as
\begin{equation}  \label{e:endomorphism of Sigma'_red}
\Sigma'_{\red}\to\sS\overset{\Fr}\longrightarrow\sS\to\Sigma'_{\red}\, .
\end{equation}
On the other hand, $\Fr_{\Sigma'_{\red}}$ 
also identifies with the composite morphism \eqref{e:endomorphism of Sigma'_red}: indeed,
by Corollary~\ref{c:Sigma'_red}, $\Sigma'_{\red}$ is the classifying stack of a flat group scheme over $\sS$ killed by the geometric Frobenius.

Uniqueness follows from the fact that $\Fr_{\Sigma'_{\red}}$ and $\Fr_{\fS'}$ have no automorphisms (to check this for $\Fr_{\Sigma'_{\red}}$, use the description of
$\Sigma'_{\red}$ as a classifying stack).

(ii) Statement (ii) was proved in Lemma~\ref{l:cO_fS'}.

(iii) The statement about $\cO_{\Sigma'_{\red}}\{ p\}$ follows from (i) and (ii). The isomorphism $$\cO_{\Sigma'_{\red}}\{ p\}\iso\sL_{\Sigma'_{\red}}^{\otimes p}$$
follows from formula~\eqref{e:BK-Tate' definition} and Lemma~\ref{l:restricting BK to Sigma_red}.

(iv) By statement (iii), it suffices to treat the case $n=0$.
The non-derived pushforward of $\cO_{\Sigma'_{\red}}$ via the morphism $\Sigma'_{\red}\to\sS$ equals $\cO_\sS$. The pushforward of 
$\cO_\sS$ via the morphism $\sS\to\fS'$ equals $\cO_{\fS'}$.
\end{proof}

\subsection{The stacks $\widetilde{\Sigma'}$ and $\widetilde{\Sigma'}_{\!\red}$}   \label{ss:tilde Sigma'}
\subsubsection{Definition of $\widetilde{\Sigma'}$}
For a scheme $S$, let $\widetilde{\Sigma'}(S)$ be the category of pairs consisting of an object $ \beta\in\Sigma'(S)$ and a trivialization of the line bundle $ \beta^*\cO_{\Sigma'}\{ 1\}$ on $S$. By Lemma~\ref{l:conservativity of cO_Sigma'{1}}, $\widetilde{\Sigma'}$ is a g-stack. It is clear that $\widetilde{\Sigma'}$ is a $\BG_m$-torsor over $(\Sigma')^{\rm g}$, where $(\Sigma')^{\rm g}$ is the g-stack under\-lying~$\Sigma'$. So $(\Sigma')^{\rm g}=\widetilde{\Sigma'}/\BG_m$.

\subsubsection{The goal}
One has $\widetilde{\Sigma'}_{\!\red}=\widetilde{\Sigma'}\times_{\Sigma'}\Sigma'_{\red}$, so $\widetilde{\Sigma'}_{\!\red}$ is a $\BG_m$-torsor over 
$(\Sigma_{\red}')^{\rm g}$. Our goal is to describe the g-stack $\widetilde{\Sigma'}_{\!\red}$ together with the action of $\BG_m$ on it. 

\subsubsection{The plan}
In \S\ref{sss:C_1,C_2} we define an affine curve $C_1$ and a curve $C_2$ finite over $C_1$. In \S\ref{sss:The group scheme H} we define a flat affine group scheme $H$ over $C_1$ and an action of $H$ on $C_2$. All these data are $\BG_m$-equivariant.

Then we construct a $\BG_m$-equivariant isomorphism $\widetilde{\Sigma'}_{\!\red}\iso C_2/H$, where $C_2/H$ is the quotient \emph{stack.} On the other hand, it turns out that the GIT quotient $C_2/\!\!/ H$ equals $C_1$, see Corollary~\ref{c:coarse moduli}(i); in other words, $C_1$ is the ``coarse moduli space'' for $\widetilde{\Sigma'}_{\!\red}$.

\subsubsection{The curves $C_1,C_2$}  \label{sss:C_1,C_2}
Let $R:=\BF_p[v_+,v_-]/(v_+v_-)$. In \S\ref{sss:fS'} we introduced the curve $C:=\Spec R$ equipped with the $\BG_m$-action such that $\deg v_\pm=\pm 1$.

Set $u_\pm:=v_\pm^p$; sometimes we will write $u$ instead of $u_+$. Let $R_1\subset R_2\subset R$ be the following subrings: $R_1$ is generated by $u_+,u_-$, and $R_2$ is generated by $u_+,v_-$.
Let $C_i:=\Spec R_i$. The action of $\BG_m$ on $R$ induces its action on $R_1$, $R_2$, $C_1$, and $C_2$.

Note that $C_1$ and $C_2$ are isomorphic to $C$ as schemes, but they are not isomorphic as schemes with $\BG_m$-action.

\subsubsection{The action of $\mu_p$ on $C_2$}  \label{sss:mu_p-action}
The subgroup $\mu_p\subset\BG_m$ acts on $C_1$ trivially. So $\mu_p$ acts on $C_2$ over $C_1$; moreover, $C_2/\!\!/ \mu_p=C_1$, where
$C_2/\!\!/ \mu_p$ is the GIT quotient.

\subsubsection{The group scheme $H$}   \label{sss:The group scheme H}
Similarly to \S\ref{sss:sG^(F)}, define a group scheme $H$ over $C_1$ as follows:  
$H$ is the scheme $W^{(F)}\times C_1$ equipped with the operation
\[
W^{(F)}\times W^{(F)}\times C_1\to W^{(F)}\times C_1, \quad (\alpha_1 ,\alpha_2, t)\mapsto (\alpha_1+\alpha_2-[u_-(t)]\cdot\alpha_1\alpha_2,t),
\]
where $\alpha_i\in W^{(F)}$, $t\in C_1$, and $u_-(t)$ is the value of $u_-\in H^0(C_1, \cO_{C_1})$ on $t$. Similarly to formula~\eqref{e:sG to W^times}, we have the  group scheme homomorphism
\begin{equation}    \label{e:H to mu_p}
H\to \mu_p\times C_1, \quad (\alpha ,t)\mapsto ( 1-u_-(t)\alpha_0,t),  
\end{equation}
where $\alpha_0$ is the 0-th component of the Witt vector $\alpha\in W^{(F)}$. 
By \S\ref{sss:mu_p-action}, it defines an action of $H$ on $C_2$ over $C_1$.

\subsubsection{$\BG_m$-equivariance}
Equip $W^{(F)}$ with the following $\BG_m$-action: $\lambda\in\BG_m$ acts as multiplication by $[\lambda^p]\in W^\times$. Combining this action with the one from \S\ref{sss:C_1,C_2}, we get an action of $\BG_m$ on the scheme $H=W^{(F)}\times C_1$. The group operation on $H$ is $\BG_m$-equivariant. So are the homomorphism \eqref{e:H to mu_p} and the action of $H$ on $C_2$.

\begin{prop}   \label{p:tilde Sigma'}
There is a canonical $\BG_m$-equivariant isomorphism $\widetilde{\Sigma'}_{\!\red}\iso C_2/H$.
\end{prop}

The proof is given in \S\ref{sss:tilde Sigma'_red to C_1}-\ref{sss:determining the action} below. The first step  (see \S\ref{sss:tilde Sigma'_red to C_1}-\ref {sss:tilde Sigma'_red to C_1/H})  is to construct a canonical morphism $\widetilde{\Sigma'}_{\!\red}\to C_1/H$, where the action of $H$ on $C_1$ is trivial (so $C_1/H$ is just the classifying stack of $H$).

\subsubsection{The morphism  $\widetilde{\Sigma'}_{\!\red}\to C_1$}   \label{sss:tilde Sigma'_red to C_1}
Let us construct a morphism
$
(\Sigma'_{\red})^{\rm g}\to C_1/\BG_m,
$
where  $(\Sigma'_{\red})^{\rm g}$ is the underlying g-stack of $\Sigma'_{\red}$. This means specifying a diagram 
\begin{equation}   \label{e:diagram to be constructed}
\cA^{\otimes p}\overset{u_-}\longrightarrow \cO_{\Sigma'_{\red}}  \overset{u_+}\longrightarrow\cA^{\otimes p}, \quad u_+u_-=0,
\end{equation}
where $\cA$ is a line bundle on $(\Sigma'_{\red})^{\rm g}$.

On $\Sigma'_{\red}$ we have a canonical diagram of line bundles
\[
\sL_{\Sigma'_{\red}}\overset{v_-}\longrightarrow \cO_{\Sigma'_{\red}}  \overset{u}\longrightarrow\sL_{\Sigma'_{\red}}^{\otimes p}, \quad uv_-=0.
\]
By Lemma~\ref{l:BK-Tate on Sigma'_red}(i,iii), $\sL_{\Sigma'_{\red}}^{\otimes p}=( \cO_{\Sigma'_{\red}}\{ 1\})^{\otimes p}$, so we get a diagram \eqref{e:diagram to be constructed}
by setting 
\[
\cA:= \cO_{\Sigma'_{\red}}\{ 1\} , \quad u_+:=u , \quad u_-:=v_-^{\otimes p}.
\]
Thus we have constructed a morphism 
\begin{equation}   \label{e:Sigma'_red to C_1/G_m}
(\Sigma'_{\red})^{\rm g}\to C_1/\BG_m.
\end{equation}
Since $\cA =\cO_{\Sigma'_{\red}}\{ 1\}$, the line bundle on $\Sigma'_{\red}$ corresponding to
the composite morphism 
$$(\Sigma'_{\red})^{\rm g}\to C_1/\BG_m\to  (\Spec\BF_p)/\BG_m$$
equals  $\cO_{\Sigma'_{\red}}\{ 1\}$. So base-changing \eqref{e:Sigma'_red to C_1/G_m} via the mophism $C_1\to C_1/G_m$, we get a  $\BG_m$-equivariant morphism $\widetilde{\Sigma'}_{\!\red}\to C_1$.

\subsubsection{The morphism  $\widetilde{\Sigma'}_{\!\red}\to C_1/H$} \label{sss:tilde Sigma'_red to C_1/H}
Let $\sH:=H/\BG_m$, then $\sH$ is a group scheme over $C_1/\BG_m$. Let $(C_1/\BG_m )/\sH$ denote the classifying stack of $\sH$ (i.e., the quotient of
$C_1/\BG_m$ by the trivial action of $\sH$). Our next goal is to upgrade the morphism \eqref{e:Sigma'_red to C_1/G_m} to a morphism
\begin{equation}   \label{e:Sigma'_red to (C_1/G_m)/H}
(\Sigma'_{\red})^{\rm g}\to (C_1/\BG_m)/\sH. 
\end{equation}
This means specifying a torsor over the pullback of $\sH$ to $(\Sigma'_{\red})^{\rm g}$. This pullback is canonically 
isomorphic\footnote{This is related to Remark~\ref{r:2strange descent}.} to the pullback of $\sG^{(F)}$ to 
$(\Sigma'_{\red})^{\rm g}$, so the realization of $\Sigma'_{\red}$ as a classifying stack (see Corollary~\ref{c:Sigma'_red}) yields the desired torsor.

We have $(C_1/\BG_m)/\sH=(C_1/H)/\BG_m$, where $C_1/H$ is the classifying stack of $H$. 
Base-changing \eqref{e:Sigma'_red to (C_1/G_m)/H} via the mophism $C_1/H\to (C_1/\BG_m)/\sH$, we get a  $\BG_m$-equivariant morphism $\widetilde{\Sigma'}_{\!\red}\to C_1/H$.

\subsubsection{The isomorphism $\widetilde{\Sigma'}_{\!\red}\times_{C_1/H}C_1\iso C_2$}   \label{sss:fiber product equals C_2}
The fiber product of $(\Sigma'_{\red})^{\rm g}$ and $C_1/\BG_m$ over $(C_1/\BG_m)/\sH$ equals $\sS^{\rm g}=C_2/\BG_m$. Moreover, by 
Lemma~\ref{l:BK-Tate on Sigma'_red}(iv), the pullback of $\cO_{\Sigma'_{\red}}\{ 1\}$ via the morphism $\sS\to\Sigma'_{\red}$ equals $\sL_{\sS}$. 
So we get a canonical $\BG_m$-equivariant isomorphism $\widetilde{\Sigma'}_{\!\red}\times_{C_1/H}C_1\iso C_2$.

\subsubsection{End of the proof of Proposition~\ref{p:tilde Sigma'}}   \label{sss:determining the action}
By virtue of \S\ref{sss:tilde Sigma'_red to C_1/H}-\ref{sss:fiber product equals C_2},
$\widetilde{\Sigma'}_{\!\red}$ is the quotient of $C_2$ by some action of $H$ on $C_2$ over $C_1$. The action is determined by the description of $\sM$ in 
Lemma~\ref{l:BK-Tate on Sigma'_red}(iv).  \qed

\begin{cor}    \label{c:coarse moduli}
(i) The morphism  $\widetilde{\Sigma'}_{\!\red}\to C_1$ from \S\ref{sss:tilde Sigma'_red to C_1} induces an isomorphism
$H^0(C_1,\cO_{C_1})\iso H^0(\widetilde{\Sigma'}_{\!\red},\cO_{\widetilde{\Sigma'}_{\!\red}})$.

(ii) The non-derived pushforward of $\cO_{\Sigma'_{\red}}\{ n\}$ via the morphism $\Sigma'_{\red}\to\fS'$ from \S\ref{sss:Sigma'_red to fS'}
equals $\sL_{\fS'}^{\otimes (n/p)}$ if $p|n$; if $p\!\not | n$ it equals $0$.
\end{cor}

\begin{proof}
Statement (i) follows from Proposition~\ref{p:tilde Sigma'}. 

The morphism $\Sigma'_{\red}\to\fS'$ from \S\ref{sss:Sigma'_red to fS'} can be factored as
\[
\Sigma'_{\red}=\widetilde{\Sigma'}_{\!\red}/\BG_m\overset{f}\longrightarrow C_1/\BG_m\to C_1/(\BG_m/\mu_p)=\fS' ,
\]
where $f$ is induced from our morphism $\widetilde{\Sigma'}_{\!\red}\to C_1$. As noted at the end of \S\ref{sss:tilde Sigma'_red to C_1}, $\cO_{\Sigma'_{\red}}\{ 1\}$ is the pullback of the natural line bundle on $C_1/\BG_m$. So statement (ii) follows from (i).
\end{proof}

\section{The c-stack $\Sigma''$}   \label{s:Sigma''}
\subsection{Introduction}  \label{ss:intro to Sigma''}

We have a g-stack $\Sigma$, a c-stack $\Sigma'$, and open immersions $j_\pm:\Sigma\to\Sigma'$ such that $\Sigma_+\cap\Sigma_-=\emptyset$, where $\Sigma_\pm\subset\Sigma$ is the image of $j_\pm$. The goal of this section is to study the c-stack $\Sigma''$ obtained from $\Sigma'$ by gluing $\Sigma_+$ with $\Sigma_-$.

Let $(\Sigma')^{\rm g}$ and $(\Sigma'')^{\rm g}$ be the g-stacks underlying $\Sigma'$ and $\Sigma''$. By Theorem~\ref{t:Sigma''}(ii), $(\Sigma'')^{\rm g}$ identifes with $\Coeq (\Sigma\overset{j_+}{\underset{j_-}\rightrightarrows} (\Sigma')^{\rm g})$, i.e., with the stack obtained from 
$(\Sigma')^{\rm g}$ by gluing $\Sigma_+$ with $\Sigma_-$. So the reader who prefers to disregard the c-structure can work with the stack $\Coeq (\Sigma\overset{j_+}{\underset{j_-}\rightrightarrows} (\Sigma')^{\rm g})$ (which is denoted by $\BZ_p^{\Syn}$ in \cite[\S 6]{Bh}).

\subsubsection{A possible order of reading}   \label{sss:order of reading about Sigma''}
In \S\ref{ss:2toy model} we discuss a certain c-stack $\fS''$, which is a toy model for $\Sigma''$ and can be easily visualized\footnote{On the other hand, the description of the $S$-points of $\fS''$ looks complicated, see \S\ref{sss:describing fS''(S)}. A pessimist would say that $\fS''$ is more complicated than it seems, but an optimist would say that $\Sigma''$ could be easier than it seems.} (see Lemma~\ref{l:warm-up} and \S\ref{sss:visualizing fS''}); in particular, its underlying g-stack $(\fS'')^{\rm g}$ equals $\nodal/\BG_m$, where $\nodal$ is the nodal curve obtained from $\BP^1_{\BF_p}$ by gluing $0$ with $\infty$. The reader could read at least a part of \S\ref{ss:2toy model} immediately after reading Definition~\ref{d:Sigma''}.

\subsubsection{Organization of  \S\ref{s:Sigma''}}  
In \ref{ss:Sigma''} we define the c-stack $\Sigma''$ and formulate some properties of it. One more property of $\Sigma''$ is formulated in \S\ref{ss:lax quotients}.

In \S\ref{ss:descent to Sigma''} we discuss descent of different types of objects from $\Sigma'$ to $\Sigma''$. In particular, in \S\ref{sss:BK-Tate on Sigma''} we use such descent to define the Breuil-Kisin-Tate module $\cO_{\Sigma''}\{1\}$.

In \S\ref{ss:Sigma'' subset sQ} we explain a point of view on $\Sigma''$, which is complementary to the one from \S\ref{ss:Sigma''}: namely, we formulate
Theorem~\ref{t:Sigma'' to Q}, which identifies $\Sigma''(S)$ with a full subcategory of  the category of pairs consisting of a ring scheme over $S$ and a quasi-ideal in it. In Question~\ref{q:Sigma'' to RStacks} we ask whether $\Sigma''(S)$ can be realized as a full subcategory of the category of ring stacks over $S$.

The theorems formulated in \S\ref{ss:Sigma''} and \S\ref{ss:lax quotients} are proved in \S\ref{ss:proving coequalizer theorems} using the abstract categorical formalism from  \S\ref{ss:lax quotients}-\ref{ss:coequalizers of categories}. In \S\ref{ss:Sigma'' to Q proof} we prove Theorem~\ref{t:Sigma'' to Q}.

In \S\ref{ss:2toy model}-\ref{ss:Sigma''_red to fS''} we discuss a very understandable c-stack $\fS''$, which is a ``toy model''  for~$\Sigma''$. More precisely, one has morphisms $\fS''\to\Sigma''_{\red}$ and $\Sigma''_{\red}\to\fS''$ (defined in \S\ref{ss:2toy model} and \S\ref{ss:Sigma''_red to fS''}, respectively); composing them in any order one gets the Frobenius. The morphism $\Sigma''_{\red}\to\fS''$ is locally understandable\footnote{It comes from the quite understandable morphism $\Sigma''_{\red}\to\fS'$ defined in \S\ref{sss:Sigma'_red to fS'}.}; globally, it is somewhat ``funny''.

The BK-Tate module $\cO_{\Sigma''}\{ 1\}$ defines a $\BG_m$-torsor $\widetilde{\Sigma}''\to (\Sigma'')^{\rm g}$, where  $(\Sigma'')^{\rm g}$ is the g-stack underlying
$\Sigma''$. In  \S\ref{ss:tilde Sigma''} we more or less describe $\widetilde{\Sigma}''_{\red}$. More precisely, as explained in \S\ref{ss:tilde Sigma''}, the morphisms
$\Sigma''_{\red}\rightleftarrows\fS''$ from \S\ref{ss:2toy model}-\ref{ss:Sigma''_red to fS''} induce morphisms 
\begin{equation}   \label{e:tilde Sigma'' to/from nodal}
\widetilde{\Sigma}''_{\red}\rightleftarrows\nodal,
\end{equation}
where $\nodal$ is the nodal curve from \S\ref{sss:order of reading about Sigma''}; composing these morphisms in any order, one gets the Frobenius.
The local picture of diagram \eqref{e:tilde Sigma'' to/from nodal} was explained in \S\ref{ss:tilde Sigma'}.

\subsection{Definition and some properties of $\Sigma''$}   \label{ss:Sigma''}

We have the canonical open immersions $j_\pm :~\Sigma\iso\Sigma_\pm\mono\Sigma'$. 

\begin{defin}   \label{d:Sigma''}
$\Sigma'':=\Coeq (\Sigma\overset{j_+}{\underset{j_-}\rightrightarrows} \Sigma')$, where $\Coeq$ stands for the coequalizer in the 2-category of c-stacks.
\end{defin}

In other words, $\Sigma''$ is the c-stack obtained by sheafifying the assignment
$$S\mapsto \Coeq (\Sigma(S)\overset{j_+}{\underset{j_-}\rightrightarrows} \Sigma'(S)),$$
where the coequalizer is computed in the 2-category of categories.

\begin{thm}   \label{t:naive Coeq is OK}
(i) For any scheme $S$, the category $\Coeq (\Sigma(S)\overset{j_+}{\underset{j_-}\rightrightarrows} \Sigma'(S))$ is also a coequalizer in the $(\infty ,2)$-category of
$(\infty ,1)$-categories.

(ii) The stack $\Sigma'':=\Coeq (\Sigma\overset{j_+}{\underset{j_-}\rightrightarrows} \Sigma')$ is also a coequalizer of the diagram 
$\Sigma\overset{j_+}{\underset{j_-}\rightrightarrows} \Sigma'$ in the sense of the $(\infty ,2)$-category of stacks of $(\infty ,1)$-categories.
\end{thm}

This theorem will be proved in \S\ref{sss:proof naive Coeq is OK}. It will be used in Remark~\ref{r:DG descent} and in the proof of Lemma~\ref{l:stacks}.

\begin{thm}  \label{t:Sigma''}
(i) The c-stack $\Sigma''$ is strongly adic in the sense of \S\ref{sss:strongly adic}. In particular, $\Sigma''$ is a formal c-stack.

(ii) Let $(\Sigma')^{\rm g}$ and $(\Sigma'')^{\rm g}$ be the g-stacks underlying $\Sigma'$ and $\Sigma''$. Then 
\[
(\Sigma'')^{\rm g}=\Coeq (\Sigma\overset{j_+}{\underset{j_-}\rightrightarrows} (\Sigma')^{\rm g}).
\]

(iii) $(\Sigma')^{\rm g}$ is $\MMor$-affine\footnote{For the notion of $\MMor$-affineness, see Definition~\ref{d:pre-algebraic stack} and Remark~\ref{r:Mor-affine}.}.

(iv) The etale sheafification of the assignment $S\mapsto \Coeq (\Sigma(S)\overset{j_+}{\underset{j_-}\rightrightarrows} \Sigma'(S))$ is equal to the fpqc sheafification (i.e., to $\Sigma''$).
\end{thm}

A proof will be given in \S\ref{sss:proof Sigma''(ii)} and \S\ref{sss:proof Sigma''(i)}-\ref{sss:proof Sigma''(iv)}. Let us note that the c-stack $\Sigma''$ is not $\MMor$-affine, see Remark~\ref{r:not Mor-affine} below.

\begin{cor}   \label{c:Sigma''}
(i) The morphism $\Sigma'\to\Sigma''$ is etale and surjective.

(ii)  It identifies the topological space $|\Sigma''|$ with $\Coeq (|\Sigma |\overset{j_+}{\underset{j_-}\rightrightarrows} |\Sigma'|)$. So $|\Sigma''|$ has two points, one of which is closed while the other is not.
\end{cor} 

\begin{proof}
To prove (i), it suffices to check that the morphism $(\Sigma')^{\rm g}\to (\Sigma'')^{\rm g}$ is etale. This follows from Theorem~\ref{t:Sigma''}(ii).

Statement (ii) follows from Theorem~\ref{t:Sigma''}(ii) combined with Corollary~\ref{c:topology on |Sigma'|}.
\end{proof} 

\subsubsection{The canonical morphism $\Sigma''\to B\BZ$}  \label{sss:Sigma'' to BZ}
Let us make a general remark. Suppose we have categories $\cD,\cC$ and two functors $\cD\to\cC$. Then we have a canonical functor
 from $\Coeq (\cD\rightrightarrows \cC)$ to $\Coeq (*\rightrightarrows*)$, where $*$ is the point (i.e., the category with a single object and a single morphism).
 On the other hand, $\Coeq (*\rightrightarrows*)$ is the classifying groupoid of $\BZ$.
 
 Thus one gets a canonical morphism $\Sigma''\to B\BZ$.

\subsection{Descent from $\Sigma'$ to $\Sigma''$}             \label{ss:descent to Sigma''}
The main goal of this subsection is to define $\sR_{\Sigma''}$ and $\cO_{\Sigma''}\{1\}$, see \S\ref{sss:R_Sigma''}-\ref{sss:BK-Tate on Sigma''}.

\begin{lem}   \label{l:O-modules on Sigma''}
A covariant $\cO_{\Sigma''}$-module is the same as a covariant $\cO_{\Sigma'}$-module $M$ equipped with an isomorphism $j_-^*M\iso j_+^*M$.
\end{lem}

\begin{proof}
For any scheme $S$, let $\fQ (S)$ be the category of quasi-coherent $\cO_S$-modules. Then $\fQ$ is a stack. A covariant $\cO_{\Sigma''}$-module is the same as a morphism $\Sigma''\to\fQ$. It remains to apply Definition~\ref{d:Sigma''} and the universal property of a coequalizer.
\end{proof}

\begin{rem}   \label{r:DG descent}
Here is a DG version of Lemma~\ref{l:O-modules on Sigma''}. Let $\cD (\Sigma'')$ be the DG category of covariant $\cO$-modules on $\Sigma''$. Then the canonical functor $\cD (\Sigma'')\to\Eq  (\cD (\Sigma')\overset{j_+^*}{\underset{j_-^*}\rightrightarrows} \cD (\Sigma ))$ is an equivalence, where $\Eq$ stands for the equalizer in the
$(\infty ,2)$-category of DG categories. This can be deduced from Theorem~\ref{t:naive Coeq is OK}(ii) by repeating the proof of Lemma~\ref{l:O-modules on Sigma''} with
$\fQ (S)$ being replaced by the underlying $(\infty ,1)$-category of $\cD (S)$.
\end{rem}

\begin{rem}   \label{r:O-modules on Sigma''}
Lemma~\ref{l:O-modules on Sigma''} and its proof remain valid if one replaces the words ``covariant $\cO$-module'' by ``weakly invertible covariant $\cO$-module'' or, say, by the words ``contravariant $\cO$-module''. Remark~\ref{r:DG descent} remains valid for the DG category of contravariant $\cO$-modules on~$\Sigma''$.

\end{rem}

\begin{lem}  \label{l:stacks}
(i) A c-stack left-fibered over $\Sigma''$ is the same as a c-stack $\sX$ left-fibered over $\Sigma'$ equipped with an isomorphism $j_-^*\sX\iso j_+^*\sX$ over $\Sigma$, where $j_\pm^*\sX$ is the base change of $\sX$ via $j_\pm :\Sigma\to\Sigma'$.

(ii) A Picard stack (resp.~ring stack) over $\Sigma''$ is the same as a Picard stack (resp.~ring stack) $\sX$ left-fibered over $\Sigma'$ equipped with an isomorphism $j_-^*\sX\iso j_+^*\sX$ over $\Sigma$.
\end{lem}

\begin{proof}
To prove (i), proceed similarly to the proof of Lemma~\ref{l:O-modules on Sigma''} but replace the category $\fQ (S)$ by the $(2,1)$-category of g-stacks over $S$;
this is possible due to Theorem~\ref{t:naive Coeq is OK}(ii). The proof of (ii) is similar.
\end{proof}

\subsubsection{The ring stack $\sR_{\Sigma''}$}   \label{sss:R_Sigma''}
By Lemma~\ref{l:stacks}(ii), the ring stack $\sR_{\Sigma'}$ from \S\ref{sss:sR' descends} descends via the isomorphism \eqref{e:sR' descends} to a ring stack 
$\sR_{\Sigma''}$ over $\Sigma''$.   
$\sR_{\Sigma''}$ is a ring stack over $\Sigma''$. Let us note that in \S\ref{sss:Sigma'' to Q} we will give a slightly different construction of $\sR_{\Sigma''}$, which does not rely on Lemma~\ref{l:stacks}, see formula~\eqref{e:Sigma''  to cQ to RStacks}.

If $S$ is a scheme then  associating to a morphism  $f:S\to\Sigma''$ the $f$-pullback of $\sR_{\Sigma''}$, we get a functor 
\begin{equation}  \label{e:Sigma'' to RStacks}
\Sigma''(S)\to\RStacks_S\, , 
\end{equation}
where $\RStacks_S$ is the $(2,1)$-category of algebraic ring stacks over $S$.

\begin{quest}   \label{q:Sigma'' to RStacks}
Is the functor \eqref{e:Sigma'' to RStacks} fully faithful?
\end{quest}

If the functor \eqref{e:Sigma'' to RStacks} is fully faithful and if its essential image is understandable \emph{then one could think of $\Sigma''$ as a moduli stack of a certain type of ring stacks.}

\subsubsection{The Breuil-Kisin-Tate module $\cO_{\Sigma''}\{1\}$}   \label{sss:BK-Tate on Sigma''}
By Lemma~\ref{l:O-modules on Sigma''} and Remark~\ref{r:O-modules on Sigma''}, the BK-Tate module $\cO_{\Sigma''}\{1\}$ from \S\ref{sss:BK-Tate on Sigma'} descends via the isomorphism \eqref{e:O_Sigma'{1} descends} to a weakly
invertible covariant $\cO_{\Sigma''}$-module. We call it the \emph{Breuil-Kisin-Tate} module on $\Sigma''$ (sometimes shortened to \emph{BK-Tate module}).
By Remark~\ref{r:trivializing O_Sigma'{1}}, the pullback of $\cO_{\Sigma''}\{1\}$ via $\Spf\BZ_p\overset{p}\longrightarrow\Sigma\to\Sigma''$  is trivialized. 
For $n\in\BZ$ set
\[
\cO_{\Sigma''}\{n\}:=(\cO_{\Sigma''}\{1\})^{\otimes n}.
\]
If $n\ge 0$ this is a weakly invertible covariant $\cO_{\Sigma''}$-module. If $n\le 0$ it is a weakly invertible \emph{contravariant} $\cO_{\Sigma''}$-module (see \S\ref{sss:Weakly invertible}).

\begin{lem}    \label{l:conservativity of cO_Sigma''{1}}
Let $S$ be a scheme and $ \beta_1, \beta_2\in\Sigma''(S)$. Then a morphism $f:  \beta_1\to \beta_2$ is an isomorphism if and only if it induces an isomorphism
$ \beta_1^*\cO_{\Sigma''}\{ 1\}\to \beta_2^*\cO_{\Sigma''}\{ 1\}$.
\end{lem}

\begin{proof}
Let us prove the ``if" statement. Fpqc-locally on $S$, one can write $f=f_1f_2\ldots f_n$, where each $f_i$ is either a $\Sigma''(S)$-isomorphism or comes from a $\Sigma'(S)$-morphism. By Lemma~\ref{l:conservativity of cO_Sigma'{1}}, the $\Sigma'(S)$-morphisms involved have to be isomorphisms.
\end{proof}

\subsubsection{The morphism $\Sigma''\to B\BA^1_m$}  
Let $\BA^1_m$ denote the multiplicative monoid (i.e., $\BA^1$ equipped with the operation of multiplication).
Let $B\BA^1_m$ be its classifying c-stack. 

The Breuil-Kisin-Tate module $\cO_{\Sigma''}\{1\}$ yields a canonical morphism $\Sigma''\to B\BA^1_m$. Combining it with
\S\ref{sss:Sigma'' to BZ}, one gets a canonical morphism
\begin{equation}    \label{e:Sigma'' to product}
\Sigma''\to B\BZ\times B\BA^1_m.
\end{equation}

\subsection{$\Sigma''(S)$ as a full subcategory of $\cQ (S)$}  \label{ss:Sigma'' subset sQ}
Instead of answering Question~\ref{q:Sigma'' to RStacks}, we construct a certain factorization
\[
\Sigma''(S)\to \cQ (S)\to\RStacks_S
\]
of the functor \eqref{e:Sigma'' to RStacks} and formulate Theorem~\ref{t:Sigma'' to Q}, which says that the functor $\Sigma''(S)\to \cQ (S)$ is fully faithful.

\subsubsection{The ring scheme $\widetilde{W}_S$}
For a scheme $S$, set
\[
\widetilde{W}_S:= \underset{\longleftarrow}\lim (\ldots\overset{F}\longrightarrow W_S\overset{F}\longrightarrow W_S\overset{F}\longrightarrow W_S) ,
\]
where $F$ is the Witt vector Frobenius. Then $\widetilde{W}_S$ is a flat affine ring scheme over $S$. The homomorphism $F:W_S\to W_S$ induces an \emph{auto}morphism of $\widetilde{W}_S$, which is still denoted by $F$. We have a homomorphism of monoids
\begin{equation}   \label{e:Aut tilde W}
H^0(S,\BZ )\to\End\widetilde{W}_S, \quad n\mapsto F^n,
\end{equation}
where $H^0(S,\BZ )$ is the group of locally constant functions $S\to\BZ$.

\begin{prop}   \label{p:Aut tilde W}
If $S$ is $p$-nilpotent then the map \eqref{e:Aut tilde W} is an isomorphism. In particular, $\End\widetilde{W}_S=\Aut\widetilde{W}_S$.
\end{prop}

\begin{proof}
Follows from Proposition~\ref{p:ring morphisms W to W_n} of Appendix~\ref{s:ring morphisms W to W_n}. 
\end{proof}

\begin{cor}    \label{c:Aut tilde W}
Let $S$ be a $p$-nilpotent scheme.

(i) Let $\cW (S)$ be the category of those ring schemes over $S$ which are fpqc-locally isomorphic to $\widetilde{W}_S$. Then $\cW (S)$ identifies with the groupoid of $\BZ$-torsors on the fqpc site of $S$.

(ii) Any ring scheme from $\cW (S)$ is etale-locally isomorphic to $\widetilde{W}_S$.
\end{cor}

\begin{proof}
Statement (i) immediately follows from Proposition~\ref{p:Aut tilde W}. Statement (ii) follows from~(i) because any $\BZ$-torsor on the fqpc site of $S$ is etale-locally trivial (this easily follows from Lemma~\ref{l:one-point compactification}).
\end{proof}

\subsubsection{The category $\cQ (S)$}    \label{sss:cQ (S)}
From now on we assume that $S$ is a $p$-nilpotent scheme. Let $\cQ (S)$ be the category of triples $(C,I,d)$, where $C\in\cW (S)$ and $(I,d:I\to C)$ is a quasi-ideal in $C$ with $I$ being a flat affine $S$-scheme. We have functors
\begin{equation}   \label{e:Q to RingStacks}
\cQ (S)\to\RStacks_S , \quad (C,I,d)\mapsto\Cone (d),
\end{equation}
\begin{equation}    \label{e:Q to BZ}
\cQ (S)\to\cW (S)=(B\BZ)(S) , \quad (C,I,d)\mapsto C.
\end{equation}
Here $(B\BZ)(S)$ (i.e.,  the groupoid of $\BZ$-torsors on $S$) identifies with $\cW (S)$ by Corollary~\ref{c:Aut tilde W}.

\subsubsection{The functor $\Phi :\Sigma' (S)\to\cQ (S)$}    \label{sss:Sigma' to Q}
We also have the following functor $\Phi :\Sigma' (S)\to\cQ (S)$: given  $(M,\xi :M\to W_S)\in\Sigma'(S)$, we set $\Phi (M,\xi ):= (C,I,d)$, where $C=\widetilde{W}_S$, 
$I=M\times_{W_S}\widetilde{W}_S$, and $d:I\to C$ is induced by $\xi :M\to W_S$.

\subsubsection{The functor $\Sigma'' (S)\to\cQ (S)$}    \label{sss:Sigma'' to Q}
By Lemma~\ref{l:F' & j_-}, the morphism $j_-j_+^{-1}:\Sigma_+\iso\Sigma_-$ takes $(M,\xi )\in\Sigma_+(S)$ to $(\overline{M},\bar\xi )\in\Sigma_-(S)$, where
$\overline{M}=M^{(1)}\times_{W_S^{(1)}}W_S$ and $\bar\xi:\overline{M}\to W_S$ is the projection. The Cartesian square
\[
\xymatrix{
\overline{M}\ar[r]^{\bar\xi} \ar[d] & W_S\ar[d]^F\\
M\ar[r]^\xi & W_S
}
\]
(whose vertical arrows are $F$-linear) induces a canonical isomorphism $\Phi (\overline{M},\bar\xi )\iso\Phi (M,\xi )$. Thus we get an isomorphism $\Phi\circ j_-\iso\Phi\circ j_+$ and therefore a functor
\begin{equation}  \label{e:Sigma'' to Q}
\Sigma'' (S)\to\cQ (S).
\end{equation}
The functor \eqref{e:Sigma'' to RStacks} is the composite functor 
\begin{equation}   \label{e:Sigma''  to cQ to RStacks}
\Sigma''(S)\to \cQ (S)\to\RStacks_S.
\end{equation}
On the other hand, composing \eqref{e:Sigma'' to Q} and \eqref{e:Q to BZ}, one gets the functor  $\Sigma''(S)\to (B\BZ)(S)$ from~\S\ref{sss:Sigma'' to BZ}.

\begin{thm}   \label{t:Sigma'' to Q}
The functor \eqref{e:Sigma'' to Q} is fully faithful.
\end{thm}

The theorem will be proved in \S\ref{ss:Sigma'' to Q proof}.

\subsection{Lax quotients and the image of $\Sigma$ in $\Sigma''$}   \label{ss:lax quotients}
The main goal of this subsection is to formulate Theorem~\ref{t:Sigma_F}. To this end, we need some general nonsense.

Let $\BZ_+\subset\BZ$ be the submonoid of non-negative integers.
\begin{defin}   \label{d:lax quotient}
The \emph{left-lax quotient} of  a category $\cC$  with respect to a functor $\Phi :\cC\to\cC$ is the following category $\cC_\Phi$: it has the same set of objects as $\cC$, for any objects $c,c'$ one has 
\begin{equation}   \label{e:left-lax quotient}
\Mor_{\cC_\Phi} (c, c'):=\bigsqcup\limits_{m\in\BZ_+}\Mor_{\cC_\Phi}^m (c, c'), \quad \Mor_{\cC_\Phi}^m (c, c'):=\Mor_{\cC} (\Phi^m(c), c'),
\end{equation}
and the composition of $f\in\Mor_{\cC} (\Phi^m(c), c')$ and $g\in\Mor_{\cC} (\Phi^n(c'), c'')$ is the morphism $g\circ\Phi^n(f)\in\Mor_{\cC} (\Phi^{m+n}(c'), c'')$.
\end{defin}

\subsubsection{Ways to think about $\cC_\Phi$}  \label{sss:Ways to think}
(i) If $\Phi$ is an equivalence there is a construction similar to $\cC_\Phi$ but with $\BZ_+$ replaced by $\BZ$. This construction is very familiar if $\cC$ is a groupoid (e.g., a set); in this case the result is called the \emph{quotient groupoid} of $\cC$ by the action of $\BZ$ corresponding to~$\Phi$. For this reason we think of $\cC_\Phi$ as a kind of quotient of $\cC$ by the action of $\BZ_+$ corresponding to $\Phi$. 

(ii) For each $c\in\cC$ one has a canonical element $f_c\in \Mor_{\cC_\Phi}^1(c, \Phi (c))$: it corresponds to~$\id_{\Phi (c)}$. In fact, $\cC_\Phi$ is generated by $\cC$-morphisms and the morphisms $f_c$ with the following defining relations: for $u\in\Mor_{\cC}(c,c')$ one has $f_{c'}u=\Phi (u)f_c$.

(iii) By definition, the category $\cC_\Phi$ is $\BZ_+$-graded. So one has a functor $$\cC_\Phi\to B\BZ_+,$$ where $B\BZ_+$ is the category with one object whose monoid of endomorphisms is $\BZ_+$. This functor is a co-Cartesian fibration. In fact, passing from $(C,\Phi )$ to $\cC_\Phi$ is a particular case of the Grothendieck construction. The functor $\Phi$ defines an action of $\BZ_+$ on $\cC$, which gives rise to a functor from $B\BZ_+$ to the 2-category of categories. The
Grothendieck construction encodes such a functor by a category co-fibered over $B\BZ_+$. 

\begin{rem} \label{r:same underlying groupoid}
The canonical functor $\cC\to\cC_\Phi$ induces an isomorphism between the underlying groupoids (because the only invertible element of the monoid $\BZ_+$ is $0$).
\end{rem}

\subsubsection{Variant for stacks}  \label{sss:stacky lax quotient}
Let $\sX$ be a c-stack and $\Phi :\sX\to\sX$ a morphism. Then for every scheme $S$ we have the category $\sX (S)_\Phi$ from Definition~\ref{d:lax quotient}. 
Let $\sX_\Phi$ be the c-stack obtained by sheafifying the assignment $S\mapsto\sX (S)_\Phi$. One has a canonical morphism $\sX\to\sX_{\Phi}$.
We think of $\sX_\Phi$ as a kind of quotient of $\sX$ by the action of $\BZ_+$ corresponding to $\Phi$ (see \S\ref{sss:Ways to think}(i)).

Using Remark~\ref{r:same underlying groupoid}, one gets the following description of the category $\sX_\Phi (S)$ for any scheme $S$: it has the same objects as 
$\sX (S)$, and its morphisms and their composition are described similarly to Definition~\ref{d:lax quotient} but with $\BZ_+$ replaced by the monoid of locally constant functions $S\to\BZ_+$.

Applying \S\ref{sss:stacky lax quotient} to the morphism $F:\Sigma\to\Sigma$, we get a c-stack $\Sigma_F$.

\begin{thm}  \label{t:Sigma_F}
The morphism $\Sigma\to\Sigma''$ factors as $\Sigma\to\Sigma_F\mono\Sigma''$, where
the morphism $\Sigma_F\mono\Sigma''$ is an open immersion.
\end{thm}

A proof will be given in \S\ref{sss:proof Sigma_F}.

\begin{rem} \label{r:not Mor-affine} 
The c-stack $\Sigma_F$ is not $\MMor$-affine because of the infinite disjoint union in \eqref{e:left-lax quotient}. By Theorem~\ref{t:Sigma_F}, this implies that
$\Sigma''$ is not $\MMor$-affine.
\end{rem}

\subsection{Lax colimits and Grothendieck construction}  \label{ss:lax colimits}
Let us recall some material from \cite{GHN} and \cite{Lu1}.
\subsubsection{Grothendieck construction}
As explained by Grothendieck, a functor $F$ from a category $\sT$ to the 2-category of categories is essentially the same as a category $\sA$ cofibered over~$\sT$.  Passing from $F$ to $\sA$ is called the \emph{Grothendieck construction.} One also says that $\sA$ is the \emph{unstraightening} of $F$; this terminology is used in \cite[\S 2.2, \S 3.2]{Lu1}.

\subsubsection{Grothendieck construction as lax colimit}  \label{sss:GHN}
A classical theorem says that the unstraightening of $F$ is the same as a \emph{left-lax colimit} of $F$, denoted by $\colim^{\llax} F$. A similar theorem for $\infty$-categories is the main result of \cite{GHN}.

\subsubsection{The universal property of $\colim^{\llax} F$}
The above-mentioned theorem is nontrivial and useful because unlike the unstraightening, the left-lax colimit is defined by a universal property (rather than a constuction). The property is as follows: a functor from $\colim^{\llax} F$ to any category $\cE$ is the same as a ``functorial collection'' of functors
\[
\cC_{c/}\to\Func (F(c),\cE ), \quad c\in\cC
\]
where $ \Func$ is the category of functors and $\cC_{c/}$ is the category of objects of $\cC$ equipped with a morphism from $c$. The words ``functorial collection'' make sense because both $\cC_{c/}$ and $\Func (F(c),\cE )$ depend on $c$ functorially.

\subsubsection{A corollary}   \label{sss:GHN-corollary}
If $F$ is a functor  from a (usual) category $\sT$ to the 2-category of (usual) categories then the $\infty$-categorical left-lax colimit of $F$ is a usual category\footnote{A similar statement for non-lax colimits  would be false: e.g., if $F(c)$ is a point for each $c\in\cC$ then the $\infty$-categorical colimit of $F$ is the nerve of $\cC$.}; equivalently, in this situation the $\infty$-categorical left-lax colimit is the same as the usual one. Indeed, a similar statement for unstraightenings is well known, and the statement for left-lax colimits follows by the result of \cite{GHN} mentioned in \S\ref{sss:GHN}.  

\subsubsection{Examples of left-lax colimits}   \label{sss:Examples of lef-lax colimits}
We need two examples. In both of them $\sT$ is the free category generated by an oriented graph $\Delta$. 

(i) Suppose that $\Delta$ has one vertex and one edge. Then $F$ amounts to a diagram $\overset{\Phi}\circlearrowright\! \cC$ (i.e., a category $\cC$ equipped with an endofunctor $\Phi$. By \S\ref{sss:Ways to think}(iii), one has
\[
\colim^{\llax} (\overset{\Phi}\circlearrowright\! \cC )=\cC_\Phi,
\]
where $\cC_\Phi,$ is the left-lax quotient.

(ii) Suppose that $\Delta$ has two vertices and an edge between them. Then $F$ amounts to a diagram of categories
$\cA\overset{\Psi}\longrightarrow\cB$, and the category $\cC:=\colim^{\llax} (\cA\overset{\Psi}\longrightarrow\cB)$ is as follows:
one has an essentially surjective functor $\nu:\cA\sqcup\cB\to\cC$ whose restrictions to $\cA$ and $\cB$ are fully faithful, and for $a\in\cA$, $b\in\cB$ one has
\[
\Mor (\nu (b),\nu (a))=\emptyset ,  \quad \Mor (\nu (a),\nu (b))=\Hom (\Psi (a), b).
\]

\begin{rem}  \label{r:lax colimits}
In the situation of \S\ref{sss:Examples of lef-lax colimits}(ii) let $j:\cB\to\cC$ be the restriction of $\nu$, then $j$ is a fully faithful left fibration, which has a left adjoint. Conversely, given \emph{any} triple $(\cB,\cC,j)$ with these properties, one has $\cC=\colim^{\llax} (\cA\overset{\Psi}\longrightarrow\cB)$, where 
$\cA :=\cC\setminus j(\cB )$ and $\Psi$ is the restriction of the left adjoint of $j$.
\end{rem}

\subsection{Coequalizers of categories}  \label{ss:coequalizers of categories}
Let $\cC,\cD$ be categories. Let $j_+:\cD\to \cC$ and $j_-:\cD\to \cC$ be functors.

\subsubsection{The category $\Coeq (j_-,j_+)$}   \label{sss:2Coeq}
Let $\Coeq (j_-,j_+)$ be the category with 
$$\Ob\Coeq (j_-,j_+)=\Ob\cC$$ 
obtained from $\cC$ by adding isomorphisms 
\begin{equation}  \label{e:2f_d}
f_d :j_-(d)\iso j_+(d), \quad d\in \cD
\end{equation}
 subject to the defining relations $f_{d'}\circ j_-(\varphi )=j_+(\varphi )\circ f_d$ for every morphism $\varphi :d\to d'$ in~$\cD$ (the relations say that $f_d$ is functorial in $d$). The category $\Coeq (j_-,j_+)$ is a coequalizer of $(j_-,j_+)$ in the sense of the 2-category of categories; however, $\Coeq (j_-,j_+)$ is defined up to unique \emph{isomorphism} (not merely up to equivalence). 
 
 Let us note that $\Coeq (j_-,j_+)$ has an $\infty$-categorical analog $\Coeq_\infty (j_-,j_+)$, whose 1-ca\-te\-gorical truncation is $\Coeq (j_-,j_+)$. The functor
 $\Coeq_\infty (j_-,j_+)\to\Coeq (j_-,j_+)$ is not an equivalence, in general.

\subsubsection{The $\BZ$-grading on $\Coeq (j_-,j_+)$}
By a $\BZ$-grading on a category $\cA$ we mean a functor\footnote{Here we are dealing with the \emph{set} of functors (as opposed to the category of functors).} $\cA\to B\BZ$, where $B\BZ$ is the category with one object whose monoid of endomorphisms is $\BZ$. 
Thus defining a grading on $\cA$ means assigning to each $\cA$-morphism $g$ an integer $\deg g$ so that $\deg (g_1g_2)=\deg g_1+\deg g_2$.

We equip $\Coeq (j_-,j_+)$ with the $\BZ$-grading such that all $\cC$-morphisms have degree $0$ and all morphisms \eqref{e:2f_d} have degree $1$.
The set of degree $n$ morphisms  $c_1\to c_2$ in $\Coeq (j_-,j_+)$ will be denoted by $\Mor^n_{\Coeq}(c_1,c_2)$.

Let $\Coeq_{\ge 0}(j_-,j_+)$ be the category obtained from $\Coeq (j_-,j_+)$ by removing all morphisms of negative degrees. Note that such morphisms exist if $\cD\ne\emptyset$ (e.g., the inverses of the isomorphisms \eqref{e:2f_d}).

\subsubsection{The goal}   \label{sss: the 3 assumptions}
We will describe $\Coeq (j_-,j_+)$ under the following assumptions:

(a) $j_-$ is a fully faithful left fibration;

(b) $j_-$ has a left adjoint $\Phi :\cC\to\cD$;

(c) $\cC_+\cap\cC_-=\emptyset$, where $\cC_\pm$ is the essential image of $j_\pm:\cD\to \cC$. 

We will use the notation
$$\Phi_\pm:=j_\pm\circ\Phi .$$

\subsubsection{The functor $\cC_{\Phi_+}\to \Coeq_{\ge 0}(j_-,j_+)$}
By adjunction, we have a canonical $\cC$-morphism $$c\to j_-(\Phi (c)), \quad c\in\cC .$$ 
Composing it with the $\Coeq$-isomorphism $f_{\Phi (c)}:j_-(\Phi (c))\iso j_+(\Phi (c))=\Phi_+(c)$,
we get a morphism
\[
c\to \Phi_+ (c)
\]
of degree $1$, which is functorial in $c\in\cC$. By \S\ref{sss:Ways to think}(ii), this gives a functor
\begin{equation}   \label{e:C_Phi to Coeq}
\cC_{\Phi_+}\to\Coeq_{\ge 0}(j_-,j_+),
\end{equation}
where $\cC_{\Phi_+}$ is the left-lax quotient.

\begin{prop}   \label{p:Coeq up to equivalence}
We keep the assumptions of \S\ref{sss: the 3 assumptions}.

(i) The functor $(\cC\setminus \cC_-)_{\Phi_+}\to\Coeq (j_-,j_+)$ induced by \eqref{e:C_Phi to Coeq} is an equivalence (here $\Phi_+$ is regarded as a functor
$\cC\setminus \cC_-\to\cC\setminus \cC_-$).

(ii) The functor $\Coeq_\infty (j_-,j_+)\to\Coeq (j_-,j_+)$  is an equivalence  (here $\Coeq_\infty$ is the $\infty$-categorical coequalizer).

(iii) The functor $\Coeq (\cD^{\rm g}\overset{j_+}{\underset{j_-}\rightrightarrows} \cC^{\rm g})\to \Coeq (\cD\overset{j_+}{\underset{j_-}\rightrightarrows} \cC)^{\rm g}$
 is an equivalence. Here the symbol g stands for the underlying groupoid.
\end{prop}

\begin{proof}
(i) For any category $\cE$, a functor $\Coeq (j_-,j_+)\to\cE$ is the same as a functor $G:\cC\to\cE$ plus an isomorphism
\begin{equation}   \label{e:Gj_+=Gj_-}
G\circ j_+\iso G\circ j_-
\end{equation}
As noted in Remark~\ref{r:lax colimits}, assumptions (a)-(b) of \S\ref{sss: the 3 assumptions} imply that 
$$\cC=\colim^{\llax} (\cC\setminus\cC_-\overset{\Phi_-}\longrightarrow\cC_-),$$
so $G$ is the same as a triple consisiting of functors $G_{\cC\setminus\cC_-}:\cC\setminus\cC_-\to\cE$,\,
$G_{\cC_-}:\cC_-\to\cE$ and a morphism $G_{\cC\setminus\cC_-}\to G_{\cC_-}\circ\Phi_-\,$. But because of \eqref{e:Gj_+=Gj_-} we have
$G_{\cC_-}\simeq G_{\cC\setminus\cC_-}\circ j_+j_-^{-1}$, so 
$$G_{\cC_-}\circ\Phi_-\simeq G_{\cC\setminus\cC_-}\circ\Phi_+\, .$$
Thus our data amounts to a functor $G_{\cC\setminus\cC_-}:\cC\setminus\cC_-\to\cE$ plus a morphism $G_{\cC\setminus\cC_-}\to G_{\cC\setminus\cC_-}\circ\Phi_+\,$.
This is the same as a functor $(\cC\setminus \cC_-)_{\Phi_+}\to\cE$.

(ii) By \S\ref{sss:GHN-corollary}, $\colim^{\llax} (\cC\setminus\cC_-\overset{\Phi_-}\longrightarrow\cC_-)$ is also the $\infty$-categorical left-lax colimit, and
$(\cC\setminus \cC_-)_{\Phi_+}$ is also the $\infty$-categorical left-lax quotient. So the proof of statement (i) shows that for any $\infty$-category $\cE$, a functor $\Coeq_\infty (j_-,j_+)\to\cE$ is the same as a functor $(\cC\setminus \cC_-)_{\Phi_+}\to\cE$, which is the same as a functor $\Coeq (j_-,j_+)\to\cE$.

(iii) The functor $(\cC\setminus\cC_- )^{\rm g}\to \Coeq (\cD^{\rm g}\overset{j_+}{\underset{j_-}\rightrightarrows} \cC^{\rm g})$ is clearly an equivalence.
The functor $(\cC\setminus\cC_- )^{\rm g}\to\Coeq (\cD\overset{j_+}{\underset{j_-}\rightrightarrows} \cC)^{\rm g}$ is an equivalence by statement (i) and Remark~\ref{r:same underlying groupoid}. 
\end{proof}

\begin{rem}  \label{r: Im (D to Coeq)}
In the situation of Proposition~\ref{p:Coeq up to equivalence}, the essential image of the canoncial functor $\cD\to \Coeq (j_-,j_+)$
identifies with $(\cC_+)_{\Phi_+}$. More precisely, this functor factors as $$\cD\overset{G_1}\epi (\cC_+)_{\Phi_+} \overset{G_2}\mono \Coeq (j_-,j_+),$$ where $G_1$ is the essentially surjective functor induced by~$j_+:\cD\to\cC_+$ and $G_2$ is the functor induced by ~\eqref{e:C_Phi to Coeq},  which is fully faithful by Propostion~\ref{p:Coeq up to equivalence}(i).
\end{rem}

\begin{prop}  \label{p:2morphisms in Coeq (j_+ to j_-)}
In addition to the assumptions of \S\ref{sss: the 3 assumptions}, suppose that $j_+$ is fully faithful. Then

(i) the functor \eqref{e:C_Phi to Coeq} is an isomorphism;  equivalently, it induces a bijection
\[
\Mor_\cC (\Phi_+^n(c_1),c_2)\iso \Mor^n_{\Coeq}(c_1,c_2) \; \mbox{ for }n\ge 0 \mbox{ and all } c_1,c_2\in\cC;
\]

(ii) $\Mor^n_{\Coeq}(c_1,c_2)=\emptyset$ if $n<-1$;

(iii) if $c_2\notin\cC_-$ then $\Mor^{-1}_{\Coeq}(c_1,c_2)=~\emptyset$; if $c_2\in\cC_-$ then one has a canonical bijection
\[
\Mor_{\cC} (c_1, j_+j_-^{-1}(c_2))\iso\Mor^{-1}_{\Coeq}(c_1,c_2)
\]
given by composition with $f^{-1}_{j_-^{-1}(c_2)}\in \Mor^{-1}_{\Coeq}(j_+j_-^{-1}(c_2),c_2)$.
\end{prop}

\begin{proof}
(a) Let us show that if $c_2\not\in\cC_-$ then the map $\Mor_\cC (\Phi_+^n(c_1),c_2)\to \Mor^n_{\Coeq}(c_1,c_2)$ induced by the functor  \eqref{e:C_Phi to Coeq} is bijective for all $n\in\BZ$. If $c_1\not\in\cC_-$ this immediately follows from Proposition~\ref{p:Coeq up to equivalence}(i). If $c_1\in\cC_-$ this follows from the commutative diagram
\[
\xymatrix{
\Mor^{n-1}_{\cC_{\Phi_+}}(\Phi_+(c_1),c_2)\ar[r]^-\sim \ar[d]^\wr & \Mor^n_{\cC_{\Phi_+}}(c_1,c_2)\ar[d]\\
\Mor^{n-1}_{\Coeq}(\Phi_+(c_1),c_2)\ar[r]^-\sim  & \Mor^n_{\Coeq}(c_1,c_2)
}
\]
whose horizontal arrows come from the canonical $\cC_{\Phi_+}$-morphism $c_1\to\Phi_+ (c_1)$ and the canonical $\Coeq$-isomorphism
$c_1\iso j_+j_-^{-1}(c_1)\simeq j_+\Phi (c_1)=\Phi_+ (c_1)$. Note that the upper horizontal arrow is bijective even for $n\le 0$: indeed, 
$\Mor^0_{\cC_{\Phi_+}}(c_1,c_2)=\Mor_\cC (c_1,c_2)=\emptyset$ because $c_1\in\cC_-$ and $c_2\not\in\cC_-\,$.

(b) Now let $c_2\in\cC_-$. To study the map $\Mor^n_{\cC_{\Phi_+}} (c_1,c_2)\to \Mor^n_{\Coeq}(c_1,c_2)$, consider the commutative diagram
\[
\xymatrix{
\Mor^n_{\cC_{\Phi_+}}(c_1,c_2)\ar[r] \ar[d] & \Mor^{n+1}_{\cC_{\Phi_+}}(c_1,\Phi_+(c_2))\ar[d]^\wr\\
\Mor^n_{\Coeq}(c_1,c_2)\ar[r]^-\sim & \Mor^{n+1}_{\Coeq}(c_1,\Phi_+(c_2))
}
\]
whose right vertical arrow is bijective by part (a) of the proof. It remains to show that the upper horizontal arrow is bijective if $n\ge 0$. This arrow is the composition of the maps
\[
\Mor_{\cC}((j_+\Phi)^n(c_1),c_2)\to\Mor_{\cC}(\Phi(j_+\Phi)^n(c_1),\Phi (c_2))\to\Mor_{\cC}((j_+\Phi)^{n+1}(c_1),j_+\Phi (c_2)).
\]
The second one is bijective because $j_+$ is fully faithful. To prove bijectivity of the first map, recall that $\Phi$ is left adjoint to $j_-$ and note that the morphism 
$c_2\to j_-\Phi (c_2)$ is an isomorphism because $c_2\in\cC_-$.
\end{proof}

\begin{cor}   \label{c:Coeq of subcategories}
In the situation of Proposition~\ref{p:2morphisms in Coeq (j_+ to j_-)}, let $\cC'\subset\cC$ be a strictly full subcategory such that 
$\Phi (\cC')\subset j_-^{-1}(\cC')=j_+^{-1}(\cC')$. Set $\cD'=j_-^{-1}(\cC')=j_+^{-1}(\cC')$. Then 

(i) the diagram $\cD'\overset{j_+}{\underset{j_-}\rightrightarrows} \cC'$ satisfies the conditions of Proposition~\ref{p:2morphisms in Coeq (j_+ to j_-)};

(ii) the canonical functor $G: \Coeq (\cD'\overset{j_+}{\underset{j_-}\rightrightarrows} \cC')\to\Coeq (\cD\overset{j_+}{\underset{j_-}\rightrightarrows} \cC )$ is fully faithful;

(iii) the preimage in $\cC$ of the essential image of $G$ equals $\cC'$.
\end{cor}

\begin{proof}
Statement (i) is clear. Statement (ii) follows from (i) and Proposition~\ref{p:2morphisms in Coeq (j_+ to j_-)}. 

Using Proposition~\ref{p:Coeq up to equivalence}(iii), one checks that $\cC'$ is stable under $\Coeq$-isomorphisms, where 
$\Coeq:=\Coeq (\cD\overset{j_+}{\underset{j_-}\rightrightarrows} \cC )$. This is equivalent to (iii).
\end{proof}

\subsection{Proofs of Theorems~\ref{t:naive Coeq is OK}, \ref{t:Sigma''}, \ref{t:Sigma_F}}  \label{ss:proving coequalizer theorems}
\subsubsection{Recollections on $j_\pm :\Sigma\mono\Sigma'$}  \label{sss:j_pm recollections}
We have a g-stack $\Sigma$, a c-stack $\Sigma'$, and morphisms $j_\pm:\Sigma\to\Sigma'$ (see \S\ref{ss:def of Sigma}, \S\ref{ss:Sigma'&ring stack}-\ref{ss:Sigma_+}, and Lemma~\ref{l:F' & j_-}).  Let us recall the properties of the diagram $\Sigma\overset{j_+}{\underset{j_-}\rightrightarrows} \Sigma'$, which will be used in the proofs.

The morphisms $j_\pm$ are affine open immersions (see Lemma~\ref{l:Sigma_+ open affine} and  \S\ref{sss:Sigma_-}).
By Lemma~\ref{l: Sigma_+ cap Sigma_-}, $\Sigma_+\cap\Sigma_-=\emptyset$, where $\Sigma_\pm\subset\Sigma$ is the image of $j_\pm$.

The morphism $j_+$ is a right fibration (see Corollary~\ref{c:Sigma_+ right-admissible}), the morphism $j_-$ is a left fibration (see \S\ref{sss:Sigma_-}).

$j_-$ has a left adjoint $F'$ (see \S\ref{ss:F'} and \S\ref{sss:id to F'_-}). Moreover, $F'\circ j_+$ is the morphism $F:\Sigma\to\Sigma$, see \eqref{e:F'j_+}.

The c-stack $\Sigma'$ is $\MMor$-affine in the sense of Definition~\ref{d:pre-algebraic stack}, see Theorem~\ref{t:Sigma' algebraic}(ii).

This property and affineness of $j_\pm$ will allow us in \S\ref{sss:proof Sigma''(i)} to apply Lemma~\ref{l:one-point compactification}. 

\subsubsection{Proof of Theorem~\ref{t:naive Coeq is OK}}   \label{sss:proof naive Coeq is OK}
 The first part of the theorem immediately follows from Proposition~\ref{p:Coeq up to equivalence}(ii).
 
The second part of the theorem is a consequence of the first one and the following statement: \emph{let $\cF$ be a presheaf of categories on some site, then its sheafification in the usual sense is also a sheafification in the $\infty$-categorical sense.} I am unable to find a reference for this well known fact. Instead, let me give two proofs of the statement, which I learned from the experts. Preliminary step for both proofs: one can assume that $\cF$ is a presheaf of groupoids\footnote{Indeed, a category $\cC$ can be viewed as a simplicial groupoid (its groupoid of $n$-simplices is the groupoid of functors $[n]\to\cC$).}. 

First argument (J.~Lurie): $\infty$-categorical sheafification is a left exact functor (i.e., it commutes with finite limits), so by Prop.~5.5.6.16 of \cite{Lu1}, it preserves
$n$-truncated objects for any value of $n$. In the case $n=1$ this means that the $\infty$-categorical sheafification of a presheaf of groupoids is a sheaf of groupoids. This implies our statement.

The second argument (N.~Rozenblyum) uses some facts from the  2017 version of \cite{Lu1} (which is available on J.~Lurie's homepage). As explained on the first two
pages of \S 6.5.3 of this version, if $\cF$ is a presheaf of groupoids then the usual sheafification of $\cF$ equals $((\cF^\dagger)^\dagger)^\dagger$. On the other hand,
it is known that for any sheaf of $\infty$-groupoids $\cG$ the canonical map  $\Hom (\cF^\dagger, \cG)\to\Hom (\cF,\cG)$ is an isomorphism (this is Lemma~6.2.2.14 of the 2017 version of \cite{Lu1}).

 \subsubsection{Proof of Theorem~\ref{t:Sigma''}(ii-iii)}    \label{sss:proof Sigma''(ii)}
 The description of $\Sigma''$ immediately after Definition~\ref{d:Sigma''} implies that $(\Sigma'')^{\rm g}$ is the stack obtained by sheafifying the assignment
$$S\mapsto \Coeq (\Sigma(S)\overset{j_+}{\underset{j_-}\rightrightarrows} \Sigma'(S))^{\rm g}.$$
Since $\Sigma (S)$ is a groupoid, Proposition~\ref{p:Coeq up to equivalence}(iii) implies that 
$$\Coeq (\Sigma(S)\overset{j_+}{\underset{j_-}\rightrightarrows} \Sigma'(S))^{\rm g}=
\Coeq (\Sigma(S)\overset{j_+}{\underset{j_-}\rightrightarrows} (\Sigma')^{\rm g}(S)).$$
 Theorem~\ref{t:Sigma''}(ii) follows. Now proving  Theorem~\ref{t:Sigma''}(iii) amounts to checking $\MMor$-affine\-ness of
 $\Coeq (\Sigma\overset{j_+}{\underset{j_-}\rightrightarrows} (\Sigma')^{\rm g}).$ It follows from $\MMor$-affine\-ness of $\Sigma'$ 
 (see Theorem~\ref{t:Sigma' algebraic}(ii)) and affineness of $j_\pm :\Sigma\to\Sigma'$.

 \subsubsection{Proof of Theorem~\ref{t:Sigma_F}}      \label{sss:proof Sigma_F}
Applying Remark~\ref{r: Im (D to Coeq)} to the diagram $\Sigma(S)\overset{j_+}{\underset{j_-}\rightrightarrows} \Sigma'(S)$, we get a factorization of the morphism 
$\Sigma\to\Sigma''$ as $\Sigma\to (\Sigma_+)_{F'_+}\iso\sU\subset\Sigma''$ for some sub\-stack $\sU\subset\Sigma''$. One has $(\Sigma_+)_{F'_+}=\Sigma_F$ because
the isomorphism $j_+:\Sigma\iso\Sigma_+$ identifies $(\Sigma,F)$ with $(\Sigma_+,F'_+:\Sigma_+\to\Sigma_+)$ (indeed, $F'_+\circ j_+=j_+\circ F'\circ j_+=j_+\circ F)$.
To check that the substack $\sU\subset\Sigma''$ is open,
it suffices to prove that the morphism
\begin{equation}   \label{e:immersion of g-stacks}
(\Sigma_F)^{\rm g}\to(\Sigma'')^{\rm g}
\end{equation}
induced by the morphism $\Sigma_F\to\Sigma''$ is an open immersion.
By Remark~\ref{r:same underlying groupoid} and Theorem~\ref{t:Sigma''}(ii), the morphism \eqref{e:immersion of g-stacks} is just the canonical morphism $\Sigma\to 
\Coeq (\Sigma \overset{j_+}{\underset{j_-}\rightrightarrows} (\Sigma')^{\rm g})$, which is an open immersion because $\Sigma_+\cap\Sigma_-=\emptyset$.

 \subsubsection{Proof of Theorem~\ref{t:Sigma''}(i)}   \label{sss:proof Sigma''(i)}
 By Remark~\ref{r:adicity in terms of g-stacks}, it suffices to check that $(\Sigma'')^{\rm g}$ is strongly adic and $\Sigma''$ is pre-algebraic.
 
 In \S\ref{sss:proof Sigma''(ii)} we proved that $(\Sigma'')^{\rm g}=\Coeq (\Sigma \overset{j_+}{\underset{j_-}\rightrightarrows} (\Sigma')^{\rm g})$.
 By Corollary~\ref{c:strongly adic}, the c-stack $\Sigma'$ is strongly adic in the sense of \S\ref{sss:strongly adic}. By Proposition~\ref{p:sX'}(ii), $\Sigma'_{\red}\cap\Sigma_\pm=(\Sigma_\pm )_{\red}\,$. So $(\Sigma'')^{\rm g}$ is strongly adic.
 
 It remains to show that $\Sigma''$ is pre-algebraic in the sense of Definition~\ref{d:pre-algebraic stack}. Let $S$ be a scheme and $x''_1,x''_2\in \Sigma''(S)$. Then we have the contravariant functor $\MMor (x''_1,x''_2)$ on the category of $S$-schemes (see \S\ref{ss:algebraic stacks}), and the problem is to show that this functor is representable by an $S$-scheme. There exists a quasi-compact faithfully flat mophism of schemes $\tilde S\to S$ such that $x''_1,x''_2$ come from some
 $x'_1,x'_2\in\Sigma'(\tilde S)$.  By Lemma~\ref{l:one-point compactification}, it suffices to show that $\MMor (x_1,x_2)\times_S\tilde S$ is a disjoint union of schemes affine over $\tilde S$. By Proposition~\ref{p:2morphisms in Coeq (j_+ to j_-)}, 
 \begin{equation}    \label{e:MMor (x_1,x_2)}
 \MMor (x_1,x_2)\times_S\tilde S=\bigsqcup\limits_{n\ge -1}Y_n,
 \end{equation}
 where 
 $\bigsqcup$ stands for the coproduct in the category of fqpc-sheaves, 
 $$Y_n:=\MMor ((F'_+)^n(x'_1),x'_2)\quad \mbox{ if }n\ge 0,$$ and $Y_{-1}$ is defined by the Cartesian square
 \[
\xymatrix{
Y_{-1}\ar[r] \ar[d] & \Sigma\ar[d]^{(j_+,j_-)}\\
\tilde S\ar[r]^-{(x'_1,x'_2)} & \Sigma'\times\Sigma'
}
\]
$Y_{-1}$ is a scheme affine over $\tilde S$ because the morphisms $j_\pm:\Sigma\to\Sigma'$ are affine. If $n\ge0$ then $Y_n$ is a scheme affine over $\tilde S$ because $\Sigma'$ is $\MMor$-affine (see Theorem~\ref{t:Sigma' algebraic}(ii).
 
 \subsubsection{Proof of Theorem~\ref{t:Sigma''}(iv)}   \label{sss:proof Sigma''(iv)}
 Corollary~\ref{c:Sigma''}(i) implies that the morphism $\tilde S\to S$ from \S\ref{sss:proof Sigma''(i)} can be chosen to be etale. Besides, the coproduct $\bigsqcup\limits_{n\ge -1}$ of fqpc sheaves (which appears in \S\ref{sss:proof Sigma''(i)}) is also a coproduct of etale (or Zariski) sheaves. Theorem~\ref{t:Sigma''}(iv) follows.

\subsection{Proof of Theorem~\ref{t:Sigma'' to Q}} \label{ss:Sigma'' to Q proof}
\subsubsection{Plan of the proof}
If $y_1,y_2\in\cQ (S)$ let $\MMor (y_1,y_2)$ be the fpqc sheaf whose set of sections over a scheme $S'$ equipped with a map $f:S'\to S$ equals $\Mor (f^*y_1,f^*y_2)$.
To prove Theorem~\ref{t:Sigma'' to Q}, it suffices to show that if $x'_1,x'_2\in\Sigma'(S)$ and $x_i''\in\Sigma''(S)$ is the image of $x'_i$ then the morphism
\begin{equation}   \label{e:what should be an iso}
\MMor (x''_1,x''_2)\to\MMor (\Phi (x'_1),\Phi (x'_2))
\end{equation}
is an isomorphism; here $\Phi :\Sigma'(S)\to\cQ (S)$ is as in \S\ref{sss:Sigma' to Q}. This will become clear once we describe  $\MMor (\Phi (x'_1),\Phi (x'_2))$ (see \S\ref{sss:grading on MMor} and Proposition~\ref{p:all comes from Constructions 1-2}). 

\subsubsection{The $\BZ$-grading on $\MMor (\Phi (x'_1),\Phi (x'_2))$}   \label{sss:grading on MMor}
Let $x'_i=(M_i,\xi_i)\in\Sigma'(S)$. Let $$\tilde M_i:=M_i\times_{W_S}\widetilde{W}_S$$ and let
$\tilde\xi_i:\tilde M_i\to\widetilde{W}_S$ be induced by $\xi_i:M_i\to W_S$. Using Proposition~\ref{p:Aut tilde W}, we see  that  a 
$\cQ (S)$-morphism $\Phi (M_1,\xi_1)\to \Phi (M_2,\xi_2)$ is a pair $(n,\tilde f)$, where $n\in H^0(S,\BZ)$ and $\tilde f:\tilde M_1\to \tilde M_2$ is an $F^n$-linear map of $\widetilde{W}_S$-module schemes such that $\tilde\xi_2\circ \tilde f=F^n\circ\tilde\xi_1$. Thus we get a grading
\[
\MMor (\Phi (x'_1),\Phi (x'_2))=\bigsqcup_{n\in\BZ}\MMor^n (\Phi (x'_1),\Phi (x'_2)).
\]
Let $\Mor^n (\Phi (x'_1),\Phi (x'_2))$ be the set of global sections of $\MMor^n (\Phi (x'_1),\Phi (x'_2))$.

\begin{prop} \label{p:all comes from Constructions 1-2}
(i) For $n\ge 0$ one has a canonical bijection
\[
\Mor ((F'_+)^n(x'_1),x'_2)\iso\Mor^n (\Phi (x'_1),\Phi (x'_2))
\]
functorial in $x'_1,x'_2\in\Sigma'(S)$. (As usual, $F'_+:\Sigma'\to\Sigma'$ is given by $F'_+:=j_+\circ F'$.)

(ii)  If $x'_1\in\Sigma_+(S)$ and $x'_2\in\Sigma_-(S)$ then all elements of $\Mor^1 (\Phi (x'_2),\Phi (x'_1))$ are isomorphisms.

(iii) In the situation of (iii) the inverses of the elements of $\Mor^1 (\Phi (x'_2),\Phi (x'_1))$ give all elements of $\Mor^{-1} (\Phi (x'_1),\Phi (x'_2))$. 
If $x'_1\not\in\Sigma_+(S)$ or $x'_2\notin\Sigma_-(S)$ then $\Mor^{-1} (\Phi (x'_1),\Phi (x'_2))=~\emptyset$.

(iv) If $S\ne\emptyset$ and $n<-1$, then $\Mor^n (\Phi (x'_1),\Phi (x'_2))=\emptyset$.

\end{prop}

\begin{proof} 
Let $(M_i,\xi_i)$ and $(\tilde M_i,\tilde \xi_i)$ be as in \S\ref{sss:grading on MMor}. An element of $\Mor^n (\Phi (x'_1),\Phi (x'_2))$ is the same as an $F^n$-linear map
$\tilde f:\tilde M_1\to\tilde M_2$ such that $\tilde\xi_2\circ \tilde f=F^n\circ\tilde\xi_1$.

(i) If $n\ge 0$ then $F^n$ maps $\Ker (\widetilde{W}_S\to W_S)$ to itself, so $\tilde f$ is the same as an $F^n$-linear map of $W_S$-module schemes
$f:M_1\to M_2$ such that $\xi_2\circ f=F^n\circ\xi_1$. Such $f$ is the same as an element of $\Mor ((F'_+)^n(x'_1),x'_2)$.

(ii) Let $x'_!\in\Sigma_+(S)$, $x'_2\in\Sigma_-(S)$, $u\in \Mor^1 (\Phi (x'_2),\Phi (x'_1))$. By (i), $u$ comes from some $v\in\Mor (F'_+(x'_2),x'_1)$. Since $\Sigma_+$ is a g-stack, $v$ is an isomorphism. Since $x'_2\in\Sigma_-(S)$, we have $F'_+(x'_2)=j_+j_-^{-1}(x_2)$, and our $u$ is essentially the isomorphism $\Phi (x'_2)\iso (\Phi\circ j_+j_-^{-1})(x'_2)$ from \S\ref{sss:Sigma'' to Q}.

(iii) Let $\tilde f\in\Mor^{-1} (\Phi (x'_1),\Phi (x'_2))$, where $x'_1,x'_2\in\Sigma'(S)$. Let us first show that $x'_2\in\Sigma_-(S)$. Set $K:=\Ker (\widetilde{W}_S\to W_S)$; then $F^{-1}:\widetilde{W}_S\to\widetilde{W}_S$ induces an isomorphism $$K/F(K)\iso W_S^{(F)}.$$
The composite morphism $K\mono\tilde M_1\overset{\tilde f}\longrightarrow\tilde M_2\epi M_2$ induces a morphism $$W_S^{(F)}=K/F(K)\to M_2$$ such that 
the composite morphism $W_S^{(F)}\to M_2\overset{\xi_2}\longrightarrow W_S$ is equal to the canonical embedding 
$W_S^{(F)}\mono W_S$. So the left vertical arrow of the usual diagram
\[
\xymatrix{
0\ar[r]&\sL_2^\sharp\ar[r] \ar[d]& M_2\ar[r]\ar[d]^{\xi_2}&M'_2\ar[r]\ar[d]&0 \\
0\ar[r]&W_S^{(F)}\ar[r] & W_S\ar[r]^F & W_S^{(1)}\ar[r] &0
}
\]
is a split epimorphism. By Proposition~\ref{p:more HHoms}(i), this implies that it is an isomorphism, which means that $x'_2\in\Sigma_-(S)$.

Let $x'_3:=F'_+(x'_2)=j_+j_-^{-1}(x'_2)\in\Sigma_+(S).$ Then using (i), we get a canonical element $\tilde g\in\Mor^{1}(\Phi (x'_2),\Phi (x'_3))$; by (iii), $\tilde g$ is an isomorphism. By (i), $\tilde g\circ\tilde f\in\Mor^0(\Phi (x'_1), \Phi (x'_3))$ comes from  a morphism $h:x'_1\to x'_3$. Since $x'_3\in\Sigma_+(S)$ and $\Sigma_+$ is a right-fibered over $\Sigma'$ (see Corollary~\ref{c:Sigma_+ right-admissible})), we have $x'_1\in\Sigma_+(S)$. Since $\Sigma_+$ is a g-stack, $h$ is an isomorphism, so $\tilde g\circ \tilde f$ is an isomorphism.
Therefore $\tilde f$ is an isomorphism. Its inverse belongs to $\Mor^1 (\Phi (x'_2),\Phi (x'_1))$.

(iv) Suppose that $u\in\Mor^{-1-m} (\Phi (x'_1),\Phi (x'_2))$ for some $m>0$. Let $x'_3:=(F'_+)^m(x_2)$; then $x'_3\in\Sigma_+(S)$.  By (i), we have a canonical element of $\Mor^m(\Phi (x'_2), \Phi (x'_3)$. Composing it with $u$, we get an element of $\Mor^{-1} (\Phi (x'_1),\Phi (x'_3))$. By (iii), we see that
$x'_3\in\Sigma_-(S)$. Thus  $x'_3\in\Sigma_+(S)\cap\Sigma_-(S)$. But $\Sigma_+\cap\Sigma_-=\emptyset$, so $S=\emptyset$.
\end{proof} 

\subsubsection{End of the proof of Theorem~\ref{t:Sigma''}}
To prove that \eqref{e:what should be an iso} is an isomorphism, compare Proposition~\ref{p:all comes from Constructions 1-2} with the description of 
$\MMor (x''_1,x''_2)$ from formula~\eqref{e:MMor (x_1,x_2)} (in the case $\tilde S=S$).

\subsection{A toy model for $\Sigma''$}    \label{ss:2toy model}

\subsubsection{The c-stack $\fS''$}   \label{sss:fS''}
In \S\ref{sss:fS'} we defined a c-stack $\fS'$ over $\BF_p$ and its open substacks $\fS_\pm\subset\fS'$ such that $\fS_+\simeq\Spec\BF_p\simeq\fS_-$. 
Let $\nu_\pm :\Spec\BF_p\mono\fS'$ be the open immersion with image $\fS_\pm$.  The categorical meaning of $\nu_\pm$ is as follows: for any $\BF_p$-scheme~$S$, the image of the unique object of $(\Spec\BF_p )(S)$ under $\nu_+$ (resp. $\nu_-$) is the initial (resp.~final) object of $\fS'(S)$.
Set 
$$\fS'':=\Coeq (\Spec\BF_p\overset{\nu_+}{\underset{\nu_-}\rightrightarrows} \fS').$$
If one believes in \S\ref{sss:(Spec F_p)^prismp}  then $\fS''=(\Spec\BF_p)^\prismpp\otimes\BF_p$. The c-stack $\fS''$ is a toy model for $\Sigma''$.
The morphism $$(\fS',\fS_+,\fS_-)\to(\Sigma',\Sigma_+,\Sigma_-)$$ constructed in \S\ref{sss:fS' to Sigma'} 
induces\footnote{Indeed,  the composite morphisms $\Spec\BF_p\iso\fS'_\pm\to\Sigma'_\pm\iso\Sigma$ are equal to each other because the category $\Sigma (\BF_p )$ is a point.} a morphism 
\[
\fS''\to \Sigma''.
\]
In Proposition~\ref{p:fS''}(vi) below we will show that if $S$ is a perfect $\BF_p$-scheme then the functor $\fS''(S)\to \Sigma''(S)$ is an equivalence.

In Propositions~\ref{p:fS''=C/Gamma''} and \ref{p:fS''(S)} below we will describe  the c-stack $\fS''$ explicitly.\footnote{One can describe $(\Spec\BF_p)^\prismpp$ similarly.} But first we will describe in Lemma~\ref{l:warm-up} the g-stack
\begin{equation}     \label{e:g-version of fS''}
\Coeq (\Spec\BF_p\overset{\nu_+}{\underset{\nu_-}\rightrightarrows} (\fS')^{\rm g}),
\end{equation}
where $(\fS')^{\rm g}$ is the g-stack underlying the c-stack $\fS'$.

\subsubsection{The ``coordinate cross" and the nodal curve}   \label{sss:C&node}
Just as in \S\ref{sss:fS'}, let  $$C:=\Spec \BF_p[v_+,v_-]/(v_+v_-);$$ 
in other words, $C$ is the ``coordinate cross" on the plane $\BA^2_{\BF_p}$ with coordinates $v_+,v_-$.
Let $C_\pm\subset C$ be the open subscheme $v_\pm\ne0$.

Let $\nodal$ be the nodal curve obtained from $\BP^1_{\BF_p}$ by gluing $0$ with $\infty$. 
The morphism $C\to\nodal$ obtained by gluing the composite maps
\[
\Spec\BF_p[v_+,v_-]/(v_-)\overset{(v_+:1)}\longrightarrow\BP^1_{\BF_p}\to\nodal \mbox{\; and\; } \Spec\BF_p[v_+,v_-]/(v_+)\overset{(1:v_-)}\longrightarrow\BP^1_{\BF_p}\to\nodal 
\]
exhibits $\nodal$ as a quotient of $C$ by an etale equivalence relation. More precisely, $\nodal$ is obtained from $C$ by gluing $C_+$ with $C_-$ via the inversion map.

\begin{lem}    \label{l:warm-up}
The g-stack \eqref{e:g-version of fS''} canonically identifies with $\nodal /\BG_m$, where the action of $\BG_m$ on $\nodal$ is induced by its usual action on $\BP^1_{\BF_p}$. \qed
\end{lem}

\begin{prop}    \label{p:fS''}
(i) The c-stack $\fS''$ is algebraic. 

(ii) Its underlying g-stack $(\fS'')^{\rm g}$ equals $\nodal/\BG_m=\Coeq (\Spec\BF_p\overset{\nu_+}{\underset{\nu_-}\rightrightarrows} (\fS')^{\rm g})$. 
In particular, $(\fS'')^{\rm g}$ is $\MMor$-affine in the sense of Definition~\ref{d:pre-algebraic stack}.

(iii) The open substack of $\fS''$ corresponding to the open point $\Spec\BF_p=\BG_m/\BG_m\subset\nodal/\BG_m$ equals $\Spec\BF_p\times B\BZ_+$, where
$B\BZ_+$ is the classifying stack of the monoid $\BZ_+$.

(iv) The c-stack $\fS''$ is also a coequalizer of the diagram 
$\Spec\BF_p\overset{\nu_+}{\underset{\nu_-}\rightrightarrows} \fS'$
in the sense of the $(\infty ,2)$-category of stacks of $(\infty ,1)$-categories.

(v) The etale sheafification of the assignment $S\mapsto \Coeq ((\Spec\BF_p)(S)\overset{\nu_+}{\underset{\nu_-}\rightrightarrows} \fS'(S))$ is equal to the fpqc sheafification (i.e., to $\fS''$).

(vi) If $S$ is a perfect $\BF_p$-scheme then the functor $\fS''(S)\to \Sigma''(S)$ is an equivalence.
\end{prop}

\begin{proof}
The properties of the diagram $\Sigma\overset{j_+}{\underset{j_-}\rightrightarrows} \Sigma'$ formulated in \S\ref{sss:j_pm recollections} also hold for the diagram 
$\fS\overset{\nu_+}{\underset{\nu_-}\rightrightarrows} \fS'$. 
So the arguments used in \S\ref{ss:proving coequalizer theorems} to prove statements about $\Sigma''$ also prove statements (i)-(v) about $\fS''$. 

Let us prove  (vi). If $S$ is a perfect $\BF_p$-scheme then the functors $\fS'(S)\to \Sigma'(S)$ and $(\Spec\BF_p )(S)\to\Sigma (S)$ are equivalences, see  Proposition~\ref{p:fS' to Sigma'} and Lemma~\ref{l:Sigma(perfect)}. Combining this with statement (v) and Theorem~\ref{t:Sigma''}(iv) we get (vi).

Let us note that statement (vi) also follows from Lemma~\ref{l:fS'' and Sigma''_red}(i) below.
\end{proof}

\subsubsection{Visualizing $\fS''$}  \label{sss:visualizing fS''}
Lemma~\ref{l:warm-up} allows one to visualize the g-stack \eqref{e:g-version of fS''}. One can  visualize the c-stack $\fS''$ using the device from \S\ref{sss:drawing Gamma'} with $C$ replaced by the projective nodal curve~$\nodal$. What follows is an attempt to translate this picture into words.

\subsubsection{$\fS''$ as a quotient of the coordinate cross}
We are going to describe $\fS''$ as $C/\Gamma''$, where $\Gamma''$ is an explicit algebraic category with $\Ob \Gamma''=C$.

Let $\Gamma'$ be the algebraic category with $\Ob \Gamma'=C$ such that $\fS'=C/\Gamma'$. Explicitly, the scheme of morphisms of $\Gamma'$ 
is the closed subscheme of $\BA^1\times C\times C$ defined by the equations 
$$\tilde v_+=\lambda v_+, \quad v_-=\lambda\tilde  v_-,$$
where $\lambda\in\BA^1_{\BF_p}$, $(v_+,v_-)\in C$ is the source, and $(\tilde v_+,\tilde v_-)\in C$ is the target.

Set $\Gamma'':=\bigsqcup\limits_{n\in\BZ}\Gamma''_n$, where $\Gamma''_n$ is the following scheme over $C\times C$:
\[
\Gamma''_0:=\Gamma', \quad \Gamma''_n=C\times C \mbox{ for }n>0, 
\quad \Gamma''_{-1}=C_+\times C_-, \quad \Gamma''_n=\emptyset \mbox{ for }n<-1.
\]

\begin{prop}    \label{p:fS''=C/Gamma''}
(i) There is a unique morphism $\Gamma''_m\times_C\Gamma''_n\to\Gamma''_{m+n}$ over $C\times C$. It makes $\Gamma''$ into a $\BZ$-graded algebraic category with $\Ob\Gamma''=C$.

(ii) The composite morphism $C\to\fS'\to\fS''$ induces an isomorphism $C/\Gamma''\iso\fS''$. 
\end{prop}

\begin{proof}    
Verifying statement (i) is straightforward.

Let us sketch the proof of statement (ii). For $i=1,2$ let $x_i \in\fS'' (C\times C)$ be the composite morphism $C\times C\overset{\ppr_i}\longrightarrow C\to\fS''$. 
We have the scheme $\MMor (x_1,x_2)$ over $C\times C$, and proving statement (ii) amounts to constructing an isomorphism
$\MMor (x_1,x_2)\iso\Gamma''$ over $C\times C$. One does it by using Proposition~\ref{p:2morphisms in Coeq (j_+ to j_-)}, just as in the proof of formula~\eqref{e:MMor (x_1,x_2)}.
\end{proof}

\begin{rem}
Similarly to \S\ref{sss:Sigma'' to BZ}, one defines a canonical morphism $\fS''\to B\BZ$; it yields a $\BZ$-torsor $\tilde\fS''\to\fS''$. One has $\fS''=\tilde\fS''/\BZ$. The description of $\fS''$ from Proposition~\ref{p:fS''=C/Gamma''}(ii) can be transformed into a description of $\tilde\fS''$ as a quotient of $C\times\BZ$ or as a quotient of the universal cover of the nodal curve $\nodal$.
\end{rem}

\subsubsection{$S$-points of $\fS''$: format of the description}
Our next goal is to describe the category $\fS''(S)$ for any scheme $S$. The format of the description given in \S\ref{sss:describing fS''(S)} below is as follows: an object of $\fS''(S)$ is given by a $\BZ$-torsor on $S$, a line bundle on $S$, and some extra data. This format is natural because we have canonical morphisms
$\fS''\to\Sigma''\to B\BZ\times B\BA^1_m$,
see \S\ref{sss:fS''} and formula \eqref{e:Sigma'' to product}.

\subsubsection{Warm-up: field-valued points of $\nodal$}  \label{sss:field-valued points of nodal}
Let $k$ be a field of characteristic $p$. Then we think of a $k$-point of the nodal curve $\nodal$ as a pair consisting of a $\BZ$-torsor $T$ and a collection of
points $x_i\in\BP^1(k)$, $i\in T$, satisfying the following \emph{admissibility conditions:}

(i) for every $i\in T$, either $x_i=0$ or $x_{i+1}=\infty$; 

(ii) there exist $i,j\in T$ such that $x_i\ne 0$ and $x_j\ne\infty$.

\subsubsection{Describing $\fS''(S)$}  \label{sss:describing fS''(S)}
Let $T$ be a $\BZ$-torsor (over a point). Let $S$ be an $\BF_p$-scheme and $\sL$ a line bundle on $S$.
Let $\{\cE_i\}_{i\in T}$ be a collection of line subbundles of $\sL\oplus\cO_S$. Let $\alpha_i:\cE_i\to\sL$ and $\beta_i:\cE_i\to\cO_S$ be the projections. Similarly to
\S\ref{sss:field-valued points of nodal}, we say that the collection $\{\cE_i\}_{i\in T}$ is admissible if

(i) $\alpha_i\otimes\beta_{i+1}=0$ for all $i\in T$;    

(ii) for every $s\in S$, there exist $i,j\in T$ such that $\alpha_i$ and $\beta_j$ are isomorphisms at $s$.

\noindent Let $\sX_{\pre}(S)$ be the following category:

(a) its objects are triples $(T,\sL, \{\cE_i\}_{i\in T})$, where $\{\cE_i\}_{i\in T}$ is an admissible collection of line subbundles of $\sL\oplus\cO_S$;

(b) by a morphism $(T,\sL, \{\cE_i\})\to (T',\sL', \{\cE'_j\})$ we mean a pair $(f,g)$, where $f:T\iso T'$ is an isomorphism of $\BZ$-torsors and $g:\sL\to\sL'$ is a morphism of line bundles such that for every $i\in T$ the morphism $g\oplus\id :\sL\oplus\cO_S\to \sL'\oplus\cO_S$ maps $\cE_i$ to $\cE'_{f(i)}$.

Let $\sX$ be the c-stack over $\BF_p$ obtained by sheafifying the assignment $S\mapsto\sX_{\pre}(S)$. (One can also describe $\sX$ directly by considering $\BZ$-torsors over $S$ instead of $\BZ$-torsors over a point.)

\begin{prop}    \label{p:fS''(S)}
One has a canonical isomorphism $\fS''\iso\sX$.
\end{prop}

\begin{proof}
Let us first construct a morphism $\fS'\to\sX$. For any $\BF_p$-scheme $S$, an $S$-point of $\fS'$ was defined in \S\ref{sss:fS'} to be a diagram
\begin{equation} \label{e:5v_+,v_-}
 \cO_S\overset{v_+}\longrightarrow\sL\overset{v_-}\longrightarrow\cO_S, \quad v_-v_+=0,
 \end{equation}
where $\sL$ is a line bundle on $S$. To such a diagram we associate the following collection of line subbundles
$\cE_i\subset\sL\oplus\cO_S$, $i\in\BZ$:
\[
\cE_0:=\im (\cO_S\overset{(v_+,1)}\longrightarrow\sL\oplus\cO_S), \quad \cE_1:=\im (\sL\overset{(1,v_-)}\longrightarrow\sL\oplus\cO_S), 
\]
\[
\cE_i:=0\oplus\cO_S\subset\sL\oplus\cO_S \mbox{ for }i<0, \quad \cE_i:=\sL\oplus 0\subset\sL\oplus\cO_S \mbox{ for }i>1.
\]
This collection is admissible. Thus we get a morphism $\fS'\to\sX$. 

Let us show that the composite morphisms
\[
\Spec\BF_p\overset{\nu_-}\mono\fS'\to\sX \quad \mbox{and}\quad \Spec\BF_p\overset{\nu_+}\mono\fS'\to\sX
\]
are canonically isomorphic to each other. Recall that the morphisms $\nu_\pm:\Spec\BF_p\mono\fS'$ correspond to the following two diagrams of type \eqref{e:5v_+,v_-}:
\[
\cO_S\overset{\id}\longrightarrow\cO_S\overset{0}\longrightarrow\cO_S, \quad  \cO_S\overset{0}\longrightarrow\cO_S\overset{\id}\longrightarrow\cO_S, 
\quad S=\Spec\BF_p\, .
\]
It is easy to check that the compatible collections $\{\cE_i\}_{i\in \BZ}$ corresponding to these diagrams are the same up to a shift of the indices $i$.

Thus we have constructed a morphism $\fS''\to\sX$. To prove that it is an isomorphism, use Proposition~\ref{p:fS''=C/Gamma''}(ii).
\end{proof}

\begin{rem}
If one believes in \S\ref{sss:(Spec F_p)^prismp}  then $\fS''=(\Spec\BF_p)^\prismpp\otimes\BF_p$, so Proposition~\ref{p:fS''(S)} can be regarded as a description of
$(\Spec\BF_p)^\prismpp (S)$ for $S$ being an $\BF_p$-scheme. In fact, for any $p$-nilpotent scheme $S$ one can describe $(\Spec\BF_p)^\prismpp (S)$ in the same spirit. The required changes are as follows. First, the equation $\alpha_i\otimes\beta_{i+1}=0$ from \S\ref{sss:describing fS''(S)} has to be replaced by the equation
$\alpha_i\otimes\beta_{i+1}=p\alpha_{i+1}\otimes\beta_i$. Second, the line subbundles $\cE_i\subset\sL\oplus\cO_S$ from the proof of Proposition~\ref{p:fS''(S)}
have to be replaced by the following ones:
\[
\cE_i:=\im (\cO_S\overset{(p^{-i}v_+,1)}\longrightarrow\sL\oplus\cO_S) \mbox{ if } i\le 0, \quad 
\cE_i:=\im (\cO_S\overset{(1,p^{i-1}v_-)}\longrightarrow\sL\oplus\cO_S) \mbox{ if } i\ge 1.
\]
\end{rem}

\subsubsection{The $\cO_{\fS''}$-module  $\sL_{\fS''}$}   \label{sss:L_fS''}
Recall that for any $\BF_p$-scheme $S$, the category $\fS' (S)$ consists of diagrams
\begin{equation}    \label{e:6v_+,v_-}
\cO_S\overset{v_+}\longrightarrow\sL\overset{v_-}\longrightarrow\cO_S, \quad v_-v_+=0,
\end{equation}
where $\sL$ is a line bundle on $S$. Recall that  $\nu_+ :\Spec\BF_p\to\fS'$ and $\nu_- :\Spec\BF_p\to\fS'$ correspond to the diagrams
$\BF_p\overset{1}\longrightarrow\BF_p\overset{0}\longrightarrow\BF_p$ and $\BF_p\overset{0}\longrightarrow\BF_p\overset{1}\longrightarrow\BF_p$,
respectively. 

In \S\ref{sss:cO_fS'} we defined $\sL_{\fS'}$ to be the weakly invertible covariant $\cO_{\fS'}$-module whose pullback via the morphism 
$S\to\fS'$ corresponding to a diagram \eqref{e:6v_+,v_-} is the line bundle $\sL$ from  the diagram. Similarly to Lemma~\ref{l:O-modules on Sigma''}, descending $\sL_{\fS'}$ to
$\fS''$ is the same as choosing an isomorphism $\nu_+^*\sL\iso\nu_-^*\sL$. The above description of $\nu_\pm$ shows that $\nu_+^*\sL$ and
$\nu_-^*\sL$ are canonically trivial, and we take the identity map $\nu_+^*\sL\iso\nu_-^*\sL$ as a descent datum.

Thus we get a weakly invertible covariant $\cO_{\fS''}$-module $\sL_{\fS''}\,$. The pullback of $\sL_{\fS''}$ to the nodal curve $\nodal$ is canonically trivial (its trivialization comes from the obvious trivialization of the pullback of $\sL_{\fS'}$ to the ``coordinate cross'' $C$).

In  terms of the description of $\fS''(S)$ from \S\ref{sss:describing fS''(S)} and Proposition~\ref{p:fS''(S)}, $\sL_{\fS''}$ is  tautological: the pullback of $\sL_{\fS''}$ via a morphism $S\to\fS''$ is the line bundle $\sL$ from \S\ref{sss:describing fS''(S)}. This fact can be reformulated in terms of the isomorphism
$(\fS'')^{\rm g}\iso\nodal/\BG_m$ from Proposition~\ref{p:fS''}(ii): the $\BG_m$-torsor $\nodal\to\nodal/\BG_m=(\fS'')^{\rm g}$ is the one corresponding to $\sL_{\fS''}$.

\subsection{The c-stack $\Sigma''_{\red}$ and the morphism $\Sigma''_{\red}\to\fS''$}   \label{ss:Sigma''_red to fS''}
As usual, $\Coeq$ stands for the coequalizer in the 2-category of c-stacks.

\begin{lem}  \label{l:Sigma''_red}
(i) The canonical morphism $\Coeq (\Sigma_{\red}\overset{j_+}{\underset{j_-}\rightrightarrows} \Sigma'_{\red)})\to\Sigma''_{\red}$ is an isomorphism.

(ii) $\Sigma''_{\red}\times_{\Sigma''}\Sigma'=\Sigma'_{\red}$.
\end{lem}

\begin{proof}
Let $\sY:=\Coeq (\Sigma_{\red}\overset{j_+}{\underset{j_-}\rightrightarrows} \Sigma'_{\red)})$. 

We have $F'(\Sigma'_{\red})\subset\Sigma_{\red}$ and $\Sigma'_{\red}\cap\Sigma_\pm=(\Sigma_\pm )_{\red}$, see Proposition~\ref{p:sX'}(ii). So
one can apply Corollary~\ref{c:Coeq of subcategories} to $\cC:=\Sigma'(S)$, $\cC'=\Sigma'_{\red}(S)$, where $S$ is any scheme. The corollary implies
that the morphism $\sY\to\Sigma''$ identifies $\sY$ with a substack of $\Sigma''$ and $\sY\times_{\Sigma''}\Sigma'=\Sigma'_{\red}$.

It remains to show that the substack $\sY\subset\Sigma''$ equals $\Sigma''_{\red}\,$. It is clear that $\sY\subset\Sigma''_{\red}\,$. 
The morphism $\Sigma'\to\Sigma''$ is etale and surjective (see Corollary~\ref{c:Sigma''}), and the stack $\sY\times_{\Sigma''}\Sigma'=\Sigma'_{\red}$ is closed in $\Sigma'$; so $\sY$ is closed in~$\Sigma''$. Finally, $|\sY| =|\Sigma''|$ by surjectivity of the map $|\Sigma'_{\red}|=|\Sigma'|\to |\Sigma''|$.
\end{proof}

\begin{prop}   \label{p:Sigma''_red as divisor}
(i) $\Sigma''_{\red}$ is an effective Cartier divisor on $\Sigma''\otimes\BF_p$.

(ii) The exact sequence \eqref{e:2BL exact} descends to an exact sequence of contravariant $\cO_{\Sigma''\otimes\BF_p}$-modules
\begin{equation}   \label{e:3BL exact}
0\to \cO_{\Sigma''\otimes\BF_p}\{1-p\}\to\cO_{\Sigma''\otimes\BF_p}\to\cO_{\Sigma''_{\red}}\to 0.
\end{equation}

(iii) $\Sigma''_{\red}$ is left-fibered over $\Sigma''$.
\end{prop}   

\begin{proof}
By Proposition~\ref{p:sX'}(i), $\Sigma'_{\red}$ is an effective Cartier divisor on $\Sigma'\otimes\BF_p$. This implies statement (i) by Lemmas~\ref{l:Sigma''_red}(ii) and \ref{l:Divp is fpqc-local}(ii).

Statement (ii) follows from Lemma~\ref{l:f_11 descends}.

Statement (iii) follows from (ii) by \S\ref{sss:SpecCoker}.  
\end{proof}

\subsubsection{Defining the morphism $\Sigma''_{\red}\to\fS''$}    \label{sss:Sigma''_red to fS''}
In \S\ref{sss:Sigma'_red to fS'} we defined a canonical map $\Sigma'\to\fS'$. The composite morphism $\Sigma\overset{j_\pm}\mono\Sigma'\to\fS'$ lands into $\fS_\pm$. Since $\fS_\pm =\Spec\BF_p$ is the final object of the category of $\BF_p$-schemes, we have a commutative diagram
\[
\xymatrix{
\Sigma_{\red}\ar@{^{(}->}[r]^-{j_\pm} \ar[d] & \Sigma'_{\red}\ar[d]\\
\Spec\BF_p\ar@{^{(}->}[r]^-{\nu_\pm} & \fS'
}
\]
So by Lemma~\ref{l:Sigma''_red}(i), we get a morphism $\Sigma''_{\red}\to\fS''$. We also have the morphism $\fS''\to\Sigma''_{\red}$ from \S\ref{sss:fS''}. Let us note that these morphisms are \emph{not flat} (see Remark~\ref{r:Why should we study them}).

\begin{lem}    \label{l:fS'' and Sigma''_red}
(i) Each of the composite morphisms $$\fS''\to\Sigma''_{\red}\to\fS''\quad \mbox{ and }\quad \Sigma''_{\red}\to\fS''\to\Sigma''_{\red}$$ is uniquely isomorphic to the Frobenius.

(ii) The pullback of the BK-Tate module $\cO_{\Sigma''}\{ 1\}$ to $\fS''$ canonically identifies with the covariant $\cO_{\fS''}$-module $\sL_{\fS''}$ from \S\ref{sss:L_fS''}.

(iii) The pullback of $\sL_{\fS''}$ to $\Sigma''_{\red}$ is canonically isomorphic to $\cO_{\Sigma''_{\red}}\{ p\}$.

(iv) The non-derived pushforward of $\cO_{\Sigma''_{\red}}\{ n\}$ via the morphism $\Sigma''_{\red}\to\fS''$ equals 
$\sL_{\fS''}^{\otimes (n/p)}$ if $p|n$; if $p\!\not | n$ it equals $0$.
\end{lem}

\begin{proof}   

Statement (i) follows from Lemma~\ref{l:fS' and Sigma'_red}(i) and the fact that the Frobenius of the stack $\Sigma_{\red}$ has no nontrivial automorphisms (indeed, 
$\Sigma_{\red}$ is the classifying stack of the group scheme $(W^*)^{(F)}\otimes\BF_p$, whose reduced part is trivial).

By Lemma~\ref{l:cO_fS'}, the pullback of $\cO_{\Sigma'}\{ 1\}$ to $\fS'$ is canonically isomorphic to~$\sL_{\fS'}$. Using \S\ref{sss:recollections on L}(ii) and Remark~\ref{r:trivializing O_Sigma'{1}}, one checks that the isomorphism from the proof of Lemma~\ref{l:cO_fS'} descends to $\fS''$. This proves (ii).

Statement (iii) follows from (i) and (ii). 

Statement (iv) follows from Corollary~\ref{c:coarse moduli}(ii).
\end{proof}

\subsection{The stacks $\widetilde{\Sigma}''$ and $\widetilde{\Sigma}''_{\red}$}   \label{ss:tilde Sigma''}
\subsubsection{Definition of $\widetilde{\Sigma}''$}
For a scheme $S$, let $\widetilde{\Sigma}''(S)$ be the category of pairs consisting of an object $ \beta\in\Sigma''(S)$ and a trivialization of the line bundle $ \beta^*\cO_{\Sigma''}\{ 1\}$ on $S$. By Lemma~\ref{l:conservativity of cO_Sigma''{1}}, $\widetilde{\Sigma}''$ is a g-stack. 
It is clear that $\widetilde{\Sigma}''$ is a $\BG_m$-torsor over $(\Sigma'')^{\rm g}$, so  $(\Sigma'')^{\rm g}=\widetilde{\Sigma}''/\BG_m$.

One has $\widetilde{\Sigma}''_{\red}=\widetilde{\Sigma}''\times_{\Sigma''}\Sigma''_{\red}$, so $\widetilde{\Sigma}''_{\red}$ is a $\BG_m$-torsor over 
$(\Sigma_{\red}'')^{\rm g}$.

\subsubsection{The morphisms $\widetilde{\Sigma}''_{\red}\rightleftarrows\nodal$}
As noted at the end of \S\ref{sss:L_fS''}, the $\BG_m$-torsor 
$$\nodal\to\nodal/\BG_m=(\fS'')^{\rm g}$$  
corresponds to $\sL_{\fS''}$. Combining this with Lemma~\ref{l:fS'' and Sigma''_red}(ii), we see that the morphism $\fS''\to\Sigma''_{\red}$ from \S\ref{sss:fS''}
induces a $\BG_m$-equivariant morphism
\begin{equation}   \label{e:nodal to tilde Sigma'''_red}
\nodal\to\widetilde{\Sigma}''_{\red}\, .
\end{equation}
Similarly, by Lemma~\ref{l:fS'' and Sigma''_red}(iii), the morphism $\Sigma''_{\red}\to\fS''$ from 
\S\ref{sss:Sigma''_red to fS''} induces a morphism
\begin{equation}   \label{e: tilde Sigma'''_red  to nodal}
\widetilde{\Sigma}''_{\red}\to\nodal ,
\end{equation}
which is $\BG_m$-equivariant if the action of $\BG_m$ on $\nodal$ is defined appropriately: namely, the original action (i.e., the one induced by the natural action of $\BG_m$ on $\BP^1_{\BF_p}$) has to be composed with the homomorphism
$\BG_m\to\BG_m$, $\lambda\mapsto\lambda^p$. The morphisms \eqref{e:nodal to tilde Sigma'''_red} and \eqref{e: tilde Sigma'''_red  to nodal} are \emph{not flat}
(similarly to \S\ref{sss:Sigma''_red to fS''} and Remark~\ref{r:Why should we study them}).
Lemma~\ref{l:fS'' and Sigma''_red}(i) implies that composing \eqref{e:nodal to tilde Sigma'''_red} and \eqref{e: tilde Sigma'''_red  to nodal} in any order one gets the Frobenius. Lemma~\ref{l:fS'' and Sigma''_red}(iv) implies the following
\begin{lem}
The non-derived pushforward of $\cO_{\widetilde{\Sigma}''_{\red}}$ with respect to the morphism \eqref{e: tilde Sigma'''_red  to nodal} equals $\cO_{\nodal}$.  \qed
\end{lem}

\begin{rem}
The preimage of the nonsingular locus of the nodal curve $\nodal$ under the morphism \eqref{e: tilde Sigma'''_red  to nodal} equals
$\widetilde{\Sigma}''_{\red}\times_{(\Sigma'')^{\rm g}}\Sigma$. By Lemma~\ref{l:restricting BK to Sigma_red}(i,iii), this stack is $\BG_m$-equivariantly isomorphic to
$\Cone ((W^\times)^{(F)}\otimes\BF_p \to\BG_m\otimes\BF_p)$. By Lemma~\ref{l:G_m^sharp modulo p}, the same stack can be described as the product of
$(\BG_m/\mu_p)\otimes\BF_p$ and the classifying stack of $W^{(F)}\otimes\BF_p$ (the latter is equipped with the trivial $\BG_m$-action).
\end{rem}

\appendix

\section{$\Sigma'$ in terms of locally free $W_S$-modules}   \label{s:SSigma'}
We will define a c-stack $\SSigma'$, which is canonically isomorphic to $\Sigma'$ but more manageable (because the definition of $\SSigma'$ involves only locally free $W_S$-modules). As an application, we describe the conormal sheaf of $(\Sigma')_{\red}$ in $\Sigma'$ (see \S\ref{ss:conormal}).

We will be using the Teichm\"uller functor (see \S\ref{ss:2Teichmuller}).

\subsection{Definition of $\SSigma'$} \label{ss:SSigma'}
Let $S$ be a scheme. Consider the following data:

(a) a line bundle $\sL$ on $S$ equipped with a morphism $v_-:\sL\to\cO_S$,

(b) an exact sequence of $W_S$-modules
\begin{equation}   \label{e:SSigma'}
0\to [\sL^{\otimes p}]\overset{i}\longrightarrow E\overset{\pi}\longrightarrow N\to 0
\end{equation}
in which $N$ is an invertible $W_S$-module,

(c) morphisms $\rho :E\to [\sL^{\otimes p}]$ and $\mu :E\to W_S$ such that $\mu\circ i=-[v_-^{\otimes p}]$ and
\begin{equation}   \label{e:p-splitting}
\rho\circ i=p.
\end{equation}

In this situation the morphism $p\mu+[v_-^{\otimes p}]\circ\rho  :E\to W_S$ kills the image of $i$, so it induces a morphism $\eta :N\to W_S$.

\begin{defin}   \label{d:SSigma'}
If $S$ is $p$-nilpotent then $\SSigma' (S)$ is the category of data (a)-(c) such that the corresponding morphism $\eta :N\to W_S$ is primitive in the sense of Definition~\ref{d:primitivity of xi}. If $S$ is not $p$-nilpotent then $\SSigma' (S)=\emptyset$.
\end{defin}

\begin{rem}   \label{r:bivariance}
By Lemma~\ref{l:g-stack}, a morphism in $\SSigma' (S)$ induces an \emph{isomorphism} between the corresponding $W_S$-modules $N$.
\end{rem}

\begin{rem}  \label{r:p-splitting}
One can interpret \eqref{e:p-splitting}  by saying that $\rho$ is a \emph{$p$-splitting} of the exact sequence~\eqref{e:SSigma'} (i.e., a splitting of the extension obtained by multiplying the extension \eqref{e:SSigma'} by $p$).
\end{rem}

\begin{rem}
In \S\ref{sss:SSigma' to Sigma'} we will see that the $W_S$-module $N$ from \eqref{e:SSigma'} is essentially $(M')^{(-1)}$, where $M$ is the admissible $W_S$-module from the definition of $\Sigma'(S)$.
\end{rem}

Objects of $\SSigma' (S)$ will be denoted by $(\sL , v_- , E, i, \rho ,\mu )$.

\begin{lem}
Let $(\sL , v_- , E, i, \rho ,\mu )\in\SSigma' (S)$. Then $\mu$ is surjective.
\end{lem}

\begin{proof}
We can assume that $S$ is the spectrum of a field of characteristic $p$. If $v_-=0$ then $\eta =p\mu$, so primitivity of $\mu$ implies that 
$\mu$ is surjective. If $v_-\ne 0$ then use the equality $\mu\circ i=-[v_-^{\otimes p}]$.
\end{proof}

\begin{rem}
If $v_-=0$ then primitivity of $\eta$ is equivalent to surjectivity of $\mu$. If $v_-$ is invertible then $\mu$ is automatically surjective, and $\eta$ is primitive if and only if the restriction of $\rho$ to the invertible $W_S$-module $\Ker\mu$ is primitive.
\end{rem}

\subsection{The isomorphism $\SSigma'\iso\sZ/\sG$}   \label{ss:SSigma'=sZ/sG}
\subsubsection{A rigidification of $\SSigma'$}
Define a c-stack $\SSigma'_{\RRig}$ as follows: an object of $\SSigma'_{\RRig}(S)$ is an object of $\SSigma'(S)$ plus a $W_S$-submodule of $E$ such that the restrictions of both $\pi :E\to N$ and $\mu :E\to W_S$ to the submodule are isomorphisms\footnote{In other words, the submodule is transversal to both $\im i$ and $\Ker\mu$.}. Equivalently, an object of $\SSigma'_{\RRig}(S)$ is an object 
$(\sL , v_- , E, i, \rho ,\mu )\in\SSigma'(S)$ plus an isomorphism $(E, i, \mu )\iso  (E_0, i_0, \mu_0)$, where 
$E_0:=[\sL^{\otimes p}]\oplus W_S$, $i_0$ is the canonical embedding $[\sL^{\otimes p}]\mono E_0$, and $\mu_0$ is the morphism $[\sL^{\otimes p}]\oplus W_S\to W_S$ given by $(-[v_-^{\otimes p}],1)$.

\subsubsection{The isomorphism $\SSigma'_{\RRig}\iso\sZ$}
Once $(E, i, \mu )$ is identified with $(E_0, i_0, \mu_0)$, our $\rho$ becomes a morphism 
$[\sL^{\otimes p}]\oplus W_S\overset{(p,\zeta )}\longrightarrow [\sL^{\otimes p}]$, where $\zeta$ is a section of the $S$-scheme 
$[\sL^{\otimes p}]$. Moreover, the morphism $\eta :N\to W_S$ from Definition~\ref{d:SSigma'} (which is required to be primitive) identifies with 
$p+[v_-^{\otimes p}](\zeta )\in W(S)$. So $\SSigma'_{\RRig}$ identifies with the c-stack $\sZ$ from~\S\ref{s:Sigma' as quotient}.

\subsubsection{The isomorphism $\SSigma'\iso\sZ/\sG$}
The above isomorphism $\SSigma'_{\RRig}\iso\sZ$ induces an isomorphism $\SSigma'\iso\sZ/\sG$, where  $\sG$ is as in \S\ref{s:Sigma' as quotient}. The group scheme $\sG$ appears here as the automorphism group of $(E_0, i_0, \mu_0)$.

\subsection{The isomorphism $\SSigma'\iso\Sigma'$}   \label{ss:SSigma'=Sigma'}
Combining the isomorphism $\SSigma'\iso\sZ/\sG$ from \S\ref{ss:SSigma'=sZ/sG} with the isomorphism $\sZ/\sG\iso\Sigma'$ from 
\S\ref{s:Sigma' as quotient}, we get a canonical  isomorphism $\SSigma'\iso\Sigma'$. Here is a direct construction of the same isomorphism.

\subsubsection{The morphism $\SSigma'\to\Sigma'$}   \label{sss:SSigma' to Sigma'}
Given $(\sL , v_- , E, i, \rho ,\mu )\in\SSigma'(S)$, one defines $(M,\xi )\in \Sigma' (S)$ as follows.

Since $\rho i=p$ and $FV=p$, we have a complex of $W_S$-modules
\begin{equation}  \label{e:2complex computing M}
0\to [\sL^{\otimes p}]^{(1)}\overset{(i^{(1)},V)}\longrightarrow E^{(1)}\oplus [\sL]\overset{(\rho^{(1)},-F)}\longrightarrow [\sL^{\otimes p}]^{(1)}\to 0.
\end{equation}
Define $M$ to be its middle cohomology. (The other cohomologies are zero because $\Ker i=0$ and $\Coker F=0$.)
The $W_S$-module $M$ is an extension of $N^{(1)}$ by $\sL^\sharp$; in particular, $M$ is admissible.
Define $\xi :M\to W_S$ to be the map induced by the morphism $$E^{(1)}\oplus [\sL]
\overset{f}
\longrightarrow W_S, \quad f:=(V\circ\mu^{(1)} ,[v_-]).$$
The equality $F\circ f=(p\rho^{(1)},[v_-^{\otimes p}]F)$ implies
 that the morphism $N^{(1)}=M'\overset{\xi'}\longrightarrow W^{(1)}_S$ induced by $\xi$ equals $\eta^{(1)}$, where 
$\eta :N\to W_S$ is as in Definition~\ref{d:SSigma'}. Since $\eta$ is primitive, so is $\xi$.

\medskip

The fact that the above morphism $\SSigma'\to\Sigma'$ is an isomorphism can be proved directly (without passing through $\sZ/\sG$) using the following description of the category of admissible $W_S$-modules in terms of locally free ones.

\begin{prop}    \label{p:admissibles&p-splittings}
For any scheme $S$, the construction of \S\ref{sss:SSigma' to Sigma'} gives an equivalence between the
category of admissible $W_S$-modules and the category of the following data:

(i) a line bundle $\sL$ on $S$ and an invertible $W_S$-module $N$;

(ii) a $W_S$-module extension 
\begin{equation}   
0\to [\sL^{\otimes p}]\overset{i}\longrightarrow E\overset{\pi}\longrightarrow N\to 0
\end{equation}
equipped with a morphism $\rho :E\to  [\sL^{\otimes p}]$ such that $\rho\circ i=p$.  
\end{prop}

\begin{proof}
This is essentially a reformulation of the following slight generalization of the isomorphism \eqref{e:2Ext(W^1,Lsharp)}: for any line bundle $\sL$ on $S$ and any invertible $W_S$-module~$N$ one has a canonical isomorphism
\[
\Ex_W(N^{(1)}, \sL^\sharp )\iso\Cone (R^{(1)}\overset{p}\longrightarrow R^{(1)}), \quad R:=\HHom_W (N, [\sL^{\otimes p}]).
\]
\end{proof}

\subsection{On the c-stacks $\Sigma'_n$}  
\subsubsection{}  \label{sss:requirement if n=1}
One can use the language of \S\ref{ss:SSigma'} to give an equivalent definition of the c-stacks $\Sigma'_n$ from \S\ref{sss:Sigma'_n}:
if $n>1$ then just replace $W_S$ by $(W_n)_S$ in the definition of 
$\SSigma' (S)$ from \S\ref{ss:SSigma'}, and if $n=1$ then add the condition that $\mu$ is surjective and require the morphisms between objects of 
$\Sigma'_1(S$) to induce isomorphisms between the modules $N$ (if $n>1$ both conditions follow from primitivity of $\eta$).

\subsubsection{} \label{sss:Sigma'_1}

Thus $\Sigma'_1(S)=\emptyset$ if $S$ is not $p$-nilpotent, and for every $p$-nilpotent scheme $S$ an object of $\Sigma'_1(S)$ is the following data:

(a) a line bundle $\sL$ on $S$ equipped with a morphism $v_-:\sL\to\cO_S$,

(b) an exact sequence of $\cO_S$-modules
\begin{equation} \label{e:Sigma'_1}
0\to \sL^{\otimes p}\overset{i}\longrightarrow E\overset{\pi}\longrightarrow N\to 0
\end{equation}
in which $N$ is  invertible.
 
(c) a morphism $\rho :E\to \sL^{\otimes p}$ and an epimorphism $\mu :E\to \cO_S$ such that
$\rho\circ i=p$, $\mu\circ i=-v_-^{\otimes p}$,  and at every point of $S$ either $v_-$ or $\rho$ vanishes.

\medskip

According to \ref{l:Sigma' as limit}(i), $\Sigma'_1$ is the formal completion of a smooth algebraic c-stack over $\BZ$ along an effective divisor $D$ in its reduction modulo $p$. This is clear from the above description of $\Sigma'_1$:
the c-stack over $\BZ$ is obtained by removing the $p$-nilpotence condition and the vanishing condition in~(c),
and the effective divisor $D$ is as follows. Note that
if $S$ is an $\BF_p$-scheme then $\rho$ is a morphism $N\to\sL^{\otimes p}$, and the preimage of $D$ in $S$ is the scheme of zeros of the morphism $v_-\rho:N\to\sL^{\otimes (p-1)}$.

\subsection{The substack $\Sigma'_{\red}\subset\Sigma'$ and its conormal sheaf}  \label{ss:conormal}
\subsubsection{Relation between $\Sigma'_{\red}$ and $D$}
It is easy to check that the effective divisor $D$ defined at the end of \S\ref{sss:Sigma'_1} is reduced. So Lemma~\ref{l:Sigma' as limit}(ii) implies that 
$\Sigma'_{\red}=\Sigma'\times_{\Sigma_1'}D$.

\subsubsection{The goal}
It is easy to see that $D$ is left-fibered over $\Sigma'_1$. So $\Sigma'_{\red}$ is left-fibered over $\Sigma'$ (on the other hand, we already know this, see Proposition~\ref{p:sX'}(iii)). So by \S\ref{ss:conormal definition}, the conormal sheaves of $\Sigma'_{\red}$ in $\Sigma'\otimes\BF_p$ and in $\Sigma'$ are defined as contravariant $\cO$-modules on $\Sigma'_{\red}$. Our goal is to describe them explicitly.

\subsubsection{The conormal sheaf of $\Sigma'_{\red}$ in $\Sigma'\otimes\BF_p$}   \label{sss: conformal in the reduction mod p}
For any scheme $S$ over $\Sigma'_1$ we have the line bundle $N\otimes\sL^{\otimes (1-p)}$, where $N$ and $\sL$ are as in \S\ref{sss:Sigma'_1}. Letting $S$ vary, we get a weakly invertible contravariant\footnote{To see contravariance, recall that in \S\ref{sss:requirement if n=1} we required the morphisms between objects of 
$\Sigma'_1(S)$ to induce\emph{ iso}morphisms between the modules $N$.} $\cO$-module (in the sense of \S\ref{ss:O-mododules on c-stacks}) on the c-stack $\Sigma'_1$. Its pullback to $\Sigma'_{\red}$ is the conormal sheaf of $\Sigma'_{\red}$ in $\Sigma'\otimes\BF_p$ (this follows from the description of $D$ at the end of \S\ref{sss:Sigma'_1}). This description of the conormal sheaf of $\Sigma'_{\red}$ in $\Sigma'\otimes\BF_p$ also follows from the exact sequence \eqref{e:2BL exact} combined with the Bhatt-Lurie isomorphism  \eqref{e:2BL}.

\subsubsection{The conormal sheaf of $\Sigma'_{\red}$ in $\Sigma'$}
This is a rank 2 vector bundle containing $\cO$, and the quotient is the conormal sheaf of $\Sigma'_{\red}$ in $\Sigma'\otimes\BF_p$.
The rank 2 vector bundle is described below.

Let $S$ be a scheme over $\Sigma'_1$. Then we have data (a)-(c) from \S\ref{sss:Sigma'_1}. Define $\bar E$ by the pushforward diagram
\[
\xymatrix{
0\ar[r]&\sL^{\otimes p}\ar[d]_{v_-} \ar[r]^i&E\ar[d] \ar[r]^\pi&N\ar@{=}[d] \ar[r]&0\\
0\ar[r]&\sL^{\otimes (p-1)}\ar[r]  \ar[r]&\bar E\ar[r]&N \ar[r]&0
}
\]
whose upper row is \eqref{e:Sigma'_1}. As $S$ varies, the vector bundles $\bar E^*$ and $N$ provide contravariant\footnote{For $N$ this follows from Remark~\ref{r:bivariance}.} $\cO_{\Sigma'_1}$-modules, denoted by $\bar\sE^*$ and $\sN$. Note that the lower row of the diagram induces a monomorphism
$\cO_{\Sigma'_1}\mono\bar\sE^*\otimes\sN$, whose cokernel is the $\cO_{\Sigma'_1}$-module from \S\ref{sss: conformal in the reduction mod p}.

\begin{prop}
(i) The conormal sheaf of $\Sigma'_{\red}$ in $\Sigma'$ canonically identifies with the pullback of $\bar\sE^*\otimes\sN$ to $\Sigma'_{\red}$.

(ii) The conormal sheaf of $D$ in $\Sigma'_1$ canonically identifies with the pullback of $\bar\sE^*\otimes\sN$ to~$D$.
\end{prop}

\begin{proof}
It suffices to prove (ii). To this end, it suffice to construct a section $$\sigma\in H^0(\Sigma'_1,\bar\sE\otimes\sN^{-1} )$$ whose locus of zeros equals $D$. 

Let $S$ be a scheme over $\Sigma'_1$.  Then we have data (a)-(c) from \S\ref{sss:Sigma'_1}. The endomorphism $p-i\rho\in\End E$ kills $\im i$, so 
$p-i\rho=\rho'\pi$ for some $\rho':N\to E$; moreover, $\pi\rho'=p$. (The morphisms $\rho$ and $\rho'$ are different ways to think about the same $p$-splitting\footnote{The notion of $p$-splitting was introduced in Remark~\ref{r:p-splitting}.} of the exact sequence \eqref{e:Sigma'_1}.) The composite morphism $N\overset{\rho'}\longrightarrow E\to \bar E$ provides a section 
$\sigma_S\in H^0(S,\bar E\otimes N^{-1})$ whose subscheme of zeros equals $D\times_{\Sigma'_1}S$. As $S$ varies, we get the desired section $\sigma\in H^0(\Sigma'_1,\bar\sE\otimes\sN^{-1} )$.
\end{proof}

\section{The $q$-de Rham prism and the proof of Proposition~\ref{p:divisors on Sigma}}   \label{s:q-de Rham}
In \S\ref{ss:q-de Rham} we recall the  \emph{$q$-de Rham prism}  introduced in \cite[\S 16]{BS}; this is the ring $\BZ_p[[q-1]]$ equipped with a certain additional structure.
In \S\ref{sss:Z_p^times-action} and Lemma~\ref{l:automorphisms of prism} we describe the automorphism group of the $q$-de Rham prism; the good news is that the group is big and the ``dynamic'' of its action on the formal scheme $Q:=\Spf\BZ_p[[q-1]]$ is understandable.

In \S\ref{ss: homomorphism f}-\ref{ss:constructing pi} we construct a faithfully flat morphism from the stack $Q/\BZ_p^\times$ to~$\Sigma$.
In \S\ref{ss:divisors on Sigma} we use this morphism to prove Proposition~\ref{p:divisors on Sigma}, which describes all effective Cartier divisors on $\Sigma$.

In \S\ref{ss:BK-Tate on Q} we describe the pullback of the BK-Tate module $\cO_\Sigma\{ 1\}$ to $Q/\BZ_p^\times$.

In \S\ref{ss:var q-de Rham} we introduce an ``economic'' variant of the $q$-de Rham prism.
In \S\ref{ss:LT variant} we discuss a Lubin-Tate version of the $q$-de Rham prism and of its ``economic'' variant.

Let us note that this Appendix overlaps with \cite[\S 3.8]{BL}.

\subsection{The $q$-de Rham prism}  \label{ss:q-de Rham}
\subsubsection{Definition}  \label{sss:A,phi,I}
According to \cite[\S 16]{BS}, the \emph{$q$-de Rham prism} is the triple $(A,\phi,I)$, where $A:=\BZ_p[[q-1]]$, $\phi :A\to A$ is the homomorphism such that $\phi (q)=q^p$, and $I\subset A$ is the ideal generated by the cyclotomic polynomial $\Phi_p(q)=\frac{1-q^p}{1-q}=1+q+\ldots+q^{p-1}$.

Since $A$ is $\BZ_p$-flat and $\phi$ is a lift of Frobenius, $A$ is a $\delta$-ring in the sense of \cite[\S 2]{BS} with $\delta :A\to A$ being given by 
$\delta (a)=\frac{\phi (a)-a^p}{p}$. The triple $(A,\phi,I)$ (or if you prefer, the triple $(A,\delta ,I)$) is a particular example of a prism in the sense of  \cite[\S 3]{BS}.

\subsubsection{The formal scheme $Q$}   \label{sss:Q}
Set $Q:=\Spf A$, where $A$ is equipped with the $(p,q-1)$-adic topology. Then $Q$ is the underlying formal scheme of the formal multiplicative group $\hat\BG_m$ over $\Spf\BZ_p$, and $\Spf (A/I)\subset Q=\hat\BG_m$ is the (formal) subscheme of primitive $p$-th roots of $1$.

We have $\End\hat\BG_m =\BZ_p$, and the morphism $\phi^*:Q\to Q$ induced by $\phi :A\to A$ corresponds to $p\in\BZ_p =\End\hat\BG_m$.

\subsubsection{Action of $\BZ_p^\times$}  \label{sss:Z_p^times-action}
The group $\BZ_p^\times =\Aut\hat\BG_m$ acts on $Q$ and therefore on $A$. The action of $\BZ_p^\times$ on $A$ is continuous; it preserves $\phi$ and $I$.

The image of $q$ under the automorphism of $A$ corresponding to $n\in\BZ_p$ will be denoted by~$q^n$. If $n\in\BZ$ then $q^n$ has the usual meaning. For any $n\in\BZ_p^\times$ one has
$$q^n=\sum\limits_{i=0}^{\infty}\frac{n(n-1)\ldots (n-i+1)}{i!}(q-1)^i .$$
\begin{lem}  \label{l:automorphisms of prism}
The inclusion $\BZ_p^\times\subset\Aut (A,\phi )$ is an equality.
\end{lem}

\begin{proof}
The subscheme of fixed points of $\phi :Q\to Q$ is defined by the equation $q^p=q$, which is equivalent to $q=1$. So any automorphism of $(A,\phi )$ preserves the ideal $(q-1)$. It is well known\footnote{This follows from \cite{LT}. One can also prove this by passing to $\BQ_p[[q-1]]$ and identifying the formal multiplicative group over $\BQ_p$ with the formal additive group using the logarithm.} that an automorphism of $(Q,\phi^* )$ preserving the subscheme $q=1$ also preserves the group operation on $Q=\hat\BG_m$. Finally, $\Aut\hat\BG_m=\BZ_p^\times$.
\end{proof}

\subsubsection{$\BZ_p^\times$ as a group scheme over $\BZ$}   \label{sss:Z_p^times as group scheme}
Any pro-finite group $G$ gives rise to the following group scheme over $\BZ$:
\[
\underset{\longleftarrow}{\lim} ((G/U)\times\Spec\BZ),
\]
where the limit is taken over all open subgroups $U\subset G$. It is convenient to denote this group scheme simply by $G$.

In particular, this applies to $G=\BZ_p^\times$, so we can view $\BZ_p^\times$ as a group scheme over $\BZ$.
The action of $\BZ_p^\times$ on $Q$ defined in \S\ref{sss:Z_p^times-action} is an action in the sense of group schemes.
So we have the quotient stack $Q/\BZ_p^\times$.

\subsection{The homomorphism $f:A\to W(A)$}   \label{ss: homomorphism f}
Let $A$ and $\phi$ be as in \S\ref{sss:A,phi,I}. By \S\ref{sss:Z_p^times-action}, the group $\BZ_p^\times$ acts on $(A, \phi)$.

\begin{lem}   \label{l:homomorphism f}
(i) The canonical epimorphism $W(A)\epi W_1(A)=A$ has a unique splitting $f:A\to W(A)$ such that $f\circ\phi =F\circ f$, where $F:W(A)\to W(A)$ is the Witt vector Frobenius.

(ii) $f(q)=[q]$, where $[q]\in W(A)$ is the Teichm\"uller representative.

(iii) Let $W(\phi ):W(A)\to W(A)$ be the homomorphism induced by $\phi :A\to A$. Then $W(\phi )\circ f=f\circ\phi =F\circ f$.

(iv) The homomorphism $f:A\to W(A)$ is $\BZ_p^\times$-equivariant.
\end{lem}

\begin{proof}
If $f:A\to W(A)$ is a splitting such that $f\circ\phi =F\circ f$ then $F(f(q))=f(q)^p$, and the image of $f(q)$ in $W_1(A)=A$ equals $q$. This implies that $f(q)=[q]$.

Conversely, the homomorphism $f:A\to W(A)$ such that $f(q)=[q]$ has the properties listed in (i,iii,iv).
\end{proof}

\subsubsection{Remark} Parts (i,iii,iv) of the lemma are corollaries of the following general statements:

(a) for \emph{any} $\delta$-ring $A$ the canonical epimorphism $W(A)\epi W_1(A)=A$ has a unique splitting $f_A:A\to W(A)$ which is a homomorphism of $\delta$-rings;

(b) $f_A$ is functorial in $A$ (with respect to $\delta$-ring homomorphisms $A\to\Tilde A$).

\noindent 
Both (a) and (b) are immediate consequences of the following fact discovered by A.~Joyal in~\cite{JoyalDelta}: the forgetful functor from the category of $\delta$-rings to that of rings has a right adjoint, which is nothing but the functor $W$.
To deduce (iii) from (b), use the fact that $\phi$ is a homomorphism of $\delta$-rings.

\subsection{The morphism $\pi :Q/\BZ_p^\times\to\Sigma$}   \label{ss:constructing pi}
\subsubsection{The $W(A)$-module $M$}  \label{sss:The W(A)-module M}
Let $M:=W(A)\cdot f(I)\subset W(A)$. Then $M$ is a free $W(A)$-module of rank $1$. The action of $\BZ_p^\times$ on $W(A)$ preserves $M$. 

Equip $W(A)$ with the weakest topology such that for every open ideal $I\subset A$ and every $n\in\BN$ the map $W(A)\to W_n(A/I)$ is locally constant. Since $M\simeq W(A)$ we also get a topology on $M$. The action of  $\BZ_p^\times$ on $W(A)$ and $M$ is continuous with respect to these topologies.

\begin{lem}  \label{l:topological primitivity}
Let $\fa\subset\BZ_p [[q-1]]$ be an open ideal and $R:=A/\fa$. Then the morphism $\xi :M\otimes_{W(A)}W(A/\fa )\to W(A/\fa )$ induced by the embedding $M\mono W(A)$ is primitive.
\end{lem}

The lemma can be deduced from the fact that $(A,\phi ,I)$ is a prism. Instead, let us give a direct proof.

\begin{proof}
We have to show that the image of $\Phi_p ([q])$ in the ring $W(\BZ_p [[q-~1]]/\fa)$ belongs to $W_{\prim }(\BZ_p [[q-1]]/\fa)$. Without loss of generality, we can assume that $\fa$ is the maximal ideal. Then $\BZ_p [[q-~1]]/\fa=\BF_p$, and the image of $\Phi_p ([q])$ in $W(\BF_p)$ equals $\Phi_p (1)=p$.
\end{proof}

\subsubsection{The morphism $\pi :Q/\BZ_p^\times\to\Sigma$}
By \S\ref{sss:Z_p^times as group scheme} we have the quotient stack $Q/\BZ_p^\times$. By \S\ref{sss:The W(A)-module M} and Lemma~\ref{l:topological primitivity}, we get a morphism $\pi :Q/\BZ_p^\times\to\Sigma$. By Lemma~\ref{l:homomorphism f}(iii), the following diagram commutes:
\begin{equation}  \label{e:pi intertwines}
\xymatrix{
Q/\BZ_p^\times\ar[r]^{\phi^*} \ar[d]_{\pi} & Q/\BZ_p^\times\ar[d]^\pi\\
\Sigma\ar[r]^F & \Sigma
}
\end{equation}
Here $\phi^*:Q/\BZ_p^\times\to Q/\BZ_p^\times$ is induced by $\phi :A\to A$.

\begin{lem}  \label{l:Q to Sigma}
The morphism $\pi :Q/\BZ_p^\times\to\Sigma$ is algebraic and faithfully flat.
\end{lem}

\begin{proof}
Algebraicity is clear because both $Q/\BZ_p^\times$ and $\Sigma$ are algebraic over $\hat\BA^1/\BG_m$ (see Lem\-ma~\ref{l:Sigma to A^1/G_m}). Faithful flatness follows from Proposition~\ref{p:Sigma as limit}(v).
\end{proof}

\subsubsection{Remark} \label{sss:functions on Sigma}
We already know that $H^0(\Sigma , \cO_{\Sigma})=\BZ_p$, see Corollary~\ref{c:density}. This fact also follows from
Lemma~\ref{l:Q to Sigma} and the equality $H^0(Q,\cO_Q)^{\BZ_p^\times}=(\BZ_p[[q-1]])^{\BZ_p^\times}=\BZ_p$. This equality follows from the fact that $\BZ_p^\times$ is infinite.

\begin{lem}   \label{l:pi for q=1}
Let $\tilde\pi :Q\to\Sigma$ be the composite morphism $Q\to Q/\BZ_p^\times\overset{\pi}\longrightarrow\Sigma$.
Let $D_0:=\Spf\BZ_p[[q-1]]/(q-1)\subset Q$. Then

(i) the restriction of $\tilde\pi$ to $D_0\simeq\Spf\BZ_p$ is the morphism $p:\Spf\BZ_p\to\Sigma$;

(ii) if $\sY\subset\Sigma$ is a closed substack such that $\tilde\pi^{-1}(\sY )\supset D_0$ then $\sY=\Sigma$.
\end{lem}

\begin{proof}
Statement (i) is clear because $\Phi_p (1)=p$. Statement (ii) follows from (i) and Proposition~\ref{p:density}.
\end{proof}

\begin{cor}   \label{c:not an iso}
The morphism $Q/\BZ_p^\times\to\Sigma$ is not an isomorphism.
\end{cor}

\begin{proof}
Let $D_0$ be as in Lemma~\ref{l:pi for q=1}. Then $D_0$ descends to a closed substack of $Q/\BZ_p^\times$. On the other hand,  by Lemma~\ref{l:pi for q=1}(ii), $D_0$ does not descend to a closed substack of $\Sigma$.
\end{proof}

\subsection{Proof of Proposition~\ref{p:divisors on Sigma}}   \label{ss:divisors on Sigma}
\subsubsection{$\BZ_p^\times$-invariant divisors on $Q$}   \label{sss:Z_p^times-invariant divisors on Q}
For each integer $n\ge 0$, let $D_n\subset Q$ be the divisor $\Phi_{p^n}(q)=0$, where $\Phi_{p^n}$ is the cyclotomic polynomial; in other words, $D_n$ is the subscheme of primitive roots of 1 of degree $p^n$. In particular, if $n=0$ we get the divisor $D_0$ from Lemma~\ref{l:pi for q=1}.

\begin{lem}   \label{l:Z_p^times-invariant divisors on Q}
The group of all $\BZ_p^\times$-invariant divisors on $Q$ is freely generated by $Q\otimes\BF_p$ and the divisors $D_n$, $n\ge 0$.
\end{lem}

\begin{proof}
It is enough to show that if $O$ is the ring of integers of a finite extension of $\BQ_p$ and a morphism $\alpha :\Spf O\to Q=\hat\BG_m$  factors through some $\BZ_p^\times$-invariant effective divisor on $Q$ then $\alpha^{p^m}=1$ for some $m\in\BN$. Indeed, the $\BZ_p^\times$-orbit of $\alpha$ is finite, so there exists $n\in\BZ_p^\times$ such that
$\alpha^n=\alpha$ and $n\ne 1$.
\end{proof}

\begin{lem}  \label{l:preimage of Delta_n}
Let $\Delta_n\subset\Sigma$ be as in \S\ref{ss:Delta_n}. Then $\Delta_n\times_\Sigma Q =D_{n+1}$. 
\end{lem}

\begin{proof}
The commutative diagram \eqref{e:pi intertwines} and the equalities 
$$\Delta_n:=(F^n)^{-1}(\Delta_0), \quad D_{n+1}=((\phi ^*)^n)^{-1}(D_1)$$ 
imply that it suffices to check that $\Delta_0\times_\Sigma Q =D_1$. This is clear.
\end{proof}

\subsubsection{Proof of Proposition~\ref{p:divisors on Sigma}}   \label{sss:divisors on Sigma}
We have to show that the monoid of effective Cartier divisors on $\Sigma$ is freely generated by $\Sigma\otimes\BF_p$ and $\Delta_n$, $n\ge 0$.

Let $\tilde\pi :Q\to\Sigma$ be the composite morphism $Q\to Q/\BZ_p^\times\overset{\pi}\longrightarrow\Sigma$.
Let $\cD$ be an effective Cartier divisor on $\Sigma$. By Lemma~\ref{l:Q to Sigma}, the morphism $\tilde\pi :Q\to\Sigma$ is algebraic and faithfully flat. So $\tilde\pi^{-1}(\cD)$ is a $\BZ_p^\times$-invariant effective divisor on $Q$, and $\cD$ is uniquely determined by $\tilde\pi^{-1}(\cD)$. 
By Lemma~\ref{l:pi for q=1}(ii), $\tilde\pi^{-1}(\cD)\not\supset D_0$. It remains to use Lemmas~\ref{l:Z_p^times-invariant divisors on Q} and \ref{l:preimage of Delta_n}.

\subsection{The pullback of $\cO_\Sigma\{ 1\}$ to $Q/\BZ_p^\times$}  \label{ss:BK-Tate on Q}
Recall that $D_0\subset Q$ is the divisor $q=1$. Let ${\bf 0}:Q\to Q$ be the composite map $Q\to\Spf\BZ_p=D_0\mono Q$.
As before, let $\tilde\pi :Q\to\Sigma$ be the composite morphism $Q\to Q/\BZ_p^\times\overset{\pi}\longrightarrow\Sigma$.

\begin{prop}
There is a canonical $\BZ_p^\times$-equivariant isomorphism 
\[
\tilde\pi^*\cO_\Sigma\{ 1\}\iso\cO_Q (D_0)\otimes {\bf 0}^*\cO_Q (-D_0).
\]
\end{prop}

\begin{proof}
Let $F:Q\to Q$ be the morphism $q\mapsto q^p$. Let $\cM:=\tilde\pi^*\cO_\Sigma\{ 1\}$. By Definition~\ref{d:BK}, formula~\eqref{e:shtuka structure}, and Lemma~\ref{l:preimage of Delta_n}, $\cM$ is an $\BZ_p^\times$-equivariant line bundle on $Q$ equipped with an isomorphism 
\begin{equation}   \label{e:shtuka structure on cM}
(F^*\cM) (-D_1)\iso\cM.
\end{equation}
Moreover, by Lemma~\ref{l:pi for q=1}(i), $\cM$ is equipped with a $\BZ_p^\times$-equivariant trivialization over $D_0$; this trivialization is compatible with the isomorphism \eqref{e:shtuka structure on cM}, assuming that the $\cO_{D_0}$-module $\cO_{D_0}(-(D_0\cap D_1))$ is trivialized using $p\in\BZ_p$.

It is easy to show that the line bundle $\cM':=\cO_Q (D_0)\otimes {\bf 0}^*\cO_Q (-D_0)$ is equipped with the same type of structure (to check this, use that $q^p-1=(q-1)\cdot\Phi_p (q)$ and $\Phi_p(1)=p$).

Now let $\cN:=\cM^{-1}\otimes\cM'$. Then $\cN$ is a $\BZ_p^\times$-equivariant line bundle on $Q$ equipped with a $\BZ_p^\times$-equivariant isomorphism
$\varphi :F^*\cN\iso\cN$ and a $\BZ_p^\times$-equivariant trivialization over $D_0\,$, which is compatible with $\varphi$. 
Using $\varphi$, one extends the trivialization from $D_0$ to $F^{-1}(D_0)$, $F^{-2}(D_0)$, etc. Since $Q=\bigcup\limits_n F^{-n}(D_0)$, we get the desired trivialization of $\cN$ over the whole~$Q$.
\end{proof}

\subsection{A variant of the  $q$-de Rham prism}  \label{ss:var q-de Rham}
Let $(A,\phi ,I)$ be as in \S\ref{sss:A,phi,I}. The epimorphism $\BZ_p^\times\epi\BF_p^\times$ has a unique splitting $\BF_p^\times\mono\BZ_p^{\times}$.
Let $A':=A^{\BF_p^\times}$ and $I':=I\cap A'$. Let $\phi':A'\to A'$ be the restriction of $\phi :A\to A$. Let $Q':=\Spf A'$; in other words, $Q'$ is the \emph{GIT quoti\-ent}~$Q/\!\!/\BF_p^\times$.
The group $H:=\BZ_p^\times/\BF_p^\times=\Ker (\BZ_p^\times\epi\BF_p^\times )$ acts on $(A',\phi', I')$.

The ring $A'$ is isomorphic to $\BZ_p [[z]]$: indeed, the action of $\BF_p^\times$ on $Q=\Spf \BZ_p[[q-1]]$ is linearizable because the order of $\BF_p^\times$ is coprime to~$p$. A particular linearization will be described in \S\ref{sss:A'_1,phi'_1,I'_1}.

It is easy to see that the ideal $I'\subset A'$ is principal. Since  $(A,\phi ,I)$ is a prism in the sense of \cite[\S 3]{BS} and 
$I'\subset A'$ is principal, it is easy to see that $(A',\phi', I')$ is also a prism. Note that $A'/I'=\BZ_p$ (while $A/I$ is the ring of $p$-adic cyclotomic  integers).

The morphism $\pi :Q/\BZ_p^\times\to\Sigma$ factors as $Q/\BZ_p^\times\to Q'/H\to\Sigma$. The morphism $Q'/H\to\Sigma$ is constructed similarly to 
\S\ref{ss:constructing pi}; it is not an isomorphism (the proof of this is similar to that of Corollary~\ref{c:not an iso}).

\subsection{Lubin-Tate version of the $q$-de Rham prism}   \label{ss:LT variant}
\subsubsection{The Lubin-Tate group law}   \label{sss:LT group law}
Let $u\in\BZ_p^\times$. By \cite{LT}, there is a unique $1$-dimensional formal group law over $\BZ_p^\times$ such that multiplication by $pu$ according to the group law is given by $x\mapsto x^p+pux$. If $c\in\BF_p$ and $[c]\in\BZ_p$ is the corresponding Teichm\"uller element then the action of $[c]$ according to the group law is given by $x\mapsto [c]\cdot x$.

\subsubsection{The case $u=1$}   \label{sss:case u=1}
By \cite{LT}, in the case $u=1$ the group law from \S\ref{sss:LT group law} is isomorphic to~$\hat\BG_m$.

\subsubsection{The prism $(A_u,\phi_u,I_u)$}
Just as $\hat\BG_m$ gives rise to the $q$-de Rham prism (see \S\ref{sss:Q}), the formal group law from \S\ref{sss:LT group law} gives rise to the prism
$(A_u,\phi_u,I_u)$, where $A_u=\BZ_p[[x]]$, $\phi_u:\BZ_p[[x]]\to \BZ_p[[x]]$ takes $x$ to $x^p+pux$, and $I_u$ is the ideal $A_u\cdot (x^{p-1}+pu)\subset A_u$.
If $u=1$ then the prism $(A_u,\phi_u,I_u)$ is isomorphic to the $q$-de Rham prism by \S\ref{sss:case u=1}.

\subsubsection{The prism $(A'_u,\phi'_u,I'_u)$}   \label{sss:A'_u,phi'_u,I'_u}
Similarly to \S\ref{sss:Z_p^times-action}, $\BZ_p^\times$ acts on $(A_u,\phi_u,I_u)$; moreover, the action of $\BF_p^\times\subset\BZ_p^\times$ 
on $\Spf A_u=\Spf\BZ_p[[x]]$ is linear by \S\ref{sss:LT group law}.

Similarly to \S\ref{ss:var q-de Rham}, set $A'_u=A_u^{\BF_p^\times}$, $I'_u:=I_u\cap A'_u$, and let $\phi'_u:A'_u\to A'_u$ be the restriction of $\phi_u :A_u\to A_u$. 
Then $(A'_u,\phi'_u,I'_u)$ is a prism, which has the following explicit description: set $y:=u^{-1}x^{p-1}$, then $A'_u=\BZ_p[[y]]$, $I'_u=A'_u\cdot (y+p)$, and 
$\phi'_u:\BZ_p[[y]]\to \BZ_p[[y]]$ takes $y$ to $u^{p-1}y(y+p)^{p-1}$. Note that $u^{p-1}$ can be any element of $1+p\BZ_p$.

\subsubsection{The prism $(A'_u,\phi'_u,I'_u)$ in the case $u=1$}    \label{sss:A'_1,phi'_1,I'_1}
If $u=1$ then the prism $(A'_u,\phi'_u,I'_u)$ from \S\ref{sss:A'_u,phi'_u,I'_u} is isomorphic to the prism $(A',\phi',I')$ from \S\ref{ss:var q-de Rham}. 
This follows from \S\ref{sss:case u=1}.

\section{Ring scheme homomorphisms $W_S\to (W_n)_S$}  \label{s:ring morphisms W to W_n}
\subsection{Formulation of the result} 
In the proof of Proposition~\ref{p:Aut tilde W} we use the following
\begin{prop}  \label{p:ring morphisms W to W_n}
Let $S$ be a $p$-nilpotent scheme and $n\in\BN$. Let $\pi_n:W_S\to (W_n)_S$ be the canonical projection. Then any ring scheme homomorphism
$f:W_S\to (W_n)_S$ has the form $\pi_n\circ F^m$ for some $m\in H^0(S, \BZ_+)$. Here $\BZ_+$ is the set of non-negative integers and $H^0(S, \BZ_+)$ is the set of locally constant functions $S\to\BZ_+$.
\end{prop}

Let us note that if $S$ is the non-p-nilpotent scheme $\Spec\BZ [p^{-1}]$ and $n>1$ then not all ring scheme homomorphisms $W_S\to (W_n)_S$ have the form 
$\pi_n\circ F^m$.

\subsection{Proof of Proposition~\ref{p:ring morphisms W to W_n}}
We keep the $p$-nilpotence assumption for $S$.

\begin{lem}   \label{l:Hom(G_m,W_n^times)}
Every group scheme homomorphism $\varphi :(\BG_m)_S\to (W_n^\times)_S$ has the form $\varphi (\lambda )=[\lambda^N]$, $N\in H^0(S,\BZ)$.
\end{lem}

\begin{proof}   
If $n=1$ the statement is well known. It remains to show that if $\varphi$ lands into $\Ker ((W_n^\times)_S\epi (W_1^\times)_S)$ then $\varphi =1$. We can assume that $S$ is affine, $S=\Spec A$. Write
\[
\varphi (\lambda )=1+\sum_{i>0}V^i [g_i(\lambda )]
\]
and suppose that $g_i=0$ for $i<j$. Then the Laurent polynomial $g_j\in A[\lambda,\lambda^{-1}]$ satisfies the identity
\[
g_j(\lambda\mu )=g_j(\lambda )+g_j(\mu )+p^jg_j(\lambda )g_j(\mu ).
\]
Using $p$-nilpotence of $A$, one deduces from this that $g_j=0$.
\end{proof}

\begin{lem}   \label{l:ring morphisms W to W_n}
Let $f:W_S\to (W_n)_S$ be a ring scheme homomorphism. Then there exists $m\in H^0(S, \BZ_+)$ such that $f([\lambda ])=[\lambda^{p^m}]$.
\end{lem}

\begin{proof}   
By Lemma~\ref{l:Hom(G_m,W_n^times)}, $f([\lambda ])=[\lambda^N]$ for some $N\in H^0(S,\BZ)$. We have to show that if $N$ is constant then $N$ is a power of $p$. Without loss of generality, we can assume that $n=1$ and $S$ is an $\BF_p$-scheme.  Then $(W_1)_S=(\BG_a)_S$ is killed by $p$, and $
W_S/pW_S=(W_1)_S=(\BG_a)_S$. So $f$ is essentially a ring homomorphism $(\BG_a)_S\to (\BG_a)_S$. Therefore $N$ is a power of $p$. 
\end{proof}

Set $\BA^\infty :=\BA^1\times\BA^1\times\ldots$.

\begin{lem}   \label{l:faithful flatness}
The morphism 
$$\BA^\infty_S\to W_S, \quad (\lambda_0,\lambda_1,\ldots )\mapsto\sum_{i=0}^\infty p^i[\lambda_i]$$
is faithfully flat.
\end{lem}

\begin{proof}   
We can assume that $S$ is an $\BF_p$-scheme. Then $p^i[\lambda_i]=V^iF^i[\lambda_i]=V^i[\lambda_i^{p^i}]$, so it is clear that our morphism is faithfully flat.
\end{proof}

\subsubsection{Proof of Proposition~\ref{p:ring morphisms W to W_n}}
Let us show that $f=\pi_n\circ F^m$, where $m$ is as in Lemma~\ref{l:ring morphisms W to W_n}. Let $\tilde f=\pi_n\circ F^m$, then
\[
f(\sum_{i=0}^\infty p^i[\lambda_i])=\tilde f(\sum_{i=0}^\infty p^i[\lambda_i]).
\]
By Lemma~\ref{l:faithful flatness}, this implies that $\tilde f=f$.  \qed

\bibliographystyle{alpha}

\end{document}